\documentclass{article}
\usepackage{zformat}
\numberwithin{equation}{section}
\usepackage{authblk}

\title{Geometry of the fixed points loci and
discretization of Springer fibers in classical types}
\usepackage{microtype}
\begin{document}
\author{Do Kien Hoang }
\date{}
\maketitle

\begin{abstract}

     Consider a simple algebraic group $G$ of classical type and its Lie algebra $\fg$. Let $(e,h,f) \subset \fg$ be an $\fsl_2$-triple and $Q_e= C_G(e,h,f)$. The torus $T_e$ that comes from the $\fsl_2$-triple acts on the Springer fiber $\spr$. Let $\fix$ denote the fixed point loci of $\spr$ under this torus action. Our main geometric result is that when the partition of $e$ has up to $4$ rows, the derived category $D^b(\fix)$ admits a complete exceptional collection that is compatible with the $Q_e$-action. The objects in this collection give us a finite set $Y_e$ that is naturally equipped with a $Q_e$-centrally extended structure. We prove that the set $Y_e$ constructed in this way coincides with a finite set that has appeared in various contexts in representation theory. For example, a direct summand $J_c$ of the asymptotic Hecke algebra is isomorphic to $K_0(Sh^{Q_e}(Y_e\times Y_e)$ (\cite[]{Cell4}, \cite[]{bezrukavnikov2001tensor}). The left cells in the two-sided cell $c$ corresponding to the orbit of $e$ are in bijection with the $Q_e$-orbits in $Y_e$ (\cite[]{bezrukavnikov2020dimensions}). Our main numerical result is an algorithm to compute the multiplicities of the $Q_e$-centrally extended orbits that appear in $Y_e$.
 
\end{abstract}
\tableofcontents
\section{Introduction}
\subsection{Representations of Lie algebras in positive characteristics and Grothendieck groups of the Springer fibers}
Consider a connected reductive algebraic group $G$ and its Lie algebra $\mathfrak{g}$. Let $\mathfrak{h}$ be a Cartan subalgebra of $\mathfrak{g}$. The Weyl group $W$ acts on its dual $\mathfrak{h}^*$ by the $\rho$-shifted action $w\cdot \lambda= w(\lambda+ \rho)- \rho$. Let $\mathbb{F}$ be an algebraically closed field of characteristic $p\gg 0$. We write $G_\mathbb{F}$ and $\mathfrak{g}_\mathbb{F}$ for the $\mathbb{F}$-forms of $G$ and $\mathfrak{g}$.

According to \cite{center}, the center of the enveloping algebra $\mathcal{U}_\mathbb{F}:= \cU(\mathfrak{g}_\mathbb{F})$ is generated by the Harish-Chandra center $\mathcal{U}^{G}_\mathbb{F}= \mathbb{F}[\mathfrak{h}]^{W}$ and the $p$-center $S(\mathfrak{g}_{\mathbb{F}}^{(1)})$, where $^{(1)}$ denotes the Frobenius twist. For each pair $\lambda\in \mathfrak{h}^*/W$ and $\chi\in {\mathfrak{g}^{*}_{\mathbb{F}}}^{(1)}$, let $\mathcal{U}^{\chi}_{\mathbb{F},\lambda}$ denote the quotient of $\mathcal{U}_\mathbb{F}$ by the ideal generated by the maximal ideal corresponding to $(\lambda, \chi)$ in the center. Each irreducible $\cU_\bF$-module factors through exactly one central reduction $\cU^\chi_{\bF,\lambda}$. Thus, to understand the irreducible modules of $\mathcal{U}_\mathbb{F}$, it is enough to study the irreducible modules of $\mathcal{U}^{\chi}_{\mathbb{F},\lambda}$.

Using standard techniques from \cite{KW}, we can first reduce the consideration of the general case to the case where $\chi$ is nilpotent. Thanks to \cite{bezrukavnikov2006localization}, we can further assume that $\lambda$ is regular. For fixed $\lambda$, the algebra $\mathcal{U}^\chi_{\mathbb{F},\lambda}$ depends only on the $G_\mathbb{F}$-orbit of $\chi$. Since $p$ is sufficiently large, the nilpotent orbits in $\fg$ and $\fg_{\bF}^{(1)}$ are in a natural bijection. Let $e$ be a nilpotent element of $\mathfrak{g}$ lying in the orbit corresponding to $\chi$. Let $G_\chi$ be the centralizer of $\chi$ in $G_\mathbb{F}$, and let $A_\chi$ be the component group of $G_\chi$. Let $\mathcal{D}^\lambda$ be the sheaf of $\lambda$-twisted differential operators on the flag variety $\mathcal{B}$ of $G_\mathbb{F}$. Let $\cU_\lambda$ be the central reduction of $\mathcal{U}_\mathbb{F}$ by the Harish-Chandra character $\lambda$.
From \cite[Theorem 3.2]{bezrukavnikov2006localization}, known as the localization theorem in positive characteristics, we first have a derived equivalence of coherent $\mathcal{D}^\lambda$-modules and finitely generated $\cU_\lambda$-modules:
\begin{equation}\label{localization char p}
    D^b(\cU_\lambda- \text{mod}) \xrightarrow{\sim} D^b(Coh(\mathcal{D}^\lambda)).
\end{equation}
This equivalence is given by $V\mapsto V\otimes^{L}_{\cU^\lambda} \cD^\lambda$ and the quasi-inverse functor is taking global sections of $\cD^\lambda$-modules. The next two paragraphs follow from \cite[Sections 2, 5]{bezrukavnikov2006localization}, an exposition can be found in \cite[Lectures 5,6]{Loseulecseries}.

The sheaf of algebras $\mathcal{D}^\lambda$ is an Azumaya algebra over $T^*\mathcal{B}^{(1)}$. This sheaf of algebras $\mathcal{D}^\lambda$ splits on a neighborhood of $\mathcal{B}_e^{(1)}$ in the following sense. Let $\hat{\cO}_{\cN^{(1)}, e}$ be the complete local ring of $e\in \cN^{(1)}$, then we have the embedding $\Spec \hat{\cO}_{\cN^{(1)}, e}\hookrightarrow \cN^{(1)}$. Let $\hat{\spr}^{(1)}$ denote the fiber product $T^*\cB^{(1)}\times_{\cN^{(1)}} \Spec \hat{\cO}_{\cN^{(1)}, e}$. The restriction of the Azumaya algebra $\mathcal{D}^\lambda$ to $\hat{\spr}^{(1)}$ splits, let $\mathcal{V}^e$ be a splitting bundle for $\mathcal{D}^\lambda|_{\hat{\spr}^{(1)}}$. We then have the equivalence $D^b(Coh_{\hat{\spr}} T^*\cB^{(1)}) \cong D^b(Coh \, \cD^\lambda |_{\hat{\spr}^{(1)}})$ by tensoring with $\cV^e$. Write $\cU_{\lambda, \bF}^{\hat{\chi}}$ for the completion of $\cU_{\lambda, \bF}$ at the central element $\chi$. From the equivalence (\ref{localization char p}), we deduce $D^b(Coh \, \cD^\lambda |_{\hat{\spr}^{(1)}})\cong D^b(\cU_{\lambda, \bF}^{\hat{\chi}}- \text{ mod})$. Thus, we have obtained the equivalence 
\begin{equation} \label{splitting bundle}
  D^b(Coh_{\hat{\spr}} T^*\cB^{(1)})\cong D^b(\cU_{\lambda, \bF}^{\hat{\chi}}- \text{mod}), \quad \mathcal{F}\mapsto R\Gamma(\mathcal{V}^e\otimes \mathcal{F}).  
\end{equation}
Let $D^b(Coh_{\spr} T^*\cB^{(1)})$ and $D^b(\cU_{\lambda, \bF}- \text{ mod}^{\chi})$ be the derived categories of coherent sheaves of $T^*\mathcal{B}^{(1)}$ set-theoretically supported on the Springer fiber $\mathcal{B}_e$ and of finitely generated $\mathcal{U}_{\lambda, \mathbb{F}}$-modules with generalized character $\chi$. We have two full embeddings $D^b(Coh_{\spr} T^*\cB^{(1)})\rightarrow$ $D^b(Coh_{\hat{\spr}} T^*\cB^{(1)})$ and $D^b(\cU_{\lambda, \bF}- \text{mod}^{\chi})\rightarrow$ $D^b(\cU_{\lambda, \bF}^{\hat{\chi}}- \text{mod})$. Furthermore, the equivalence (\ref{splitting bundle}) sends $D^b(Coh_{\spr} T^*\cB^{(1)})$ to $D^b(\cU_{\lambda, \bF}- \text{mod}^{\chi})$, hence inducing an equivalence between these two categories. Taking the Grothendieck groups, we obtain $K_0(\mathcal{U}_{\lambda,\mathbb{F}}- \text{mod}^\chi)= K_0(\mathcal{B}_e)$. As $K_0(\mathcal{U}_{\lambda,\mathbb{F}}- \text{mod}^\chi)$ and $K_0(\mathcal{U}^\chi_{\lambda,\mathbb{F}}- \text{mod})$ are the same, we have
\begin{equation} \label{1.3}
    K_0(\mathcal{U}^\chi_{\lambda,\mathbb{F}}- \text{mod})= K_0(\mathcal{B}_e).
\end{equation}
Let $Q_\mathbb{F}\subset G_\mathbb{F}$ be the reductive part of the centralizer of $e\in \mathfrak{g}$. Let $T_0$ be a maximal torus of $Q_\mathbb{F}$, and let $\underline{Q}_\mathbb{F}$ be the centralizer of $T_0$ in $Q_\mathbb{F}$. According to \cite[Section 3.4]{bezrukavnikov2020dimensions}, we can equip $\mathcal{V}$ with a $\underline{Q}_\mathbb{F}$-equivariant structure. The equivalence (\ref{splitting bundle}) is compatible with the natural actions of $\underline{Q}_\mathbb{F}$ on both sides. Therefore, similarly to how we obtain (\ref{1.3}), we have an isomorphism between equivariant K-groups.
\begin{equation} \label{1.4}
    K_0(\mathcal{U}^\chi_{\lambda, \mathbb{F}}-mod^{\underline{Q}_\mathbb{F}})\cong K_0^{\underline{Q}_\mathbb{F}}(\mathcal{B}_e)
\end{equation}

\subsection{The finite set $Y_e$ and its centrally extended structure} \label{subsect 1.2}
Starting with this section, we assume that the Lie algebra $\fg$ is of classical type. Write $Y_e$ for the finite set of isomorphism classes of irreducible $\mathcal{U}^\chi_{\mathbb{F},\lambda}$-modules. This section explains the close relationship between $Y_e$ and $\spr$.

First, consider the case where $e$ is a distinguished nilpotent element. Let $A_\chi$ be the component group of $Q_\bF$, then $Q_\mathbb{F}$ and $\underline{Q}_\mathbb{F}$ coincide with $A_\chi$. The group $A_\chi$ permutes the elements of $Y_e$. For $y\in Y_e$, let $V_y$ be the corresponding irreducible $\mathcal{U}^\chi_{\mathbb{F},\lambda}$-module. Write $A_y$ for the stabilizer of $y$ in $Y_e$.

The group $A_\chi$ acts on $\mathcal{U}^\chi_{\mathbb{F},\lambda}$ by algebra automorphisms. Consequently, each element $a\in A_y$ gives rise to an isomorphism $\varphi_a$ of $\mathcal{U}^\chi_{\mathbb{F},\lambda}$-modules: $V_y\rightarrow V_y$ defined up to a scalar by the Schur lemma. For an arbitrary choice of the set $\{\varphi_a\}_{a\in A}$, the ratios of the two isomorphisms $\varphi_{a'}\circ\varphi_{a}$ and $\varphi_{a'a}$ give us a $2$-cocycle of the group $A_y$. The corresponding Schur multiplier $\psi_y\in H^2(A_y, \bC^\times)$ of this $2$-cocycle depends only on $V_y$. Thus, we obtain an $A_\chi$-equivariant assignment $y\mapsto \psi_y \in H^2(A_y, \mathbb{C}^\times)$. This assignment (together with the action of $A_\chi$ on $Y_e$) is called an $A_\chi$-\textit{centrally extended} structure of the finite set $Y_e$. 

Consider a subgroup $A\subset A_\chi$. The set $Y_e$ has a natural $A$-centrally extended structure by restricting its $A_\chi$-centrally extended structure. With respect to this centrally extended structure, we have the following category. Let $\Sh^A(Y_e)$ denote the category of $A$-equivariant sheaves of projective representations on $Y_e$ where the fiber over a point $y\in Y_e$ has the corresponding class $\psi_y|_{A\cap A_y}$. This category $\Sh^A(Y_e)$ is equivalent to the category of semisimple objects in $\mathcal{U}^\chi_{\lambda, \mathbb{F}}-mod^A$. Thus, we have $K_0(Sh^A(Y_e))= K_0(\mathcal{U}^\chi_{\lambda, \mathbb{F}}-mod^A)$. Since the isomorphism (\ref{1.4}) holds for any subgroup of $\underline{Q}_\mathbb{F}= A_\chi$, we obtain $K_0(Sh^A(Y_e))= K_0(\mathcal{U}^\chi_{\lambda, \mathbb{F}}-mod^A)= K^A_0(\mathcal{B}_e)$. 

For a general nilpotent element $e$ (not necessarily distinguished), we have a realization of the finite set $Y_e$ coming from the noncommutative Springer resolution as follows (see, e.g., \cite[Section 2.5]{cocenter} for more details). In \cite[Section 1.5, 1.6]{bezrukavnikov2012representations}, the authors construct an algebra $\cA^0$ by taking the endomorphism algebra of a tilting generator $\cE$ on the Springer resolution $\tilde{\cN}$. This noncommutative algebra is equipped with a derived equivalence 
$$D^b(\tilde{\cN})\cong D^b(\cA^0-\text{mod}).$$

For each nilpotent element $e\in \cN$, we have the corresponding maximal ideal $\fm_e\subset \cO(\cN)$. Since $\cO(\cN)$ lies in the center of $\cA^0$, we can consider $\cA^{0}_e:=$ $\cA^0/\fm_e\cA^0$, the fiber of $\cA^0$ over $e\in \cN$. From \cite[Theorem 5.1.1]{bezrukavnikov2012representations}, we have a bijection between the set of irreducible $A_e^{0}$-modules and the set of irreducible $\mathcal{U}^{\chi}_{\lambda, \mathbb{F}}$-modules. Thus, we can consider the set $Y_e$ as the set of classes of irreducible modules of $\cA_e^{0}$. From \cite[Section 6.3.2]{bezrukavnikov2012representations}, we have a natural $G$-equivariant structure on $\cA^0$. Hence, the centralizer of $e$ in $G$ acts on $\cA_e^{0}$ by algebra automorphisms. Let $Q_e$ be the reductive part of the centralizer of $e$. Then the set $Y_e$ admits a $Q_e$-centrally extended structure (for the definition of a centrally extended set for a general algebraic group, see \cite[Section 4.2]{bezrukavnikov2001tensor} or \Cref{subsect 2.1}). For the same reason as in the case where $e$ is distinguished, we have natural isomorphisms $$K_0(\Sh^{Q_e}(Y_e))\cong K_0(\cA_0^{e}-\text{mod}^{Q_e})\cong K^{Q_e}_0(\mathcal{B}_e).$$

These equalities of the K-groups imply a close relationship between $\spr$ and $Y_e$. The theme of this paper is that we study the geometry of $\spr$ to understand $Y_e$. Our objective is to obtain a description of the $Q_e$-centrally extended structure of $Y_e$. In particular, we want to determine which centrally extended orbits appear in $Y_e$, and then find their multiplicities. A motivation for this task is to understand how simple $\cU^\chi_{\bF,\lambda}$-modules are related to simple equivariant $\cU^\chi_{\bF,\lambda}$-modules, the latter are parametrized in \cite[Section 8]{bezrukavnikov2020dimensions}. In addition, information about $Y_e$ has applications in studying the asymptotic Hecke algebra $J$ and cells in affine Weyl groups as follows. Let $c$ be the two-sided cell corresponding to the orbit of $e$. The direct summand $J_c$ of $J$ is isomorphic to $K_0(Sh^{Q_e}(Y_e\times Y_e)$ (\cite[Section 10]{Cell4}, \cite[Theorem 4]{bezrukavnikov2001tensor}). The left cells in $c$ are in bijection with the $Q_e$-orbits in $Y_e$ (\cite[Proposition 8.25]{bezrukavnikov2020dimensions}).

\subsection{Content of the paper}
In Section 2, we compute the groups of Schur multipliers of subgroups of $Q_e$. Using this computation, we show that the $Q_e$-centrally extended structure of $Y_e$ is controlled by a certain subgroup $A_e$ of $Q_e$. We choose this subgroup $A_e$ so that it projects surjectively to $A_\chi$ and $A_e$ is elementary $2$-abelian (when $e$ is distinguished, we have $A_e\cong Q_e= A_\chi$). The problem is then reduced to understanding the centrally extended structure of $Y_e$ with respect to the action of an elementary $2$-group as in the distinguished case. In this paper, we mainly work with type C (i.e., $G= Sp(2n)$). We often assume that the partition of $e$ consists only of even parts. The case of a general nilpotent orbit in $Sp(2n)$ and the case of types B, D will be explained in Section 9.

In Section 3, we study certain numerical invariants of $\spr$ that help to determine $Y_e$. For a finite abelian group $A$ and an $A$-variety $X$, we define two sets of numbers $\{F^{A}_a(X)\}_{a\in A}$ and $\{S_i^{A}(X)\}_{i\geqslant 0}$ as follows. Let $R(A)$ be the representation ring of $A$. Then $\bC\otimes_\bZ K_0^{A}(X)$ is a $\bC\otimes_\bZ R(A)$-module. The set $\{F^{A}_a(X)\}_{a\in A}$ records the multiplicities of the irreducible summands of $\bC\otimes_\bZ K_0^{A}(X)$. Next, let $F_X^{A}: K_0^{A}(X)\rightarrow K_0(X)^{A}$ be the forgetful map from the equivariant K-group to the $A$-invariant part of the K-group of $X$. The set $\{S_i^{A}(X)\}_{i\geqslant 0}$ consists of the multiplicities of the cyclic summand $\cyclic{2^i}$ in the decomposition of coker$F_X^{A}$ into the direct sum of cyclic groups. In Section 3.1, finite models $Y$ of $X$ are defined as certain centrally extended sets so that $X$ and $Y$ have the same numerical invariants $\{F_a^{A}\}$ and $S_i^{A}$ (\Cref{finite model}). The main result of Section 3 is a recursive algorithm to compute the numbers $\{F^{A}_a(X)\}_{a\in A}$ when $X$ is the Springer fiber $\spr$ and $A\subset A_e$. 

For a topological space $X$, let $\chi(X)$ denote its Euler characteristic.
\begin{thm*}[\Cref{recursive}] We have the following formula:
    $$F_a^{A}(\spr)= \frac{1}{|A|}(\sum_{a'\in A}\chi(\spr^a\cap \spr^{a'})).
    $$
Here the connected components of the variety $\spr^{a,a'}:= \spr^a\cap \spr^{a'}$ are all isomorphic. In particular, they are isomorphic to the products of certain smaller Springer fibers.
\end{thm*}

For any subgroup $A\subset A_e\subset Q_e$, we have $K_0(Sh^A(Y_e)) =$ $K^A_0(\mathcal{B}_e)$. Therefore, the numbers $\{F^{A}_a(\spr)\}_{a\in A}$ calculated in Section 3 give (partial) information on the $A_e$-centrally extended structure of $Y_e$. In Section 4, we first recall the explicit Springer correspondences in \cite[chapter 13]{carter1993finite}. We use these correspondences to compute the characters of $A_e$ that appear in the $A_e$-module $\bC\otimes_\bZ K_0(\spr)$. This character computation gives some restrictions on the types of $A_e$-orbits in $Y_e$. In particular, we can determine the structure of $Y_e$ when the partition of $e$ has $2$ or $3$ rows. When the associated partition of $e$ has more rows, finer restrictions are discussed in Section 5.

Section 5 starts with a brief introduction and some facts about the theory of Kazhdan-Lusztig cells and the asymptotic affine Hecke algebra $J$. This algebra has a decomposition $J= \oplus J_\chi$ where $\chi$ runs over the nilpotent orbits of $\fg$. Each direct summand $J_\chi$ is shown to be isomorphic to $K_0(\Sh^{Q_e}(Y_e\times Y_e))$ by Bezrukavnikov and Ostrik in \cite{bezrukavnikov2001tensor}. Simple $J_\chi$-modules are classified in \cite{Cell4}. When $Q_e$ is finite abelian (equivalently, $e$ is distinguished), \cite[Section 10]{etingof2009fusion} provides another classification of simple modules in terms of the same parameters. We show that these two parameterizations match numerically (in terms of dimensions of irreducibles). We then obtain some restrictions on the types of $A_e$-orbit that appear in $Y_e$. Applying these restrictions to the case where the partition of $e$ has $4$ rows, we deduce that the structure of $Y_e$ is completely determined by the numerical invariants $\{F_a^{A}\}$ and $\{S_i^A\}$. As we have mentioned above, in the case of $3$-row partition, knowing the numbers $\{F_a^{A}\}$ is enough to understand $Y_e$. In the case of $4$-row partitions, knowledge of $\{S_i^A\}$ is crucial.

Consider an $\fsl_2$ triple $(e,h,f)$ in $\fg$. Let $T_e\subset G$ be the torus with the Lie algebra $\bC h$. The adjoint action of $T_e$ scales $e$, so $T_e$ naturally acts on $\spr$. Let $\fix:= \spr^{T_e} \subset \spr$ be the corresponding fixed point subvariety. The actions of $Q_e$ and $T_e$ on $\spr$ commute, so $Q_e$ and $A_e$ naturally act on $\fix$. In Section 6, we show that the two varieties $\fix$ and $\spr$ share the same numbers $S_i^A$ and $F_a^{A}$ for $A\subset A_e$. This result is an application of an "integral version" of the localization theorem in equivariant K-theory from \cite{intloc}. The equalities of the numerical invariants suggest that $\fix$ and $\spr$ share the same finite model $Y_e$. A result of similar flavor has been proved in \cite[Theorem A]{bezrukavnikov2023geometric} and \cite[Theorem 1.5.2]{Propp}. In particular, they show that $J_\chi\cong $ $K_0^{Q_e}(\fix\times \fix)$, (so $K_0(\Sh^{Q_e}(Y_e\times Y_e))\cong $  $K_0^{Q_e}(\fix\times \fix)$). 

While $\spr$ is connected, $\fix$ has many connected components. The group $A_e$ permutes the components of $\fix$. The $A_e$-orbits of components of $\fix$ are parameterized by tuples of weights of $T_e$. For each tuple $\alpha$, we write $(\fix)_\alpha$ for the corresponding variety. Then we have the following conjecture (for the definition of finite models, see \Cref{finite model}). 
\begin{conj*}[\Cref{union}]
    Each variety $(\fix)_\alpha$ admits a unique $A_e$-finite model $Y_\alpha$. Furthermore, $Y_e$ is the disjoint union of all $Y_\alpha$.
\end{conj*}
 
Section 7 develops the tools to obtain the set $Y_\alpha$ from the bounded derived category $D^{b}((\fix)_\alpha)$. Consider an $A$-variety $X$. Assume that $D^b(X)$ admits a full exceptional collection that is compatible with the action of $A$ in a suitable sense (see \Cref{compatible}). Then the objects of such a collection give us a finite set $Y$, equipped with a natural $A$-centrally extended structure. We call $Y$ a categorical finite model of $X$. The following is one of our main theorems.
\begin{thm*}[\Cref{main}]
    Assume that the partition $\lambda$ of $e$ has up to $4$ parts. There exists a full exceptional collection in $D^b((\fix)_\alpha)$ that is compatible with the action of $A_e$ (and $Q_e$). The corresponding categorical finite model $Y_\alpha$ is the unique finite model of $(\fix)_\alpha$.
\end{thm*}

This theorem and the results of Section 5 show that \Cref{union} is true when $\lambda$ has up to four rows. Another application is that we can construct a basis of $K_0^{Q_e}(\fix\times \fix)$ so that $J_\chi\cong $ $K_0^{Q_e}(\fix\times \fix)$ as two based algebras over $\bZ$. Therefore, in this case, we obtain a slightly stronger version of \cite[Theorem A]{bezrukavnikov2023geometric}.

\Cref{main} follows from our geometric description of $\fix$ in Section 8. In \Cref{2 row}, we show that when $\lambda$ has two parts, the connected components of $\fix$ are towers of $\bP^1$-bundles and the varieties $(\fix)_\alpha$ have either one or two connected components. When $\lambda$ has three or four parts, \Cref{geometric 3 rows} and \Cref{geometric main 4 rows} show that the varieties $(\fix)_\alpha$ belong to two collections of varieties, $\cC_3$ and $\cC_4$. These two collections (see Definitions \ref{collection C_3}, \ref{collection C_4}) are generated from some base varieties by constructing towers of projective bundles and blowing up along smooth subvarieties. The finite set $Y_e$ is then assembled from the finite models of these base varieties. The first consequence is that we can completely determine the centrally extended orbits that appear in $Y_e$ when $\lambda$ has four parts. The second consequence is that we obtain an algorithm to compute the multiplicities of these orbits from the numerical invariants $\{F_{a}^{A}\}$. Compared to the results in Section 5 (where we need to make use of $\{S_i^{A}\}$), this fact is somewhat surprising. It shows evidence for the conjecture about the linear independence of different classes of left cell modules in the corresponding Grothendieck groups (\Cref{combi conjecture}). 

Section 9 concerns the structure of $Y_e$ in the general case and related conjectures. We first show that for $e\in \fsp_{2n}$, the centrally extended structure of $Y_e$ can be recovered from the centrally extended structure of $Y_{e'}$ where $e'$ is an even nilpotent element of $\fsp_{2n'}$ for some $n'\leqslant n$. We then discuss the case where $\fg$ is of classical type B or D. Next, we state the conjecture that the structure of $Y_e$ can be understood by studying some base cases. Lastly, we construct a set $S_e$ of centrally extended orbits that is likely to appear in $Y_e$. This set of orbits is conjectured to be a complete list when $\lambda$ has five parts. We plan to return to this case in the future. When $\lambda$ has more parts, there are orbits outside of $S_e$ that appear in $Y_e$, this phenomenon is under investigation. The primary technical difficulty is that the author does not have good control over the set of base varieties (and whether the corresponding derived categories admit full exceptional collections) when $\lambda$ has six parts or more.  

Appendix A consists of the geometric descriptions of the base varieties of the collection $\cC_4$.
\subsection{Acknowledgements}
I am grateful to Ivan Losev for introducing this problem and for many stimulating discussions and countless valuable pieces of advice. I would like to thank Roman Bezrukavnikov, Pablo Boixeda Alvarez,
Stefan Dawydiak, Gurbir Dhillon, Ivan Karpov, George Lusztig, Victor Ostrik, Oron Propp, Eric Sommers, and Yaochen Wu for insightful conversations and comments related to the project. I also thank Hang Le, Yutong Nie, and Elad Zelingher for their help with coding. This work is partially supported by the NSF
under grant DMS-2001139.

\section{The Schur multipliers of the subgroups of $Q_e$}
From this section to the end of Section 8, we will consider $G= Sp(2n)$ and consider a nilpotent element $e\in \fg$ with the associated partition $2x_1+...+2x_k= 2n$. Let $Q_e$ be the reductive part of the centralizer of $e$ in $G$. Let $Q_e^{o}$ be the identity component of $Q_e$. For an algebraic group $Q$, we write $M(Q)$ for the group of (algebraic) Schur multipliers of $Q$ (see Section 2.2). As explained in Section 1.2, the $Q_e$-centrally extended structure of $Y_e$ is recorded by a dataset $\{(Q_y, \psi_y\in M(Q_y))\}_{y\in Y_e}$. Hence, to study $Y_e$, the first task is to understand the group $M(Q_y)$. This is done in Section 2.2 and Section 2.3. In particular, we compute the groups $M(Q)$ for $Q_e^{o}\subset Q\subset Q$.

Another goal of Section 2 is to introduce a finite abelian subgroup $A_e\subset Q_e$ with the following properties.
\begin{enumerate}
    \item $A_e\cong (\cyclic{2})^{\oplus k}$, and $A_e$ surjects to the component group $A_\chi:= Q_e/Q_e^{o}$.
    \item For $y\in Y_e$, let $A_y$ be $Q_y\cap A_e$. Then the natural pullback $f_y^{*}: M(Q_y)\rightarrow M(A_y)$ is injective.
\end{enumerate}
We explain the meaning of these conditions. First, $Y_e$ is naturally an $A_e$-centrally extended set. The surjectivity of $A_e\rightarrow A_\chi$ means that every $Q_e$-orbit in $Y_e$ is a single $A_e$-orbit. Consequently, $A_y$ is the stabilizer of $y$ in $A_e$, and $Q_y$ is generated by $Q_e^{o}$ and $A_y$. Next, the injectivity of $f_y^{*}: M(Q_y)\rightarrow M(A_y)$ means that we can recover $\psi_y$ from the knowledge of $f_y^*(\psi_y)$. In other words, this group $A_e$ satisfies the following proposition.
\begin{pro} \label{Ae determines Qe}
    The $Q_e$-centrally extended structure of $Y_e$ is fully recovered from its $A_e$-centrally extended structure.
\end{pro}

In Section 2.1, we first describe $Q_e$ and a particular choice of $A_e$. The property $A_e\twoheadrightarrow A_\chi$ will follow from our choice. The maps $f_y^{*}: M(Q_y)\rightarrow M(A_y)$ will be described explicitly in our computations of $M(Q_y)$. 

\subsection{Setting for $G= \Sp_{2n}$} \label{subsect 2.1}
Consider the case  $G= \Sp_{2n}$. Let $e\in \fg$ be a nilpotent element with the associated partition $ \lambda= (2x_1,2x_2,...,2x_k)$ where $x_1\geqslant x_2...\geqslant x_k$. The transpose partition of $\lambda^t$ has the form $(y_1,y_{-1}, y_2, y_{-2},...,y_M, y_{-M})$ where $y_i = y_{-i}$ and $y_1\geqslant y_2...\geqslant y_M$. For convenience, we sometimes write $N_\lambda$ for $e$. Let $\mathcal{B}_e$, $\mathcal{B}_\lambda$ or $\mathcal{B}_{(2x_1,...,2x_k)}$ denote the Springer fiber over $e$.

Now we give a concrete realization of $\fg$ and $e$ in terms of matrices. For $1\leqslant i\leqslant k$, consider $k$ symplectic vector spaces $V^i$ of dimensions $2x_i$. Let $\langle, \rangle^i$ be the symplectic form of $V^i$. We choose a basis $\{v^i_j\}_{1\leqslant j\leqslant 2x_i}$ of $V^i$ so that
\begin{itemize}
    \item $\langle v^i_j,v^i_l \rangle^{i}= 0$ if $j+l\neq 2x_i+1$,
    \item $\langle v^i_j,v^i_l \rangle^i= 1$ if $j+l= 2x_i+1$ and $j<l$.
\end{itemize}

Let $V_\lambda$ be the direct sum $\oplus_{i=1}^k V^i$. We identify $\fg= \mathfrak{sp}_{2n}$ with $\mathfrak{sp}(V_\lambda)$. The action of $e$ on $V_\lambda$ is given by $ev^i_{j}= ev^i_{j-1}$. In other words, the matrix of $e= N_\lambda$ is in Jordan form with all eigenvalues $0$, and the number $2x_i$ records the size of the $i$-th block.
\begin{exa} \label{basis}
Consider the element $e$ with the associated partition $(2,2,4,6)$ in $\fsp_{14}$. We give an illustration by the "centered" Young diagram of the basis $\{v_j^{i}\}$. $$\ytableausetup{centertableaux, boxsize= 3em}
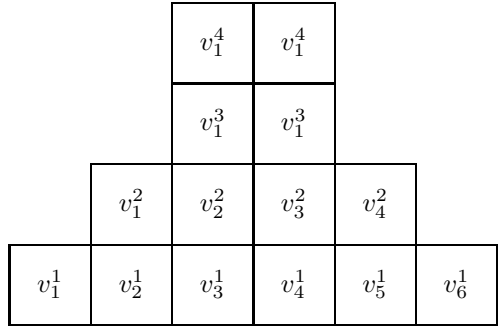
\begin{ytableau}
\none & \none & v^{4}_{1} & v^{4}_{1} &\none &\none \\
\none & \none & v^{3}_{1} & v^{3}_{1} &\none &\none \\
\none &v^{2}_{1} &v^{2}_{2} &v^{2}_{3} &v^{2}_{4} &\none \\
v^{1}_{1} & v^{1}_{2} &v^{1}_{3} &v^{1}_{4} &v^{1}_{5} &v^{1}_{6} \\
\end{ytableau}$$\\ 
In this diagram, the $i$-th row from the bottom represents a basis of $V^i$. And the action of $e$ sends a basis vector in a box to the one on its left.
\end{exa}
From \cite[Chapter 5]{collingwood1993nilpotent}, the reductive part $Q_e$ of the centralizer of $e$ in $G$ is isomorphic to $\prod_{j=1}^M O_{y_j- y_{j+1}}$ (some factors of this product may be trivial). We specify a realization of $Q_e$ in $\Sp(V_\lambda)$ with respect to the basis $\{v^i_j\}$ as follows. Let $\ell$ be the number of distinct parts of $\lambda$. Let $i_1<...< i_{\ell-1}< i_\ell$ be the indexes that satisfy $x_{i_m}> x_{i_m+ 1}$ for $1\leqslant m\leqslant \ell$, here $i_\ell= k= y_1$. As in \cite[Page 71, 72]{collingwood1993nilpotent}, the symplectic form on $V_\lambda$ induces a nondegenerate symmetric bilinear form on each vector space Span$(v_1^{i_m+ 1},..., v_1^{i_{m+1}})$. 

With the choice of our symplectic form on $V_\lambda$, the induced bilinear forms on these vector spaces are the standard dot products. The group $Q_e$ is then realized as $\prod_{m= 1}^{\ell} O_{i_{m}- i_{m-1}}$ with respect to these dot products. In \Cref{basis}, we have $i_1= 1, i_2= 2, i_3= 4$. Hence $Q_e\cong O_1\times O_1\times O_2= (\cyclic{2})^{\oplus 2}\times O_2$.

We choose a subgroup $A_e$ of $Q_e$ as follows. Consider the elements $z_i\in Q_e$ that act as $\Id$ on $V^i$ and $-\Id$ on $V^j$ for $j\neq i$. Let $A_e$ be the group generated by $z_i$ for $1\leqslant i\leqslant k$. In \Cref{basis}, we have $Q_e\cong O_2\times (\cyclic{2})^{\oplus 2}$, so $A_e\cong (\cyclic{2})^{\oplus 4}$. In terms of matrices, $A_e$ lives inside $Q_e$ as diagonal matrices with entries $1$ or $-1$. For convenience, we write $z_{i_1i_2...i_m}$ for the element $z_{i_1}z_{i_2}...z_{i_m}\in A_e$.

Write $A_\chi$ for the component group of $Q_e$, then $A_\chi \cong (\cyclic{2})^{\oplus \ell}$. We write $Q_e$ as a semi-product $(\prod_{m= 1}^{\ell} (\SO_{i_{m}- i_{m-1}})) \rtimes (\cyclic{2})^{\oplus \ell}$ $\cong Q_e^{o}\rtimes A_\chi$ . The second factor can be realized as the group generated by $\{ z_{i_m}\}_{1\leqslant m\leqslant \ell}$ living in $Q_e$. This realization gives us a lift of $A_\chi$ in $A_e\subset Q_e$. Consider the projection $A_e \hookrightarrow Q_e \rightarrow A_\chi$. With the chosen lift of $A_\chi$, this projection sends $z_j$ to $z_{i_m}$ where $m$ is the index that satisfies $i_m+1\leqslant j\leqslant i_{m+1}$. In the context of \Cref{basis}, the component group $A_\chi$ is realized as the subgroup generated by $z_{1}, z_3$ and $z_4$. The projection $A_e\rightarrow A_\chi$ is given by $z_1\mapsto z_1, z_2\mapsto z_1, z_3\mapsto  z_3$ and $z_4\mapsto z_4$. In summary, we have a chain of maps with the composition being an isomorphism:
\begin{equation} \label{first chain}
    A_\chi \hookrightarrow A_e\hookrightarrow Q_e\twoheadrightarrow A_\chi.
\end{equation}

For a point $y\in Y_e$, its stabilizer $Q_y$ must contain $Q_e^{o}$. So $Q_y$ has the form $(\prod_{j=1}^\ell \SO_{y_j- y_{j-1}})\rtimes (\cyclic{2})^{\oplus m}$ for some $m\leqslant \ell$. The next section develops techniques for computing the group of Schur multipliers of a semiproduct of algebraic groups.
\subsection{Schur multipliers of algebraic groups} \label{subsect 2.2}
Let $G$ be an algebraic group over $\mathbb{C}$. By an algebraic projective representation of $G$ we mean a finite-dimensional vector space $V$ and an algebraic group homomorphism $\rho: G\rightarrow PGL(V)$. Let $M(G)$ be the set of equivalence classes of algebraic projective representations of $G$. The equivalence relation is given as follows: $V_1$ and $V_2$ are equivalent if and only if $V_1\otimes V_2^*$ is a linear representation of $G$, here $V_2^*$ is the dual projective representation of $V_2$. The set $M(G)$ has a group structure induced from the tensor product of projective representations. When $G$ is a finite group, $M(G)$ is simply $H^2(G, \mathbb{C}^\times)$.  

Next, consider the case where $G$ is the semidirect product of two algebraic groups $G_1 \rtimes G_2$. From \cite[Lemma 2.3]{Schur}, the obstruction of lifting a projective representation $W$ of $G$ to an honest representation $W$ is recorded by a certain regular function $\alpha_W$ (not defined canonically). The function $\alpha_W: G\times G\rightarrow \mathbb{C}^\times$ has the property that $\alpha_W(g_1g_2, h_1h_2)= \alpha_W(g_2,h_2)\alpha_W(g_2, h_1)\alpha_W(g_1, h_1)$ for $g_i, h_i\in G_i$, $i= 1,2$. In particular, this function is determined by its restriction to $G_1\times G_1$, $G_2\times G_2$, and $G_2\times G_1$. On the other hand, given $\alpha_W$, we can reconstruct the class of projective representations of $W$. The group of Schur multiplier of a semidirect product of two finite groups is studied in \cite[Section 2.2]{Schur}. Here, we have similar results for the case of algebraic groups under certain conditions in \Cref{semi} and \Cref{direct product} below.

Let $G_1$ be an algebraic group, and let $G_2$ be a finite group acting on $G_1$. We write $\Hom(G_1, \mathbb{C}^\times)$ for the set of algebraic group homomorphisms from $G_1$ to $\bC^\times$. We then have a natural action of $G_2$ on the set $\Hom(G_1, \mathbb{C}^\times)$. Thus, it makes sense to talk about $\gH^1(G_2, \Hom(G_1, \mathbb{C}^\times))$, the first group cohomology of $G_2$ with coefficients in $\Hom(G_1, \mathbb{C}^\times)$. 
\begin{lem} \label{semi}
 Let $G=$ $G_1\rtimes G_2$ be the semidirect product with respect to the action of $G_2$ on $G_1$. We have a set injection
$$M(G)\hookrightarrow M(G_1)\times M(G_2)\times \gH^1(G_2, \Hom(G_1, \mathbb{C}^\times))$$
\end{lem}
\begin{proof}
We first define a set injection $M(G)\rightarrow M(G_1)\times M(G_2)\times \text{Map}(G_2\times G_1\rightarrow \mathbb{C}^\times)$ as follows. For each class of $M(G)$, we pick a representative $W$ and send it to $W|_{G_1}\times W|_{G_2}\times \beta_W$ where $W|_{G_i}$ are the classes of $W$ naturally viewed as a projective representation of $G_i$, and $\beta_W$ is obtained by restricting the function $\alpha_W $ on $G\times G= G_1\rtimes G_2\times G_1\rtimes G_2$ to $1\times G_2\times G_1\times 1$. From part (iii) of \cite[Lemma 2.2.4 ]{Schur}, we find that $\beta$ satisfies $\beta(g_2h_2,g_1)=$ $\beta(h_2,g_1)\beta(g_2, h_2.g_1)$ and $\beta(g_1, h_1)\beta(g_2.g_1,g_2.h_1)^{-1}=$ $\beta(g_2, g_1)\beta(g_2,g_1h_1)^{-1}\beta(g_2,h_1)$ for $g_1, h_1\in G_1$ and $g_2, h_2\in G_2$. We realize $\beta$ as a map from $G_2$ to $\Hom(G_1, \bC^\times)$ by $\beta(g_1): g_2\mapsto \beta(g_2,g_1)$. From \cite[Page 34]{Schur}, $\beta$ can be considered an element of $\gH^1(G_2, \Hom(G_1, \mathbb{C}^\times))$. Thus, the map we have defined has the image in $\Hom(G_1, \gH^1(G_2, \Hom(G_1, \mathbb{C}^\times)))$ $\times$ $M(G_1)$ $\times M(G_2)$. The injectivity of this map follows from part (ii) of Lemma 2.2.3 in \cite{Schur}.
\end{proof}
 \begin{rem} \label{direct product} We can think of $\Hom(G_1, \bC^\times)$ as an abelian group where the multiplication is given by $(f_1\times f_2)(x)= f_1(x)f_2(x)$ for $f_1, f_2\in \Hom(G_1, \bC^\times)$. When $G= G_1\times G_2$, the group cohomology $\gH^1(G_2, \Hom(G_1, \mathbb{C}^\times))$ is $\Hom(G_2, \Hom(G_1, \mathbb{C}^\times))$, the set of group homomorphisms between $G_2$ and $\Hom(G_1, \mathbb{C}^\times)$. The latter expression makes sense without the assumption that $G_2$ is finite. By the same argument as in the proof of Lemma 2.2 we have an injection  
$$M(G)\hookrightarrow M(G_1)\times M(G_2)\times \Hom(G_2, \Hom(G_1, \mathbb{C}^\times)).$$ 
 In this paper, we always choose $G_1$ so that $\Hom(G_1, \mathbb{C}^\times)$ is a finite abelian group or $\bZ$. Moreover, our $G_2$ is of classical type, so the set $\Hom(G_2, \Hom(G_1, \mathbb{C}^\times))$ has finite cardinality.
 \end{rem}
We have a straightforward numerical consequence of \Cref{semi} and $\Cref{direct product}$.
 \begin{cor} \label{upperbound}
When $G_2$ is finite or when $G_2$ acts trivially on $G_1$, we have
$$|M(G_1\rtimes G_2)|\leqslant |M(G_1)|\times |M(G_2)|\times |\gH^1(G_2, \Hom(G_1, \mathbb{C}^\times))|,$$
in which we assume that all the sets on the right side have finite cardinalities.
 \end{cor}
 
In addition to the lemma, we need results on the Schur multipliers of some basic groups. The following groups of Schur multipliers are known. We have $M(O_k)= M(\SO_h)= \mathbb{Z}/2\mathbb{Z}$ for $k\geqslant 2$ and $h\geqslant 3$, and $M(\SO_2)= M(\mathbb{C}^\times)= 0$.

The result $M(\bC^\times)= 0$ can be explained as follows. Consider an algebraic homomorphism $\phi: \bC^{\times} \rightarrow PGL(V)$. Then $\phi$ can be lifted to an algebraic homomorphism $\phi': \bC^{\times} \rightarrow \GL(V)$. Indeed, let $t\subset \bC^\times$ be a generic element, the morphisms $\phi$ and $\phi'$ are determined by the images of $t$. We choose $\phi'(t)$ to be a representative of $\phi(t)$ in $\GL(V)$. Then $\phi'$ is a lift of $\phi$.

The results about projective representations of $O_k, \SO_h$ for $k\geqslant 2, h\geqslant 3$ come from representations of their simply connected covers, the Pin groups and Spin groups. More details about these groups are explained in \Cref{subsect 7.3.1}. 
\begin{rem} \label{Schurz2}
Another well-known result is about the Schur multipliers of elementary $2$-groups. As in \cite{Schur}, the group $M((\cyclic{2})^{\oplus N})$ can be identified with the $\bF_2$-vector space of the $2$-forms on $A= \bF_{2}^{N}$. Here, we point out a realization of a basis of $M((\cyclic{2})^{\oplus N})$ which will be useful for the next results. Let $z_1,...,z_N$ be the generators of the group $A= (\cyclic{2})^{\oplus N}$, they can be viewed as a basis for the vector space $\bF_{2}^{N}$. We consider the 2-forms $\omega_{[i,j]}$ on $A$ with $1\leqslant i< j\leqslant N$,
 $$\omega_{[i,j]}(z_k, z_l)= \bigg\{
 \begin{aligned}
     1 \text{ if } i\leqslant k\neq l\leqslant j, \\     
     0 \text{ in other cases}.
 \end{aligned}
$$
We are considering forms with coefficients in $\bF_2$, so $-1= 1$, and the above expression makes sense. The elements $\omega_{[i,j]}$ form a basis of $M((\cyclic{2})^{\oplus N})$. Each $\omega_{[i,j]}$ is obtained by pulling the unique Schur multiplier of $O_{j-i+1}$ back along the map $$A=\Span(z_1,...,z_N)\rightarrow\Span(z_i,...,z_j)\rightarrow O_{j-i+1}.$$ 
The first map is the projection sending $z_k$ to $0$ if $k\notin [i,j]$, the last inclusion is obtained by sending $z_i$ to the diagonal matrix with $i$-th entry $-1$ and other entries $1$. 
 \end{rem}

\subsection{The Schur multipliers of $Q_y$} \label{subsect 2.3}
Next, we demonstrate the computation of $M(G)$ for various types of groups $G$ of interest. In particular, all the formulas for $M(Q_y)$ can be worked out.

First, we summarize some notation from \Cref{subsect 2.1}. We let $k$ (resp. $\ell$) denote the number of parts (resp. distinct parts) of the partition $\lambda$ of $e$. And we have realized $Q_e$ as $(\prod_{j= 1}^{\ell} (\SO_{i_{m}- i_{m-1}})) \rtimes A_\chi$. This lift of $A_\chi\subset A_e\subset Q_e$ is generated by $\{z_{i_m}\}_{1\leqslant m\leqslant \ell}$. Write $Z_m$ for $z_{i_m}$, then in this semidirect product, $Z_l$ acts trivially on $\SO_{i_{m}- i_{m-1}}$ if $l\neq m$, and the action of $Z_l$ on $\SO_{i_l-i_{l-1}}$ is by conjugation.

For a point $y\in Y_e$, its stabilizer $Q_y$ is isomorphic to $(\prod_{m= 1}^{\ell} (\SO_{i_{m}- i_{m-1}})) \rtimes (\cyclic{2})^{\oplus p}$ for some $p\leqslant \ell$. And the group $A_y= Q_y\cap A_e$ is isomorphic to $(\cyclic{2})^{\oplus k-\ell+p}$. In the semiproduct decomposition of $Q_y$, we write $Q_e^{\circ}$ for $\prod_{m= 1}^{\ell} (\SO_{i_{m}- i_{m-1}})$, and write $A_{\chi,y}$ for the factor $(\cyclic{2})^{\oplus p}\subset A_\chi$. The action of $A_{\chi,y}$ on $\prod_{m= 1}^{\ell} (\SO_{i_{m}- i_{m-1}})$ is induced by the action of $A_\chi$. 

Let $\ell'$ be the number of indexes $m$ such that 
\begin{enumerate}
    \item $y_{i_m}- y_{m_{i-1}}= 1$, or
    \item $y_{i_m}- y_{m_{i-1}}= 2$ and $A_{\chi,y}$ acts trivially on this factor $\SO_2$.
\end{enumerate}
We have the following lemma. 
\begin{lem} \label{Schur}
    Consider the group $Q_y= Q_e^{\circ}\rtimes A_{\chi, y} \cong$ $(\prod_{m=1}^\ell \SO_{i_m- i_{m-1}})\rtimes (\cyclic{2})^{\oplus p}$ in the above setting. Then we have a group isomorphism $M(Q_y)\cong $ $\bF_2^{\ell+ p(p-1)/2-\ell'}$. And the pullback map $M(Q_y)\rightarrow M(A_y)$ is injective.
\end{lem} 
The proof of this lemma will be given later. As mentioned in \Cref{Schurz2}, we have $M(A_{\chi, y})= M((\cyclic{2})^{\oplus p})\cong \bF_2^{\frac{p(p-1)}{2}}$. The first step of the proof is to understand the contribution of the factor $Q_e^{\circ}$. The next lemma is a special case of \Cref{Schur}. From now on we write $k_m$ for $y_{i_m}- y_{m_{i-1}}$. 
\begin{lem} \label{schur product of SO}
    Consider the case $Q_y= \prod_{m=1}^\ell \SO_{k_m}$. Let $\ell'$ be the number of indexes $m$ such that $k_m\geqslant 3$. Then $M(Q_y)= \bF_2^{\ell'}$.
\end{lem}
\begin{proof}
Recall that $\SO_2= \bC^\times$, and $M(\mathbb{C}^\times)= 0$. For any algebraic group $G$ such that the set of characters of $G$ is finite, we have that the group $\Hom(\mathbb{C^\times},\Hom(G,\mathbb{C}^\times))$ is trivial. Applying \Cref{direct product}, we see that the cardinalities of $M(G\times \mathbb{C}^\times )$ and $M(G)$ are the same. Thus, the classes of algebraic projective representations of $G\times \mathbb{C}^\times$ are obtained by pulling back the classes of $G$.

For $j\geqslant 3$, the set $\Hom (\SO_j,\bC^\times)$ consists only of the trivial character. Applying \Cref{upperbound}, we have $|M(G\times \SO_j)|\leqslant |M(G)|\times |M(\SO_j)|= 2|M(G)|$. Hence, by induction, we have $M(\prod_{m=1}^\ell \SO_{k_m})\leqslant 2^{\ell'}$. Next, we can construct $\ell'$ elements of $M(\prod_{m=1}^\ell \SO_{k_m})$ as follows. For each index $m$ such that $k_m\geqslant 3$, consider the projection $\prod_{m=1}^\ell \SO_{k_m}\rightarrow \SO_{k_m}$, we have $\ell'$ such projections. Pulling back $\ell'$ nontrivial Schur multipliers of these $\SO_{k_m}$, we obtain $\ell'$ elements of $M(Q_y)$. 

Recall that we have $Q_y\subset Q_e= \prod_{m=1}^\ell O_{k_m}$, and we let $A_y$ denote $Q_y\cap A_e$. As $A_e$ lives in $Q_e$ as the set of diagonal matrices with diagonal entries $1$ or $-1$, we have $A_y$ lives in $Q_y$ as diagonal matrices with diagonal entries $1$ or $-1$ in $\SO_{k_i}$. From the $\ell'$ elements of $M(Q_y)$ we have constructed, the pullback map $p: M(Q_y)\rightarrow M(A_y)$ gives us $\ell'$ elements of $M(A_y)$. These elements are linearly independent in $M(A_y)$, one can verify this by realizing $M(A_y)$ as the space of 2-forms as in \Cref{Schurz2}. 

Since $p$ is a group homomorphism, the image of $p$ contains subgroup of $M(A_y)$ generated by these $\ell'$ linear independent elements. Thus, the image of $p: M(Q_y)\rightarrow M(A_y)$ has cardinality at least $2^{\ell'}\geqslant |M(Q_y)|$. So $p$ is injective and $M(Q_y)$ is realized as an $\bF_2$-subspace of $M(A_y)$ of dimension $\ell'$.  
\end{proof} 

Let $\ell_1$ (resp. $\ell_2$) be the number of indexes $m$ such that $i_{m}-i_{m-1}=$ $1$ (resp. $2$). Then we have $M(\prod_{m=1}^{\ell}\SO_{i_{m}-i_{m-1}})= \bF_2^{\ell- \ell_1- \ell_2}$. Thus, by \Cref{upperbound}, we have $|M(Q_y)|\leqslant$ $2^{\ell-\ell_1- \ell_2 + \frac{m(m-1)}{2}}$ $\times$ $|H^1((\cyclic{2})^{\oplus p}, \Hom((\prod_{m=1}^{\ell}\SO_{i_{m}-i_{m-1}}), \bC^\times))|$. Now, we proceed to the proof of \Cref{Schur}.

\begin{proof} The proof consists of three steps.

    \textbf{Step 1}.We first show that the cardinality of $M(Q_y)$ is at most $2^{\ell+ p(p-1)/2-\ell'}$.\\  
    Consider the set $\Hom((\prod_{m=1}^{\ell}\SO_{i_{m}-i_{m-1}}), \bC^\times)$. We have nontrivial homomorphims from $\SO_{i_{m}-i_{m-1}}$ to $\bC^\times$ if and only if $i_{m}-i_{m-1}= 2$. In that case, $\Hom(\SO_2= \bC^\times, \bC^\times)= \bZ$. Hence $\Hom((\prod_{m=1}^{\ell}\SO_{i_{m}-i_{m-1}}), \bC^\times)$ can be identified with $(\bZ)^{\oplus \ell_2}$. Recall that the action of an element $x\in A_{\chi,y}\subset A_\chi$ on a factor $SO_{i_{m}-i_{m-1}}$ of $Q_e^\circ$ is either the trivial action or the conjugation action. As a consequence, the actions of elements of $A_{\chi, y}$ on a summand $\bZ$ of $(\bZ)^{\oplus \ell_2}$ are either trivial or $x\mapsto -x$. 
    
    Now we proceed to compute the group cohomology $\gH^1(A_{\chi, y}, \bZ^{\oplus \ell_2})$. For a summand $\bZ$ with trivial $A_{\chi,y}$-action, we have $\gH^1{(A_{\chi,y}, \bZ)}= \Hom((\cyclic{2})^{\oplus p}, \bZ)= 0$. For a summand $\bZ$ that $A_{\chi, y}$ acts nontrivially on, we have a decomposition of $A_{\chi, y}$ as $\cyclic{2}\times (\cyclic{2})^{\oplus p-1}$ so that the factor $\cyclic{2}$ acts on $\bZ$ by $x\mapsto -x$, and the factor $(\cyclic{2})^{\oplus p-1}$ acts trivially on $\bZ$. As a consequence, we have $\gH^1(A_{\chi, y}, \bZ)= \gH^1(\cyclic{2}, \bZ)= \bF_2$ when the action of $A_{\chi,y}$ on $\bZ$ is nontrivial. Thus, we obtain the upper bound $2^{\ell_1+ \ell_2- \ell'}$ for the cardinality of $\gH^1(A_{\chi,y}, \Hom(Q_e^{\circ}, \bC^\times))$. As a consequence of \Cref{upperbound}, we get $|M(Q_y)|\leqslant 2^{\ell+ p(p-1)/2-\ell'}$.
    
    \textbf{Step 2}. We construct $\ell-\ell'+ \frac{p(p-1)}{2}$ elements of $M(Q_y)$.\\
    Recall that $A_{\chi,y}$ is the component group of $Q_y$ and $A_{\chi, y}\subset A_e$ is realized as a subgroup of $Q_y\subset Q_e$. So we have two natural maps $A_{\chi, y}\xhookrightarrow{i_y} Q_y$ and $Q_y\xrightarrow{i^y} A_{\chi, y}$ satisfying $i^y\circ i_y= \Id$. Hence, the map $i_y^{*}: M(Q_y)\rightarrow M(A_{\chi, y})$ is surjective, and the map $(i^y)^*: M(A_{\chi, y})\rightarrow M(Q_y)$ is injective. We choose a basis of $M(A_{\chi, y})\cong \bF_2^{p(p-1)/2}$. The image of this basis under $(i^y)^*$ gives us $\frac{p(p-1)}{2}$ elements of $M(Q_y)$.
    
    Next, we find $\ell- \ell'$ elements in the kernel of the map $i_y^*$. To do this, we construct several homomorphisms from $Q_y$ to orthogonal groups as follows. For indexes $j$ such that $A_{\chi,y}=(\cyclic{2})^{\oplus p}$ acts trivially on $\SO_{y_j- y_{j-1}}$ and $y_j- y_{j-1}\geqslant 3$, consider the homomorphism $Q_y\rightarrow \SO_{y_j- y_{j-1}}\rightarrow O_{y_j-y_{j-1}}$. If $A_{\chi,y}$ acts nontrivially on $\SO_{y_j- y_{j-1}}$ and $y_j- y_{j-1}\geqslant 2$, we have a decomposition of $A_{\chi,y}$ as $\cyclic{2}\times (\cyclic{2})^{\oplus p-1}$ so that the first factor acts by conjugation and the second factor acts trivially on $\SO_{y_j- y_{j-1}}$. Then we have a projection $Q_y\rightarrow \SO_{y_j- y_{j-1}}\rtimes \cyclic{2}\cong O_{y_j- y_{j-1}}$. Pulling these unique Schur multipliers of $O_{y_j- y_{j-1}}$ back, we obtain $\ell-\ell'$ elements of $M(Q_y)$. They are in the kernel of the pullback map $M(Q_y)\xrightarrow{i_y^*} M(A_{\chi, y})$. 

    \textbf{Step 3}. Write $f_y$ for the inclusion $A_y\hookrightarrow Q_y$. Recall that $A_y$ and $M(A_y)$ are finite dimensional $\bF_2$-vector spaces. In Step 2, we have constructed $\frac{p(p-1)}{2}+ \ell- \ell'$ elements of $M(Q_y)$. We show that the images of these elements under the pullback map $f_y^{*}: M(Q_y)\rightarrow M(A_y)$ are linearly independent in $M(A_y)$. A consequence is that $f_y^{*}$ is injective and $M(Q_y)$ is a $\bF_2$-subspace of $M(A_y)$ of dimension $\frac{p(p-1)}{2}+ \ell- \ell'$.
    
    Recall that from (\ref{first chain}) we have an isomorphism
    $$A_\chi \hookrightarrow A_e\hookrightarrow Q_e\rightarrow A_\chi.$$ 
    Taking the intersections of these groups with $Q_y$, we obtain a chain of maps 
    \begin{equation} \label{chain of maps}
        A_{\chi, y}\hookrightarrow A_y\xhookrightarrow{f_y} Q_y\rightarrow A_{\chi, y}
    \end{equation}
    with the composition being an isomorphism. Therefore, the following composition of pullback maps is an isomorphism $$ M (A_ {\chi, y})\rightarrow M(Q_y)\xrightarrow{f_y^{*}} M(A_y)\rightarrow M(A_{\chi, y}).$$
    As a consequence, the kernel of the map $f_y^{*}$ must lie in the kernel of the pullback map $i_{y}^{*}: M(Q_y)\rightarrow M(A_{\chi ,y})$. 
    
    From (\ref{chain of maps}), we have a direct sum decomposition $A_y= \oplus_{m=1}^{\ell} ((\cyclic{2})^{\oplus k_m- 1})\oplus (\cyclic{2})^{\oplus p}$ that is compatible with the inclusion $f_y: A_y\xhookrightarrow{} Q_y$ $= \prod_{m=1}^{\ell}\SO_{k_m}\rtimes A_{\chi, y}$ in the following sense. We have that $f_y$ maps each summand $ ((\cyclic{2})^{\oplus k_m- 1})$ to the diagonal matrices with entries $\pm 1$ in the corresponding $\SO_{k_m}$; and $f_y$ maps $(\cyclic{2})^{\oplus p}$ isomorphically to $A_{\chi, y}$.

    In Step 2, we have constructed $\ell-\ell'$ elements of $M(Q_y)$ which are in the kernel of $i_{y}^{*}: M(Q_y)\rightarrow M(A_{\chi ,y})$. It is left to show that these elements remain linearly independent after being pulled back by $f_y$. Let $b_m$ denote the element obtained by pulling back the nontrivial Schur multipliers of $O_{k_m}$ along the maps $Q_y\rightarrow O_{k_m}$. The element $f_y^*(b_m)$ has the property that, as a 2-form on $A_y$, its restriction to the summand $(\cyclic{2})^{\oplus k_j- 1}\oplus A_{\chi, y}$ of $A_y$ is nonzero iff $j= m$ (see the direct sum decomposition of $A_y$ in the previous paragraph). Therefore, for a set of different $\ell- \ell'$ indexes $m$, the pullbacks of the corresponding elements are linearly independent in $M(A_y)$.
\end{proof}
We often make use of the following applications of \Cref{Schur}
\begin{cor} \label{easy schur} \leavevmode

\begin{itemize}
    \item Let $G$ be a group whose group of characters is finite. Then the pullback map $M(G)\rightarrow M(G\times \bC^\times)$ is an isomorphism. 
    \item Let $G_1$ be a group with $M(G_1)= \{1\}$ and $\Hom(G_1, \bC^\times)= \{1\}$. Then the pullback map $M(G)\rightarrow M(G\times G_1)$ is an isomorphism. In particular, for $G_1= \Sp_{2l}$ for some $l$, we have $M(G\times \Sp_{2l})= M(G)$.
    \item Assume $k_i\geqslant 2$ for $1\leqslant k\leqslant n$, we have $M(\prod_{i=1}^{n} O_{k_i})\cong$ $\mathbb{F}_2^{\frac{n(n+1)}{2}}$.
\end{itemize}    
\end{cor}

Next, we have the following lemma that fits the context of the first two cases of \Cref{easy schur}. Consider two algebraic groups $G$ and $G_1$. Assume that $G_1$ is connected and that the pullback map $M(G'\times G_1)\rightarrow M(G')$ is an isomorphism for any subgroup $G'\subset G$. Let $Y$ be a $G\times G_1$-centrally extended finite set. 
\begin{lem} \label{centrally extended structure of a product}
    The $G\times G_1$-centrally extended structure of $Y$ is determined by its $G$-centrally extended structure. 
\end{lem}
\begin{proof}
    As $G_1$ is connected, it acts trivially on the finite set $Y$. Hence, for a point $y\in Y$, its $G$ orbit coincides with its $G\times G_1$-orbit. Let $G_y$ be the stabilizer of $y$ in $G$. The central extensions associated to the orbit of $y$ can then be read from the isomorphism $M(G_y\times G_1)\rightarrow M(G_y)$.

    Conversely, this proof shows that we can naturally equip a $G$-centrally extended finite set $Y$ with a $G\times G_1$-centrally extended structure. We call this centrally extended structure the lift of the $G$-centrally extended structure on $Y$.
\end{proof}

When proving \Cref{Schur}, we have shown in the third step that the pullback map $M(Q_y)\rightarrow M(A_y)$ is injective. Moreover, we have obtained an explicit description of the map $M(Q_y)\rightarrow M(A_y)$ (see the basis in Step 2 of the proof of \Cref{Schur}). Therefore, similar to \Cref{centrally extended structure of a product}, knowing the $A_e$-centrally extended structure of $Y_e$ allows us to know the $Q_e$-centrally extended structure of $Y_e$. An advantage of working with $A_e$ is that we can understand centrally extended structures through some numerical invariants; this is the theme of the next section.

Lastly, we introduce a subgroup of $A_e$ that will be useful for explicit calculations. Let $A_{e}^{'}$ be the subgroup of $A_e$ generated by $z_1,...,z_{k-1}$. It is noted that the element $z_{12...k}$ acts trivially on $\spr$, so acts trivially on $Y_e$. Therefore, to determine $Y_e$ as an $A_e$-centrally extended set is equivalent to determine $Y_e$ as an $A_{e}^{'}$-centrally extended set. 

\section{Some numerical invariants of the Springer fibers}
Let $A$ be a finite abelian group acting on a quasi-projective variety $X$. The equivariant K-group $K^A_0(X)$ is a module over the representation ring $R(A)$ of $A$. We obtain a decomposition $\mathbb{C}\otimes_\bZ R(A)= \bigoplus_{a\in A}\mathbb{C}_a$ as a direct sum of $\bC$-algebras. Consequently, the complexified equivariant K-group of $X$ decomposes as $\mathbb{C}\otimes_\bZ K^A_0(X)= \bigoplus_{a\in A} K^A_0(X)_{(a)}$. Similar constructions apply to any subgroup $A'\subseteq A$ and $a\in A'$. We then obtain a set of numbers as follows.
\begin{defin} \label{numeric F}
   For $a\in A'$, let $F_a^{A'}(X)$ denote the dimension of $K^{A'}_0(X)_{(a)}$ over $\mathbb{C}$. 
\end{defin}

Another set of numerical invariants comes from forgetful maps from equivariant K-groups to invariant parts of K-groups of $X$, $F_X^{A'}: K_0^{A'}(X)\rightarrow K_0(X)^{A'}$. The cokernel of this map is an abelian group that is annihilated by the order of $A'$ (see \Cref{annihilated by cardinality}). For a subgroup $A'$ of $A$ and $1\leqslant i\leqslant \log_2|A'|$, let $S^{A'}_i(X)$ be the multiplicities of the abelian group $\cyclic{2^i}$ in the decomposition of coker$F_X^{A'}$ into the direct sum of cyclic groups. 

When $a$ and $A'$ vary over all elements and subgroups of $A$, the numbers $F_a^{A'}$ and $S_{i}^{A'}(X)$ give us a set of numerical invariants of $X$. The definitions of the numbers $F_a^{A'}$ and $S_{i}^{A'}$ of an $A$-variety $X$ carry over to an $A$-centrally extended set $Y$. In particular, we substitute the K-groups $K_0^{A'}(X)$ in those definitions by $K_0(Sh^{A'}(Y))$. 

In Section 3.1, we introduce the notion of a finite model $Y$ (not uniquely defined) of an $A$-variety $X$. This finite model $Y$ will share the same numbers $\{F_{a}^{A'}, S_i^{A'}\}$ with $X$. In \Cref{subsect 3.2}, we consider $X= \spr$ and $A= A_e$ (see \Cref{subsect 2.1} for a detailed description of $A_e$). Then we specify how to compute the numbers $F_a^{A'}(\mathcal{B}_e)$. The numbers $S_i^{A'}(\spr)$ are discussed in later sections.
\subsection{Finite models of an $A$-variety} \label{subsect 3.1}
In this subsection, we consider a reductive group $A$, and we do not require $A$ to be finite or abelian. Consider a quasi-projective variety $X$ acted on by $A$. For any algebraic subgroup $A'$ of $A$, write $\Coh^{A'}(X)$ for the category of $A'$-equivariant coherent sheaves on $X$. For an element $a\in A$, we have an $A'$-equivariant morphism $a:X\rightarrow X$, $a(x)= a.x$. The pullback $a^*$ gives an autoequivalence of the category $\Coh^{A'}(X)$. Hence, the space $\bC\otimes K_0^{A'}(X)$ becomes a representation of $A$ (where $A'$ acts trivially). For a subgroup $A''$ of $A'$, we have the forgetful map $K_0^{A'}(X)\rightarrow K_0^{A''}(X)$.

For a finite $A$-set $Y$ and a point $y\in Y$, write $A_y$ for the stabilizer of $y$. When $A$ is finite, recall from \Cref{subsect 1.2} that an $A$-centrally extended structure is an $A$-equivariant assignment $y\mapsto \psi_y\in H^2(A_y, \bC^\times)$. When $A$ is reductive, we substitute $H^2(A_y, \bC^\times)$ by $M(A_y)$ defined in \Cref{subsect 2.2}.

An equivalent definition of finite centrally extended sets is given in \cite[Page 9]{bezrukavnikov2001tensor}. From \cite[Theorem 3]{bezrukavnikov2001tensor}, the category $\Sh^{A}(Y)$ (defined in \Cref{subsect 1.2}) is equivalent to $Mod_{\text{Rep}(A)}(B)$, the category of $A$-equivariant modules of some semisimple $\bC$-algebra $B$ equipped with a rational action of $A$ by automorphisms. Thus, for each subgroup $A'\subset A$, we have a natural action of $A/A'$ on $\Sh^{A'}(Y)$. This action realizes $\bC\otimes K_0(\Sh^{A'}(Y))$ as a representation of $A/A'$. 

We define a finite model of an $A$-variety $X$ as follows.
\begin{defin} \label{finite model}
    Let $X$ be an $A$-variety. An $A$-centrally extended finite set $Y$ is called an $A$-\textit{finite model} of $X$ if the following conditions are satisfied.
    \begin{itemize}
        \item For each subgroup $A'\subset A$, we have an $A/A'$-equivariant isomorphism of K-groups $e_{A'}: K_0^{A'}(X)\xrightarrow{\sim} K_0(\Sh^{A'}(Y))$.
        \item For $A''\subset A'\subset A$, we have the following commutative diagram of forgetful maps.

            \begin{center}
            \begin{tikzcd}
K_0^{A'}(X) \arrow{d}{F_X^{A', A''}} \arrow{r}{e_{A'}} \& K_0(\Sh^{A'}(Y)) \arrow{d}{F_Y^{A', A''}} \\
K_0^{A''}(X) \arrow{r}{e_{A''}} \& K_0(\Sh^{A''}(Y))
            \end{tikzcd}
            \end{center}
    \end{itemize}
\end{defin}
Next, we give some examples of finite models.
\begin{exa} \label{finite model of a product}(A finite model of the product). For $i= 1,2$, let $X_i$ be an $A_i$-variety. Assume that $X_i$ admits an $A_i$-finite model $Y_i$. Now we have a natural action of $A= A_1\times A_2$ on $X= X_1\times X_2$. Consider the finite set $Y=Y_1\times Y_2$. For a point $y= (y_1, y_2)\in Y$, recall that we have two pairs $(A_{1, y}, \psi_1\in M(A_{1, y}))$ and $(A_{2, y}, \psi_2\in M(A_{2, y}))$ that record the centrally extended structures of $Y_1$ and $Y_2$. 

Let $p_1$ and $p_2$ be the projections from $A$ to $A_1$ and $A_2$, respectively. Write $A_y$ for $A_{1,y}\times A_{2,y}$. Let $\psi$ be $(p_1^*\psi_1).(p_2^*\psi_2)\in $ $M(A_y)$. Assigning to the point $y= (y_1,y_2)$ the pair $(A_y, \psi)$, we obtain an $A$-centrally extended structure on $Y$.

By the Kunneth formula, we have an isomorphism $K_0^{A_1}(X_1)\otimes_{\bZ} K_0^{A_2}(X_2)$ $\cong$$ K_0^{A_1\times A_2}(X_1\times X_2)$ sending $([\cF],[\cG]$) to $[\cF\boxtimes \cG]$. This isomorphism commutes with the forgetful maps in \Cref{finite model}. Therefore, we get $Y$ as an $A$-finite model of $X= X_1\times X_2$.
\end{exa}

\begin{exa} \label{finite model of the square set}(A finite model of the square set). Consider an $A$-variety $X$ with an $A$-finite model $Y$. Consider the finite set $Y\times Y$ and the diagonal action of $A$ on this set. For each pair of points $y, y'\in Y$, we have correspondingly two pairs $(A_y, \psi_y\in M(A_y))$ and $(A_{y'}, \psi_{y'}\in M(A_{ y'}))$. We give $Y\times Y$ an $A$-centrally extended structure by assigning to each point $(y,y')\in Y\times Y$ the pair $(A_{y}\cap A_{y'}, \psi_y (\psi_{y'})^{-1})$. In Section 7, we will consider certain varieties $X$ with nice properties. In those cases, $Y\times Y$ with this centrally extended structure is an $A$-finite model of $X\times X$ with respect to the diagonal action of $A$.   \end{exa}
    
\begin{rem} \label{same numerics unique}
    Assume that $A$ is an abelian $2$-group. From the definition, a finite model $Y$ of $X$ satisfies $F_{a}^{A'}(X)= F_{a}^{A'}(Y)$ and $ S_i^{A'}(X)= S_i^{A'}(Y)$. In general, both the existence and the uniqueness of a finite model are not guaranteed. From \Cref{subsect 1.2}, $Y_e$ serves as an $A_e$-finite model (and a $Q_e$-finite model) of $\spr$. Throughout this paper, the theme is that we consider the cases in which we can verify the following conditions.
    \begin{enumerate}
        \item The finite model of $\spr$ is unique. Therefore, it must be $Y_e$.
        \item The structure of $Y_e$ is recovered from the numerical invariants $\{F_{a}^{A'}, S_i^{A'}\}$.
    \end{enumerate}
The next subsection discusses an algorithm to determine the numbers $F_{a}^{A'}(\spr)$. 
\end{rem}

\subsection{An algorithm to compute $F_{a}^{A'}(\spr)$} \label{subsect 3.2}
We first recall the setting from \Cref{subsect 2.1} and introduce some new notations. We make the assumption that the associated partition of $e$ is $\lambda= (2x_1,..., 2x_k)$. The group $A_e$ is isomorphic to $(\cyclic{2})^{\oplus k}$ with a set of generators $\{z_1,z_2,...,z_k\}$. Write $[1,k]$ for the set $\{1,2,...,k\}$. The elements of $A_e\cong \bF_2^{k}$ can be regarded as subsets of $[1,k]$ as follows. If $a=z_{i_1i_2...i_h}$ with $1\leqslant i_1<i_2<...< i_h\leqslant k$, then $a$ corresponds to the subset $\{i_1,...,i_h\}$ of $[1,k]$. 

For an arbitrary subset $a\subset [1,k]$, let $\lambda_a$ denote the partition $(2x_i)_{i\in a}$ and write $|\lambda_a|= \sum_{i\in a} x_i$. Since $\lambda$ consists only of even parts, $\lambda_a$ has the same property. Therefore, $\lambda_a$ is a partition of type C. Let $N_{\lambda_a}\in \mathfrak{sp}_{2|\lambda_a|}$ be a nilpotent element with partition $\lambda_a$. Write $\cB_{\lambda_a}$ for the corresponding Springer fiber.  

Let $\chi(X)$ denote the Euler characteristic of a variety $X$. Consider two elements $a,a'\in A_e$. Since we can regard both $a$ and $a'$ as subsets of $[1,k]$, it makes sense to use set operations for $a$ and $a'$. For example, we have $a\cap a'$, $a\setminus a'$, $a'\setminus a$, and $[1,k]\setminus(a\cup a')$ as subsets of $[1,k]$. And we have the corresponding partitions $\lambda_{a\cap a'}$, $\lambda_{a\setminus a'}$, $\lambda_{a'\setminus a}$, $\lambda_{[1,k]\setminus(a\cup a')}$. With this notation, we now state the main theorem of this section. 
\begin{thm} \label{recursive} \leavevmode
\begin{enumerate}
    \item The number $F_a^{A'}(\spr)$ is computed by
    $$F_a^{A'}(\spr)= \frac{1}{|A'|}(\sum_{a'\in A'}\chi(\spr^a\cap \spr^{a'})).
    $$
    \item The connected components of the variety $\spr^{a,a'}$ are all isomorphic. They are isomorphic to $\cB_{\lambda_{a\cap a'}}\times$ $\cB_{\lambda_{a\setminus a'}}\times$ $\cB_{\lambda_{a'\setminus a}}\times$ $\cB_{\lambda_{[1,k]\setminus(a\cup a')}}$.
\end{enumerate}      
\end{thm}

We first prove Statement 2 of the theorem. The first step is to describe the variety $\spr^a$.
\begin{lem} \label{fix}
    The fixed-point variety $\spr^a$ is the disjoint union of ${n\choose |\lambda_a|}$ copies of the product $\cB_{\lambda_a}\times \cB_{\lambda_{[1,k]\setminus a}}$.
\end{lem}
\begin{proof}
    Consider $a= z_{i_1}z_{i_2}...z_{i_h}$, $i_1<...<i_h$, then we can regard $a$ as a set $\{i_1,...,i_h\}\subset [1,k]$. As $a$ acts on $V_\lambda$ with two eigenvalues $1$ and $-1$, we have a decomposition $V_\lambda= V_{\lambda_a}\oplus V_{\lambda_{[1,k]\setminus a}}$. Here $V_{\lambda_a}$ is the invariant space of $a$, having dimension $\sum_{i\in a} 2x_i$. The centralizer of $a$ in $\fsp(V_\lambda)$ is the pseudo-Levi $\fl_a= \fsp(V_{\lambda_a})\times \fsp(V_{\lambda_{[1,k]\setminus a}})$. Write $W$ and $W_{\fl_a}$ for the Weyl groups of $\fsp(V_\lambda)$ and $\fl_a$.
    
    Recall that $\mathcal{B}_e^a$ consists of the Borel subalgebras of $\fsp(V_\lambda)$ stabilized by $a$. We then have a map $\mathcal{B}_e^a\rightarrow \cB_{\lambda_a}\times \cB_{\lambda_{[1,k]\setminus a}}$ sending a Borel subalgebra $\fb$ to $\fb\cap \fl_a$. This map is an isomorphism on each connected component of $\spr^a$, and the fiber of this map consists of ${n\choose |\lambda_a|}$ points. Here, we obtain the number ${n\choose |\lambda_a|}$ because we are counting the $W_{\fl_a}$ cosets of $W$. Therefore, $\spr^a$ is the disjoint union of ${n\choose |\lambda_a|}$ copies of the product $\cB_{\lambda_a}\times \cB_{\lambda_{[1,k]\setminus a}}$.

\end{proof}
Next, we analyze the action of $A_e$ on $\spr^a$.
\begin{rem}(An $A_e$-finite model of $\spr^a$) \label{A_e action}
\Cref{fix} shows that each connected component of $\spr^a$ is isomorphic to $\cB_{\lambda_a}\times \cB_{\lambda_{[1,k]\setminus a}}$. For the two nilpotent elements $N_{\lambda_a}\in \mathfrak{sp}(2|\lambda_a|)$ and $N_{\lambda_{[1,k]\setminus a}}\in \mathfrak{sp}(2|\lambda_{[1,k]\setminus a}|)$, we can construct two vector spaces $V_{\lambda_a}$ and $V_{\lambda_{[1,k]\setminus a}}$ as in \Cref{subsect 2.1}. We have two corresponding groups $A_{\lambda_a}$ and $A_{\lambda_{[1,k]\setminus a}}$. 

From the proof of \Cref{fix}, we can identify the two spaces $V_{\lambda_a}$ and $V_{\lambda_{[1,k]\setminus a}}$ with the two eigenspaces of eigenvalue $1$ and $-1$ with respect to the action of $a$ on $V_{\lambda}$. Then $A_e$ is naturally identified with $A_{\lambda_a}\times A_{\lambda_{[1,k]\setminus a}}$. Furthermore, the actions of $A_e$ on each component $\cB_{\lambda_a}\times \cB_{\lambda_{[1,k]\setminus a}}$ of $\spr^a$ come from the actions of $A_{\lambda_a}$ and $ A_{\lambda_{[1,k]\setminus a}}$ on the corresponding factors.

The Springer fibers $\cB_{\lambda_a}$ and $ \cB_{\lambda_{[1,k]\setminus a}}$ admit finite models $Y_{N_{\lambda_a}}$ and $ Y_{N_{\lambda_{[1,k]\setminus a}}}$. As in \Cref{finite model of a product}, we have a centrally extended structure on $Y_{N_{\lambda_a}}\times Y_{N_{\lambda_{[1,k]\setminus a}}}$ that gives a finite model of $\cB_{\lambda_a} \times  \cB_{\lambda_{[1,k]\setminus a}}$. Therefore, we have an $A_e$-finite model $Y_e^{a}$ of $\spr^a$ consisting of $n\choose |\lambda_a|$ copies of the product $Y_{N_{\lambda_a}}\times Y_{N_{\lambda_{[1,k]\setminus a}}}$.

\end{rem}

\begin{proof}[Proof of part 2 of \Cref{recursive}]
    From \Cref{A_e action}, the fixed-point variety $(\cB_{\lambda_a}\times \cB_{\lambda_{[1,k]\setminus a}})^{a'}$ is the same as $\cB_{\lambda_a}^{a'}\times \cB_{\lambda_{[1,k]\setminus a}}^{a'}$. Applying \Cref{fix} to each factor, we have the following. 
    \begin{enumerate}
        \item $\cB_{\lambda_a}^{a'}$ is isomorphic to $|\lambda_a| \choose |\lambda_{a\cap a'}|$ copies of $\cB_{\lambda_{a\cap a'}}\times $ $\cB_{\lambda_{a\setminus a'}}$.
        \item $\cB_{\lambda_{[1,k]\setminus a}}^{a'}$ is isomorphic to $n- |\lambda_a| \choose |\lambda_{a'\setminus a}|$ copies of $\cB_{\lambda_{a'\setminus a}}\times$ $\cB_{\lambda_{[1,k]\setminus(a\cup a')}}$.
    \end{enumerate}
    Thus, $\spr^{a,a'}$ is a disjoint union of copies of $\cB_{\lambda_{a\cap a'}}\times$ $\cB_{\lambda_{a\setminus a'}}\times$ $\cB_{\lambda_{a'\setminus a}}\times$ $\cB_{\lambda_{[1,k]\setminus(a\cup a')}}$.
\end{proof}

\begin{rem} \label{no_odd}
    In \cite[Theorem 3.9]{DLP}, it was shown that Springer fibers do not have odd cohomologies. It follows that the varieties $\spr^{a,a'}$ do not have odd cohomologies. For this reason, we have $\dim K_0(\spr^{a,a'})= \chi(\spr^{a,a'})$.  
\end{rem}
To prove the remaining parts of \Cref{recursive}, we will make use of \cite[Theorem 5.10.5]{CG} multiple times. We recall this theorem below.

Let $A$ be an abelian reductive group and let $R(A)$ be its representation ring. Elements of $\bC\otimes R(A)$ can be regarded as functions on $A$ by taking the traces of the representations. For $a\in A$, let $S_a$ be the multiplicative set of functions in $\bC\otimes R(A)$ that do not vanish at $a$. We write $R(A)_a$ for the localization of $\bC\otimes R(A)$ at $S_a$. For a $R(A)$-module $M$, let $M_a$ denote $M\otimes_{R(A)}R(A)_a$. In this subsection, all K-groups are complexified K-groups.
\begin{thm}\label{CG_localization}(Localization theorem in equivariant-K-theory)
Given an element $a\in A$, let $X^a$ be the $a$-fixed point locus of an $A$-variety $X$. Write $i$ for the natural embedding $X^a\hookrightarrow X$. Then the pushforward
$$i_*: K^A(X^a)_a\cong K^A(X)_a$$
induces an isomorphism between localized K-groups.
\end{thm}
We now consider $A$ to be a finite abelian group, recall that $R(A)\otimes \bC= \oplus_{a\in A} \bC_a$. For $a'\neq a$, the summand $\bC_{a'}$ is annihilated by $1\in \bC_a$, so $R(A)_a= \bC_a$. Therefore, a numerical consequence of \Cref{CG_localization} is $F_a^{A}(X)= F_a^{A}(X^a)$. Since $a$ acts trivially on $X^a$, we have $F_a^{A}(X^a)= F_1^{A}(X^a)$. Therefore,
\begin{equation} \label{character}
    F_a^{A}(X)= F_a^{A}(X^a)= F_1^{A}(X^a).
\end{equation} 
We further assume that $X$ has an $A$-finite model $Y$ and $X$ has no odd cohomologies. An application of \Cref{CG_localization} is the following lemma.
\begin{lem} \label{fiber at 1}
    For a subgroup $A'$ of $A$, we have 
    $$F_1^{A'}(X)= \frac{1}{|A'|}(\sum_{a'\in A'}\chi(X^{a'})).$$
\end{lem}
\begin{proof}
    Let $H^*(X)^{A'}$ denote the subspace of fixed points of $H^*(X)$ under the $A'$-action. We claim $F_1^{A'}(X)=\dim K^{A'}_0(X)_{(1)}= \dim H^*(X)^{A'}$. Indeed, we have 
\begin{equation} \label{fiber is invariant}
    \dim K^{A'}_0(X)_{(1)}= \dim K^{A'}_0(Y)_{(1)}= \dim K_0(Y)^{A'}= \dim K_0(X)^{A'} =\dim H^*(X)^{A'}.
\end{equation}
The second equality follows from the fact that each $A'$-orbit contributes $1$ to the fiber over the identity element and also $1$ to the dimension of the $A$-fixed point subspace of $K_0(Y)$. Next, we explain how to calculate $\dim H^*(X)^{A'}$.

Since $A'$ acts on the finite set $Y$, we can regard $K_0(Y)$ as a permutation representation of the group $A'$. So we have 
$$\dim K_0(Y)^{A'}= \frac{1}{|A'|}\sum_{a\in A'}\chi(a),$$
in which $\chi(a)$ is the character of $a$ in the representation $K_0(Y)= H^*(X)$. To conclude the proof, we show $\chi(a)= \chi(X^a)$. This is another application of the localization theorem. Consider the action of the group $\{1,a\}$ on the finite set $Y$ (note that $M(\bZ/2\bZ)= \{1\}$, so $Y$ is an ordinary $\cyclic{2}$-set). We see that both $\dim K^{\{1,a\}}_0(Y)_{a}$ and $\chi(a)$ count the number of points in $Y$ fixed by $a$. Therefore, $\chi(a)= F_a^{\{1,a\}}(X)$. The latter is identified with $F_1^{\{1,a\}}(X^{a})$ by the equality (\ref{character}). As the group $\{1,a\}$ acts trivially on $X^a$, we have $F_1^{\{1,a\}}(X^{a})=\dim K_0(X^a)$. With the assumption that $X$ does not have odd cohomologies, we get $\chi(a)= \chi(X^a)$.
\end{proof}

We now present the proof of the statement $F_a^{A'}(\spr)= \frac{1}{|A'|}(\sum_{a'\in A'}\chi(\spr^a\cap \spr^{a'}))$ in \Cref{recursive}.

\begin{proof} 
From the description of $\spr^a$ in \Cref{fix} and from \Cref{no_odd}, we see that $X= \spr^a$ does not have odd cohomologies. From the analysis in \Cref{A_e action}, $\spr^a$ admits a finite model as in \Cref{finite model of a product}. Hence $X= \spr^a$ and $A= A_e$ satisfy the hypothesis of \Cref{fiber at 1}.

From (\ref{character}), we get $F_a^{A'}(\spr)= F_1^{A'}(\spr^a)$. Now, applying \Cref{fiber at 1} to $X=\spr^a$, we have $$F_1^{A'}(\spr^a)= \frac{1}{|A'|}(\sum_{a'\in A'}\chi((\spr^{a})^{a'}))= \frac{1}{|A'|}(\sum_{a'\in A'}\chi(\spr^a\cap \spr^{a'})).$$  
\end{proof}

\begin{rem}
    From the beginning of this section, we have assumed that all parts of the associated partition $\lambda$ are even. In fact, we can consider the case of a general partition $\lambda$, and let $A'$ be an arbitrary finite abelian subgroup of $Q_e$. The numerical result of \Cref{recursive} remains correct, while the description of the components of the fixed-point variety depends on the decomposition of the eigenspaces of $V_\lambda$ under the action of $A'$.
\end{rem}

A simple application of \Cref{recursive} is a formula for the number of left cells in a two-sided cell in affine Weyl group. In particular, let $c$ the two-sided cell corresponding to the orbit of $e$ (see Section 5.1). Then the number of left cells in $c$ is the number of $Q_e$-orbits (and hence $A_e$-orbit) in $Y_e$ (see, e.g., \cite[Proposition 8.25]{bezrukavnikov2020dimensions}). This number is $F_{1}^{A_e}$, so we have obtained the following corollary.
\begin{cor}
    The number of left cells in $c$ is 
    $$F_{1}^{A_e}(\spr)= \dim H^*(\spr)^{A_e}= \frac{1}{|A_e|}(\sum_{a\in A_e} \chi(\spr^a)).$$
\end{cor}

\subsection{Some explicit formulas}
\Cref{recursive} has provided a recursive formula to calculate $F_{a}^{A'}$ in terms of $\chi(\spr)$ with various $e$. We now cite the results on the Euler characteristics of the Springer fibers. 

Let $\lambda= (1^{r_1},2^{r_2},...)$ be a partition corresponding to a nilpotent adjoint orbit of $\fsp_{2n}$. Let $EC(\lambda)$ denote the Euler characteristic of $\mathcal{B}_{N_\lambda}$. According to \cite[Theorem 3.1]{Eulerchar}, we have the recursive formula
\begin{equation} \label{Euler recursive}
  EC(\lambda)= \sum_{i\geqslant 2, r_i odd} EC(\lambda^{h,i})+ \sum_{i\geqslant 1, r_i\geqslant 2}2\floor{\frac{r_i}{2}}EC(\lambda^{v,i}),
\end{equation}
in which $\lambda^{h,i}$ (resp. $\lambda^{v,i}$) are the partitions obtained by removing a horizontal (resp. vertical) domino in a row (resp. two consecutive rows) of length(s) $i$ from the Young diagram of $\lambda$. In particular, with $\lambda= (1^{r_1},2^{r_2},...,(i-2)^{r_{i-2}},(i-1)^{r_{i-1}}i^{r_i},...)$, we have $\lambda^{h,i}= (1^{r_1},2^{r_2},...,(i-2)^{r_{i-2}+ 1},(i-1)^{r_{i-1}},i^{r_i- 1},...)$ and $\lambda^{v,i}= (1^{r_1},2^{r_2},...,(i-2)^{r_{i-2}},(i-1)^{r_{i-1}+2},i^{r_i- 2},...)$, provided $r_i\geqslant 1, 2$, respectively. For example, if $\lambda= (2k, 2j), j< k$, then we have $\lambda^{h, 2j}= (2k, 2j-2), \lambda^{h, 2k}= (2k-2, 2j)$. If $\lambda= (k,k)$, then $\lambda^{v,k}= (k-1, k-1)$. 

In principle, these relations allow us to calculate the dimension of the total cohomology recursively. In the example below, we first give a closed formula for $EC(\lambda)$ when $\lambda$ is a $2$-row partition, and then discuss the finite model $Y_{e}$.
\begin{exa} \label{numeric 2 row}(The case of $2$-row partitions).
The recursive formula for $\lambda= (2k, 2j)$ reads:
$$EC(2k, 2j)= EC(2k-2, 2j)+ EC(2k, 2j-2) \indent \text{if } j\neq k, $$
$$EC(2j, 2j)= 2EC(2j-1, 2j-1).$$
For partitions $\lambda= (2k, 2j)$, we assume $j\leqslant k$. With the base case $EC(2,0)= 1, EC(1,1) = 2$, it is easy to prove by induction that $$ EC(j,j)= 2^j,\indent EC(2k, 2j)= 2\sum_{i\leqslant j}{j+k\choose i}+ {j+k\choose j}. $$
For $e=N_{(2k, 2j)}$, the group $A_e$ $= (\cyclic{2})^{\oplus 2}$ is generated by two elements $z_1$ and $z_2$. The product $z_1z_2$ acts trivially on $\spr$ because it acts as $-\Id$ on $V_\lambda$. Hence, we have two possibilities for irreducible summands of the $A_e$-module $K_0(\spr)= H^*(\spr)$. We call $(1,1)$ the trivial character and $(-1,-1)$ the sign character. The result of \cite[Proposition 8.3]{Eulerchar} gives us the multiplicities of the trivial representation and the sign representation of $A_e$ in $H^*(\mathcal{B}_{(2k, 2j)})$ as follows:
\begin{equation} \label{trivial}
 \dim H^*(\mathcal{B}_{(2k, 2j)})^{A_e}= \dim H^*(\mathcal{B}_{(2k, 2j)})_{Id}= \sum_{i\leqslant j} {k+j\choose i} ,  
\end{equation}
\begin{equation} \label{sign}
  \dim H^*(\mathcal{B}_{(2k, 2j)})_{sgn}= \sum_{i< j} {k+j\choose i} . 
\end{equation}
We now discuss the structure of $Y_e$ in this case. The centrally extended action of $A_e$ on $Y_e$ factors through the action of the quotient $A_e/\{1,z_{12}\}\cong \cyclic{2}$, so there is no nontrivial central extension. The two possible types of $A_e$-orbits in $Y_e$ are the trivial orbit and the 2-point orbit fixed by $z_{12}$. The multiplicities of the characters in Equations (\ref{trivial}) and (\ref{sign}) show that both types of orbits appear in $Y_e$. The multiplicity of the 2-point orbit is the multiplicity of the sign representation, $\sum_{i< j} {k+j\choose i}$. The multiplicity of the trivial orbit is ${k+j}\choose{j}$.
This agrees with Lusztig's calculation for the subregular case of $\fsp_{2n}$ (e.g. see \cite{lusztig_subregular}). 
\end{exa}

\begin{exa}\label{characters 3 rows}(The case of $3$-row partitions).
Consider the case $\lambda =(2k, 2j, 2i)$. In this case $A_e$ is generated by $z_1, z_2$ and $z_3$. From the proof of \Cref{recursive}, we see that $\chi(z_{1})= \dim H^*(\mathcal{B}^{z_{1}}_{(2k, 2j, 2i)})$. From \Cref{fix}, $\cB_{(2k, 2j, 2i)}^{z_1}$ is the disjoint union of $i+j+k\choose i$ copies of $\cB_{(i)}\times \cB_{(2k, 2j)}$. And since $\cB_{(i)}$ is a single point, we get that
$$\chi(z_{1})= \dim H^*(\mathcal{B}^{z_{1}}_{(2k, 2j, 2i)})= {i+ j+ k \choose i}\dim H^*(\mathcal{B}_{(2k, 2j)})= {i+j+k\choose i}EC(2k, 2j).$$
Similarly,  
$$\chi(z_{2})= \dim H^*(\mathcal{B}^{z_{2}}_{(2k, 2j, 2i)})= {i+ j+ k \choose j}\dim H^*(\mathcal{B}_{(2k, 2i)})= {i+j+k\choose j}EC(2k, 2i),$$
$$\chi(z_{3})= \dim H^*(\mathcal{B}^{z_{3}}_{(2k, 2j, 2i)})= {i+ j+ k \choose k}\dim H^*(\mathcal{B}_{(2j, 2i)})= {i+j+k\choose k}EC(2j, 2i).$$
\end{exa}

Now we give the formula of $\chi(1)$, the dimension of $H^*(\cB_{(2k, 2j,2i)})$. In principle, one can compute $EC(2k,2j,2i)$ using the recursive formula in (\ref{Euler recursive}). In the following, we state a result that directly leads to a closed formula. 
\begin{lem} \label{3 rows}
For $\lambda= (2k, 2j, 2i)$, $k\geqslant j\geqslant i$, we have
$$\chi(1)= \chi(z_1)- \chi(z_2)+ \chi(z_3).$$
Hence, 
$$\dim H^*(\spr)= \dim K_0(\spr)= {i+j+k\choose i}EC(2k, 2j)- {i+j+k\choose j}EC(2k, 2i)+ {i+j+k\choose k}EC(2j, 2i),$$
in which $EC(2k, 2j)= 2\sum_{i< j}{j+k\choose i}+ {j+k\choose j}$.
\end{lem}
The proof is given in the next section, together with the analysis of the structure of $Y_e$ when $\lambda$ has three parts.

\section{The $A_e$-centrally extended structure of $Y_{(2k, 2j, 2i)}$}
In \Cref{subsect 4.1}, we study the irreducible characters of $A_e$ that appear in the representation $\bC\otimes K_0(Y_e)$. As an application of this result, we prove \Cref{3 rows} in \Cref{subsect 4.2}. We then proceed to give a full description of the $A_e$-centrally extended structure of $Y_e$ when $\lambda= (2k, 2j, 2i)$. 
\subsection{The characters of $A_e$ in $K_0(Y_e)$} \label{subsect 4.1}
In \Cref{subsect 2.1}, we have a natural surjection from $A_e$ to $A_\chi$, the component group of $Q_e$. The action of $A_e$ on $H^*(\spr)$ factors through $A_\chi$. In \cite[Theorem 2.5]{Shoji1983}, Shoji showed that a character of $A_\chi$ appears in $H^*(\spr)$ if and only if it appears in $H^{2\dim \spr}(\spr)$, the Springer representation. The task now is to analyze the characters of $A_\chi$ in the Springer representation. This information is given by the explicit Springer correspondence in \cite[Section 13.3]{carter1993finite}. 

Here, we give a brief exposition for the case of the group $\Sp(V_\lambda)$. This is done in two steps. The first is to assign a symbol $S_\lambda$ to the partition $\lambda$. The characters of $A_\chi$ that appear in $H^{2\dim \spr}(\spr)$ are then parametrized by certain permutations of the entries of $S_\lambda$. The second step is to describe how one can read the characters from the permutations.

\textbf{Step 1}. In this step, we follow the notation of \cite[Section 13.3]{carter1993finite}. Write the associated partition of $e$ in the form $\lambda= (\lambda_{2k}\geqslant...\geqslant \lambda_1)$. We add the number $0$ if necessary to get an even number of terms in the partition. Define $\lambda_{i}^{*}= \lambda_i +i -1$, and we separate the set $\{\lambda_i^*\}$ into its even and odd parts. Since $\lambda$ is a partition of type C, we get exactly $k$ odd terms and $k$ even terms in $\{\lambda_i^*\}$. Let them be $2\xi_{1}^*+1<...< 2\xi_{k}^*+1$ and $2\eta_{1}^*<...< 2\eta_{k}^*$. We then define $\eta= \eta^*- (i-1)$ and $\xi= \xi^*- (i-1)$. We now have
$$0\leqslant \xi_1\leqslant...\leqslant \xi_k,$$
$$0\leqslant \eta_1\leqslant...\leqslant \eta_k.$$

If $\eta_1= 0$, we omit this part and write $\eta_{i-1}$ for $\eta_i$. Otherwise, we add a number $0$ as $\xi_1$, and let $\xi_{i+1}$ be $\xi_i$. After this process, we obtain a symbol $S_\lambda$
$$\bigg(
 \begin{aligned}
     \xi_1 \indent \xi_2+2 \indent \xi_3+4 \indent \xi_4+6 ...\\
      \eta_1+1 \indent \eta_2+ 3 \indent \eta_3 +5 ... 
 \end{aligned} \bigg).
$$

The irreducible characters of $A_\chi$ in $H^{2\dim \spr}(\spr)$ are parametrized by all the symbols
$$
\bigg(
 \begin{aligned}
     \xi'_1 \indent \xi'_2+2 \indent \xi'_3+4 \indent \xi'_4+6 ...\\
     \eta'_1+1 \indent \eta'_2+ 3 \indent \eta'_3 +5 ... 
 \end{aligned} \bigg),
$$
in which $0\leqslant \xi'_1\leqslant...\leqslant \xi'_k$, $1\leqslant \eta'_1\leqslant...\leqslant \eta'_k$ so that the two multisets $\{\xi_1 \text{  ,  } \xi_2+2 \text{  ,  } \xi_3+4 \text{  ,  } \xi_4+6 \text{  ,  } ...,
     \eta_1+1 \text{  ,  } \eta_2+ 3 \text{  ,  } \eta_3 +5 ...\}$ and $\{\xi'_1 \text{  ,  } \xi'_2+2 \text{  ,  } \xi'_3+4 \text{  ,  } \xi'_4+6 \text{  ,  } ...,
     \eta'_1+1 \text{  ,  } \eta'_2+ 3 \text{  ,  } \eta'_3 +5 ...\}$ are equal.
     
In other words, from the symbol $S_\lambda$, we consider the permutations of its entries that satisfy the following two conditions.
\begin{enumerate}
    \item If one of the entry of $S_\lambda$ is $0$, the number $0$ must stay at the first row, since the smallest entries of the second row is $\eta'+1\geqslant 1$.
    \item If $x$ and $y$ are in the same row, then $|x-y|\geqslant$ $2$. 
\end{enumerate}
We give an example of producing the symbol $S_\lambda$.
\begin{exa} \label{symbol}
    Consider $\lambda= (2,4,6,6)$, we get a corresponding nilpotent element $e\in \fsp(18)$. Following Step 1, we have 
    $$(2, 4, 6, 6)\rightarrow (2, 5, 8, 9)\rightarrow (2, 4) \text{  and   } (1, 4).$$
    Hence the symbol is $\bigg( \begin{aligned}  0 \quad 4 \quad 8\\ 2 \quad 7 \end{aligned} \bigg)$. Now we list the permutations that parameterize irreducible characters of $A_\chi$. The conditions are that $0$ is in the first row and that the two numbers $7$ and $8$ are not in the same row. 
    With these restrictions, we obtain the following permutations: $\bigg( \begin{aligned} 0 \quad 4 \quad 8\\ 2 \quad 7 \end{aligned} \bigg)$, $\bigg( \begin{aligned} 0 \quad 2 \quad 8\\ 4 \quad 7 \end{aligned} \bigg)$, $\bigg( \begin{aligned} 0 \quad 2 \quad 7\\ 4 \quad 8 \end{aligned} \bigg)$, $\bigg( \begin{aligned} 0 \quad 4 \quad 7\\ 2 \quad 8 \end{aligned} \bigg)$. The interpretation is that there are $4$ different irreducible characters of $A_\chi$ in $H^*(\spr)$.
\end{exa}
\textbf{Step 2}.
We first recall some notation related to $A_e$ in \Cref{subsect 2.1}. Consider $e\in \fsp(V_\lambda)$ with partition $\lambda= (2x_1,...,2x_k)$. The group $A_e$ is isomorphic to $(\cyclic{2})^{\oplus k}$ with the generators $z_1,...,z_k$. We now explain how to obtain the characters of $A_e$ from the permutations.

Consider the symbol $(\xi, \eta)$ obtained in Step 1. The non-zero entries of the symbol are pairwise distinct. Assume that these entries are $a_1< a_2<...< a_m$. An observation is that the number of nonzero entries in $S_\lambda$ is $k$, dimension of $A_e$ as an $\bF_2$-vector space. Consider a permutation $(\xi', \eta')$ of the symbol $(\xi, \eta)$, and let $\psi$ be the corresponding character of $A_e$. We obtain this character $\psi$ by comparing the positions of the entries $a_i$' in the original symbol and the permutation as follows. We assign $\psi(z_i)= 1$ if $a_i$ remains in the same row and $\psi(z_i)= -1$ otherwise. As $\{z_1,...,z_k\}$ generates $A_e$, the character $\psi$ is uniquely determined in this way.

As an example, we apply this step to the permutations listed in \Cref{symbol}.  
\begin{exa}    
    We are in the context of \Cref{symbol}. Recall that we have the elements $z_1,z_2,z_3,z_4$ defined in \Cref{subsect 2.1}. The group $A_e= (\cyclic{2})^{\oplus 4}$ is generated by these elements. Because the element $z_3z_4$ belongs to the connected component of $C_G(e)$, it acts trivially on $H^*(\spr)$. For the component group $A_\chi$, we can choose a set of generators as $z_1, z_2, z_3$. Then the surjection $A_e\rightarrow A_\chi$ has the kernel $\{1, z_3z_4\}$. In the table below, we give the correspondence between the permutations of the symbol and the characters. The characters are presented by their value at the elements $z_i$'s.
    \begin{center}
\begin{tabular}{| c| c| c| }
\hline
 Permutations & Characters of $A_\chi$ & Characters of $A_e$ \\ 
 \hline
 $\bigg( \begin{aligned} 0 \quad 4 \quad 8\\ 2 \quad 7 \end{aligned} \bigg)$ & $(1,1,1)$ & $(1,1,1,1)$ \\  
 \hline
 $\bigg( \begin{aligned} 0 \quad 2 \quad 8\\ 4 \quad 7 \end{aligned} \bigg)$ & $(-1,-1,1)$ & $(-1,-1,1,1)$ \\
 \hline
 $\bigg( \begin{aligned} 0 \quad 2 \quad 7\\ 4 \quad 8 \end{aligned} \bigg)$ & $(-1,-1,-1)$ & $(-1,-1,-1,-1)$ \\
 \hline
 $\bigg( \begin{aligned} 0 \quad 4 \quad 7\\ 2 \quad 8 \end{aligned} \bigg)$ & $(1,1,-1)$ & $(1,1,-1,-1)$ \\
 \hline
\end{tabular}
\end{center}
\end{exa}

\subsection{The structure of $Y_{(2k, 2j, 2i)}$} \label{subsect 4.2}
We begin with the proof of \Cref{3 rows}.
\begin{proof}
We first claim that the character $\psi(z_1,z_2,z_3)$ $=$ $(-1,1,-1)$ does not appear in the total Springer representation $H^*(\spr)$. Following the steps in \Cref{subsect 4.1}, from the partition $(2k, 2j, 2i)$, $k\geqslant j\geqslant i$, we have
$$(0,2i, 2j, 2k)\rightarrow (0,2i+1, 2j+2, 2k+3)\rightarrow (i, k+1) \text{  and   } (0, j+1).$$
Hence we obtain the symbol
$$
\bigg(
 \begin{aligned}
     i \quad k+2\\
     j+1
 \end{aligned} \bigg).
$$
For the character $(-1,1,-1)$ to appear, we need to have all three entries $i$, $j+1$ and $k+2$ on the second row, which cannot happen. 

Consequently, the three possible characters of $A_\chi$ that can appear are $(1,1,1)$, $(-1, -1, 1)$, and $(1,-1,-1)$. We write $m_\chi$ for the multiplicities of the characters $\chi$ in $H^*(\spr)$. Then we have $$\chi(1)= m_{(1,1,1)}+ m_{(-1,-1,1)}+ m_{(1,-1,-1)},$$ 
$$\chi(z_1)=  m_{(1,1,1)}- m_{(-1,-1,1)}+ m_{(1,-1,-1)},$$ 
$$\chi(z_2)=  m_{(1,1,1)}- m_{(-1,-1,1)}- m_{(1,-1,-1)},$$ 
$$\chi(z_3)= m_{(1,1,1)}+ m_{(-1,-1,1)}- m_{(1,-1,-1)}.$$ 
Therefore, $\chi(z_1)-\chi(z_2)+\chi(z_3)= \chi(1)$.
\end{proof}

Now we have enough information to study $Y_{e}$ when the partition of $e$ is $(2k, 2j, 2i)$. First, the action of $A_e= \langle z_1, z_2, z_3 \rangle$ factors through its quotient $A_e/\langle 1, z_{123}\rangle$ $\cong(\cyclic{2})^{\oplus 2}$. Because we do not have the character $(-1,1,-1)$, there are at most four types of $A_e$-orbits in the set $Y_e$ as follows.
\begin{enumerate}
    \item The trivial orbit.
    \item The ordinary $2$-point orbit with stabilizer $\langle z_1, z_{23}\rangle$.
    \item The ordinary $2$-point orbit with stabilizer $\langle z_3, z_{12}\rangle$.
    \item The special orbit fixed by $A_e$ with the Schur multiplier pulled back from the unique Schur multiplier of $A_e/\langle 1, z_{123}\rangle$.
\end{enumerate}
Here, we use \textit{special} (resp. \textit{ordinary}) to describe the orbits with (resp. without) nontrivial Schur multiplier. Let the multiplicities of these orbits in $Y_e$ be $a_0, a_{1}, a_{3}$ and $s$, respectively.

We then have a system of equations:
$$
\systeme*{a_0+ s+ 2a_{1}= \chi(z_{1}) , a_0+ s= \chi(z_{2}) , a_0+ 2a_{12}+ s= \chi(z_{12}) , a_0+ a_{1}=F_{z_1}^{A_e}  }
$$
We explain why these equalities are true by examining the structure of $Y_e$. For example, only the orbits fixed by $z_1$ have nontrivial contributions to the value of the character at $z_1$. For those orbits, the contributions to $\chi(z_1)$ equal to the numbers of points in the orbits. Therefore, $\chi(z_{1})= a_0+ s+ 2a_{1}$, the other equations are obtained in a similar way. 

Recall that the computation of the characters was given in \Cref{characters 3 rows}. Solving the system of equations then gives us a closed formula for the multiplicities of the orbits in $Y_e$ as follows. 
$$s= {i+j+k\choose j}\sum_{l=0}^{i-1} {i+k\choose l}$$
$$a_0= {i+j+k\choose j}\sum_{l=0}^{i} {i+k\choose l}$$
$$a_{1}= \frac{{i+j+k\choose i}EC(2j,2k)- {i+j+k\choose j}EC(2i,2k)}{2}$$ 
$$a_{12}= \frac{{i+j+k\choose k}EC(2i,2j)- {i+j+k\choose j}EC(2i,2k)}{2}$$
The multiplicity of the orbit with nontrivial Schur multiplier is $s= {i+j+k\choose j}\sum_{l=0}^{i-1} {i+k\choose l}$. An observation is that this number is positive unless $i= 0$, so we always have the presence of special orbits when the partition has three rows.
\begin{rem}
    From \Cref{recursive}, for $A'= \{1, z_i\}$ and $a= z_i$, we have $\chi(z_i)= F_{z_\alpha}^{\{1,z_\alpha\}}$. Therefore, the structure of $Y_{(2k,2j,2i)}$ is determined by the numbers $\{F_a^{A'}\}_{a\in A'\subset A_e}$.
\end{rem}
\begin{rem} \label{not same fix point}
    For $a\in A_e$, we have $K_0(Y_e^a)= K_0(\spr^{a})$ as two vector spaces, since both their dimensions are $\chi(a)$. As $A_e$ is abelian, it acts on $Y_e^{a}$ and $\spr^{a}$. Unfortunately, it is not true that $K_0(Y_e^a)$ and $K_0(\spr^{a})$ are isomorphic as representations of $A_e$.
    
    In particular, take $a= z_2$. In the four types of orbit in $Y_e$ we listed above, the only two types of orbit fixed by $z_2$ are the trivial orbit and the special orbit. Both are fixed by $A_e$, so the action of $A_e$ on $Y^{z_2}$ is trivial. As explained in \Cref{fix}, the fixed-point locus $\spr^{z_2}$ consists of disjoint copies of the product of two Springer fibers, $\mathcal{B}_{{(k,i)}}\times \mathcal{B}_{{(j)}}$. This product is identified with the Springer fiber $\mathcal{B}_{{(k,i)}}$ because $\cB_{{(j)}}$ is a point. In \Cref{numeric 2 row}, we have seen that the sign character appears in the total cohomology of this Springer fiber. Hence $A_e$ does not act trivially on $K_0(\spr^{z_2})$.
\end{rem}
\section{Cell theory from different perspectives}
In Section 4, we have deduced restrictions on the types of orbit in $Y_e$ from the characters of $A_e$ that appear in $H^*(\spr)$. This section introduces some finer restrictions. These conditions come from the list of $A_e$-characters in $H^*(\spr^a)$ for $a\in A_e$. 

We start by realizing the algebra $K_0(\Sh^{Q_e}(Y_e\times Y_e))$ as a direct summand $J_c$ of the asymptotic Hecke algebra $J$. We then compare two parameterizations for the irreducibles of $K_0(\Sh^{Q_e}(Y_e\times Y_e))$ and $J_c$. This comparison results in certain restrictions for the centrally extended structure orbits that appear in $Y_e$.

\subsection{Cells in affine Weyl groups and the asymptotic affine Hecke algebra} \label{subsect 5.1}
Let $G$ be a connected reductive group. Write $P$ for its weight lattice. Let $W_f$ be the Weyl group of $G$. Let $W= W \ltimes P$ be the extended afffine Weyl group. Let $\mathcal{H}$ be the Hecke algebra associated with the group $W$. In particular, the affine Hecke algebra $\cH$ is a free $\bC[v^\pm]$-module generated by the elements $T_w$ for $w\in W$. The relations are $T_wT_{w'}= T_{ww'}$ if $l(ww')= l(w)+ l(w')$ and $T_s^2 = (v-v^{-1})T_s +1$ for simple reflections $s\in W_f\subset W$.

Over $\mathbb{C}[v^\pm]$, the algebra $\mathcal{H}$ has another basis $C_w, w\in W$, the \textit{Kazhdan-Lusztig basis}, see, e.g., \cite[Chapter 5]{lusztig2014hecke}. Let $h_{x,y,z}\in\mathbb{C}[v^\pm]$ be the structure constants of $\mathcal{H}$ with respect to this basis: $C_xC_y=\sum_{z\in W} h_{x,y,z}C_z$. In \cite{Cell4}, Lusztig defines two functions $a: W\rightarrow \mathbb{Z}_{\geqslant 0}$ and $\gamma: W\times W\times W\rightarrow \mathbb{Z}_{\geqslant 0}$ such that $v^{a(z)}h_{x,y,z}-\gamma_{x,y,z^{-1}}\in v\mathbb{Z}[v]$. The asymptotic affine Hecke algebra $J$, as a module over $\bZ$, is freely generated by a basis $\{t_w\}_{w\in W}$. The multiplication structure is given by $t_xt_y= \sum_{z\in W}\gamma_{x,y,z}t_{z^{-1}}$. These relations define $J$ as a \textit{based ring}.

Next, we define some equivalence relations in $W$. We say that a left (right, two-sided) ideal $I\subset \mathcal{H}$ is a \textit{KL-ideal} if $I$ has a basis over $\mathbb{C}[v^\pm]$ consisting of some elements $C_x$. For an element $z\in W$, let $I_{z}^{L}$ (resp. $I_{z}^{R}$, $I_z$) be the smallest left (right, two-sided) KL-ideal that contains $C_z$. For $x,y\in W$, we say $x\leqslant_L y$ if $I^{L}_{x}\subset I^{L}_{y}$. The relations $x \leqslant_R y$ and $x\leqslant_{LR} y$ are defined in a similar way.

These preorders induce three equivalence relations $\sim_L, \sim_R, \sim_{LR}$. These equivalence relations partition $W$ into left (right, two-sided) cells. For any two-sided cell $c$, let $J_c\subset J$ be the $\mathbb{Z}$-submodule generated by $t_x, x\in c$. The following properties were proved by Lusztig in \cite{Cell4}.
\begin{itemize}
    \item Each $J_c$ has a natural structure of a based ring. We have a decomposition of based algebras $J=\oplus_c J_c$.
    \item There is a bijection between the set of two-sided cells and the set of nilpotent orbits of $\fg$.
    \item For a two-sided cell $c$, let $e$ be an element in the corresponding nilpotent orbit. Let $C_G(e)$ be the centralizer of $e$ in $G$. Fix a choice of the maximal reductive subgroup $Q_e$ of $C_G(e)$. Let $s$ be a semisimple element of $Q_e$. Let $A_s$ be the component group of the centralizer of $s$ in $C_G(e)$. Write $\spr^s$ for the $s$-fixed-point subvariety of $\spr$. Then $A_s$ naturally acts on $H^*(\spr^s)$. The irreducible modules of $\mathbb{Q}\otimes J_c$ are classified by conjugacy classes of pairs $(s,\rho)$ in which $\rho$ is an irreducible representation of $A_s$ that appears in $H^*(\spr^s)$. 
    \item These irreducible modules can be realized as multiplicity spaces $ \Hom_{A_s}(\rho, H^*(\spr^s))$.
\end{itemize}
Fix a two-sided cell $c$ and a nilpotent element $e$ in the corresponding orbit. The left cells of $c$ are stable under the multipilication of $J_c$ from the left. Therefore, the span of elements in a left cell naturally becomes a $J_c$-modules. This is called a left cell module. In the following, we explain the realization of $J_c$ and its left cell modules in terms of $Y_e$.

The finite set $Y_e\times Y_e$ has a natural $Q_e$-centrally extended action on it (see \Cref{finite model of the square set}). We have an algebra structure by convolution on $K_0(\Sh^{Q_e}(Y_e\times Y_e))$ (see, e.g., \cite[Section 5]{bezrukavnikov2001tensor}). This algebra has a basis given by the classes of the simples in $\Sh^{Q_e}(Y_e\times Y_e)$. Then we have the following facts (see, e.g., \cite[Proposition 8.25]{bezrukavnikov2020dimensions}).
\begin{enumerate}
    \item There is an isomorphism of based algebras $J_c\cong K_0(\Sh^{Q_e}(Y_e\times Y_e))$.
    \item The left cells in the two-sided cell $c$ are in bijection with the $Q_e$-orbits in $Y_e$.
    \item Consider a left cell $c_L\subset c$, let $\bO$ be the corresponding $Q_e$-orbit in $Y_e$. Then the left cell module of $c_L$ is isomorphic to $K_0(\Sh^{Q_e}(Y_e\times \bO))$. Here we view $\bO$ as a centrally extended set whose centrally extended structure is restricted from $Y_e$.
\end{enumerate}
These facts explain one motivation for the study of $Y_e$. In particular, knowing which centrally extended orbits appear in $Y_e$ helps us to understand the left cell modules up to isomorphism.

\subsection{Irreducible representations of $K_0(\Sh^{Q_e}(Y_e\times Y_e))$, distinguished case}
When $e$ is distinguished, we have $Q_e= A_e$, so $J_c\cong K_0(\Sh^{Q_e}(Y_e\times Y_e))$ $= K_0(\Sh^{A_e}(Y_e\times Y_e))$. Recall that $A_e$ is an elementary 2-group (see \Cref{subsect 2.1}). According to \cite[Section 10.3]{etingof2009fusion}, the algebra $\mathbb{Q}\otimes K_0(\Sh^{A_e}(Y_e\times Y_e))$ is semisimple over $\mathbb{Q}$. In this subsection, we first recall the classification of simple modules of $K_0(\Sh^{A_e}(Y_e\times Y_e))$ in \cite{etingof2009fusion}. Then, we explain how the result in \cite{etingof2009fusion} agrees with Lusztig's classification of simple modules of $J_c$.

Let $A'$ be a subgroup of $A_e$, view it as an $\bF_2$-vector space. Recall from \Cref{Schurz2} that we have a group isomorphism $M(A')\cong H^2(A',\mathbb{C}^\times)$. The latter is parameterized by the set $\{\psi\}$ of 2-forms on $A'$. Recall that an $A_e$-centrally extended orbit is determined by a pair $(A'\subset A_e,\psi\in H^2(A',\mathbb{C}^\times))$. Here, $A$ is the stabilizer of a point in the orbit and $\psi$ records the centrally extended structure. Let $\cC(A)$ be the set of all pairs $(A',\psi)$. Let $A_e^*$ be the group of characters of $A_e$. Then $A_e\times A_e^{*}$ is naturally a symplectic vector space over $\bF_2$. From \cite[Proposition 10.3]{etingof2009fusion}, we have a bijection from $\cC(A)$ to the set of Lagrangian subspaces of $A_e\times A_e^*$ given by
\begin{equation} \label{L}
    L: (A',\psi)\rightarrow \{(x,\rho)|x\in A', \psi(x, y)= \rho(y) \forall y\in A'\}.
\end{equation}
The irreducible modules of $\mathbb{Q}\otimes K_0(\text{Sh}^{A_e}(Y_e\times Y_e))$ are classified by certain pairs $(s,\rho)\in A_e\times A_e^*$ with $s\in A_e$ and $\rho(s)= 1$. For such a pair $(s,\rho)$, let $V_{(s,\rho)}$ be the corresponding irreducible module. For each pair $(s,\rho)$, following \cite[Section 10.3]{etingof2009fusion}, we have a proposition describing how to compute $\dim V_{(s,\rho)}$.
\begin{pro} \label{equality of dimensions}
    We say that an orbit $(A', \psi)$ is associated with the pair $(s,\rho)$ if the set $L(A', \psi)$ defined by the map in (\ref{L}) contains $(s,\rho)$. Then the dimension of $V_{(s,\rho)}$ is the number of its associated orbits in $Y_e$. \footnote{The author thanks Roman Bezrukavnikov and Victor Ostrik for sharing the email correspondence that explains this result.}   
\end{pro}

\Cref{equality of dimensions} and the results from \Cref{subsect 5.1} give two parameterizations for the irreducibles of $J_c= K_0(\text{Sh}^{A_e}(Y_e\times Y_e))$. The two methods of constructing these simple modules are in similar fashions (see, e.g., \cite[Section 3.6]{dawydiak2023asymptotic}). The two sets of labels are the same. The next proposition shows that two irreducibles with the same label have the same dimension. This is done by relating both quantities to the numerical invariants of $Y_e$ defined at the beginning of Section 3.

\begin{pro} \label{dimensions from F_i}
    We have $\dim V_{(s,\rho)}= \dim \Hom_{A_e}(\rho, H^*(\spr^s))$.
\end{pro}

\begin{proof}
Fix $s$ and $\rho$ and let $A_\rho\subset A_e$ be the kernel of $\rho$. The number $\dim \Hom_{A_e}(\rho, K_0(\spr^s))$ can be expressed in terms of $F^{A_e}_{1}(\spr^s)$  and $F^{A_\rho}_{1}(\spr^s)$ as follows.

$$\dim \Hom_{A_e}(\rho, K_0(\spr^s))= \bigg\{
\begin{aligned}
    F^{A_e}_{1}(\spr^s) \, \text{if} \, \rho= 1\\
    F^{A_\rho}_{1}(\spr^s)- F^{A_e}_{1}(\spr^s) \, \text{if} \,   \rho \neq 1
\end{aligned}
$$
(We have written $1$ for the trivial character on $A_e$. Since the context is clear, it should not be confused with the element $1$ of $A_e$)

For any subgroup $A\subset A_e$, we have $F^{A}_{1}(\spr^s)$ $= F^{A}_s(\spr)$ from (\ref{character}). Hence, we have $F^{A_e}_{1}(\spr^s)=$ $F^{A_e}_s(\spr)$ and $F^{A_\rho}_{1}(\spr^s)- F^{A_e}_{1}(\spr^s)=$ $F^{A_\rho}_s(\spr)- F^{A_e}_s(\spr)$. The numerical invariants $F^{A_\rho}_s(\spr)$ and $F^{A_e}_s(\spr)$ have interpretations on the side of the finite set $Y_e$ as follows. $F_s^{A_e}(\spr)$ $= F_s^{A_e}(Y_e)$ is the number of orbits $(A',\psi)$ with $s\in A'$ and $\psi(s,x)= 1$ for $x\in A'$. Thus $F_s^{A_e}(\spr)$ $=\dim V_{(s,1)}$. When $\rho\neq 1$, the equality $F^{A_\rho}_s(Y_e)- F^{A_e}_s(Y_e)$ $= \dim V_{(s,\rho)}$ is verified in a similar way.
\end{proof}

We then have an important corollary.
\begin{cor} \label{vanish}
    If a character $\rho$ does not appear in $H^*(\spr^s)$, the set $Y_e$ does not contain any orbit associated with the pair $(s,\rho)$.
\end{cor}

\begin{rem} \label{dimensions are known}
    From the proof of Proposition 5.3, we have obtained the description of the dimensions of the irreducible modules of $J_c$ in terms of $\{F_a^{A'}(\spr)\}_{a\in A'\subset A_e}$. We have an algorithm to compute these numbers from \Cref{recursive}. Thus, when the two-sided cell $c$ corresponds to a distinguished nilpotent orbit, the structure of $J_c\otimes \bQ$ as a direct sum of matrix algebras is known. In particular, we can compute the cardinality of the corresponding two-sided cells for distinguished nilpotents $e$ (when $e$ is not distinguished, the two-sided cell is not finite).

\end{rem}
\subsection{Possible types of orbits in $Y_{(2x_1, 2x_2, 2x_3, 2x_4)}$} \label{4 row Uniqueness}
In \Cref{subsect 3.1} we have seen that for an element $a\in A_e\subseteq Q_e$, the fixed point locus $\spr^a$ is realized as the disjoint union of several copies of the product of two smaller Springer fibers $\mathcal{B}_{e_1}\times \mathcal{B}_{e_2}$. Since $A_e$ is commutative, we have an action of $A_e$ on $\spr^a$. Consequently, we have an $A_e$-action on $K_0(\spr^a)$. From the analysis in \Cref{A_e action}, this action of $A_e$ is naturally identified with the action of $A_{e_1}\times A_{e_2}$ on $K_0(\mathcal{B}_{e_1}\times \mathcal{B}_{e_2})$, which is $K_0(\mathcal{B}_{e_1})\times K_0(\mathcal{B}_{e_2})$ by the Kunneth formula.

The explicit Springer correspondence (see \Cref{subsect 4.1}) gives us the list of irreducible characters of $A_i$ in $K_0(\mathcal{B}_{e_i})$, $i= 1, 2$. Therefore, the characters of $A_e$ that appear in $K_0(\spr^a)$ are known. We consider the set $\{\rho\}$ of $A_e$-characters that do not appear in $K_0(\spr^a)$. \Cref{vanish} implies the absence of orbits associated with $(a,\rho)$ in $Y_e$. We now demonstrate in examples how this relation works.
\begin{exa} \label{Example 5.6}
Consider the case where the associated partition of $e$ is $\lambda= (2k, 2j, 2i)$. In the notation of \Cref{subsect 4.2}, let $s$ denote the multiplicity of the special orbit. Let $a_0, a_1, a_2, a_3$ and $a$ denote the number of ordinary orbits whose stabilizers are $A_e$, $\langle z_1, z_{23}\rangle$, $\langle z_2, z_{13}\rangle$, $\langle z_3, z_{12}\rangle$, $\langle z_1, z_{23}\rangle$, and $\langle z_{123}\rangle$, respectively. Let $\rho$ be the character of $A_e$ taking values $(-1,1,-1)$ on $(z_1, z_2, z_3)$. 

From the analysis in \Cref{subsect 4.2}, we have $\dim \Hom_{A_e}(\rho, K_0(\spr))= 0$, so $\dim V_{(1, (-1,1,-1))}= 0$. In terms of multiplicities of orbits, the dimension of $V_{(1, (-1,1,-1))}$ is given by $a_2+ a$, so we get $a_2= a= 0$. This gives an explanation for the fact that there are at most $4$ types of orbit that appear in $Y_e$. As we know from \Cref{subsect 4.2}, there are precisely $4$ different types of orbit in $Y_e$ when $e$ is distinguished. 
\end{exa}
We now discuss the case where the associated partition of $e$ has $4$ rows. We will demonstrate how various choices of $(s,\rho)$ give restrictions on the types of $A_e$-orbit in $Y_e$.

We consider $\lambda= (2l, 2k, 2j, 2i)$ where $l\geqslant k\geqslant j\geqslant i$. We have $A_e$ is $(\cyclic{2})^{\oplus 4}$, generated by the elements $z_1, z_2, z_3$ and $z_4$ as in \Cref{subsect 2.1}. From the discussion at the end of \Cref{subsect 2.3}, the centrally extended action of $A_e$ factors through the centrally extended action of a smaller group $A_{e}^{'}$ generated by $z_1, z_2, z_3$. This example studies $Y_e$ as a centrally extended $A_{e}^{'}$-set. We first give the notation to work with.
\begin{notation}[Orbits and multiplicites] \label{Notation for types of $A_{e}^{'}$ orbits}
Recall that we write $z_{i_1...i_h}$ for $z_{i_1}...z_{i_h}$ in $A_e^{'}$. Hence, an element of $A_{e}^{'}$ has the form $1, z_i, z_{ij}$ or $z_{ijk}$. We think of $i$ or $ij$ or $ijk$ as the indexes of these elements. Let $\alpha$ be a set of indexes; for example, $\alpha= \{1\}$, $\alpha= \{12\}$, $\alpha= \{1, 2\}$,... We write $\bO_\alpha$ for the ordinary orbit with the stabilizer generated by the elements indexed by $\alpha$. For example, $\bO_{1,2}$ is the ordinary orbit of $A_{e}^{'}$ that is isomorphic to $A_{e}^{'}/\{1, z_1, z_2, z_{12}\}$. Any ordinary orbits can be written as $\bO_\alpha$ for some $\alpha$ (non-canonically, for example $\bO_{1,2}$ is the same as $\bO_{1,12}$). For the trivial orbit, we simply write $\bO_0$. 

We now proceed to the notation for special orbits (those having nontrivial centrally extended structure). Write $\bO(y)$ for the orbit of a point $y\in Y_e$. For $\bO(y)$ to be special, the stabilizer $A_y$ must have at least $4$. If $|A_y|= 4$, then there is a unique nontrivial Schur multiplier. Let $\bO_{\alpha}^{+}$ denote the special orbit of this type, here we still use $\alpha$ as the indexes to describe the stabilizer. For example, $\bO_{1,2}^{+}$ is the orbit stabilized by $\langle z_1, z_2\rangle$ with the corresponding nontrivial Schur multiplier.

Lastly, we consider orbits with the stabilizer is $A_e^{'}$. We have $H^2(A_{e}^{'},\bC^\times)\cong $ $\bigwedge^2 A'_e$ (see \Cref{Schurz2}). The vector space $A_{e}^{'}$ has dimension $3$ over $\bF_2$, so each nontrivial $2$-form has a one-dimensional kernel. We write $\bO_{0}^{\alpha}$ for the orbit with the Schur multiplier having $z_{\alpha}$ in the kernel. For example, $\bO_{0}^{12}$ denotes the orbit $(A_{e}^{'}, z_3^{*}\wedge(z_1^{*}+ z_2^{*}))$.

We let $a_\alpha$, $a_\alpha^+$, and $a_{0}^{\alpha}$ denote the multiplicities of the orbits $\bO_\alpha$, $\bO_\alpha^+$ and $\bO_{0}^{\alpha}$ in the finite set $Y_e$.
\end{notation}
Next, we make use of the list of characters of $A_e'$ in the representations $H^*(\spr^s)$ to detect variables that are zero.
For $s=1$, following the explicit Springer correspondence for $\lambda= (2l, 2k, 2j, 2i)$ (see \Cref{subsect 4.1}), we find that the characters of $A'_{e}$ that appear in $K_0(\spr)$ are $(1,1,1)$, $(-1,-1,1)$, $(1,-1,-1)$, $(-1,-1,-1)$, $(1,1,-1)$, and $(-1,1,1)$. Therefore, the characters that do not appear in $H^*(\spr)$ are $(-1,1,-1)$ and $(1,-1,1)$. In other words, $\dim V_{(1, (1,-1, 1))}$ $=\dim V_{(1, (-1,1,-1))}$ $=0$.

For $s= z_1$, the characters of $\langle z_2, z_3\rangle$ that do not appear in $H^*(\cB_{(2k, 2j, 2i)})$ are $(-1,1)$ (see \Cref{Example 5.6}), so we get $\dim V_{(z_1, (1,-1, 1))}=0$. Similarly, the dimensions of $V_{(z_2, (-1,1,1))}$, $V_{(z_3, (-1,1,1))}$, and $V_{(z_{123}, (-1,1,-1))}$ are $0$. From \Cref{equality of dimensions}, we get the following equalities in terms of the multiplicities of the centrally extended orbits in $Y_e$.
$$ \dim V_{(1,(-1,1,-1))}= a_{2}+ a_{13}+ a_{2, 13}+ a_{123}$$
$$ \dim V_{(1,(1,-1,1))}= a_{1}+ a_{3}+ a_{13}+ a_{1,3}$$
$$ \dim V_{(z_1,(1,-1,1))}= a_{1}+ a_{1,3}+ a_{1,2}^{+}+ a_{1,23}^{+} +a_0^{3} +a_0^{13}$$
$$ \dim V_{(z_2,(-1,1,1))}= a_{2}+ a_{2,3}+ a_{1,2}^{+}+ a_{2,13}^{+} + a_0^{3}+ a_0^{23}$$
$$ \dim V_{(z_3,(-1,1,1))}= a_{3}+ a_{2,3}+ a_{1,3}^{+}+ a_{3,12}^{+} + a_0^{2}+ a_0^{23}$$
By \Cref{vanish}, the multiplicities that appear in these five equations are all $0$. Excluding them, we are left with $12$ types of centrally extended orbits that may appear in $Y_e$. They are $\bO_{0}$, $\bO_{12}$,  $\bO_{23}$, $\bO_{1,2}$, $\bO_{1,23}$, $\bO_{3,12}$, $\bO_{12,23}$, $\bO_{0}^{1}$, $\bO_{0}^{12}$, $\bO_{0}^{123}$, $\bO_{2,3}^+$, and $\bO_{12,23}^+$. By \Cref{equality of dimensions}, the multiplicities of these orbits satisfy the following system of equation.
\begin{equation} \label{all equations}
  \left\{
    \begin{aligned}
&\dim V_{(z_1,(1,1,1))}= a_0+ a_{1,2}+ a_{1,23}+ a_0^{1}\\
&\dim V_{(z_1,(1,1,-1))}= a_{1,2}+ a_{0}^{12}\\ 
&\dim V_{(z_1,(1,-1,-1))}= a_{1,23}+ a_{0}^{123}\\
&\dim V_{(z_2,(1,1,1))}= a_0+ a_{1,2}\\
&\dim V_{(z_2,(1,1,-1))}= a_{1,2}+ a_0^{1}+ a_{12}+ a_{2,3}^{+}\\
&\dim V_{(z_2,(-1,1,-1))}= a_0^{123}+ a_{2,3}^{+}\\
&\dim V_{(z_3,(1,1,1))}= a_0+ a_{3,12}\\
&\dim V_{(z_3,(1,-1,1))}= a_0^{1}+a_{2,3}^+\\
&\dim V_{(z_3,(-1,-1,1))}= a_0^{123}+ a_{2,3}^{+}+ a_{0}^{12}+a_{3,12}\\
&\dim V_{(z_{12},(-1,-1,1))}= a_{12}+ a_{3,12}+ a_{12,23}^+ + a_0^{123}\\
&\dim V_{(z_{13},(1,1,1))}= a_0+ a_{12,23}\\
&\dim V_{(z_{23},(1,1,1))}= a_0+ a_{12,23}+ a_{23}+ a_{1,23}\\
&\dim V_{(z_{12},(1,1,1))}= a_0+ a_{12}+ a_{1,2}+ a_{12,23}+ a_{3,12}+ a_0^{12}
    \end{aligned}
  \right.
\end{equation}
This system breaks into two separate systems
\begin{equation} \label{system 1}
  \left\{
    \begin{aligned}
&a_{1,2}+ a_{0}^{12}= \dim V_{(z_1,(1,1,-1))}\\
&a_0+ a_{1,2}= \dim V_{(z_2,(1,1,1))}\\
&a_0+ a_{3,12}= \dim V_{(z_3,(1,1,1))}\\
&a_0+ a_{12,23}= \dim V_{(z_{13},(1,1,1))}\\
&a_{12}+a_{3,12}=\dim V_{(z_{12},(1,1,1))}- \dim V_{(z_1,(1,1,-1))}- \dim V_{(z_{13},(1,1,1))}
    \end{aligned}
  \right.
\end{equation}
and 
\begin{equation} \label{system 2}
  \left\{
    \begin{aligned}
&a_{1,23}+ a_{0}^{123}= \dim V_{(z_1,(1,-1,-1))}\\
&a_0^{123}+ a_{2,3}^{+}= \dim V_{(z_2,(-1,1,-1))}\\
&a_0^{1}+a_{2,3}^+= \dim V_{(z_3,(1,-1,1))}\\
&a_{23}+ a_{1,23}= \dim V_{(z_{23},(1,1,1))}- \dim V_{(z_{13},(1,1,1))}\\
&a_{12,23}^{+} + a_{0}^{123}= \dim V_{(z_{12},(-1,-1,1))}+ \dim V_{(z_1,(1,1,-1))}+ \dim V_{(z_{13},(1,1,1))}- \dim V_{(z_{12},(1,1,1))}
    \end{aligned}
  \right.
\end{equation}
We now have two systems of linear equations, each having $6$ variables and $5$ equations. We explain how the numerical invariants $\{S_i^{A'}\}$ will help us to determine the multiplicities. Since $A'_e$ is $(\cyclic{2})^{\oplus 3}$, the dimension of a genuine projective representation of any subgroup $A'\subset A'_e$ is $2$. Hence, the cokernel of the forgetful map $F_{Y_e}^{A'}:$ $ K_0(\Sh^{A'}(Y_e))\rightarrow $ $K_0(Y_e)^{A'}$ is a $2$-torsion group. In other words, the cokernel of $F_{Y_e}^{A'}$ takes the form $(\cyclic{2})^{\oplus S_1^{A'}}$. Therefore, the number $S_1^{A'}$ counts the number of special orbits with respect to the centrally extended action of $A'$ on $Y_e$. Let $A'$ vary, we get the following relations:
\begin{equation} 
    \begin{cases}
    S_1^{A'_e}= a_{12,23}^+ + a_{2,3}^+ + a_0^{1}+ a_0^{12}+ a_0^{123}\\
    S_1^{1,2}=a_0^{1}+ a_0^{12}+ a_0^{123}\\
    S_1^{1,3}= a_0^{12}+ a_0^{123}\\
    S_1^{2,3}= 2a_{2,3}^+ + a_0^{1}+ a_0^{12}+ a_0^{123}\\
    S_1^{12,3}= a_0^{1}\\
    S_1^{1,23}= a_0^{12}\\
    S_1^{2,13}=a_0^{1}+a_0^{12}\\
    \end{cases}
\end{equation}
In the above system, for notational convenience, we use the indexes of the generators to refer to $S^{A'}_1$. For example, $S^{1,2}_1$ is $S^{\langle z_1, z_2\rangle}_1$.
Solving this system of equations, we then get 
\begin{equation}  
    \begin{cases}
    a_{12,23}^+= S_1^{A_e'}- \frac{1}{2}(S_1^{2,3}+ S_1^{1,2})\\
    a_0^{123}= S_1^{1,3}- S_1^{1,23}\\
    a_{2,3}^+ = \frac{1}{2}(S_1^{2,3}- S_1^{1,2})\\
    a_0^{1}= S_1^{12,3}\\
    a_0^{12}= S_1^{1,23}\\   
    \end{cases}
\end{equation}
Therefore, we obtain all the multiplicities of the special orbits from $\{S^{A'}\}$, and the structure of $Y_e$ is determined from the remaining equations of the two systems (\ref{system 1}) and (\ref{system 2}). In summary, the centrally extended structure of $Y_e$ is determined by the numerical invariants $\{S_i^{A'}, F_a^{A'}\}$. 

In \Cref{subsect 3.1}, we have defined the notion of finite models of an $A$-variety $X$. Consider $A= A_e$ and $X= \spr$. Let $Y'$ be an arbitrary $A_e$-finite model of $\spr$, then $Y'$ share the same set of numerical invariants with $\spr$ and $Y_e$. Using the same argument, we can obtain the structure of $Y'$ as we did for $Y_e$ from the same set of numerical invariants. Thus, we have the following corollary.
\begin{cor} \label{unique}
    When the associated partition of $e$ has up to $4$ rows. The centrally extended set $Y_e$ is the unique $A_e$-finite model of $\spr$.
\end{cor}
In Section 8, the geometric analysis of the Springer fibers will show that $a_{2,3}^{+}=0$ and $a_{0}^{12}= 0$ (see \Cref{geometric main 4 rows} and \Cref{types of orbits}). With this information, two systems (\ref{system 1}) and (\ref{system 2}) become two linear systems of five equations in five variables. Now we can state the main theorem of this section.
\begin{thm} \label{can solve equations}
    The $A_e$-centrally extended structure of $Y_e$ is determined by the numerical invariants $\{F_a^{A'}\}$ when the associated partition of $e$ has $4$ rows. Hence, we have an algorithm to compute the multiplicities of the types of orbit in $Y_e$.
\end{thm}
\begin{proof}
    Plugging in $a_{0}^{12}= 0$ in the system (\ref{system 1}), we solve for the multiplicities and get
    \begin{equation}
  \left\{
    \begin{aligned}
&a_{1,2}= \dim V_{(z_1,(1,1,-1))}\\
&a_0= \dim V_{(z_2,(1,1,1))}- \dim V_{(z_1,(1,1,-1))}\\
&a_{3,12}= \dim V_{(z_3,(1,1,1))}-(\dim V_{(z_2,(1,1,1))}- \dim V_{(z_1,(1,1,-1))})\\
&a_{12,23}= \dim V_{(z_{13},(1,1,1))}- (\dim V_{(z_2,(1,1,1))}- \dim V_{(z_1,(1,1,-1))})\\
&a_{12}=\dim V_{(z_{12},(1,1,1))}- 2\dim V_{(z_1,(1,1,-1))} - \dim V_{(z_{13},(1,1,1))}- \dim V_{(z_3,(1,1,1))} \\& \quad \quad +\dim V_{(z_2,(1,1,1))}
    \end{aligned}
  \right.
\end{equation}
Plugging in $a_{2,3}^{+}= 0$ in the system (\ref{system 2}), we solve for the multiplicities and get
\begin{equation}
  \left\{
    \begin{aligned}
&a_{1,23}= \dim V_{(z_1,(1,-1,-1))}- \dim V_{(z_2,(-1,1,-1))}\\
&a_0^{123}= \dim V_{(z_2,(-1,1,-1))}\\
&a_0^{1}= \dim V_{(z_3,(1,-1,1))}\\
&a_{23}= \dim V_{(z_{23},(1,1,1))}- \dim V_{(z_{13},(1,1,1))}- (\dim V_{(z_1,(1,-1,-1))}- \dim V_{(z_2,(-1,1,-1))})\\
& a_{12,23}^{+}= \dim V_{(z_{12},(-1,-1,1))}- \dim V_{(z_{12},(1,1,1))} + \dim V_{(z_1,(1,1,-1))}+ \dim V_{(z_{13},(1,1,1))}\\& \quad \quad - \dim V_{(z_2,(-1,1,-1))}
    \end{aligned}
  \right.
\end{equation}
    As mentioned in \Cref{dimensions are known}, the dimensions of the irreducible representations $V_{(s,\rho)}$ can be calculated from the numerical invariants $\{F_a^{A'}\}$. Therefore, the theorem is proved. 
\end{proof}

\section{The variety $\fix$ and its properties}
We first introduce a subvariety of $\spr$ as follows. According to the Jacobson-Mozorov theorem, there exists an $\mathfrak{sl}(2)$-triple $(e,f,h)\in \mathfrak{g}$. Let $T_e$ be the image of the embedding $\mathbb{C}^\times \rightarrow G$ sending $t$ to $t^h$. The adjoint action of $T_e$ on $\mathfrak{g}$ rescales $e$, so $T_e$ acts on $\mathcal{B}_e$. Let $\fix$ be the fixed-point locus of the $T_e$-action on $\mathcal{B}_e$. While $\spr$ is connected and not irreducible, $\fix$ is smooth and has many components.

As we have seen in \Cref{4 row Uniqueness}, the numerical invariants$\{F_{a}^{A'}$ and $\{S^{A'}\}$ are crucial in determining the centrally extended structure of $Y_e$. The main theorem of this section is that $\spr$ and $\fix$ share the same set of numerical invariants $\{F_a^{A'}\}$ and $\{S^{A'}\}$. We then shift the focus to this subvariety in later sections.

\subsection{Integral localization theorem} \label{subsect 6.1}
We first give an exposition of an integral version of the localization theorem in equivariant K-theory following \cite[page 75]{intloc}.

All the schemes here are over $\Spec(\mathbb{C})$. Let $D$ be a commutative reductive algebraic group. Let $M= \Hom(D, \mathbb{G}_m)$, the abelian group of characters of $D$. Consequently, the representation ring of $D$ is $R(D)= \mathbb{Z}M$. For a prime ideal $\mathfrak{p}\subset \mathbb{Z}M$, let $M_\mathfrak{p}$ be the maximal subgroup of $M$ such that $\mathfrak{p}$ belongs to the image of the natural map $\Spec(\mathbb{Z}[M/M_p])\rightarrow \Spec(\mathbb{Z}M)$. A concrete description of $M_\fp$ is that it consists of elements $m\in M$ such that $m- 1\in \mathfrak{p}$. Let $D_\mathfrak{p}$ be the unique closed subgroup of $D$ such that the inclusion $D_\mathfrak{p}\hookrightarrow D$ induces the quotient map $\mathbb{Z}M\rightarrow \mathbb{Z}[M/M_\mathfrak{p}]$. 

In this section, for notational convenience, we will simply write $K$ instead of $K_0$ for the $0$-th Grothendieck group.
\begin{thm} \label{integral loc}
Consider a $\mathbb{C}$-variety $X$ with a $D$-action on it. Let $X^{(\mathfrak{p})}$ be the $D_\mathfrak{p}$-fixed point scheme of $X$ with its reduced scheme structure. Consider the embedding $i_X: X^{(\mathfrak{p})}\rightarrow X$, then 
$$(i_X)_*: K^D(X^{(\mathfrak{p})})_\mathfrak{p} \rightarrow K^D(X)_\mathfrak{p}$$
is an isomorphism of $R(D)_\mathfrak{p}$-modules. Here the subscript $\mathfrak{p}$ denotes the localization of the $R(D)$-modules at the prime $\mathfrak{p}$.
\end{thm}

\begin{rem} \label{relate localizations}
\Cref{CG_localization} is a consequence of \Cref{integral loc} in the following sense. First, consider $\bC M= \bC\otimes_\bZ R(D)$ instead of $\bZ M$. For each point $x\in D$, we have the maximal ideal $\fp_x\subset \bC M$ corresponding to $x$. Localizing at this prime ideal $\fp_x$ gives us the result of \Cref{CG_localization}. 
\end{rem}
We now specify the context to apply \Cref{integral loc}. Let $D$ be an abelian subgroup of $T_e\times A_e$ $= \bC^\times \times A_e$ that contains $T_e$. Then $D$ has the form $\bC^\times \times A$ with $A= (\cyclic{2})^{\oplus m}$ for some $m\leqslant k$. We have $M= \bZ\oplus (\cyclic{2})^{\oplus m}$, and $R(D)= R(\bC^\times)\otimes R(A)= \mathbb{Z}[v^{\pm 1}]\otimes R(A)=  \mathbb{Z}M$. In other words, $R(D)= \mathbb{Z}[v^{\pm 1}, \chi_1,...,\chi_m]/(\chi_i^2-1)_{i=1,\ldots,m}$. The prime ideal $\mathfrak{p}$ we consider here is the ideal generated by $\chi_1- 1,...,\chi_m- 1$, and $2$.

Next, we show that with this choice of $\fp$, the group $D_\fp$ in the statement of \Cref{integral loc} is the torus $T_e$. An equivalent task is to identify $M_\mathfrak{p}$ with $(\mathbb{Z}/2\mathbb{Z})^{\oplus m}$, embedded in $M$ as $\{1\}\times A\rightarrow T_e\times A$. 

Recall that $M_\mathfrak{p}$ consists of elements $m\in M$ such that $m- 1\in \mathfrak{p}$. Therefore, it follows from $\chi_i-1\in \fp$ for $1\leqslant i\leqslant m$ that $A\subset M_\fp$. Next, assume that $M_\mathfrak{p}$ has an element of the form $(a,a_1,...,a_m)\in \mathbb{Z}\oplus (\mathbb{Z}/2\mathbb{Z})^{\oplus m}$ with $a\neq 0$, then $v^a\prod \chi_i^{a_i}- 1\in \mathfrak{p}$. And from $\chi_i -1\in \mathfrak{p}$, we have $\prod \chi_i^{a_i}-1\in \mathfrak{p}$, so $v^a- 1\in \mathfrak{p}$, this is a contradiction.

Thus $M_\mathfrak{p}= (\mathbb{Z}/2\mathbb{Z})^{\oplus m}$, so $M/M_\mathfrak{p}= \mathbb{Z}$. We get that $D_\mathfrak{p}$ is $T_e\times \{1\}\subset T_e\times A$. Applying \Cref{integral loc}, we obtain an isomorphism 
\begin{equation} \label{localize TA}
    i_*: K^{T_e\times A}(\fix)_\mathfrak{p}\xrightarrow{\sim}K^{T_e\times A}(\mathcal{B}_e)_\mathfrak{p}.
\end{equation}
In addition, we consider $D= T_e$. Localizing at the prime $\mathfrak{q}=(2)\subset \mathbb{Z}[v^{\pm}]$ gives us an isomomorphism
\begin{equation} \label{localize T}
    i_*:K^{T_e}(\fix)_{(2)}\xrightarrow{\sim} K^{T_e}(\mathcal{B}_e)_{(2)}.
\end{equation}
Since the actions of $A$ and $T_e$ on $\spr$ commute, $A$ acts on both sides. The isomorphism in (\ref{localize T}) is $A$-equivariant by construction.
\subsection{Applications}
We have obtained the two isomorphisms (\ref{localize TA}) and (\ref{localize T}). We can relate the left (resp. right) hand sides of these two isomorphisms by forgetful maps as follows.

For an $A$-module $P$, we write $P^A$ for its invariant submodule. We have a natural action of $R(T_e)$ on $(K^{T_e}(\mathcal{B}_e))^A$. Next, letting all $\chi_i$ act trivially, we obtain an $R(D)$-module structure on $(K^{T_e}(\mathcal{B}_e))^A$. This module structure makes the forgetful map $F_1: K^{A\times T_e}(\mathcal{B}_e)\rightarrow (K^{T_e}(\mathcal{B}_e))^A$ an $R(D)$-homomorphism. We then localize at $\mathfrak{p}$ to get $F_1: K^{A\times T_e}(\mathcal{B}_e)_\mathfrak{p} \rightarrow (K^{T_e}(\mathcal{B}_e))^A_{(2)}$. A similar construction gives $F_2: K^{A\times T_e}(\fix)_\mathfrak{p}\rightarrow (K^{T_e}(\fix))^A_{(2)}$.

Note that $(K^{T_e}(\mathcal{B}_e))_{(2)}^A$ here is obtained by first taking $A$-fixed points and then taking the localization at the prime ideal $(2)$. The actions of $A$ and of $R(T_e)$ are compatible in the sense $a.(f\mathcal{F})= f(a.\mathcal{F})$ for $a\in A, f\in R(T_e)$. Hence, we can think of $(K^{T_e}(\mathcal{B}_e))_{(2)}^A$ as the $A$-fixed points of $(K^{T_e}(\mathcal{B}_e))_{(2)}$. A similar identification holds for $(K^{T_e}(\fix))_{(2)}^A$. With these identifications we have the following commutative diagram that involves two forgetful maps $F_1$ and $F_2$.
\begin{equation} \label{cokernel diagram}
    \begin{tikzcd} 
K^{A\times T_e}(\fix)_\mathfrak{p} \arrow{d}{i_*} \arrow{r}{F_1} \&  ((K^{T_e}(\fix))_{(2)})^A \arrow{d}{i_*} \\
K^{A\times T_e}(\mathcal{B}_e)_\mathfrak{p} \arrow{r}{F_2} \&  ((K^{T_e}(\mathcal{B}_e))_{(2)})^A
\end{tikzcd}
\end{equation}

In this diagram, $i_*$ on the left is the isomorphism in (\ref{localize TA}), and the isomorphism $i_*$ on the right is obtained by restricting (\ref{localize T}) to the $A$-invariant parts. Consequently, the cokernels of $F_1$ and $F_2$ are isomorphic as $R(T_e\times A)_\mathfrak{p}$-modules. Next, the actions of $\chi_i\in R(T_e\times A)_\mathfrak{p}$ on $K^{T_e}(\mathcal{B}_e)^A$ and $K^{T_e}(\fix)^A$ are trivial. Consider $\mathbb{Z}[v,v^{-1}]_{(2)}$ as the quotient of $R(T_e\times A)_\mathfrak{p}$ by the ideal generated by $\{\chi_i-1\}$. Then the cokernels of $F_1$ and $F_2$ are isomorphic as $\mathbb{Z}[v,v^{-1}]_{(2)}$-modules. We will use this to prove the following proposition.
\begin{pro}
The cokernels of the two forgetful maps $F_1^{A}: \forget{\fix}{A}$ and $F_2^{A}: \forget{\mathcal{B}_e}{A}$ are isomorphic as $\mathbb{Z}$- modules.
\end{pro}
\begin{proof}

Since $T_e$ acts trivially on $\fix$, the top row of (\ref{cokernel diagram}) becomes $$R(T_e)\otimes_\mathbb{Z}K^A(\mathcal{B}_e)_{\mathfrak{p}}\xrightarrow{F_1}R(T_e)\otimes_\mathbb{Z}(K(\fix))^A_{(2)}.$$ 
Furthermore, the map $F_1$ decomposes as $\operatorname{Id}\otimes F^A_1$. Thus, $\text{coker}F_1= (R(T_e)\otimes_\mathbb{Z} \text{coker} F^A_1)_\mathfrak{p}$. As $R(T_e)_{(2)}$-modules, we have $$\text{coker}F_1= R(T_e)_{(2)}\otimes_\mathbb{Z} \text{coker} F^A_1= \mathbb{Z}[v, v^{-1}]_{(2)}\otimes_\mathbb{Z} \text{coker} F^A_1.$$

Next, we describe $\text{coker}F_2$. We have a natural $A\times T_e$-finite centrally extended structure on $Y_e$ so that $K^{A\times T_e}(\mathcal{B}_e)= K(\text{Sh}^{A\times T_e}(Y_e))$. From \Cref{easy schur}, the pullback map $M(A)\rightarrow M(A\times T_e)$ is an isomorphism. So the $A\times T_e$-centrally extended structure on $Y_e$ is a lift from its $A$-centrally extended structure (see \Cref{centrally extended structure of a product}). Recall that this lift has the property that $T_e$ acts trivially on points of $Y_e$, and the centrally extended structure is read from the isomorphism $M(A)\rightarrow M(A\times T_e)$. Hence, we get the decomposition $K(\text{Sh}^{A\times T_e}(Y_e))= R(T_e)\otimes_\mathbb{Z} K(\text{Sh}^A(Y_e))= R(T_e)\otimes_\mathbb{Z}K^A(\mathcal{B}_e)$. We then obtain
$K^{A\times T_e}(\mathcal{B}_e)=  R(T_e)\otimes_\mathbb{Z}K^A(\mathcal{B}_e)$ and $F_2= \operatorname{Id}\otimes F^A_2$.

In summary, we get
$$\mathbb{Z}[v, v^{-1}]_{(2)}\otimes_\mathbb{Z} \text{coker} F^A_1= \text{coker }F_1=\text{coker }F_2=\mathbb{Z}[v, v^{-1}]_{(2)}\otimes_\mathbb{Z} \text{coker} F^A_2.$$
Evaluating at $v= 1$, we have $\mathbb{Z}_{(2)}\otimes_\mathbb{Z} \text{coker} F^A_1 =\mathbb{Z}_{(2)}\otimes_\mathbb{Z} \text{coker} F^A_2$. Thus, for any $l\in \mathbb{N}$, the multiplicities of the summand $\cyclic{2^l}$ in coker$F^A_{1}$ and in coker$F^A_{2}$ are the same. We then finish the proof of the proposition thanks to the following lemma.
\end{proof}

\begin{lem} \label{annihilated by cardinality}
Let $A$ be a finite group, $X$ be an $A$-variety. The cokernel of the map $F^A_{X}: K^A(X)\rightarrow K(X)^A$ is annihilated by $|A|$.
\end{lem}
\begin{proof}
We have an induction map ind: $K(X)\rightarrow K^A(X)$: $[\cF]\mapsto$ [Ind$_{\{1\}}^A\mathcal{F}$]. Note that [Ind$_{\{1\}}^A\mathcal{F}$] $=  [\oplus_{a\in A} a.\cF$] in $K(X)$. Therefore, for $\cF \in K(X)^A$, we see that [Ind$_{\{1\}}^A \cF$] $= [\cF^{\oplus |A|}$] in $K(X)^A$. In other words, the element $[\cF]+ [\cF]+...+ [\cF]$ ($|A|$ times) is in the image of $F^A_X$, the lemma is proved.
\end{proof}

An important numerical application of Proposition 6.3 is that the two sets of numbers $S_i^{A}(\spr)$ and $S_i^{A}(\fix)$ coincide for any subgroup $A$ of $A_e$. Furthermore, we can show that $\spr$ and $\fix$ share the same set of $\{F_{a}^{A}\}$ using a similar technique as follows.

For any element $a\in A$, we have $\chi_i(a)\in \{\pm 1\}$. Let $\fp$ be the ideal of $\bZ M$ generated by the elements $\chi_i \pm 1$ that vanish at $a$; we have $\bZ M/\fp= \bZ[v,v^{-1}]$ so $\fp$ is prime. In this case, $M_\fp$ is $T_e\times \{1, a\}$. Apply \Cref{integral loc} for this prime ideal $\fp$ and then tensor both localized modules with $\bC$. We get an isomorphism of $R(T_e)\times \bC_a$-modules 
$$\bC \otimes K^{T_e\times A}(\spr)_{\fp}\cong \bC\otimes K^{T_e\times A}((\fix)^a)_{\fp}.$$

Recall that $T_e$ acts trivially on both $H^*(\spr)$ and $H^*((\fix)^a)$. Hence, both modules are free over $R(T_e)= \bZ[v^\pm]$. Forgetting the $R(T_e)$-modules structure, in other words, specializing $v$ to $1$, we get $\bC\otimes K^A(\spr)_{(a)}= \bC\otimes K^A((\fix)^a)_{(a)}$. The latter module is isomorphic to $\bC\otimes (K^A(\fix))_{(a)}$ by \Cref{CG_localization}. Therefore, we have $F_{a}^A(\spr)= F_{a}^A(\fix)$.
\begin{cor} \label{same numerics}
    The two varieties $\spr$ and $\fix$ share the same set of numerical invariants $\{F_a^{A'}, S_i^{A'}\}$.
\end{cor}
The next subsection discusses finite models of $\spr$ and $\fix$.
\subsection{Finite models of $\spr$ and $\fix$}
In Section 3.1, for a reductive group $A$, we have defined $A$-finite models of an $A$-variety $X$. A finite model is an $A$-centrally extended set that is comparable to $X$ in some senses (see \Cref{finite model}). When $A=A_e$ and $X= \spr$, $Y_e$ serves as an $A_e$-finite model of $\spr$. From the results of Section 6.2 and how we define finite models, we get that any $A_e$-finite model of $\spr$ is an $A_e$-model of $\fix$ and vice versa. Therefore, $Y_e$ is an $A_e$-finite model of $\fix$. 

While the variety $\spr$ is connected, the variety $\fix$ has many connected components, and the action of $A_e$ permutes them. Let $(\fix)_\alpha$ be an $A_e$-orbit of connected components of $\fix$. It is natural to expect that each $(\fix)_\alpha$ admits a finite model $Y_e^\alpha$. We then have a conjecture that is a weaker form of the conjectures stated in \cite[Section 1.3]{lusztig2021discretization}.
\begin{conj} \label{union}
For each $A_e$ orbit $(\fix)_\alpha$ of connected components of $\fix$, there exists a finite model $Y^{\alpha}_{e}$ of $(\fix)_\alpha$ so that $\cup_{\alpha}Y_{e}^{\alpha}$ is isomorphic to $Y_e$ as $A_e$-centrally extended finite sets.
\end{conj}
\begin{rem} \label{same finite model}
When the partition of $e$ has up to $4$ rows, we have from \Cref{unique} that $Y_e$ is the unique finite model of $\spr$. Hence, $Y_e$ is the unique finite model of $\fix$. Thus, the statement of \Cref{union} reduces to the existence of $Y^e_{\alpha}$. This is done in Sections 7 and 8. The method is to use exceptional collections to construct finite models of $(\fix)_\alpha$. 

\end{rem}

\section{Exceptional collections and finte models of an $A$-variety}
\subsection{Exceptional collection, semi-orthogonal decomposition, and mutation}
For more about exceptional collections, the readers are referred to \cite[Section 2]{böhning2005derived} and \cite{Bondal_1990} for detailed expositions. In this subsection, we only present selective definitions and properties of exceptional collections that are needed in the paper.  
\subsubsection{Definitions}
Let $\mathcal{A}$ be a triangulated category.
\begin{defin}
An object $E$ in $\cA$ is called \textit{exceptional} if
$$\Hom(E,E)=\mathbb{C} \quad \text{and} \quad  \Hom(E,E[i])= 0 \quad \forall i\neq 0 $$ 
\end{defin}
\begin{defin}
An $n$-tuple $(E_1,...,E_n)$ of exceptional objects in $\cA$ is called an \textit{exceptional sequence} or an \textit{exceptional collection} if:
$$\Hom(E_j, E_l[i])= 0 \indent \indent \forall 1\leqslant l<j\leqslant n \indent \text{and} \indent \forall i\in \mathbb{Z}.$$ 
If the collection $(E_1,E_2,...,E_n)$ generates $\cA$ as a triangulated category, we say it is a \textit{complete exceptional sequence} or a \textit{full exceptional collection} .
\end{defin}
The following is a classical example of exceptional sequences in the derived category of coherent sheaves on projective spaces given by Beilinson in \cite{Beilinson1978CoherentSO}.
\begin{exa} \label{Pn}
In the category $\cA= D^b(\Coh(\mathbb{P}^n))$, we have a full exceptional collection of line bundles $(\mathcal{O}, \mathcal{O}(1),..., \mathcal{O}(n))$. 
\end{exa}
A more general notion to describe $\cA$ in terms of blocks is \textit{semi-orthogonal decomposition}. Our exposition below follows \cite[Page 19]{böhning2005derived}. 

Let $\cB$ be a full triangulated subcategory of $\cA$. The \textit{left orthogonal} of $\cB$ in $\cA$ is defined to be the full subcategory $\cC$ of $\cA$ consisting of the objects $C$ such that $\Hom(C,B)= 0$ for all $B\in \cB$. We say that a triangulated subcategory $\cB$ is \textit{left-admissible} if for arbitrary $A\in \cA$, we have a distinguished triangle 
$$C\rightarrow A\rightarrow B\rightarrow C[1]$$
for some $B\in \cB$ and $C\in \cC$. Similarly, we have the definition of a \textit{right orthogonal} subcategory and a \textit{right-admissible} subcategory of $\cA$. A full triangulated subcategory of $\cA$ is \textit{admissible} if it is both left-admissible and right-admissible. We then have the following definition.
\begin{defin} \label{SOD}
    A triangulated category $\cA$ admits a semi-orthogonal decomposition $\langle \cA_1, \cA_2,...,\cA_n \rangle$ if the following conditions are satisfied.
    \begin{enumerate}
        \item The subcategories $\cA_i$ are admissible, and they generate $\cA$.
        \item The subcategory $\cA_j$ belongs to the left orthogonal of $\cA_i$ for $1\leqslant i< j\leqslant n$. 
    \end{enumerate}   
\end{defin}

In such a decomposition, if each subcategory $\cA_i$ has a complete exceptional sequence, then $\cA$ has one. Next, we proceed to a functorial construction of \textit{mutation} following \cite[Section 2.2.1]{jiang2021derived}. Let $\cB$ be an admissible subcategory of $\cA$. Let $\cB^\perp$ and $^\perp\cB$ be the left orthogonal and right orthogonal of $\cB$ in $\cA$. Write $i_{\cB^\perp}$ and $i_{^\perp\cB}$ for the corresponding embeddings.
\begin{defin}
    The \textit{left (resp. right) mutation through $\cB$} is $L_{\cB}:= i_{\cB^\perp} i^*_{^\perp\cB}$ (resp. $R_{\cB}:= i_{^\perp\cB} i^*_{\cB^\perp}$).
\end{defin}
For an object $A\in \cA$, we write $L_A$ (resp.$R_A$) for the mutation through the full subcategory of $\cA$ generated by $A$.

We now consider the case $\cA= D^b(X)$, the bounded derived category of coherent sheaves on a smooth projective variety $X$. We introduce some properties of mutations of exceptional objects. For two objects $A, B\in D^b(X)$, there is a canonical morphism $$\operatorname{lcan}: \RHom(A,B)\otimes^\text{L} A\rightarrow B.$$
The left mutation of $B$ via $A$, $L_A(B)$ then fits in the distinguished triangle below 
\begin{equation} \label{triangle}
\RHom(A, B)\otimes A\rightarrow B\rightarrow L_A(B)\rightarrow.    
\end{equation}

When $A$ is an exceptional object, applying the functor $\RHom(A, *)$ to (\ref{triangle}), we obtain $\RHom(A, L_AB)= 0$. With a further assumption that $(A,B)$ is an exceptional sequence, the following properties were proved in \cite{gorodentsev_1990},
\begin{equation} \label{mutation property}
  \RHom(B,C)\cong \RHom(L_AB, L_AC) \indent \forall C\in D^b(X), \quad \text{and } L_{L_AB}L_AC\cong L_AL_BC.  
\end{equation}
In particular, if we have $\RHom(A, B)= 0$, then $L_{B}L_{A}C\cong L_{A}L_{B}C$.

For an exceptional sequence $(E_1,...,E_n)$ in $D^b(X)$, consider the following sequences of mutations.
$$E_i^{\vee}:= L_{E_1}L_{E_2}...L_{E_{n-i}}E_{n-i+1}$$
The collection $(E_1^{\vee},...,E_n^{\vee})$ is called the \textit{right dual} of $(E_1,...,E_n)$.
According to \cite[Proposition 2.1.12]{böhning2005derived}, we have that
\begin{equation} \label{property of dual collection}
    \Ext^k(E_i, E_j^\vee)= \begin{cases}
\mathbb{C} \text{ if } i+j= n+1, i=k+1, \\ 0 \text{ otherwise. }
\end{cases}
\end{equation}

Moreover, if $(E_1,...,E_n)$ is complete, then its right dual is complete (\cite[Lemma 2.2]{Bondal_1990}).
\begin{exa}
    In $D^b(\bP^n)$, we have a complete exceptional sequence $(\Omega^n(n),...., \Omega^1(1), \cO)$ of sheaves of differential forms twisted by $\cO(k)$. The right dual of this collection is the collection $(\cO,\cO(1)..., \cO(n))$ in \Cref{Pn}. More details are given in \cite[Example 2.1.13]{böhning2005derived}.
\end{exa}
\subsubsection{Interactions with group actions}
We discuss some interactions between group actions on a variety $X$ and exceptional collections in $D^b(X)$. Let $X$ be a smooth projective variety, and let $A$ be a finite group acting on $X$. Then $A$ acts on $D^b(X)$ by exact equivalences. We define the compatibility of a full exceptional sequence and this $A$-action as follows.
\begin{defin} \label{compatible}
    Let $\mathcal{E}$ be a full exceptional collection of $D^b(X)$. We say that this collection is \textit{compatible} with the action of $A$ on $X$ if $A$ permutes the objects of $\cE$ (up to isomorphisms).
\end{defin}
Consider a collection $\cE$ that is compatible with the $A$-action. We divide the objects of $\cE$ into blocks as follows. First, we claim that the $\Ext^\bullet$ groups between two different objects in an $A$-orbit of $\cE$ vanish. Consider an $A$-orbit $E_{i_1},...,E_{i_k}$ in $\cE$ with $i_1<...< i_k$,  then $\RHom (E_{i_l}, E_{i_1})= 0$ for $2\leqslant l\leqslant k$ as $\cE$ is exceptional. For each $j$, choose an $a_j\in A$ so that $a_j E_{i_j}= E_{i_1}$. Then for $l\neq j$, we have $\RHom(E_{i_l}, E_{i_{j}})= \RHom (a_jE_{i_l}, a_jE_{i_{j}})$ $= \RHom (a_jE_{i_l}, E_{i_1})= 0$ since $a_jE_{i_l}\cong E_{i_{l'}}$ for some $1< l'\leqslant k$.

Next, consider another $A$-orbit $E_{i'_1},...E_{i'_{k'}}$ in $\cE$ with $i_1<...< i_{k'}$. Assuming $i_1< i'_1$, arguing in a similar way as above, we then have $\RHom(E_{i'_l}, E_{i_{j}})= 0$ for any $1\leqslant l \leqslant k'$, $1\leqslant j\leqslant k$. Hence, given a compatible full exceptional collection $\cE$, we can arrange the objects of $E$ in blocks:
$$\mathcal{E}= \begin{pmatrix}
E_1^{(1)} & E_1^{(2)} & & & E_1^{(n)}\\
. & . & & & .\\
. & . & & & .\\
. & . &. &. & .\\
E_{k_1}^{(1)} & E_{k_2}^{(2)} & & & E_{k_n}^{(n)}
\end{pmatrix}.$$
Here each block is an $A$-orbit, the linear order of two objects in the same block is given by their subscripts. The linear order of two objects in two different blocks is given by the order of the superscripts. The next proposition shows that this block decomposition is compatible with the process of mutations.

\begin{pro} \label{dual is compatible}
The right dual of $\mathcal{E}$ admits a block decomposition with respect to the action of $A$,
$$\mathcal{E}^\vee= 
\begin{pmatrix}
{E_{k_n}^{(n)}}^\vee & {E_{k_{n-1}}^{(n-1)}}^\vee & & & {E_{k_1}^{(1)}}^\vee\\
. & . & & & .\\
. & . & & & .\\
. & . &. &. & .\\
{E_{1}^{(n)}}^\vee & {E_{1}^{(n-1)}}^\vee & & & {E_{1}^{(1)}}^\vee
\end{pmatrix}.
$$
Morever, for an element $a\in A$, if $aE_{j_1}^{(i)}= E_{j_2}^{(i)}$ in the $i$-th block, then $a {E_{j_1}^{(i)}}^\vee= {E_{j_2}^{(i)}}^\vee$.
\end{pro}
\begin{proof}
Consider the $i$-th block of $\mathcal{E}$. For $1\leqslant j_1\neq j_2\leqslant k_i$, we have $\RHom(E_{j_1}^{(i)},E_{j_2}^{(i)})= 0$ (see the argument after \Cref{compatible}). By \Cref{mutation property}, $L_{E_l^{(i)}}E_{j}^{(i)}= E_{j}^{(i)}$ and the mutations $L_{E^{(i)}_j}$ commute (up to isomorphisms) for $1\leqslant j\leqslant k_i$. Hence, it makes sense to write $L_i$ for the sequence of mutations $L_{E_1^{(i)}}L_{E_2^{(i)}}...L_{E_{k_i}^{(i)}}$.

From the definition of mutation, we see that the action of the group $A$ commutes with the mutation: $aL_C(B)\cong L_{aC}(aB)$. Thus for any $B\in D^b(X)$, we have
$$a(L_iB)= a(L_{E_1^{(i)}}L_{E_2^{(i)}}...L_{E_{k_i}^{(i)}}B)= L_{aE_1^{(i)}}L_{aE_2^{(i)}}...L_{aE_{k_i}^{(i)}}(aB)= L_{E_1^{(i)}}L_{E_2^{(i)}}...L_{E_{k_i}^{(i)}}(aB)= L_i(aB).$$
The last equality follows from that $a$ permutes the objects of the block $\{E^{(i)}_j| 1\leqslant j\leqslant k_i\}$ and $L_{E^{(i)}_j}$ commute. Recall that the dual object ${E_{j}^{(i)}}^\vee$ is obtained by applying $L_1L_2...L_{i-1}L_{E_1^{(i)}}...L_{E_{j-1}^{(i)}}$ to $E_{j}^{(i)}$. As noted above, $L_{E_l^{(i)}}E_{j}^{(i)}= E_{j}^{(i)}$ for any $l\neq j$; thus ${E_{j}^{(i)}}^\vee= L_1L_2...L_{i-1}E_{j}^{(i)}$. And $a{E_{j}^{(i)}}^\vee= aL_1L_2...L_{i-1}E_{j}^{(i)}= L_1L_2...L_{i-1}(aE_{j}^{(i)})= (aE_{j}^{(i)})^\vee$.
\end{proof}
\subsection{From full exceptional collections to finite models}
\subsubsection{Cocycles, twisted representations, and twisted sheaves}
This subsection provides the necessary terminology for the next subsections. We follow the notation from \cite{Equi}.

Let $G$ be a finite group. Let $\alpha$ be a $2$-cocycle of $G$ with coefficient in $\mathbb{C}^\times$. Let $V$ be a vector space. An $\alpha$-\textit{representation} structure on $V$ is a map $R: G\rightarrow GL(V)$ such that $R(g)R(h)=\alpha(g,h)R(gh)$ for any $g,h\in G$.

Let $X$ be an algebraic variety over $\mathbb{C}$ acted on by $G$. An $\alpha$-$G$-\textit{equivariant sheaf} on $X$ is a coherent sheaf $\mathcal{F}$ together with a set of isomorphisms $\theta_g: \mathcal{F}\rightarrow g^* \mathcal{F}$ for all $g\in G$ such that $\alpha(g,h)\theta_{gh}= h^*{\theta_g}\circ \theta_h$. In other words, $\mathcal{F}$ has the structure of an $\alpha$-$G$-equivariant sheaf means that there exists an $\alpha^{-1}$-representation $V$ so that $\mathcal{F}\otimes V$ can be equipped with a $G$-equivariant structure. When the group is fixed, we will call $\mathcal{F}$ an $\alpha$-sheaf for convenience.

An $\alpha$-representation $V$ has the structure of an honest representation of $G$ if and only if the class of $\alpha$ in $H^2(G,\mathbb{C}^\times)$ is trivial. Consider two cocycles $\alpha$ and $\beta$ with the same cohomology class. Then a sheaf $\mathcal{F}$ admits an $\alpha$-sheaf structure if and only if it admits a $\beta$-sheaf structure.

\subsubsection{Twisted equivariant structure on exceptional objects} 

Suppose that $E$ is an exceptional coherent sheaf on $X$ and $g^*E\cong E$. According to \cite[Section 2]{Equi}, we can equip $E$ with an $\alpha$-sheaf structure for some $2$-cocycle $\alpha$ of $G$ as follows. Fix an isomorphism $\theta_g: E\rightarrow g^*E$ for every $g$. Since $E$ is exceptional, the two isomorphisms $\theta_{gh}$ and $h^*(\theta_g)\circ \theta_h$ : $E\rightarrow (gh)^*E$ differ by a nonzero scalar $\alpha(g,h)$. In this way, we obtain a $2$-cocyle $\alpha$ of $G$. The isomorphisms $\theta_g$ define an $\alpha$-sheaf structure on $E$.

Similarly, assume that we have a $G$-oribt of exceptional coherent sheaves $E_1,..., E_k$. Let $H$ be the stabilizer of the isomorphism class of $E_1$. The same argument shows that there is an $\alpha_H$-sheaf structure on $E_1$ for some $2$-cocyle $\alpha_H$ of $H$. Let $V$ be an $\alpha_H^{-1}$-representation of $H$. Then $E_1\otimes V$ is a $H$-equivariant sheaf, and $\Coind_{H}^G(E_1\otimes V)$ is a $G$-equivariant sheaf. If we forget the $G$-equivariant structure, the latter object is just $(E_1\oplus E_2\oplus... \oplus E_k)\otimes V$. In the next section, equivariant sheaves on $X$ are produced in this way.
\subsubsection{Categorical finite models}
The main result of this subsection is a corollary of \cite[Theorem 2.3]{Equi}. We first cite the theorem here. Let $X$ be a smooth projective variety acted on by a finite group $G$. Assume that we have a full exceptional collection $\cE$ of $D^b(X)$ that is compatible with the $G$-action. By the discussion after \Cref{compatible}, we have a block decomposition of $\cE$ of the form
$$\mathcal{E}= \begin{pmatrix}
E_1^{(1)} & E_1^{(2)} & & & E_1^{(n)}\\
. & . & & & .\\
. & . & & & .\\
. & . &. &. & .\\
E_{k_1}^{(1)} & E_{k_2}^{(2)} & & & E_{k_n}^{(n)}
\end{pmatrix},$$
in which each column consists of objects in a single $G$-orbit. 
\begin{thm} \label{main Elagin}
Let $H_i$ be the stabilizer of the isomorphism class of $E_1^{(i)}$ in the above block decomposition. Make $E_1^{(i)}$ an equivariant sheaf with respect to some $2$-cocycle $\alpha_i$ of the group $H_i$. Denote the $\alpha_i$-$H_i$-sheaf we obtain by $\mathcal{E}^{(i)}$. Let $V_1^{(i)},..., V_{k_i}^{(i)}$ be the pairwise non-isomorphic irreducible $\alpha_i^{-1}$-representations of $H_i$ over $k$. Then the collection
$$\mathcal{E}_A= \begin{pmatrix}
\Coind_{H_1}^G(\mathcal{E}^{(1)}\otimes V_1^{(1)}) & \Coind_{H_2}^G(\mathcal{E}^{(2)}\otimes V_1^{(2)}) & & & \Coind_{H_n}^G(\mathcal{E}^{(n)}\otimes V_1^{(n)})\\
. & . & & & .\\
. & . & & & .\\
. & . &. &. & .\\
\Coind_{H_1}^G(\mathcal{E}^{(1)}\otimes V_{k_1}^{(1)}) & \Coind_{H_2}^G(\mathcal{E}^{(2)}\otimes V_{k_2}^{(2)}) & & & \Coind_{H_n}^G(\mathcal{E}^{(n)}\otimes V_{k_n}^{(n)})
\end{pmatrix}$$
is a full exceptional collection in the category $D^G(X)$. 
\end{thm}
Recall that we have defined finite models of a $G$-variety (\Cref{finite model}). A direct consequence of this theorem is the following corollary. 
\begin{cor} \label{categorical finite model}
With the same hypothesis for $X$ and $G$, we have a finite model $Y_\cE$ of $X$ as follows. This set $Y_\cE$ consists of points $e_1^{(1)},..., e_{k_1}^{(1)}, $ $e_1^{(2)},..., e_{k_2}^{(2)}$ $,..., e_1^{(n)},..., e_{k_n}^{(n)}$. Each subset $\mathbb{O}_i= \{e_1^{(i)},..., e_{k_i}^{(i)}\}$ is a $G$-orbit with the stabilizer $H_i$, and the centrally extended structure on this orbit is given by the $2$-cocycle $\alpha_i^{-1}$. We then have a natural bijection between the objects of $\cE_A$ and the classes of simple objects of $\Sh^A(Y_\cE)$.
\end{cor}
\begin{rem}
Recall from Section 7.2.2 that the cocycle assigned to a $G$-invariant exceptional object $E$ was defined up to a choice of a set of isomorphisms $\theta_g: E\rightarrow g^*E$, $g\in G$. As we rescale each of these isomorphisms $\theta_g$ by a scalar $b_g$, the induced cocyle is changed by the image of the $1$-coboundary $f: G\rightarrow \mathbb{C}^\times$, $f(g)= b_g$. Thus, once we fix an exceptional collection $\mathcal{E}$ of $X$, the centrally extended set $Y_\cE$ in \Cref{categorical finite model} is independent of the choice of the cocycles $\alpha_i$.
\end{rem}
With this remark, it makes sense to define $Y_\cE$ as the finite model associated with $\cE$.
\begin{defin}
We call $Y_\cE$ defined in \Cref{categorical finite model} the \textit{categorical finite model} of $X$ with respect to the exceptional collection $\mathcal{E}$. 
\end{defin}

\begin{rem}
There is a more general version of \Cref{main Elagin} in which $G$ is a reductive group (see \cite[Theorem 2.12]{Equi}). This theorem requires some slight technical changes for the setting; the readers who are interested can find the details in \cite[Section 2]{Equi}. 

In summary, the author defines the notions of cocycles, twisted representations, and twisted $G$-equivariant sheaves so that they make sense for reductive groups $G$. Then we have a version of \Cref{main Elagin} for $G$ reductive. In particular, the construction of categorical finite models in \Cref{categorical finite model} works well for the case $G$ is reductive, provided that $X$ has a full exceptional collection that is compatible with the action of $G$.

\end{rem}
\subsubsection{Categorical finite models of projective bundles and blow-ups}
\Cref{categorical finite model} gives us a method to produce finite models from exceptional collections that are compatible with group actions. In this subsection, we make use of classical results on derived categories of smooth varieties to construct finite models of equivariant projective bundles and equivariant blow-ups.

Let $A$ be a reductive group. Let $X$ be an $A$-variety. Assume that $D^b(X)$ has a full exceptional collection $\cE$ that is compatible with the action of $A$. Let $Y_\mathcal{E}$ be the corresponding categorical finite model. Let $\cV$ be an $A$-equivariant vector bundle of rank $r$ over $X$. Consider the corresponding projective bundle $\mathbb{P}(\cV)\xrightarrow{\pi} X$.
\begin{lem} \label{projective bundle}
The category $D^b(\mathbb{P}(\cV))$ has a full exceptional collection $\langle\ \mathcal{O}(-r+ 1)_{\mathbb{P}(E)/X}\otimes \pi^*\mathcal{E},..., \mathcal{O}(-1)_{\mathbb{P}(E)/X}\otimes \pi^*\mathcal{E}, \pi^*\mathcal{E} \rangle$. This collection is compatible with the $A$-action on $\mathbb{P}(\cV)$. The corresponding finite model of $\mathbb{P}(E)$ is the disjoint union of $r$ copies $Y_\mathcal{E}$.
\end{lem}
\begin{proof}
The first assertion is a direct consequence of \cite[Corollary 2.7]{projbundle} in which we have a semi-orthogonal decomposition $D^b(\bP(E))= \langle\ \mathcal{O}(-r+ 1)_{\mathbb{P}(E)/X}\otimes \pi^*D^b(X),..., \mathcal{O}(-1)_{\mathbb{P}(E)/X}\otimes \pi^*D^b(X), \pi^*D^b(X) \rangle$. The second statement follows from the construction of the corresponding categorical finite model in \Cref{categorical finite model}.
\end{proof}
A simple example is as follows. Let $V$ be a $r$-dimensional representation of a group $A$, then $A$ naturally acts on the projective space $\bP(V)$. The categorical finite model corresponding to the collection $\langle \cO(1-r),...,\cO(-1), \cO \rangle$ is the set of $r$ points with trivial $A$-action.
\begin{notation}
In Section 8, we will often realize the connected components of $\fix$ as towers of projective bundles over some varieties. We use the notation $(\bP^{i_1},\bP^{i_2},...,\bP^{i_l}, X)$ for a tower of projectizations of vector bundles over $X$. For example, $(\bP^1, \bP^2)$ denotes a projective $\bP^1$-bundle over $\bP^2$; $(\bP^2, \bP^1, X)$ denotes a projective $\bP^2$-bundle over a projective $\bP^1$-bundle over $X$.

In an equivariant setting, when we write $(\bP^{i_1},\bP^{i_2},...,\bP^{i_l}, X)$, it is implicitly understood that each layer of the tower comes from an equivariant vector bundle. A corollary of \Cref{projective bundle} is that we can construct a categorical finite model of $(\bP^{i_1},\bP^{i_2},...,\bP^{i_l}, X)$ by taking $\prod_{j=1}^{l}(i_j+1)$ copies of a categorical finite model of $X$.  
\end{notation}
We proceed to categorical finite models of equivariant blow-ups. Let $Z$ be a smooth subvariety of a smooth projective variety $X$. Let $\tilde{X}$ denote the smooth projective variety obtained by blowing up $X$ along the center $Z$. Consider the Cartesian square
$$
\begin{tikzcd}
\tilde{Z} \arrow{r}{j} \arrow{d}{p} \& \tilde{X} \arrow{d}{\pi}\\
Z \arrow{r}{i} \& X
\end{tikzcd}
$$
We consider an action of a reductive group $A$ on $X$ and assume that $Z$ is stable under this action. Then we have a natural action of $A$ on $\tilde{X}$. The map $\tilde{Z}\xrightarrow{p} Z$ is the projection from an $A$-linearizable projective bundle of rank $r-1$ with $r= \text{Codim}_Z(X)$. 

We assume that both $X$ and $Z$ admit exceptional collections $\mathcal{E}_X, \mathcal{E}_Z$ that are compatible with the actions of $A$. Let $Y_{\mathcal{E}_X}$ and $Y_{\mathcal{E}_Z}$ be the corresponding categorical finite models of $X$ and $Z$.
\begin{lem} \label{finite model blow up}
 The blow-up variety $\tilde{X}$ admits a full exceptional collection $\langle j_*(\mathcal{O}(-r+ 1)_{\tilde{Z}/Z}\otimes p^*\mathcal{E}_Z),...,$ \\ $ j_*(\mathcal{O}(-1)_{\tilde{Z}/Z}\otimes p^*\mathcal{E}_Z), \pi^*\mathcal{E}_X \rangle$. This collection is compatible with the $A$-action on $\tilde{X}$. The corresponding finite model of $\tilde{X}$ is then realized as the disjoint union of $Y_{\mathcal{E}_X}$ and $r- 1$ copies of $Y_{\mathcal{E}_Z}$.
\end{lem}

\begin{proof}
The first assertion is a direct consequence of Corollaries 4.3 and 4.4 in \cite{projbundle} in which we have a semi-orthogonal decomposition of $D^b(\tilde{X})$ by $\langle j_*(\mathcal{O}(-r+ 1)_{\tilde{Z}/Z}\otimes p^*D^b(Z)),..., j_*(\mathcal{O}(-1)_{\tilde{Z}/Z}\otimes p^* D^b(Z)),$ $\pi^*D^b(X) \rangle$. The second statement follows from the construction of categorical finite models in \Cref{categorical finite model}.
\end{proof}

In Section 8, we have a process of describing the geometry of connected components of $\fix$ as follows. We start with a base variety $X$, then in each step we build a tower of projective bundles or blow-up along a smooth subvariety. Therefore, \Cref{projective bundle} and \Cref{finite model blow up} are crucial in understanding categorical finite models of components of $\fix$. The last ingredient is categorical finite models of base varieties, which are often realized as quadrics and towers of quadric fibrations. This topic is discussed in the next subsection.
\subsection{Spin representations and spinor bundles on quadrics}
This subsection discusses categorical finite models of quadrics and quadric fibrations with respect to suitable actions of orthogonal groups $O_n$. We will explain how nontrivial central extension naturally arises in these cases.

In \Cref{subsect 2.2}, we have mentioned that $O_n$ has a unique nontrivial Schur multiplier and $M(O_n)= \cyclic{2}$. Here, we first interpret the unique Schur multiplier of $O_n$ in terms of representation theory. 
\subsubsection{The spin representations} \label{subsect 7.3.1}
Consider an $n$-dimensional vector space $V$ over $\bC$ and a nondegenerate quadratic form $Q$ on $V$. Let $T(V)$ be the tensor algebra of $V$. Let $I_Q$ be the two-sided ideal of $T(V)$ generated by elements of the form $v\otimes v- Q(v),$ $v\in V$. The \textit{Clifford algebra}, $\text{Cl}_n(V)$, is defined as $T(V)/I_Q$. The dimension of $\Clif_n(V)$ is $2^n$.

When $n$ is even, $\Clif_n(V)$ is isomorphic to the matrix algebra $M_{2^\frac{n}{2}}$. When $n$ is odd, $\Clif_n(V)$ is isomorphic to the direct sum of two matrix algebras $M_{(2^\frac{n-1}{2})}\oplus M_{(2^\frac{n-1}{2})}$. Thus, we can write $\Clif_n(V)$ as $\End(W)$ or $\End(W_1)\oplus \End(W_2)$ for some vector spaces $W, W_1, W_2$. Let $\text{Pin}(V)$ be the group that consists of the elements of $\Clif_n(V)$ that have the form $v_1v_2...v_k$ with $v_i\in V, Q(v_i)= 1$. Then $W$ (or $W_1, W_2$) is an irreducible representation of $\text{Pin}(V)$, we call this the \textit{spin representation}. 

It is well known that $\text{Pin}(V)$ is a double cover of $O(V)$, and the identity component of $\text{Pin}(V)$ is a double cover of $\SO(V)$. This identity component is denoted by $\text{Spin}(V)$, the \textit{spin group}. When $n$ is odd, $W_1$ (isomorphic to $W_2$) is an irreducible representation of $\text{Spin}(V)$. When $n$ is even, we have a direct sum decomposition of $W$ into two irreducible $\text{Spin}(V)$-modules $W= W^+\oplus W^-$. The two representations $W^+$ and $W^{-}$ are called the \textit{half-spin representations}. A good reference with a detailed explanation of these classical results is \cite[Chapter 6]{Goodman_book}.

When $\dim V \geqslant 3$, \text{Spin}$(V)$ is simply connected. Therefore, it serves as the unique nontrivial cover of $SO(V)$. Correspondingly, the spin representation $W_1$ (when $n\geqslant 3$ odd) and the half-spin representation $W^+$ (when $n\geqslant 4$ even) represent the unique nontrivial class of projective representation of $SO(V)$. 

Similarly, the unique nontrivial class of projective representation of $O_n$ is represented by $W$ (when $n\geqslant 2$ even) and by $W_1$ (when $n\geqslant 3$ odd). For the case of $SO_2$, as we discussed in \Cref{subsect 2.2}, $SO_2\cong \bC^\times$ and $M(SO_2)= 1$.

\subsubsection{Smooth quadrics and their categorical finite models}
Consider the variety $\cQ(V)\subset \bP(V)$ that parmeterizes the isotropic lines in $V$ with respect to the quadratic form $Q$. The variety $\cQ(V)$ is a quadric of dimension $n- 2$. When $n\geqslant 3$, on this quadric we have the \textit{spinor bundle(s)}: one vector bundle $\cS$ if $n$ odd, and two vector bundles $\cS^\pm$ when $n$ even. The construction can be found in \cite[Section 4.3]{Kapranov1988} and \cite{Ottaviani1988SpinorBO}. The spinor bundle(s) have the following properties:
\begin{itemize}
    \item When $n$ is even, the collection of bundles $\langle \cS^+(-n+2), \cS^-(-n+2), \cO(-n+3),..., \cO \rangle$ is exceptional, generating the category $D^b(\cQ(V))$.
    \item When $n$ is odd, the collection of bundles $\langle \cS(-n+2), \cO(-n+3),..., \cO \rangle$ is exceptional, generating the category $D^b(\cQ(V))$.
    \item Any automorphism of $\cQ(V)$ either fixes the spinor bundles or exchanges $\cS^+$ and $\cS^-$ (up to isomorphisms). As a consequence, any automorphism of $\cQ(V)$ preserves the two exceptional collections above.
\end{itemize}
The following are some examples that we will use in Sections 8 and 9.
\begin{exa}\label{quadrics}(Categorical finite model of $\cQ(\bC^n)$ with respect to the $O_n$-action)

Letting $n= \dim V$, we have a natural action of $O_n$ on $\cQ(V)$. This $O_n$-action fixes $\cS$ when $n$ is odd, and permutes $\cS^\pm$ when $n$ is even. Taking the global section of the spinor bundle(s), we get the spin representation (the half-spin representations) of $SO_n$; see, e.g. \cite[Pages 17-18]{Addington2013SPINORSA}. Because these representations are not linear representations of $SO_n$, the bundles $\cS$ (when $n$ odd) and $\cS^+ \oplus \cS^-$ (when $n$ even) do not admit $O_n$-equivariant structures. Instead, they have structures of $\alpha$-$O_n$ sheaves, where $\alpha$ is the unique nontrivial element of $M(O_n)$. Thus, we have the following categorical finite models.
\begin{itemize}
    \item Consider $n$ is even, $n= 2m$. The model $Y_{2m}$ of $\cQ(\bC^n)$ has $2m$ points. In these, there are $2m-2$ ordinary points fixed by $O_{2m}$, these points come from $2m-2$ line bundles $\cO(-2m+3),...,\cO(1), \cO$. The two other points corresponding to $\cS^\pm(-2m+2)$ are permuted by $O_{2m}$. These two points are special points, and the Schur multiplier comes from the half-spin representations of $\SO_{2m}$ (two half-spin representations give rise to the same class of projective representations).
    \item Consider $n= 2m-1$. The model $Y_{2m-1}$ of $\cQ(\bC^n)$ has $2m-2$ points, all fixed by $O_{2m-1}$. In these, $2m-3$ of them are ordinary, and one is special. The Schur multiplier of this special point comes from the spin representation of $O_{2m-1}$. 
\end{itemize}
\end{exa}
In Section 8, the base varieties that we use to construct the connected components of $\fix$ are the orthogonal Grassmannians $\OG(k,w)$ for $1\leqslant k\leqslant \frac{2}{2}$ and $w\leqslant 4$. For later use, we list their categorical finite models below.
\begin{exa} \label{small OG} \leavevmode
\begin{enumerate}
    \item The variety $\OG(1,2)= \cQ(\bC^2)$ consists of two points which are permuted by $O_2$. The corresponding finite model is a 2-point orbit.
    \item The variety $\OG(1,3)= \cQ(\bC^3)$ is isomorphic to $\bP^1$. Its categorical finite model consists of a trivial orbit and a special orbit.
    \item The variety $\OG(1,4)= \cQ(\bC^4)$ is isomorphic to $\bP^1\times \bP^1$. Its categorical finite model consists of two trivial points and a special 2-point orbit.
    \item The variety $\OG(2,4)$ is a disjoint union of two $\bP^1$. Its categorical finite model consists of an ordinary 2-point orbit and a special 2-point orbit.
\end{enumerate}
In this context, when we say special orbit, we mean that the attached Schur multipliers are the unique nontrivial Schur multipliers of $O_3$ and of $SO_4$. 
\end{exa}

\subsubsection{Quadric fibrations}
Consider a smooth variety $X$ and a vector bundle $\cE$ of rank $n$ on $X$. Let $Q$ be a fiberwise nondegenerate quadratic form on $\cE$. Let $\cQ(\cE)\subset \bP(\cE)$ be the subvariety of the projective bundle $\mathbb{P}(\cE)\xrightarrow{\pi} X$ that is the vanishing locus of $Q$. Under certain technical assumptions (see (A.1) and (A.2) in \cite[pages 43,44]{böhning2005derived}), we have a semi-orthogonal decomposition of the derived category $D^b(\cQ(\cE))$ into blocks that come from $D^b(X)$. In particular, this is the case when $X$ is a homogeneous space. The following is a consequence of \cite[Theorem 3.2.7]{böhning2005derived}.

\begin{lem} \label{quadric fibration}
Assume that $X$ is a homogeneous space of a Lie group $G$ of classical type. We have a semi-orthogonal decomposition of $D^b(\cQ(\cE))$ as follows $$D^b(\cQ(\cE))=\langle \pi^*(D^b(X))\otimes \cS_{X}^{\pm}(-n+2), \pi^*(D^b(X))\otimes \cO(-n+3), ..., \pi^*(D^b(X))\otimes \cO(0) \rangle$$
in which $\cS_{X}^{\pm}$ are certain relative spinor bundle(s) on $X$ (see \cite[Section 3.2]{böhning2005derived}). Thus, if $X$ has a full exceptional collection, then $\cQ(\cE)$ does.
\end{lem}

As applications of this lemma, we work out categorical finite models of partial flag varieties for $O_n$.
\begin{exa} \label{finite model partial flag}
Let $V$ be $\bC^n$ with the $O_n$-action as in the context of \Cref{quadrics}. For $k\leqslant \frac{n}{2}$, let $X_{[1,k]}$ be the variety that parameterizes the partial isotropic flags $V^1\subset V^2\subset... \subset V^k\subset V$ with $\dim V^i =i$ for $1\leqslant i \leqslant k$. 

Consider the projection $\pi_i: X_{[1,i]}\rightarrow X_{[1,i-1]}$ that forgets the space $V^i$. Let $\cV_i$ be the vector bundle with fiber $(V^{i-1})^\perp/V^{i-1}$ over each partial flag of $X_{[1,i-1]}$. Then $\pi_i$ is realized as a quadric fibration in the projective bundle $\bP(\cV_i)$. Next, we consider the chain of projections $X_{[1,k]}\rightarrow X_{[1,k-1]}\rightarrow... \rightarrow X_{[1]}$. This realizes $X_{[1,k]}$ as a tower of quadric fibrations over $X_{[1]} =\cQ(\bC^n)$. Applying \Cref{quadric fibration}, we can inductively construct an $O_n$-finite model $Y_{[1,k]}$ of $X_{[1,k]}$.

Next, we claim that $Y_{[1,k]}$ is the unique $O_n$-finite model of $X_{[1,k]}$. Consider an $O_n$-finite model $Y$ of $X_{[1,k]}$. An $O_n$-orbit with finite cardinality can only have $1$ or $2$ points. Moreover, $M(O_n)= \cyclic{2}$ and $M(SO_n)= \cyclic{2}$, so $Y$ can only have at most four types of orbit. For $i= 1,2$, we write $\bO_i$ and $\bO_i^{s}$ for the ordinary and special $O_n$ orbits with $i$ points. Write $a_i$ and $a_i^{s}$ for the corresponding multiplicities of these orbits in $Y$. Consider the finite subgroup $A_n\subset O_n$ of diagonal matrices with entries $\pm 1$. As shown in \Cref{Schur}, the $O_n$-centrally extended structure of $Y$ is determined by its $A_n$-centrally extended structure. Hence, the uniqueness of $O_n$-finite models will follow from the uniqueness of $A_n$-finite models. This is proved as follows.

With the action of $A_n$ on $X_{[1,k]}$, we have a corresponding set of numerical invariants for $X_{[1,k]}$ as in \Cref{subsect 3.2}. Write $1$ for $\diag(1,1,...,1)$, $z_1$ for $\diag(-1,1,1,...,1)$ and $z_{12}$ for $\diag(-1,-1,1,...,1)$ in $A_n$. Then in terms of the multiplicities of the orbits, we have a system of linear equations:
$$\left\{ 
\begin{aligned}
    F_{1}^{A_n}= a_1+ a_1^{s}+ a_2+ a_2^{s}\\
    F_{z_1}^{A_n}= a_1\\
    F_{z_1}^{\{1,z_1\}}= a_1+ a_1^{s}\\
    F_{z_{12}}= a_1+ a_2
\end{aligned} 
\right.
$$
This system has unique solution. Therefore, the $A_n$-finite model of $X_{[1,k]}$ is unique (similar to the proof of \Cref{unique}).

\end{exa}
\begin{rem} \label{finite model OG}
    The existence of a full exceptional collection is not guaranteed for general partial flag varieties in types B, C, and D. However, we can obtain $O_n$-finite models of these varieties indirectly as follows. In the context of \Cref{finite model partial flag}, write $X_{i_1,...,i_l}$ for the variety of partial isotropic flags $V^{i_1}\subset...\subset V^{i_l}\subset V$ with $\dim V^{i_j}= i_j$. We think of $X_{i_1,...,i_l}$ as the base and realize $X_{[1,i_l]}$ as a tower of projective bundles over $X_{i_1,...,i_l}$ by maps that forget the spaces $V^j$, $j\notin \{i_1,...,i_l\}$. This suggests that if $X_{i_1,...,i_l}$ admits a finite model $Y_{i_1,...,i_l}$, then we have a finite model $Y_{[1,i_l]}$ of $X_{[1,i_l]}$ consisting of identical copies of $Y_{i_1,...,i_l}$. Conversely, provided that we have a unique finite model $Y_{[1,i_l]}$ of $X_{[1, i_l]}$ from \Cref{finite model partial flag}, we can obtain the (only possible) finite model of $X_{i_1,...,i_l}$ by dividing the multiplicities of orbits in $Y_{[1,i_l]}$ by a constant. One can work out the details to see that this constant is $N= \prod_{j=1}^l (i_j - i_{j-1}- 1)!$ and that the multiplicities of orbits in $Y_{[1,i_l]}$ are divisible by $N$. Thus, we obtain a (unique) finite model of $X_{i_1,...,i_l}$ in this way.
\end{rem}
\begin{exa}
    For $n=4$ and $k=2$, the variety $X_{[1,2]}$ is a disjoint union of two copies of $\OG(1,\bC^4)\cong \bP^1\times \bP^1$. Its categorical finite model consists of two $2$-point ordinary orbits and two $2$-point special orbits. Now $X_{[1,2]}$ has another realization as a $\bP^1$-bundle over $\OG(2,4)$, so it gives another explanation for the finite model of $\OG(2,4)$ we mentioned in \Cref{quadrics}.
\end{exa}
\subsection{Main results} \label{subsect 7.4}
We consider a nilpotent element $e\in \fsp_{2n}$ ($\fso_{2n}$, or $\fso_{2n+1}$) with the associated partition $\lambda$. Let $(\fix)_\alpha$ denote an $A_e$-orbit of connected components of $\fix$.
\begin{thm} \label{main}
\begin{enumerate}
We consider the case where $\lambda$ has up to $4$ parts.
    \item There exists a full exceptional collection in $D^b((\fix)_\alpha)$ that is compatible with the action of $A_e$. The corresponding finite model $Y_\alpha$ is the unique finite model of $(\fix)_\alpha$.
    \item Taking the union over $A_e$-orbits of connected components, we get a finite model $\bigcup_\alpha Y_\alpha$ of $\fix$. This is the unique finite model of $\fix$ and $\spr$. Thus, $Y_e= \bigcup_\alpha Y_\alpha$ as $A_e$-centrally extended sets. 
    \item The conclusions of Parts 1 and 2 hold if we replace $A_e$ by $Q_e$.
\end{enumerate}
\end{thm}
\begin{proof}
The existence of full exceptional collections is proved in Section 8. The uniqueness of the finite models of $\fix$ and $\spr$ is explained in \Cref{unique} and \Cref{same finite model}. Then the two finite models $Y_e$ and $\bigcup_\alpha Y_\alpha$ must coincide, and Part 2 of the theorem follows. For Part 3, the fact that the $Q_e$-centrally extended structure of $Y_e$ can be recovered from its $A_e$-centrally extended structure is discussed in Section 2.
\end{proof}

The existence of a full exceptional collection has further application in the study of $D^b(\fix\times \fix)$. In the following, we construct a basis of $K_0^{Q_e}(\fix\times \fix)$ and an isomorphism of based algebras $K_0(\Sh^{Q_e}(Y_e\times Y_e))\cong $  $K_0^{Q_e}(\fix\times \fix)$.

Consider a reductive group $A$ and an $A$-variety $X$. Assume that $X$ has a full exceptional collection $\cE$ that is compatible with the action of $A$ (see \Cref{compatible}). In the notation of \Cref{dual is compatible}, let $\mathcal{E}^\vee= \{{E_j^{(i)}}^\vee\}$ be the right dual collection of $\cE$. We showed in \Cref{dual is compatible} that $\mathcal{E}^\vee$ is compatible with the action of $A$. Next, we claim that the collection $\mathcal{E}^*$ obtained by applying the functor $\RHom(*, \mathcal{O}_X)$ to objects in $\mathcal{E}$ gives another $A$-compatible exceptional collection of $X$. Indeed,
$$\RHom(\RHom(\mathcal{F}, \mathcal{O}_X), \RHom(\mathcal{G}, \mathcal{O}_X))=\RHom(\mathcal{O}_X, \RHom(\mathcal{G}, \mathcal{O}_X)\otimes \mathcal{F})=\RHom(\mathcal{G}, \mathcal{F}).$$ 
Because $A$ permutes objects of $\cE$, $A$ permutes objects of $\cE^*$. We write ${E_{j}^{(i)}}^*$ for the sheaf $\RHom(E_{j}^{(i)}, \mathcal{O}_X)$ in $\mathcal{E}^*$. Now, according to \cite[Corollary 3.1]{basic}, we have a full exceptional collection $\mathcal{E}^\vee \boxtimes \mathcal{E}^*$ of $D^b(X\times X)$ consisting of objects ${E^{(i)}_j}^\vee \boxtimes {E^{(i')}_{j'}}^*$. And since $A$ permutes objects of $\mathcal{E}^\vee$ and $\mathcal{E}^*$, we have an induced action of $A$ on the collection $\mathcal{E}^\vee \boxtimes \mathcal{E}^*$. This verifies the compatibility of the $A$-action for the collection $\mathcal{E}^\vee \boxtimes \mathcal{E}^*$.

Now, we describe how \Cref{main Elagin} applies to the collection $\mathcal{E}^\vee \otimes \mathcal{E}^*$. Let $E^{(ij)}_{l_ih_j}$ denote ${E^{(i)}_{l_i}}^\vee\boxtimes {E^{(j)}_{h_j}}^*$. The stabilizer of $E^{(ij)}_{l_ih_j}$ in $A$ is $H_{ij}:= H_i\cap H_j$. Consider an $A$-orbit of some sheaf in the collection $E^{(ij)}_{l_ih_j}$. Let $\alpha_{ij}$ be the cocycle of $H_{ij}$ obtained by taking the product of two cocycles $\alpha_i|_{H_{ij}}$ and $\alpha_j^{-1}|_{H_{ij}}$. Make $E^{(ij)}_{l_ih_j}$ an $\alpha_{ij}$-$H_{ij}$-sheaf with respect to the cocycle $\alpha_{ij}$. Let $\mathcal{E}_{l_ih_j}^{(ij)}$ denote this $\alpha_{ij}$-$H_{ij}$-sheaf. Let $W_1,...,W_{m_{ij}}$ denote the set of irreducible $\alpha_{ij}^{-1}$-representations of $H_{ij}$. From \Cref{main Elagin}, the collection $\{ \Coind_{H_{ij}}^A(\mathcal{E}_{l_ih_j}^{(ij)}\otimes W_{1})$,..., $\Coind_{H_{ij}}^A(\mathcal{E}_{l_ih_j}^{(ij)}\otimes W_{m_{ij}}) \}$ is a full exceptional collection of the equivariant derived category $D^{b,A}(X\times X)$. Write $(\mathcal{E}^\vee \otimes \mathcal{E}^*)_A$ for this collection. We obtain a basis of $K_0^A(X\times X)$ by taking the classes of the equivariant sheaves in $(\mathcal{E}^\vee \otimes \mathcal{E}^*)_A$ in the equivariant Grothendieck group.

Let $Y$ be the categorical finite model of $X$ with respect to the collection $\cE^\vee$. We form the set $Y\times Y$, this finite set has an $A$-centrally extended structure induced by $Y$'s (see \Cref{finite model of the square set}). The abelian group $K_0(\Sh^A(Y\times Y))$ has an algebra structure over $\bZ$ defined by twisted convolutions \cite[Section 5]{bezrukavnikov2001tensor}. On the coherent side, we have an algebra structure on $K_0^A(X\times X)$ defined by convolution in equivariant K-theory.
\begin{pro} \label{Xi based proof}
The two $\mathbb{Z}$-algebras $K_0^A(X\times X)$ and $K_0(\Sh^A(Y\times Y))$ are isomorphic as based algebras. 
\end{pro}
\begin{proof}

By \cite[Section 4.2]{bezrukavnikov2001tensor}, the category $\Sh^A(Y)$ has the structure of a module category over the category $Rep(A)$. And $\Sh^{A}(Y\times Y)$ is identified with $Fun_A(Y, Y)$, the category of module functors from $\Sh^A(Y)$ to itself. Now, the algebra $K_0^{A}(X\times X)$ acts on $K_0^{A}(X)= K_0(\Sh^A(Y))$ by the convolution action. Therefore, each equivariant sheaf $\cF$ gives us a map $K_0(\Sh^A(Y))\rightarrow K_0(\Sh^A(Y))$. Next, for each equivariant sheaf $\cF$ in the collection $(\mathcal{E}^\vee \otimes \mathcal{E}^*)_A$, we have a corresponding module functor $\cF'\in Fun_A(Y, Y)$ as follows. 

First, on the level of objects, we want $\cF$ and $\cF'$ to give us the same induced morphism $K_0(\Sh^A(Y))\rightarrow K_0(\Sh^A(Y))$. Consider the general form of $\cF$ as $\Coind_{H_{ij}}^A(\mathcal{E}_{l_ih_j}^{(ij)}\otimes W_{l})$. On the nonequivariant level, the convolution action with the class of $\mathcal{E}_{l_ih_j}^{(ij)}= {E^{(i)}_{l_i}}^\vee\boxtimes {E^{(j)}_{h_j}}^*$ sends the collection $\cE^\vee$ to itself. In particular, let $c_{j,h_j}$ be the position of $E^{(j)}_{h_j}$ in the collection $\cE$ (from the left). By (\ref{property of dual collection}), we have
\begin{equation*} 
    [\mathcal{E}_{l_ih_j}^{(ij)}* {E^{(i')}_{l_{i'}}}^\vee]= 
    \begin{cases}
(-1)^{c_{j,h_j}-1}[{E^{(i)}_{l_i}}^\vee] \, \text{ if } j+ i'=n+1 \, \text{ and } \, h_{j}+ l_{i'}= k_j\\
0 \text{ otherwise. }
\end{cases}
\end{equation*}
Hence, on the equivariant level, the convolution with $\cF$ sends the collection $\cE^\vee_A$ to itself. Recall from the construction in \Cref{categorical finite model} that the objects of the collection $\cE^\vee _A$ can be realized as simple objects in $\Sh^A(Y)$. The category $\Sh^A(Y)$ is semisimple, so the map defined for simple objects gives us a map $\cF': \Sh^A(Y)\rightarrow \Sh^A(Y)$. As $\cF'$ is defined by convolution action, it is a functor, and there is a natural module functor structure of $\cF'$ coming from the equivariant structure of $\cF$ (see, e.g., \cite[Section 5]{bezrukavnikov2001tensor}). 

Taking the corresponding classes in K-groups of $\cF$ and $\cF'$, we have an algebra homomorphism 
$$c: K_0^{A}(X\times X)\rightarrow K_0(Fun_A(Y))= K_0(\Sh^A(Y\times Y))$$ Explicity, we have $c([\Coind_{H_{ij}}^A(\mathcal{E}_{l_ih_j}^{(ij)}\otimes W_{l_{ij}})])$ $= (-1)^{c_{j,h_j}-1}[\Coind_{H_{ij}}^A(W_{p_{ij}})]$. Here $W_{p_{ij}}$ is a projective representation supported on $(e^{(i)}_{l_i},e^{(j)}_{h_j})$ in $Y\times Y$ with respect to the cocycle $\alpha_i^{-1}\alpha_j$. As $c$ maps a basis of $K_0^A(X\times X)$ to a basis of $K_0(\Sh^A(Y\times Y))$, it is an isomorphism of based $R(A)$-algebras. 

\end{proof}

\section{The geometry of $\fix$}
Consider a nilpotent element $e\in \fsp_{2n}$ with associated partition having up to $4$ rows. The primary purpose of this section is to study the connected components of the variety $\fix$ defined in Section 6. Throughout the section, we will introduce various morphisms from the connected components of $\fix$ to some connected components of $\cB_{e'}^{gr}$, the fixed-point loci of possibly different Springer fibers. These morphisms, called reduction maps, are projections from projective bundles or blow-up morphisms along smooth subvarieties. Furthermore, we will see that the centrals of the blow-ups are isomorphic to some connected components of $\cB_{e''}^{gr}$ for some $e''$. This inductive process will give us the proof of \Cref{main}.

We recall the setting and study some basic properties of the connected components of $\fix$ in the first two subsections. The third and fourth subsections are devoted to general results about reduction maps. In the last two subsections, we describe the geometry of $\fix$ in detail when the partitions of $e$ have $3$ and $4$ rows.

\subsection{Setting} \label{new setting} 
We keep most of the notation from \Cref{subsect 2.1} and make a minor modification to the indexes of the basis vectors of $V_\lambda$. We consider a nilpotent element $e\in \fsp_{2n}= \fsp(V_\lambda)$ with the associated partition $ \lambda= (2x_1,2x_2,...,2x_k)$ with $x_1\geqslant x_2...\geqslant x_k$. We write the transpose partition of $\lambda^t$ in the form $(y_1,y_{-1}, y_2, y_{-2},$ $...,y_M, y_{-M})$  where $y_i = y_{-i}$ and $y_1\geqslant y_2...\geqslant y_M$. As in \Cref{subsect 2.1}, we have a direct sum decomposition $V= \oplus_{i=1}^k V^i$ so that $e$ acts on each $V^i$ as a single Jordan block. Furthermore, we can find an $\fsl_2$-triple $(e,h,f)\subset \fsp(V_\lambda)$ such that each $V^i$ is a simple $\fsl_2$-submodule of $V_\lambda$. Write $T_e$ for the image of the embedding $\bC^\times \rightarrow \Sp(V_\lambda)$ sending $t$ to $t^h$. Recall that $\fix$ is the $T_e$-fixed point locus of $\spr$.

The action of $T_e$ on $V_\lambda$ induces the weight space decomposition $V_\lambda= \oplus_{j= 1-M}^{M} V_{2j-1}$ (the number $2M$ is the number of parts in the dual partition $\lambda^t$). The intersection $V^i\cap V_{2j-1}$ has dimension at most $1$ for $1\leqslant i\leqslant k$ and $1-M\leqslant j\leqslant M$. If the dimension is $1$, we pick a nonzero vector $v_{2j-1}^{i}$ in the intersection. The set of vectors $\{v_{2j-1}^{i}\}$ gives us a basis of $V_\lambda$. We can rescale the vectors of this basis to have $ev_{2j-1}^{i}= v_{2j+1}^{i}$. This is precisely the basis we have chosen in \Cref{subsect 2.1} with a slight change of lower indexes to record the weights of $T_e$. 

\begin{exa} \label{example basis}
Let $e= N_{(2,6,6,8)}$. 
The action of $T_e$ on $\mathcal{B}_{(2,6,6,8)}$ induces a weight space decomposition $V_{(2,6,6,8)}= \oplus_{i= -3}^4 V_{2i-1}$. The diagram below illustrates the actions of $e$ and $T_e$ on the basis we have chosen.
$$\ytableausetup{centertableaux, boxsize= 3em}
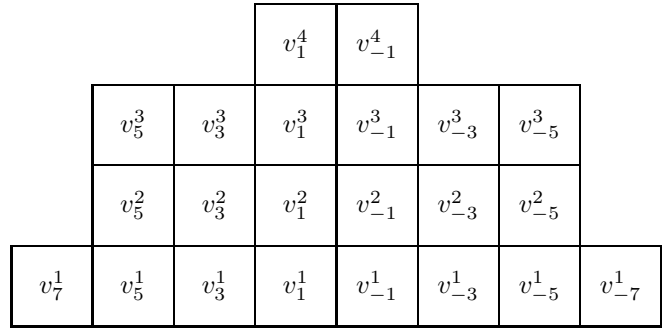
\begin{ytableau}
\none & \none&\none &v_{1}^{4} &v_{-1}^{4} &\none   &\none &\none \\
\none & v_{5}^{3} &v_{3}^{3} &v_{1}^{3} &v_{-1}^{3} &v_{-3}^{3}   &v_{-5}^{3} &\none\\
\none & v_{5}^{2} &v_{3}^{2} &v_{1}^{2} &v_{-1}^{2} &v_{-3}^{2}   &v_{-5}^{2} &\none\\
 v_{7}^{1} & v_{5}^{1} &v_{3}^{1} &v_{1}^{1} &v_{-1}^{1} &v_{-3}^{1}   &v_{-5}^{1} &v_{-7}^{1}\\
\end{ytableau}\\
$$
The action of $e$ moves the vector in a box to its left.  
\end{exa} 
The partition of $e$ can be read from the sizes of the rows of the diagram. The dimensions of the $T_e$-weight spaces are given by the sizes of the columns. Conversely, given a partition $2x_1+...+ 2x_k= 2n$ with $x_1\geqslant x_2\geqslant...\geqslant x_k$, we have a pyramid with $k$ rows that have sizes $2x_1,...,2x_k$ from bottom to top. The boxes are labeled with $v_{2j-1, i}$ accordingly.
\begin{defin}[Basis diagram] \label{basis diagram}
The tableaux obtained by the above process is called the basis diagram of $V_\lambda$.
\end{defin}
\begin{rem}
Given a basis diagram, we can reconstruct a pair $(V_\lambda, e\in Sp(V_\lambda))$ with the basis and the action of $e$ described above. Later, it will be convenient to refer to a basis diagram by the sizes of its column, so we also write $\mathcal{B}_{(w_{2M-1},...,w_{1-2M})^\intercal}$ for $\mathcal{B}_e$ in which $w_{2j-1}$ is the dimension of $V_{2j-1}$ and the superscript $\intercal$ stands for transpose. For example, for $e= N_{(2,6,6,8)}$ as above, we have two ways to write the Springer fiber over $e$: $\mathcal{B}_{(2,6,6,8)}$ and $\mathcal{B}_{(1,3,3,4,4,3,3,1)^\intercal}$. Note that in this notation we do not have the standard order for the parts of $\lambda^\intercal$. 
\end{rem} 

Consider an $\fsl_2$-triple $(e,h,f)\subset \fg$ and the Slodowy slice $S_e:= (e+ Z_\fg(f))\cap \cN$. Let $\pi: \tilde{\cN}\rightarrow \cN$ be the Springer resolution and let $\tilde{S}_e$ be the preimage $\pi^{-1}(S_e)$. The torus $T_e$ acts on $\tilde{S}_e$, and $\fix$ can be realized as the $T_e$-fixed point locus of $\tilde{S}_e$. Since the variety $\tilde{S}_e$ is smooth, $\fix$ is smooth. The next section describes its connected components.

\subsection{Connected components of $\fix$}
\subsubsection{Associated tuples and standard tuples}
An element of $\mathcal{B}_e$ is an $e$-stable maximal isotropic flag $U^\bullet$ of $V_\lambda$. We represent a flag $U^\bullet$ by a chain of isotropic subspaces $U^1\subset U^2\subset....\subset U^n= U$ or a sequence of vectors $(u_1,u_2,...,u_n)$ so that $U^i=\text{Span}(u_1,...,u_i)$. A flag $U^\bullet \in \fix$ is $T_e$-stable, so it can be represented by a sequence $(u_1,...,u_i)$ with each $u_i$ belonging to some $V_{2j-1}$. 
\begin{notation} \label{flag by vectors}
    In this paper, when we regard an element of $\fix$ as a sequence of vectors $(u_i)$, it is automatically assumed that $u_i$ are eigenvectors of the $T_e$-action.
\end{notation}

Let $d_{i}= \dim (V_{2i-1}\cap U)$. Each flag $(u_1,...,u_n)$ gives rise to a tuple of $n$ numbers in which there are $d_{i}$ numbers $2i-1$. This is done by substituting the vector $u_j$ in the sequence by the number $2i-1$ if $u_j\in V_{2i-1}$. It is clear that two flags with different associated tuples lie in two different connected components of $\fix$.

Conversely, for each tuple $\alpha= (p_1,...,p_n)$ of $n$ weights (not necessarily distinct) of the $T_e$-action on $V_\lambda$, let $(\fix)_\alpha$ denote the subvariety of isotropic flags in $\fix$ that can be represented by $(u_1,...,u_{n})$ with $u_i\in V_{p_{i}}$. We call $(\fix)_\alpha$ the subvariety of $\fix$ associated with $\alpha$. Then $(\fix)_\alpha$ is empty or a union of connected components of $\fix$. We now give a realization of $(\fix)_\alpha$ as a closed subvariety of the product of Grassmannians.

\begin{rem}[An alternative description of $(\fix)_\alpha$] \label{alternative}
Fix an $n$-tuple of $T_e$-weights $\alpha= (p_1,...,p_n)$. Recall that we write $w_i$ for $\dim V_i$, and write $d_{i}$ for the multiplicities of the weights $i$ in the tuple $\alpha= (p_1,...,p_n)$. A flag $U^\bullet \in (\fix)_\alpha$ can be represented by $(u_1,...,u_n)$ with $u_i\in V_{p_i}$. For each weight $i$ of $T_e$, we take the intersection of $U^\bullet$ with $V_i$ to get a partial flag $W_i^1\subset...W_i^{d_i}\subset V_i$. The flag $U^\bullet$ is then completely recovered from the set of subspaces $\{W_{i}^j= V_i\cap U^j\}_{i,j}$. Thus, we obtain an embedding of $(\fix)_\alpha$ into the product of Grassmannians $Gr^{\lambda}_{\alpha}:= \prod_{i= 1-2M}^{2M-1} \prod_{j=1}^{d_{i}} \Gr(j, w_{i})$, the first product runs over odd indexes $i$. 

Recall that we have chosen a basis $\{v_{i}^j\}_{1\leqslant j\leqslant d_i}$ for each $V_i$ in \Cref{new setting}, so it makes sense to talk about the standard inner products (dot products) of $V_i$ and orthogonal Grassmannians of $V_i$. The following conditions characterize $(\fix)_\alpha$ as a subvariety in $Gr^{\lambda}_\alpha$.
\begin{enumerate}
    \item $W_i^{j}\subset W_i^{j+1}$. 
    \item Let $p_l$ be the $(j+1)$-th weight $i$ from the left, and let $j'$ be the number of weights $i-2$ in the set $\{p_1,...,p_l\}$. Then $eW_{i-2}^{j'}\subset W_i^{j}$.
    \item Assume $i>0$. Note that $e^{i}$ is an isomorphism between $V_{-i}$ and $V_{i}$. We require the subspace $e^{i}W_{-i}^{d_{-i}}$ to be the orthogonal complement of the space $W_{i}^{d_{i}}$ in $V_{i}$. A simple consequence is that $d_i+ d_{-i}= w_i$.
\end{enumerate}

The first condition is to make $U^\bullet$ a flag, and the second condition means that $U^\bullet$ is stable under the action of $e$. We elaborate on why the last condition is equivalent to the flag $U^\bullet$ being isotropic. Recall that we have a $T_e$-weight decomposition $V_\lambda=\oplus_{i= 1-2M}^{2M-1} V_i$, $i\notin 2\bZ$. The symplectic form of $V_\lambda$ has the property that $\langle V_i, V_j\rangle= 0$ for $i+j\neq 0$. Therefore, $U^\bullet$ is isotropic if and only if $\langle W_i^{d_i}, W_{-i}^{d_{-i}} \rangle= 0$ for $1-2M\leqslant i\leqslant 2M- 1$. For $u\in W_{i}, v\in W_{-i}$, up to a sign, the pairing $\langle u, v\rangle$ is equal to the dot product of $u$ and $e^{i}v$ (with respect to our basis $\{v_{i}^{j}\}$ of $V_i$). Therefore, we get the third condition. 
\end{rem}
\begin{notation} \label{alpha}
    We write $X_\alpha$ for the variety $(\fix)_\alpha$. \Cref{alternative} shows that, given a tuple $\alpha$, we can recover the dual partition $\lambda^\intercal$ by $w_i= d_i+ d_{-i}$ (the sum of the multiplicities of the weights $i$ and $-i$ in $\alpha$). Therefore, in the expression $X_\alpha$, once $\alpha$ is specified, the element $e$ and the partition $\lambda$ are implicitly understood. We will continue to alternate between the notation $X_\alpha$ and $(\fix)_\alpha$, depending on the context.   
\end{notation}

Keep the assumption that $i\geqslant 0$. A consequence of the second and third conditions in \Cref{alternative} is that $\langle e^{i}u, u\rangle =0$ for any $u\in W_{-i}^{d_{-i}}$. With our chosen basis of $V_i$ and $V_{-i}$ in \Cref{new setting}, $\langle e^{i}u, u\rangle =0$ means that $u$ is an isotropic vector of $V_{-i}$  with respect to the dot product. We use the notation $\OG(j,V_i), 2j\leqslant \dim V_i$ for the variety of $j$-dimensional isotropic spaces of $V_i$. The following is an important corollary that will be used often.
\begin{cor} \label{isotropic}
     We have $W_i^{d_i}\in \OG(d_i,V_i)$ for any $i< 0$. In other words, $W_i^{j}\in \OG(j, V_i)$ for any $i<0$ and $1\leqslant j\leqslant d_i$.
\end{cor}
With \Cref{alternative}, a quick observation is that two different $n$-tuples of weights may have isomorphic associated varieties. This reduces our consideration to a certain type of tuples. 
\begin{defin}[Standard tuple] \label{standard tuple}
An $n$-tuple of $T_e$-weights $(p_1,...,p_n)$ is standard if $|p_j-p_{j-1}|\leqslant 2$ for $j\in \{2,...,n\}$.
\end{defin}

\begin{pro} \label{standard}
For an arbitrary $n$-tuple $\alpha$, we can find a standard tuple $\alpha'$ such that $(\fix)_{\alpha'}$ and $(\fix)_\alpha$ are isomorphic.
\end{pro}
\begin{proof}
Let $i$ be the smallest index such that $p_{i-1}> p_{i}+2$ or $p_{i-1}< p_i-2$, and let $s_i(\alpha)$ be the tuple of weights obtained by swapping $p_i$ and $p_{i-1}$ in $\alpha$. In terms of \Cref{alternative}, since $s_i(\alpha)$ is a permutation of $\alpha$, Conditions 1 and 3 for $X_\alpha$ and $X_{s_i(\alpha)}$ are the same. Given $p_{i-1}> p_{i}+2$ (or $p_{i-1}< p_i-2$), Condition 2 is not affected by the change from $\alpha$ to $s_i(\alpha)$. Therefore $X_\alpha\cong X_{s_i(\alpha)}$. In terms of vector representatives, an isomorphism is given by sending the flag represented by $(u_1,...u_{i-1}, u_{i},...,u_n)$ to the flag represented by $(u_1,...u_{i}, u_{i-1},...,u_n)$. 

Applying this argument inductively, we obtain a standard tuple $\alpha'$ with the desired property.
\end{proof}
\begin{rem} \label{what if the difference is 2}
    In the proof of \Cref{standard}, we have shown that if $|p_{i-1}- p_i|\neq 2$, then $X_\alpha \cong X_{s_i(\alpha)}$. Checking how Condition 2 in \Cref{alternative} changes, we find that if $p_{i-1}= p_i- 2$, then $X_\alpha$ is a subvariety of $X_{s_i(\alpha)}$. Similarly, if $p_{i-1}= p_i+ 2$, then $X_{s_i(\alpha)}$ is a subvariety of $X_\alpha$. In the following sections, we will encounter blow-ups of $X_\alpha$ where the exceptional sets are often given by $X_{s_i(\alpha)}$ for certain $p_{i-1}= p_i+ 2$.
\end{rem}

\begin{rem} \label{good tuples}
    We have mentioned that for an $n$-tuple of weights $\alpha$, the set $(\fix)_\alpha$ is either empty or a union of connected components of $\fix$. If $(\fix)_\alpha$ is nonempty, we call $\alpha$ a \textbf{good tuples}. The alternative description in \Cref{alternative} is enough to determine all good tuples. However, the condition for $\alpha$ to be good is not very relevant to the topics of this paper, so we do not go into the details here. In this paper, once we specify $\lambda$ and consider a tuple $\alpha$, it is understood that we only consider good tuples.
\end{rem}
\begin{exa} Consider $e= N_{(4,4,6)}$.
The weight tuple $\alpha= (5,3,3,3,1,1,-1)$ is standard, while $\alpha= (3,5,1,3,3,1,1)$ and $\alpha= (5,3,1,3,-1,3,1)$ are not standard.

In this case, the good standard tuples $\alpha$ without negative weights are $(5,3,3,3,1,1,1)$, $(3,3,5,3,1,1,1)$, $(3,3,1,1,5,3,1)$, $(3,3,1,5,3,1,1)$, $(3,1,3,1,5,3,1)$, $(3,1,3,5,3,1,1)$, $(3,1,5,3,3,1,1)$, $(3,1,5,3,1,3,1)$, $(3,5,3,3,1,1,1)$, $(3,5,3,1,3,1,1)$, $(3,5,3,1,1,3,1)$.
\end{exa}
\subsubsection{ Embeddings of $(\fix)_\alpha$} \label{subsect 8.2.2}
In this section we realize $(\fix)_\alpha$ as a subvariety of a variety coming from fixed points loci of smaller Springer fibers. Recall that $\lambda^\intercal$ is $(w_{2M-1},...,w_{1-2M})$ with $w_{2M-1}\leqslant ... \leqslant w_1$ and $w_i = w_{-i}=\dim V_{i}$. In the following, we sometimes use the notation $\mathcal{B}_{(w_{2M-1},...,w_{1-2M})^\intercal}$ for $\mathcal{B}_e$.

We say that a partition is symmetric if the multiplicities of its parts are all even. The condition that $\lambda$ consists only of even parts is equivalent to $\lambda^\intercal$ being symmetric. Let $A$ be a nonempty subset of $\{1,3,...,2M-1\}$. The tuple $(w_{i}, |i|\in A)$ gives a symmetric partition of $2n_A= 2\sum_{i\in A} \dim V_{i}$. Hence, the partition $\lambda_A:= (w_{i}, |i|\in A)^\intercal$ is the associated partition of a nilpotent element $e_A$ in $\fsp_{2n_A}$. We then have the Springer fiber $\mathcal{B}_{e_A}$ (or $\mathcal{B}_{(w_i, |i|\in A)^\intercal}$). 

Next, we want to construct a symplectic vector space $V_{\lambda_A}$ such that $e_A\in \fsp(V_{\lambda_A})$. Assuming $A= \{i_1\leqslant ...\leqslant i_l\}$, consider $V_{\lambda_A}= \oplus_{|i|\in A} V_i\subset V_\lambda$. We let $e_A$ act on $V_{\lambda_A}$ by $e_A|_{V_{i_j}}= e^{i_{j+1}-i_j}$; as a consequence, $e_A: V_{i_j}\xrightarrow{} V_{i_{j+1}}$. To have $e_A\in \mathfrak{sp}(V_{\lambda_A})$, we define a symplectic form on $V_{\lambda_A}$ as follows. We use the original symplectic form of $V_{i_1}\oplus V_{-i_1}$ and for $j>1$ we twist the symplectic forms on each $V_{i_j}\oplus V_{-i_j}$ by $-1$ if needed. In terms of basis diagrams (\Cref{basis diagram}), the basis diagram of $e_A$ is obtained from the basis diagram of $e$ by removing the columns corresponding to $\{V_{\pm i}\}_{i\notin A}$.

The construction of $V_{\lambda_A}$ from $V_\lambda$ induces a map $p_A: (\fix)_\alpha \rightarrow \mathcal{B}_{e_A}^{gr}$ as follows. Let $\alpha= (p_1,...,p_n)$ and consider a maximal isotropic flag $U^\bullet\in (\fix)_\alpha$. Regarding $V_{\lambda_A}$ as a vector subspace of $V_\lambda$, this map is $p_A(U^\bullet)= U^\bullet \cap V_{\lambda_A}$. In particular, we describe $p_A$ in terms of representative vectors as follows. Recall that we can represent this flag by weight vectors $(u_1,u_2,...,u_n)$ with $u_i\in V_{p_i}$. Picking out the vectors $(u_i, |p_i|\in A)$, we get a collection of vectors in $V_{\lambda_A}$. This collection represents a maximal isotropic flag in $V_{\lambda_A}$ because the twists by $-1$ of the symplectic forms of $V_{i_j}\oplus V_{-i_j}$ preserve the property that a flag is isotropic. Next, to refine the target of the morphism $p_A$, we introduce the following definition.

\begin{defin}[Induced tuple] \label{induced tuple}
Let $\alpha$ be an $n$-tuple $(p_1,...,p_n)$ of weights with $p_i\in \{\pm 1, \pm 3,..., \pm 2M-1\}$. For a subset $A\subset \{1,3,...,2M-1\}$, write $A$ as $\{i_1\leqslant ...\leqslant i_l\}$. First, we pick out all $p_i\in \pm A$ while keeping their relative order. Then, substituting the numbers $\pm i_j$ by $\pm (2j-1)$, we obtain a tuple of weights $\alpha_A$. 
\end{defin}
\begin{exa}
Consider $\alpha= (7,5,5,3,5,3,1,1,-1,-1,-3)$.
\begin{itemize}
    \item If $A=\{3,5,7\}$, then $\alpha_{A}= (5,3,3,1,3,1,-1)$. Here $i_1= 3$, $i_2= 5$, and $i_3= 7$, so we replace the weights $-3,3,5$ and $7$ by $-1, 1, 3$ and $5$, respectively.
    \item If $A=\{1,3,7\}$, then $\alpha_{A}= (5,3,3,1,1,-1,-1,-3)$. Here $7$ is $i_3$, so we replace $7$ by $5$.
    \item If $A=\{5\}$, then $\alpha_{A}= (1,1,1)$. Here $5$ is $i_1$, so we replace $5$ by $1$.
    \item If $A=\{1,7\}$, then $\alpha_{A}= (3,1,1,-1,-1)$. Here $7$ is $i_2$, so we replace $7$ by $3$.
\end{itemize}    
\end{exa}
Recall that we write $X_\alpha$ for $(\fix)_\alpha$ (see \Cref{alpha}). We see that the image of $p_A$ is contained in the associated variety $X_{\alpha_A}\subset \mathcal{B}_{e_A}^{gr}$. So, for any subset $A$ of $\{1,3,...,2M-1\}$, we have defined a morphism
\begin{equation} \label{pA}
    p_A: X_\alpha \rightarrow X_{\alpha_A}.
\end{equation}
Next, consider a partition of $\{1,3,...,m\}$ into two subsets $A$ and $B$. We have an embedding of $X_\alpha$ as follows. 
\begin{pro} \label{AB}

The morphism $i_{A,B}: x\mapsto (p_A(x), p_B(x))$ gives a closed embedding of $X_\alpha$ into $X_{\alpha_A}\times X_{\alpha_B}$.
\end{pro}
\begin{proof}
From the alternative description of $X_\alpha$ in \Cref{alternative}, the variety $X_\alpha$ is a closed subvariety of the product of Grassmannians $\Gr_{\alpha}^{\lambda}$ defined by three conditions $1, 2$ and $3$. Similarly, we can realize $X_{\alpha_A}\times X_{\alpha_B}$ as a subvariety of $\Gr_{\alpha_A}^{\lambda_A}\times \Gr_{\alpha_B}^{\lambda_B}$. From the way we define $\alpha_A, \alpha_B, \lambda_A$ and $\lambda_B$, we have $\Gr_{\alpha}^{\lambda}\cong \Gr_{\alpha_A}^{\lambda_A}\times \Gr_{\alpha_B}^{\lambda_B}$. In this identification, the set of defining relations of the variety $X_{\alpha_A}\times X_{\alpha_B}$ is obtained by removing certain defining relations of $X_\alpha$. In particular, in the second condition of \Cref{alternative}, we remove the relation $eW_i^{j}\subset W_{i+2}^{j'}$ if $|i|\in A$ and $|i+2|\in B$ or vice versa. Thus, the morphism $i_{A,B}$ is a closed embedding.  
\end{proof}
We continue to study the properties of the morphism $p_A$ defined in (\ref{pA}). For a nonempty subset $A'\subset A\subset \{1,3,...,2M-1\}$, we can reiterate the construction to obtain $X_\alpha\xrightarrow{p_A} X_{\alpha_A}\xrightarrow{p^A_{A'}} X_{\alpha_{A'}}$. Note that $p_{A'}= p^A_{A'}\circ p_A$. 
\begin{pro} \label{Cartesian}
 Consider a partition of $\{1,3,..., 2M-1\}$ into three sets $A, B$ and $C$. Assume that $a-c\neq \pm 2$ for any $a\in A$ and $c\in C$. Then the following is a Cartesian commutative diagram
$$\begin{tikzcd} 
    X_{\alpha_{A\cup B\cup C}} \arrow{r}{i_{A\cup B,C}}\arrow{d}{p_{B\cup C}} \& X_{\alpha_{A\cup B}}\times X_{\alpha_C} \arrow{d}{p^{A\cup B}_{B}\times \Id}\\
     X_{\alpha_{B\cup C}} \arrow{r}{i_{B, C}} \& X_{\alpha_B}\times X_{\alpha_C} 
     \end{tikzcd} $$
Above, note that $X_{\alpha_{A\cup B\cup C}}$ is $X_\alpha$. 

\end{pro}
\begin{proof}
    Similarly to the proof of \Cref{AB}, we embed $X_\alpha, X_{\alpha_{A\cup B}}\times X_{\alpha_C}$ into $\Gr_\alpha^{\lambda}$ and embed $X_{\alpha_{B\cup C}}, X_{\alpha_B}\times X_{\alpha_C}$ into $\Gr_{\alpha_{B\cup C}}^{\lambda_{B\cup C}}$. Then the four varieties in the above diagram are describe by sets of relaltions in \Cref{alternative}. Recall from the proof of \Cref{AB} that the embedding $i_{A\cup B, C}$ comes from forgetting relations $eW_i^{j}\subset W_{i+2}^{j'}$ if $|i|\in A\cup B$ and $|i+2|\in C$ or vice versa. With the hypothesis that $|a-c|\neq 2$ for $a\in A$ and $c\in C$, the condition that $|i|\in A\cup B$ and $|i+2|\in C$ is equivalent to $|i|\in B$ and $|i+2|\in C$. Thus, the embeddings $i_{A\cup B,C}$ and $i_{B,C}$ are defined by forgetting the same set of relations, and the commutative square is Cartesian.
\end{proof}

\begin{exa}
Consider $e= N_{(2,6,6,8)}$ as in \Cref{example basis}. First, we write the Springer fiber $\mathcal{B}_e$ as $\mathcal{B}_{(1,3,3,4,4,3,3,1)^\intercal}$. Here $2M-1= 7$, we choose $A= \{5\}$, $B=\{3,7\}$ and $C= \{1\}$. For a general $8$-tuple $\alpha$ of weights, the Cartesian square in the proposition has the form

$$\begin{tikzcd} 
    (\mathcal{B}^{\text{gr}}_{(1,3,3,4,4,3,3,1)^\intercal})_\alpha \arrow{r}{i_{A\cup B, C}}\arrow{d}{p_{B\cup C}} \& (\mathcal{B}^{\text{gr}}_{(1,3,3,3,3,1)^\intercal})_{\alpha_{A\cup B}}\times (\mathcal{B}^{\text{gr}}_{(4,4)^\intercal})_{\alpha_C} \arrow{d}{p_B^{A\cup B}\times \operatorname{Id}}\\
     (\mathcal{B}^{\text{gr}}_{(1,3,4,4,3,1)^\intercal})_{\alpha_{B\cup C}} \arrow{r}{i_{B,C}} \& (\mathcal{B}^{\text{gr}}_{(1,3,3,1)^\intercal})_{\alpha_{B}} \times  (\mathcal{B}^{\text{gr}}_{(4,4)^\intercal})_{\alpha_{C}} \\
     \end{tikzcd} $$
Consider $\alpha= (7,5,5,3,5,3,1,1,-1,-1,-3)$. We then have $\alpha_{A\cup B}= (5,3,3,1,3,1,-1)$, $\alpha_C= (1,1,-1,-1)$, $\alpha_{B\cup C}= (5,3,3,1,1,-1,-1,-3)$, and $\alpha_{B}= (3,1,1,-1)$. If we use the notation $X_\alpha$, the diagram reads

$$\begin{tikzcd} 
    X_{(7,5,5,3,5,3,1,1,-1,-1,-3)} \arrow{r}{i_{A\cup B,C}}\arrow{d}{p_{B\cup C}} \& X_{(5,3,3,1,3,1,-1)}\times X_{(1,1,-1,-1)} \arrow{d}{p^{A\cup B}_{B}\times \Id}\\
     X_{(5,3,3,1,1,-1,-1,-3)} \arrow{r}{i_{B, C}} \& X_{(3,1,1,-1)}\times X_{(1,1,-1,-1)} 
     \end{tikzcd} $$

\end{exa}

\subsubsection{Reduction maps}
We say that a nilpotent element $e'$ is smaller than a nilpotent element $e$ if we can obtain the basis diagram of $e'$ by removing some boxes from the basis diagram of $e$ (see \Cref{basis diagram} for definition of basis diagram). 
\begin{defin}[Reduction maps] \label{reduction maps}
A \textit{reduction map} is a morphism $F: (\fix)_\alpha\xrightarrow{} (\mathcal{B}_{e'}^{gr})_{\alpha'}$ in which $\alpha\neq \alpha'$ and $e'$ are smaller than or equal to $e$.
\end{defin}
An example of reduction maps is the isomorphisms considered in the proof of \Cref{standard}. We now introduce two other types of reduction maps.

\begin{exa}[Distinguished reduction] \label{distinguished reduction}
 Consider the case where $e$ is distinguished; in other words, $\lambda$ has pairwise distinct parts. Then the intersection of the kernel of $e$ and a weight space $V_i$ has dimension $0$ or $1$. Consider an arrangement of weights $\alpha$ of the form $(p_1,..., p_n)$. Let $v^{l}_{p_1}$ be the basis vector with weight $p_1$ annihilated by $e$. A point of $(\fix)_\alpha$, regarded as a flag $U^\bullet$, has the property $U_1= \bC v^{l}_{p_1}$. Therefore, the quotient flag $U^\bullet/ U^1$ can be considered as an isotropic flag of the symplectic vector space $\Span(v_j^{i})_{(i,j)\neq (l, \pm p_1)}$. Let $e'$ be a nilpotent element of $\fsp_{2n-2}$ having the associated partition $(2x_1,..., 2x_{l-1}, 2x_l-2, 2x_{l+1},...,2x_{k})$. Let $\alpha'$ be the tuple of weights $(p_2,...,p_n)$. The map $U^\bullet \mapsto U^\bullet/ U^1$ is then an isomorphism between $(\fix)_\alpha$ and $(\cB^{gr}_{e'})_{\alpha'}$. 
 
 When $e$ is not necessarily distinguished, similar maps still make sense for certain weight tuples $\alpha$. In particular, we consider the weight $2m+1> 0$ such that $\dim V_{2m+1}= \dim V_{2m+3}+ 1$. And consider a tuple $\alpha$ that contains a weight $2m+1$ that does not have any weights $2m-1$ or $2m+3$ in front of it. Then we can move this weight $2m+1$ to the first position as in the proof of \Cref{standard} to obtain a tuple $\alpha''$ with property $X_{\alpha''}\cong X_\alpha$. Now, for a flag $U^\bullet$ in $X_{\alpha''}$, we have $U^1\subset \ker e\cap V_{2m+1}$. Because $\dim V_{2m+1}= \dim V_{2m+3}+ 1$, the vector space $U^1\subset \ker e\cap V_{2m+1}$ has dimension $1$. Hence, we have $U^1= \ker e\cap V_{2m+1}$. Let $\alpha'$ be the tuple obtained from $\alpha$ by removing this weight $2m+1$. Using the same argument as in the previous paragraph, we have $X_\alpha \cong X_{\alpha''}\cong X_{\alpha'}$. We call this type of isomorphism \textbf{distinguished reduction map}. 
\end{exa}

A simple consequence is the following corollary.    
\begin{cor} \label{distinguished}
    Consider a distinguished element $e\in \fsp_{2n}$ and an $n$-tuple of weights $\alpha$. There exists a smaller, nondistinguished element $e' \in \fsp_{2n-2j}$ so that $(\fix)_\alpha \cong (\cB^{gr}_{e'})_{\alpha'}$ for some $j>0$ and an $n-j$-tuple of weights $\alpha'$.  
\end{cor}
The next example of reduction maps is a special case of the map $p_A$ defined in (\ref{pA}).
\begin{exa}[Reduction maps $F_i$] \label{ Reduction maps F}
First, $p_A$ is a reduction map. Next, consider a weight $i\in \{1,3,...,2M-1\}$. We focus on the case where $A$ has the form $A_i= \{1,3,..., 2M-1\}\setminus \{i\}$ and write $F_i$ for $p_{A_i}$. We write $\alpha_i$ for the induced tuple $\alpha_{A_i}$ (\Cref{induced tuple}). To obtain this induced tuple $\alpha_i$, we first remove all the weights $\pm i$ in $\alpha$, then shift the weights $k\neq \pm i$ by $-2$ (resp. $2$) if $k> i$ (resp. $k<-i$). In terms of the alternative characterization of $(\fix)_\alpha$ in \Cref{alternative}, the map $F_i$ is obtained by forgetting the subspaces $W_i^j$ and $W_{-i}^j$ for $1\leqslant j\leqslant  d_i$ and $1\leqslant j\leqslant  d_{-i}$, respectively.
\end{exa}
The next remark discusses some decompositions of $F_i$ into the composition of simpler maps.

\begin{rem} \label{decompositions}
For simplicity of notation, we assume $d_{-i}= 0$. Then $F_i$ is the map that forgets all subspaces $W_i^{j}$. A natural way to think about $F_i$ is to fix a permutation $(j_1,..., j_{d_i})$ of $(1,2...,d_i)$ and consider the decomposition of $F_i$ into maps that forget a single subspace $W_i^{j_{l}}$ for $1\leqslant l\leqslant d_i$. To describe this type of decomposition, we first define the varieties $X_{\alpha_i}^A$ with $A$ being a subset of $\{1,2,...,d_i\}$.

In \Cref{alternative}, we have the description of $X_\alpha$ in terms of a collection of vector spaces $\{W_p^{j}\}_{|p|\in \{1,3,...,2M-1\}}$ subject to the three sets of relations. These relations can be rewritten equivalently as follows
\begin{enumerate}
    \item $W_i^{j}\subset W_i^{j+r}$ for $r> 0$.
   \item Let $p_l$ be the $j$-th weight $i$ from the left, and let $j_{i+2t}$ be the number of weights $i+2t$ with $t>0$ in the set $\{p_1,...,p_l\}$. Then  $e^{t}W_i^{j}\subset W_{i+2t}^{j_{i+2t}}$.
   \item For $i> 0$, the subspace $e^{i}W_{-i}^{d_{-i}}$ is the orthogonal complement of the space $W_{i}^{d_{i}}$ in $V_{i}$. And $W_{-i}^j\in \OG(j, w_i)$ (see  \Cref{isotropic}).
\end{enumerate}

For a subset $A\subset \{1,2,...,d_i\}$, we let $X_{\alpha_{i}}^{A}$ be the variety that parameterizes the collection of vector spaces $\{W_p^{j}\}$, for $p\neq i$ or $p=i$ and $j\in A$, subject to the relations obtained from the relations of $X_\alpha$ in the following sense. We remove the vector spaces $W_{i}^{j}, j\notin A$ and the relations related to them in the three sets of relations $1, 2$ and $3$ above. With this notation, $X_\alpha= X_{\alpha_{i}}^{\{1,2,...,d_{i}\}}$ and $X_{\alpha_i}\cong X_{\alpha_i}^{\emptyset}$. Now, for $A_1\subset  A\subset \{1,2,...,d_{i}\}$, we have a natural map $X_{\alpha_{i}}^{A}\rightarrow X_{\alpha_{i}}^{A_1}$ that forgets the spaces $W_i^{j}$ with $j\in A\setminus A_1$.

Consider a permutation $(j_1,..., j_{d_i})$ of $(1,2...,d_i)$, let $A_l= \{j_1,...,j_{l}\}$. Then we can realize the map $F_i:X_\alpha \rightarrow X_{\alpha_i}$ as the composition of maps
$$X_{\alpha}= X_{\alpha_i}^{A_{d_i}}\rightarrow X_{\alpha_i}^{A_{d_i- 1}}\rightarrow ....\rightarrow X_{\alpha_i}^{A_1}\rightarrow X_{\alpha_i}^{\emptyset} = X_{\alpha_i}.$$
We will consider the permutations $(j_1,..., j_{d_i})$ with $j_1= d_i$, so the subspace $W_{i}^{j_1}$ is determined uniquely. Then the map $X^{A_1}_{\alpha_i}\rightarrow X_{\alpha_i}^{\emptyset}$ is an isomorphism. For $l\geqslant 2$, we let $F_i^{j_l,...,j_2}$ denote the map $X_{\alpha_i}^{A_{l}} \rightarrow X_{\alpha_i}^{A_{l-1}}$, and we have $F_i= F_i^{j_2}\circ....\circ F_i^{j_{d_i},...,j_2}$. In this notation, the first superscript from the left indicates the space that the map forgets.

When the partition $\lambda$ is small, we have a more illustrative way to denote the decompositions of $F_i$. For example, with $\alpha= (5,3,3,1,3,1,1)$, the map $F_3$ can be decomposed as follows
$$F_3: X_{\alpha}\xrightarrow{F_3^{1,2}} X_{(5,\Box,3,1,3,1,1)}\xrightarrow{F_3^{2}} X_{(5,\Box,\Box,1,3,1,1)}\cong X_{\alpha_3}= X_{(3,1,1,1)},$$
or
$$F_3: X_{\alpha}\xrightarrow{F_3^{2,1}} X_{(5,3,\Box,1,3,1,1)}\xrightarrow{F_3^{1}} X_{(5,\Box,\Box,1,3,1,1)}\cong X_{\alpha_3}= X_{(3,1,1,1)}.$$
Here we use square boxes to indicate which spaces were forgotten.
\end{rem}
The next section makes use of reduction maps to describe the variety $\fix$ for $\lambda$ having equal parts.

\subsection{The case when $\lambda$ has equal parts}
In this section, we prove some general results which do not depend on the number of parts of $\lambda$. As applications, we obtain geometric descriptions and finite models of $\fix$ for certain $\lambda$.

\subsubsection{Cases where $F_i$ is the projection from a tower of projective bundles}
Recall that we write $w_i$ for the dimension of the weight space $V_i$, and write $d_i$ for the multiplicity of the weight $i$ in the tuple $\alpha$. Let $2M-1$ denote the highest weight of $T_e$-action on $V_\lambda$.
\begin{lem} \label{equal part}
    Assume that $w_{2M-1}= w_{2M-3}$ and $\alpha$ is a standard tuple. Then the morphism $F_{2M-1}: X_\alpha\rightarrow X_{\alpha_{2M-1}}$ is the projection from a tower of projective bundles. These projective bundles are associated to $A_e$-equivariant (and $Q_e$-equivariant) vector bundles.
\end{lem}
\begin{proof}
    \textbf{Step 1}: We first give the proof of the lemma when $d_{1-2M}= 0$. In this case, we have $d_{2M-1}= w_{2M-1}$ (see Condition 3 in \ref{alternative}). Then $F_{2M-1}$ is the map that forgets the spaces $W_{2M-1}^{j}$ for $1\leqslant j\leqslant d_{2M-1}= w_{2M-1}$. We think of $F_{2M-1}$ as the composition of maps that forget the spaces $W_{2M-1}^{1}$, $W_{2M-1}^{2},..., W_{2M-1}^{w_{2M-1}}$ in the order from left to right. In the notation of \Cref{decompositions}, we have $F_{2M-1}= F_{2M-1}^{-1+w_{2M-1}}\circ ....\circ F_{2M-1}^{1,2,...,-1+w_{2M-1}}$. The lemma follows from the claim that for each $1\leqslant l\leqslant w_{2M-1}$, the morphism $F_{2M-1}^{l,...,-1+w_{2M-1}}$ below is the projection from a projective bundle. 
    $$F_{2M-1}^{l,...,-1+w_{2M-1}}: X_{\alpha_{2M-1}}^{\{l,l+1,...,w_{2M-1}\}} \rightarrow X_{\alpha_{2M-1}}^{\{l+1,...,w_{2M-1}\}}$$ 
    The map $F_{2M-1}^{l,...,-1+w_{2M-1}}$ forgets the space $W_{2M-1}^{l}$, and the conditions to recover this subspace $W_{2M-1}^{l}$ are $W_{2M-1}^{l}\subset W_{2M-1}^{l+1}$ and $W_{2M-1}^{l}\supset eW_{2M-3}^{l'}$ for certain $l'$ that depends on $\alpha$ and $l$ (see Conditions 1 and 2 in \Cref{alternative}). Since we consider the case $w_{2M-1}= w_{2M-3}$, we have that $e$ is an isomorphism between the two weight spaces $V_{2M-1}$ and $V_{2M-3}$. Therefore, we can realize the tautological vector bundle $\cW_{2M-3}^{l'}:=(X^{l+1,...,w_{2M-1}}_{\alpha_{2M-1}}, W_{2M-3}^{l'})$ as a vector subbundle of the tautological vector bundle $\cW_{2M-1}^{l+1}:=(X^{l+1,...,w_{2M-1}}_{\alpha_{2M-1}}, W_{2M-1}^{l+1})$. Let $\cW$ be the quotient vector bundle $(X^{l+1,...,w_{2M-1}}_{\alpha_{2M-1}}, W_{2M-1}^{l+1}/ (eW_{2M-3}^{l'}))$.
    Then the map $F_{2M-1}^{l,...,-1+w_{2M-1}}$ is the projection from $\bP(\cW)$, the projectization of $\cW$.

    Consider the action of $A_e$ (or $Q_e$) on $X_\alpha$. The two tautological vector bundles $\cW_{2M-3}^{l'},\cW_{2M-1}^{l+1}$ admit natural equivariant structures. Therefore, $\bP(\cW)= \bP(\cW_{2M-3}^{l'}/\cW_{2M-1}^{l+1})$ is the projectization of an equivariant vector bundle.

    \textbf{Step 2}: Now assume that $d_{1-2M}= p$ for some $0< p\leqslant \frac{w_{2M-1}}{2}$. First, we can decompose $F_{2M-1}$ as $F^{+}_{2M-1}\circ F^{-}_{2M-1}$ in which the first map forgets the spaces $W^{j}_{2M-1}$ for $1\leqslant j\leqslant w_{2M-1}- p$ and the second map forgets the spaces $W^{j}_{1-2M}$ for $1\leqslant j\leqslant p$. Next, we think of $F^{-}_{2M-1}$ as the composition of maps that forget the spaces $W_{1-2M}^{p},..., W_{1-2M}^{1}$ from left to right. Since $\alpha$ is a standard tuple, the constraints to recover $W^{j}_{1-2M}$ depend only on the relative position of the weights $3-2M$ and $1-2M$. Arguing as in Step 1, we have $F^{-}_{2M-1}$ is a composition of projections from projective bundles (we use the fact that $e^{-1}: V_{3-2M}\rightarrow V_{1-2m}$ is a vector space isomorphism here). Next, given $W_{1-2M}^{p}$, the space $W_{2M-1}^{w_{2M-1}-p}$ is determined as the orthogonal complement of $e^{2M-1} W_{1-2M}^{p}$ in $W_{2M-1}$. We then view $F_{2M-1}^{+}$ as the composition of maps that forget the spaces $W_{2M-1}^{1}$, $W_{2M-1}^{2},..., W_{2M-1}^{w_{2M-1}-p}$ in the order from left to right. The proof is concluded by repeating the same argument as in Step 1.
\end{proof}

Two other general results that can be proved similarly are the following lemmas.
\begin{lem} \label{no negative}
    If the tuple $\alpha$ does not contain any negative weights, the variety $(\fix)_{\alpha}$ is isomorphic to a tower of projective bundles over a point. And the corresponding categorical finite models of $X_\alpha$ consist only of ordinary points fixed by $Q_e$. 
\end{lem}
\begin{proof}
    We claim that $F_1$ realizes $X_\alpha$ as a tower of projective bundles over $X_{\alpha_1}$, the conclusion of the lemma then follows by induction. Decompose $F_1$ as $F_1^{1}\circ ...\circ F_1^{w_1,...,2,1}$. We consider the map $F_1^{l,...,2,1}: X_{\alpha_1}^{\{l,...,1\}}\rightarrow X_{\alpha_1}^{\{l-1,...,1\}}$. The condition to recover the space $W_1^{l}$ for an element of $X_{\alpha_1}^{\{l-1,...,1\}}$ is that $W_1^{l-1}\subset W_1^l\subset e^{-1}W_3^{l'}$ for some $l'$ depends on $\alpha$. Arguing as in Step 1 of the proof of \Cref{equal part}, we obtain that $F_1^{l,...,2,1}$ is the projection from a $\bP^{l'+ w_1- w_3- l}$-bundle. Furthermore, this projective bundle is the projectization of a $Q_e$-equivariant vector bundle. Therefore, we obtain the corresponding finite model of $(\fix)_\alpha$ by repeatedly applying \Cref{projective bundle}.
\end{proof}

Consider the map $F_i^j$ that forgets a single vector space $W_i^j$ from $X_\alpha$. We further assume that in $\alpha$, between the $j$-th weight $i$ and the $j+1$-th weight $i$, there is no weight $i-2$ or $i+2$ (if $\alpha$ is standard, it is equivalent to that these two weights $i$ are consecutive in $\alpha$).
\begin{lem} \label{forget only one}
    In the above context, the map $F^{j}_i$ is the projection from a $\bP^1$-bundle that is associated to an $A_e$-equivariant (and $Q_e$-equivariant) vector bundle. 
\end{lem}
\begin{proof}
    With the assumption that there is no weight $i-2$ or $i+2$ between the $j$-th weight $i$ and the $j+1$-th weight $i$, the condition to recover $W_i^{j}$ is simply $W_i^{j-1} \subset W_i^{j}\subset W_i^{j+1}$. Therefore, we obtain a $\bP^1$-bundle.
\end{proof}

\subsubsection{Geometry of $\cB_{(2x,2x,...,2x)}^{gr}$}
We consider the case where the partition $\lambda$ of $e\in \fsp_{2n}$ takes the form $(2x, 2x,..., 2x)$ with $x= \frac{n}{k}$, here $k$ is the number of parts of $\lambda$. Let $\alpha$ be a weight tuple of $V_\lambda$. Let $\beta= \alpha_{\{1\}}$ be the tuple obtained from $\alpha$ by picking out the weights $1$ and $-1$ in $\alpha$, keeping their orders. For example, if $\alpha= (3,1,3,3,1,-1,-1,-3)$, then $\beta= (1,1,-1,-1)$. Applying \Cref{equal part} repeatedly, we have the following corollary. 
\begin{cor} \label{equal partition}
    The variety $X_\alpha$ is a tower of projective bundles over the variety $X_\beta$. Furthermore, these projective bundles are the projectizations of $O_k$-equivariant vector bundles. 
\end{cor}
We proceed to describe $X_\beta$. Assume that $\beta$ consists of $l$ weights $-1$ and $k- l$ weights $1$ for some $1\leqslant l\leqslant \frac{k}{2}$. We show that $X_\beta$ is a tower of projective bundles over the orthogonal Grassmannian $\OG(l,k)$ as follows. Consider the projection $p$ from $X_\beta$ to $\OG(l,k)$ that sends a flag in $X_\beta$ to its intersection with the weight space $V_{-1}$. Now we regard $p$ as the map forgetting $W_1^{1},...,W_1^{k-l}, W_{-1}^{1},..., W_{-1}^{l-1}$ in order from left to right. Then $p$ is the projection from a tower of projective bundles; the details of the argument are similar to Step 1 of \Cref{equal part}. 

Similar arguments apply for the case where $\lambda$ has equal parts that are odd. We have the following proposition.

\begin{pro}\leavevmode
\begin{enumerate}
    \item Assume that $\lambda$ is of the form $((2x)^k)$. Let $l$ be the multiplicity of weight $-1$ in $\alpha$. Then $X_\alpha$ is a tower of projective bundles over $\OG(l,k)$.
    \item Assume that $\lambda$ is of the form $((2x+1)^{2k})$. Then $X_\alpha$ is a tower of projective bundles over $LG(\bC^{2k})$. Here, $LG(\bC^{2k})$ denotes the Lagrangian grassmannian of a symplectic vector space of dimension $2k$.
\end{enumerate}
    
\end{pro}
Recall that we have discussed the existence and uniqueness of the finite model of $\OG(l,k)$ in \Cref{finite model partial flag} and \Cref{finite model OG}. In summary, for $k=2$, $\OG(1,2)$ is two points; for $k\geqslant 3$, the finite model of $\OG(l,k)$ is unique, consisting of two types of orbits. If $2l<k$, the two types of orbits are the ordinary orbit fixed by $O_k$ and the special orbit fixed by $SO_k$. If $2l=k$, the two types of orbits are the ordinary orbit fixed by $SO_k$ and the special orbit fixed by $SO_k$. The finite model of $X_\alpha$ then consists of copies of the finite model of $\OG(l,k)$. 
\subsubsection{Applications}
We give some examples as illustrations of the results in Section 8.3.2. 
\begin{exa} \label{base case equal}
This example describes $\cB_{(2^k)}^{gr}$ and the corresponding varieties $X_\beta$ for $k= 2,3,4$. 
    \begin{enumerate}
        \item The variety $\cB_{(2,2)}^{gr}$ has three connected components. One of them is $X_{(1,1)}\cong \bP^1$. The other two come from $X_{(1,-1)}= \OG(1,2)$, which consists of two points.
        \item The variety $\cB_{(2,2,2)}^{gr}$ has three connected components. They are $X_{(1,1,1)}= (\bP^2, \bP^1)$, $X_{(1,1,-1)}= (\bP^1, \OG(1,3))$ and $X_{(1,-1,1)}= \OG(1,3)$.
        \item The variety $\cB_{(2,2,2,2)}^{gr}$ has eight connected components that come from six varieties $X_\beta$. They are $X_{(1,1,1,1)}\cong (\bP^3, \bP^2, \bP^1)$, $X_{(1,1,1,-1)}\cong (\bP^2,\bP^1, \OG(1,4))$, $X_{(1,1,-1,1)}\cong (\bP^1,\bP^1, \OG(1,4))$, $X_{(1,-1,1,1)}$ $\cong (\bP^1, \OG(1,4))$, $X_{(1,1,-1,-1)}$ $\cong (\bP^1, \bP^1, \OG(2,4))$ and $X_{(1,-1,1,-1)}\cong (\bP^1, \OG(2,4))$. Here, each of the last two varieties has two connected components.        
    \end{enumerate}
\end{exa}

Consider the case $\lambda= (2x_1, 2x_2)$. From \Cref{distinguished}, we see that a connected component of $\fix$ is isomorphic to a connected component of $\cB_{(2x',2x')}^{gr}$ for some $x'\leqslant x_2$. By \Cref{equal part}, a connected component of $\cB_{(2x',2x')}^{gr}$ is a tower of projective bundles over some connected component of $\cB_{(2,2)}^{gr}$. Because the dimensions of weight spaces of $\lambda$ are at most $2$, the only non-trivial bundle we have is $\bP^1$-bundle. Therefore, we have the following result.
\begin{pro} \label{2 row}
    Consider $\lambda= (2x_1, 2x_2)$ and a tuple of weights $\alpha$. The variety $(\fix)_\alpha$ is a single connected component of $\cB_{(2x_1, 2x_2)}$ if and only if the weights of $\alpha$ are all positive. Otherwise, $(\fix)_\alpha$ has two connected components that are permuted by $Q_e$. In either case, a connected component of $(\fix)_\alpha$ is a tower of $\bP^1$-bundles. 
\end{pro}
We have similar results for the cases $\lambda= (2x_1, 2x_1, 2x_1)$ and $\lambda= (2x_1, 2x_1, 2x_1, 2x_1)$. 
\begin{pro} \label{3 row equal}
    Consider $\lambda =(2x_1, 2x_1, 2x_1)$ and a tuple of weights $\alpha$. If $\alpha$ consists only of positive weights, $(\fix)_\alpha$ is a tower of projective bundles over a point. Otherwise, $(\fix)_\alpha$ is a tower of projective bundles over the variety $\OG(1,3)$.
\end{pro}

\begin{pro} \label{4 row equal}
    Consider $\lambda= (2x_1, 2x_1, 2x_1, 2x_1)$ and a tuple of weights $\alpha$. Recall that we write $w_{-1}$ for the multiplicity of the weight $-1$ in $\alpha$. We have three cases:
    \begin{enumerate}
        \item If $w_{-1}= 0$, then $X_\alpha$ is a tower of projective bundles over a point. 
        \item If $w_{-1}= 2$, then $X_\alpha$ is a tower of projective bundles over $\OG(1,4)$.
        \item If $w_{-1}= 2$, then $X_\alpha$ is a tower of projective bundles over $\OG(2,4)$.
    \end{enumerate}
    In the first two cases, $X_\alpha$ is a connected component of $\fix$. In the last case, $X_\alpha$ consists of two connected components of $\fix$. 
\end{pro}

In terms of finite models, 
\begin{enumerate}
    \item $Y_{(2x_1, 2x_2)}$ has two types of orbits as shown in \Cref{numeric 2 row}.
    \item The $O_3$-finite model of $\OG(1,3)$ consists of two points, both stabilized by $O_3$. One of them is ordinary; the other is special with respect to the unique Schur multiplier of $O_3$. Then \Cref{3 row equal} explains why $Y_{(2x_1,2x_1,2x_1)}$ consists of two types of orbits as found in \Cref{3 rows}.
    \item The $O_4$-finite model $Y_{(2x_1,2x_1,2x_1,2x_1)}$ has three different types of orbits coming from finite models of $\OG(1,4)$ and $\OG(2,4)$ (see \Cref{small OG}). This can be worked out from the system of equations in Section 5.3 without using geometric arguments. 
\end{enumerate}

\subsection{Techniques to describe reduction maps}
We have described $\fix$ when $\lambda$ has equal parts in the previous section. When $\lambda$ has distinct parts, a new phenomenon is that reduction maps $F_i$ may involve blow-up maps. Section 8.4.1 is devoted to techniques related to blow-ups. The results up to Section 8.4.2 are enough to describe the reduction maps for the case $\lambda$ having $3$ rows. Section 8.4.3 discusses more technical results that we need for the case $\lambda$ having $4$ rows. 

\subsubsection{Blow-ups}
Let $V$ be a vector space of dimension $w$. Fix an integer $1\leqslant j\leqslant w-1$. Let $X\cong \bP^{w-1} \times \Gr(j, w) $ be the variety that parameterizes pairs of subspaces $(W^1, W^j \subset V)$ where $\dim W^l= l$ for $l= 1, j$. Let $X'$ be the variety that parameterizes triples of subspaces $(W^1, W^j, W^{j+1} \subset V)$ such that $\dim W^{j+1}= j+1$ and $W^{j+1}\supset W^1\cup W^j$. We consider the projection $p: X'\rightarrow X$ that forgets the subspace $W^{j+1}$ in the triple.
\begin{lem} \label{main blow up}
   The morphism $p: X'\rightarrow X$ is the blow-up along the smooth subvariety $Y\subset X$ which parametrizes pairs $(W^1\subset W^j)$.
\end{lem}

\begin{proof}
    First, the restriction of $p$ to $X\setminus Y$ is an isomorphism. The variety $Y$ can be regarded as a projective $\bP^{j-1}$-bundle over $\Gr(j, w)$, so it has codimension $w-j$ in $X$. The fiber $p^{-1}(Y)$ is a $\bP^{w-j-1}$-bundle over $Y$. Hence, $p: X'\rightarrow X$ is birational. Next, we explain why $X'\cong \Bl_Y X$.

    We have an embedding $X\hookrightarrow \bP^{{w \choose j}- 1}\times \bP^{w-1}=: X_1$ by the Plucker embedding. Let $a_i$ and $b_{i'}$ be the homogeneous coordinate functions of the two projective spaces $\bP^{{w \choose j}- 1}$ and $\bP^{w-1}$. The condition $W^1\subset W^j$ is equivalent to $W^1\wedge W^j= 0$. So, the defining relation of $Y$ inside $X$ is given by $w\choose j+1$ functions $f_l's$ which are linear combinations of $a_ib_{i'}$. The graph of the functions $f_l's$ gives us a map $X\setminus Y\xrightarrow{g} (X\setminus Y)\times \bP^{{w\choose j+1} - 1}$. The classical definition of blow-up tells us that the blow-up variety $\Bl_{Y}X$ can be realized as the closure of the image of $g$ in $(X\setminus Y)\times \bP^{{w\choose j+1} - 1}$. 
    
    By the way we define the functions $f_l's$, they satisfy the Plucker relations of the Plucker embedding $\Gr(j+1, w)\hookrightarrow \bP^{{w\choose j+1} - 1}$. Hence, the image of $g$ lives in $X\times \Gr(j+1, w)\subset X\times \bP^{{w\choose j+1} - 1}$. This image of $g$ can be realized as the variety that parameterizes the triple of spaces $(W^1\nsubseteq W^j, W^{j+1}= W^1\cup W^j)$, an open subset of $X'$. Therefore, we get a closed embedding $\Bl_{Y}X\hookrightarrow X'$. This closed embedding is birational and hence is an isomorphism.
\end{proof}
Fix a nondegenerate bilinear form of $V$. By taking orthogonal complements, we obtain a dual version of \Cref{main blow up} as follows.
\begin{cor} \label{intersection}
    Let $X$ be the variety that parameterizes pairs $(W^{w-j}, W^{w-1})$ and $X'$ be the variety that parameterizes triples $(W^{w-j-1}\subset W^{w-j}\cap W^{w-1})$. Then $X'\cong \Bl_Y X$ where $Y$ is the variety that parameterizes $(W^{w-j}\subset W^{w-1})$.
\end{cor}

Another lemma that we will use frequently is the following. 

\begin{lem} \label{transversal}
Let $X$ be a smooth variety over $\mathbb{C}$, and $Y, Z$ be smooth subvarieties of $X$. Assume further that $Y$ and $Z$ intersect transversally. Then we have a Cartesian square
$$\begin{tikzcd} 
    \text{Bl}_{Y\cap Z}Z \arrow{r}\arrow{d} \& \text{Bl}_Y X \arrow{d}\\
    Z\arrow{r} \& X
     \end{tikzcd} $$
\end{lem}
\begin{proof}
Let $\mathcal{I}$ and $\mathcal{J}$ be the two ideal sheaves which correspond to the embedding of $Y\xhookrightarrow{f} X$ and $ Z\xhookrightarrow{g} X$. First, we have a natural surjective map of sheaves on $Z$, $g^*\mathcal{I}\rightarrow \mathcal{I}(\mathcal{O}_X/\mathcal{J})=: \mathcal{I}'$. This map induces a closed embedding $\textbf{Proj}(\oplus \mathcal{I}'^n) \rightarrow (\textbf{Proj}(\oplus (g^*\mathcal{I})^n))$. In other words, we get a closed embedding $i: \text{Bl}_{Y\times_X Z}Z\hookrightarrow$ $g^*\text{Bl}_Y X$.

We have two natural projections $p_1: \text{Bl}_{Y\times_X Z}Z\rightarrow Z$ and $p_2: g^*\text{Bl}_Y X\rightarrow Z$ that satisfy $p_1= p_2\circ i$. Since $Y$ and $Z$ intersect transversally, $Z\times_X Y$ is a proper subvariety of $Z$. Hence, the projection $p_1$ is birational. In particular, $p_1$ is an isomorphism over the open set $Z\setminus Z\times_X Y$. Next, the projection $p_2$ is an isomorphism over the same open set, and the subvariety $p_2^{-1}(Y\cap Z)$ has dimension $\dim (Y\cap Z)+ \dim X- \dim Y- 1= \dim Z- 1$. Hence, $p_2$ is birational. Thus, the closed embedding $i: \text{Bl}_{Y\times_X Z}Z\rightarrow g^*\text{Bl}_Y X$ is birational and is an isomorphism between the two varieties. 
\end{proof}

Next, we give some examples of how these results are applied in studying reduction maps. 

\begin{exa} \label{blow up six six four}
    Consider the case $\lambda= (6,6,4)$, we want to show that the following morphism is a blow-up map
    $$F_3: X_{(5,3,5,3,1,1,3,1)}\rightarrow X_{(3,3,1,1,1)}.$$
    The basis diagram of $V_\lambda$ is as follows.
    $$\ytableausetup{centertableaux, boxsize= 3em}
    \begin{ytableau}
    \none & v_{3}^{3} &v_{1}^{3} &v_{-1}^{3}& v_{-3}^{3} &\none\\
    v_{5}^{2} &v_{3}^{2} &v_{1}^{2} &v_{-1}^{2} &v_{-3}^{2}&v_{-5}^{2}\\
    v_{5}^{1} &v_{3}^{1} &v_{1}^{1} &v_{-1}^{1} &v_{-3}^{1} &v_{-5}^{1}\\
    \end{ytableau}\\    
    $$    
As in \Cref{alternative}, we have an embedding $X_{(3,3,1,1,1)}\hookrightarrow \Gr^{(4,4,2)}_{(3,3,1,1,1)}$ $= \Gr(1,2)\times \Gr(1,3)\times \Gr(2,3)$. Here, $\Gr(1,2)$ pararmetrizes $W_{3}^{1}\subset V_3= \bC^2$; $\Gr(1,3)$ and $\Gr(2,3)$ pararmetrize $W_{1}^{1}, W_{1}^{2} \subset V_1= \bC^3$. Let $Y'$ be the subvariety of $\Gr^{(4,4,2)}_{(3,3,1,1,1)}$ defined by the condition $W_1^{2}= e^{-1}W_3^{1}$. The variety $X_{(3,3,1,1,1)}$ is defined in $\Gr^{(4,4,2)}_{(3,3,1,1,1)}$ by the condition $W_1^{1}\subset W_1^{2}$, so it intersects transversally with $Y'$ defined by the conditions $W_1^{2}= e^{-1}W_3^{1}$. This intersection is isomorphic to $X_{(3,1,1,3,1)}$, which is embedded into $X_{(3,3,1,1,1)}$ as follows. In the alternative description in \Cref{alternative} of $X_{(3,3,1,1,1)}$, we add the constraint that $eW_{1}^{2}= W_{3}^1$. 

Consider the map $F_3: X_{(5,3,5,3,1,1,3,1)}\rightarrow X_{(3,3,1,1,1)}$. Since $W_3^{2}$ is determined by $eW_{1}^2$, $F_3$ is the map that forgets $W_3^{1}$. And to recover $W_3^{1}$, the condition is $W_3^{1}\subset e^{-1}W_5^{1}\cap W_3^{2}$ $= e^{-1}W_5^{1}\cap eW_1^{2}$. This description of $X_{(5,3,5,3,1,1,3,1)}$ and the result of \Cref{intersection} for $V= V_3$ (so $\dim V= 2)$ and $j=1$ give us a Cartesian square

$$\begin{tikzcd} 
    X_{(5,3,5,3,1,1,3,1)} \arrow{r}\arrow{d} \& \text{Bl}_{Y'} \Gr^{(4,4,2)}_{(3,3,1,1,1)} \arrow{d}\\
    X_{(3,3,1,1,1)}\arrow{r} \& \Gr^{(4,4,2)}_{(3,3,1,1,1)}
     \end{tikzcd} $$
Then we have an isomorphism $X_{(5,3,5,3,1,1,3,1)}\cong \Bl_{X_{(3,1,1,3,1)}}X_{(3,3,1,1,1)}$ by applying \Cref{transversal}. 
\end{exa}

\begin{exa} \label{first blowup}
    We consider $\lambda= (4,4,2)$, $\alpha= (3,1,3,1,-1)$ and the morphism $$X_{(3,1,3,1,-1)}\xrightarrow{F_3} X_{(1,1,-1)}.$$
    The basis diagram of $V_\lambda$ is as follows.
    $$\ytableausetup{centertableaux, boxsize= 3em}
    \begin{ytableau}
    \none &v_{1}^{3} &v_{-1}^{3}&\none\\
    v_{3}^{2} &v_{1}^{2} &v_{-1}^{2} &v_{-3}^{2}\\
    v_{3}^{1} &v_{1}^{1} &v_{-1}^{1} &v_{-3}^{1} \\
    \end{ytableau}\\    
    $$    
    Note that $W_3^{2}$ is determined by $eW_{1}^{2}$, so $F_3$ is the map that forgets $W_3^{1}$. The constraint is $W^2:= e^{-1}W_3^{1}\supset W_1^{1}$. In other words, we want to find a $2$-dimensional vector space $W^2\subset V_1$ such that $W^2$ contains both $W_1^{1}$ and $\bC v_{1}^{3}$. This fits the context of \Cref{main blow up}. Arguing similarly to \Cref{blow up six six four}, we have that $F_3$ is a blow-up map. The exceptional set in $X_{(1,1,-1)}$ is characterized by the constraint $W_1^{1}= \bC v_{1}^{3}$, and we can realize it as $X_{(1,3,1,-1,3)}$ for the following reasons. In the alternative description in \Cref{alternative} of $X_{(1,3,1,-1,3)}$, we must have $W_1^{1}= \bC v_{1}^{3}$, and $W_3^{1}$ is uniquely determined by $eW_{1}^{2}$. Hence, we have an embedding $X_{(1,3,1,-1,3)}\hookrightarrow X_{(1,1,-1)}$, and $X_{(3,1,3,1,-1)}\cong \Bl_{X_{(1,3,1,-1,3)}} X_{(1,1,-1)}$.
\end{exa}

\subsubsection{Reducible tuples}
We first introduce the notion of reducible tuples $\alpha$. Recall that in \Cref{reduction maps}, a reduction map is a morphism $F: (\fix)_\alpha\xrightarrow{} (\mathcal{B}_{e'}^{gr})_{\alpha'}$.
\begin{defin} [Reducible tuples] \label{reducible tuples}
    We say that a tuple $\alpha$ is \textit{reducible with respect to a reduction map $F$} if the morphism $F$ satisfies the following conditions.
    \begin{enumerate}
        \item $F$ is a composition of projections from towers of projective bundles and blow-ups along smooth subvarieties.
        \item In addition, we require these projective bundles to be associated to $A_e$-equivariant vector bundles, and the exceptional sets of the blow-ups to be $A_e$-stable. 
    \end{enumerate}
     We say that $\alpha$ is \textit{reducible} if such a reduction map $F$ exists.
\end{defin}
The following example gives a recollection of reducible tuples and the corresponding reduction maps that we have come across in this section.
\begin{exa}\label{collection of reducibles}
We have
    \begin{enumerate}
        \item If $e$ is distinguished, any tuple $\alpha$ is reducible with respect to distinguished reduction maps (\Cref{distinguished reduction}).  
        \item Consider $\lambda$ such that $\dim V_{2m+1}= \dim V_{2m+3}+ 1$ for some $m> 0$. Assume that $\alpha$ contains a weight $2m+1$ that does not have any weight $2m-1$ or $2m+3$ in front of it. Then $\alpha$ is reducible (\Cref{distinguished reduction}).  
        \item If $w_{2M-1}= w_{2M-3}$, then $\alpha$ is reducible with respect to $F_{2M-1}$ (\Cref{equal part}).
        \item If $\alpha$ does not have any negative weights, $\alpha$ is reducible with respect to $F_1$ (\Cref{no negative}).
        \item Assume that we have a weight $i$ such that the map $F_i: X_\alpha\rightarrow X_{\alpha_i}$ is the map that forgets a single space $W_i^{j}$. Assume further that the $j$-th and $j+1$-th weight $i$ are next to each other in $\alpha$, then $\alpha$ is reducible with respect to $F_i$ by \Cref{forget only one}.
        \item For the following $\lambda$, any tuple $\alpha$ is reducible.  
        \begin{itemize}
            \item $\lambda= (2x_1, 2x_2)$, see \Cref{2 row}.
            \item $\lambda= (2x_1, 2x_1, 2x_1)$, see \Cref{3 row equal} and the discussion before it.
            \item $\lambda= (2x_1, 2x_1, 2x_3)$ or $(2x_1, 2x_3, 2x_3)$ with $x_1\geqslant x_3+ 2$. These are special cases of \Cref{equal part}.
            \item $\lambda= (4,2,2)$ and $\lambda= (4,4,2)$ (\Cref{ four two two} and \Cref{four four two}).            
        \end{itemize}
    \end{enumerate}
\end{exa}

\subsubsection{Other techniques and reduction diagrams}
We give three examples that describe some features of the reduction maps for the case where $\lambda$ has $4$ parts. In each example, we give the details of a new technique for describing reduction maps $F_i$.

The following example demonstrates the importance of the choice of the reduction maps $F_i$. In particular, when a tuple $\alpha$ is not reducible with respect to $F_{2m-1}$, we consider $F_{2m-3}$ or $F_{2m+1}$.
\begin{exa} \label{use another reduction map}
    We consider $\lambda$ of the form $(2m+2, 2m+2, 2m+2, 2m), m\geqslant 2$, then we have $\dim V_{2m+1}= 3$ and $\dim V_{2m-1}= \dim V_{2m-3}= 4$. Assume that the relative position of the three weights $2m+1$, $2m-1$ and $2m-3$ in $\alpha$ is $(2m+1, 2m-1, 2m-1,2m+1, 2m+1, 2m-1,2m-3, 2m-3, 2m-3, 2m-1, 2m-3)$. We first claim that $F_{2m-1}$ may not be surjective; this is already the case when $m=2$. 
    
    For $m=2$, we have the map
    $$F_3: X_{(5,3,3,5,5,3,1,1,1,3,1)}\rightarrow X_{(3,3,3,1,1,1,1)}.$$
    The corresponding basis diagram is
    $$\ytableausetup{centertableaux, boxsize= 3em}
\begin{ytableau}
\none & \none&v_{3}^{4} &v_{1}^{4} &v_{-1}^{4} &v_{-3}^{4}   &\none &\none \\
\none & v_{5}^{3} &v_{3}^{3} &v_{1}^{3} &v_{-1}^{3} &v_{-3}^{3}   &v_{-5}^{3} &\none\\
\none & v_{5}^{2} &v_{3}^{2} &v_{1}^{2} &v_{-1}^{2} &v_{-3}^{2}   &v_{-5}^{2} &\none\\
 \none & v_{5}^{1} &v_{3}^{1} &v_{1}^{1} &v_{-1}^{1} &v_{-3}^{1}   &v_{-5}^{1} &\none\\
\end{ytableau}\\
$$
    From \Cref{alternative}, elements of $X_{(5,3,3,5,5,3,1,1,1,3,1)}$, described by sets of subspaces $\{W_i^{j}\}$, must satisfy conditions $W_3^{3}= eW_1^{3}$ and $W_3^{2}= e^{-1}W_{5}^{1}\supset \bC v_3^{4}$. So $eW_1^{3}\supset \bC v_3^4$, or $W_1^3\supset \bC v_1^{4}$. This condition $W_1^3\supset \bC v_1^{4}$ is carried over to the image of $F_3$ in $X_{(3,3,3,1,1,1,1)}$, but we do not have such a restriction for elements of $X_{(3,3,3,1,1,1,1)}$. Thus, the map $F_3$ is not surjective.   

    Returning to the general case, we notice that the relative position of $2m-1$ and $2m-3$ is $(2m-1, 2m-1, 2m-1, 2m-3, 2m-3, 2m-3, 2m-1, 2m-3)$. We claim that the map $F_{2m-3}$ makes $\alpha$ a reducible tuple. Note that the process of recovering the spaces $W_{2m-3}^{j}$ for $1\leqslant j\leqslant 4$ depends only on the relative position of $2m-1, 2m-3, 2m-5$ in $\alpha$. Since we have the condition $W_3^{2m-3}= e^{-1} W_3^{2m-1}$, to describe $F_{2m-3}$, it is left to understand how to recover $W_{2m-3}^{1}$ and $W_{2m-3}^2$. In other words, we have $F_{2m-3}= F_{2m-3}^{2}\circ F_{2m-3}^{1,2}$ (see \Cref{decompositions} for notation). And regardless of the relative position of the weights $2m-5$ and $2m-3$, we have that both $F_{2m-3}^{2}$ and $F_{2m-3}^{1,2}$ are projections from a projective $\bP^l$-bundle for some $l$. Therefore, in this example, although $\alpha$ may not be reducible with respect to $F_{2m-1}$, it is reducible with respect to $F_{2m-3}$.  
\end{exa}
The next example demonstrates that, in some cases, it is necessary to slightly modify the tuple $\alpha$ first and then consider the reduction maps $F_i$. 
\begin{exa} \label{swap}
    Consider $\lambda$ of the form $(2m+4, 2m+4, 2m+2, 2m)$. Let $\alpha$ be a standard tuple such that the relative position of the weights $2m+1, 2m-1$ and $2m-3$ is $(2m+1, 2m-1, 2m-1, 2m+1, 2m+1, 2m-1, 2m-3, 2m-3, 2m-1, 2m-3, 2m-3)$. 
    
    We first treat the case $m= 2$. The four positive weight spaces are $V_7, V_5, V_3$ and $V_1$ with the corresponding dimensions $4,4,3$ and $2$. Below is the basis diagram in this case 
    $$\ytableausetup{centertableaux, boxsize= 3em}
\begin{ytableau}
\none & \none&v_{3}^{4} &v_{1}^{4} &v_{-1}^{4} &v_{-3}^{4}   &\none &\none \\
\none & v_{5}^{3} &v_{3}^{3} &v_{1}^{3} &v_{-1}^{3} &v_{-3}^{3}   &v_{-5}^{3} &\none\\
v_{7}^{2} & v_{5}^{2} &v_{3}^{2} &v_{1}^{2} &v_{-1}^{2} &v_{-3}^{2}   &v_{-5}^{2} &v_{-7}^{2}\\
 v_{7}^{1} & v_{5}^{1} &v_{3}^{1} &v_{1}^{1} &v_{-1}^{1} &v_{-3}^{1}   &v_{-5}^{1} &v_{-7}^{1}\\
\end{ytableau}\\
$$

    Let $\alpha$ be $(7, 5, 3, 3, 7, 5, 5, 3, 1, 1, 3, 1, -1)$, then $\alpha_3= (5,3,3,5,3,1,1,1,-1)$. We claim that the map $F_3: X_{\alpha}\rightarrow X_{\alpha_3}$ is not surjective. The argument is as follows. For an element of $X_{(7, 5, 3, 3, 7, 5, 5, 3, 1, 1, 3, 1, -1)}$, by the description in \Cref{alternative}, we have the following relations: $W_3^{3}\supset W_{3}^2= e^{-1}W_5^{1}$ and $W_{3}^3\supset eW_{1}^{2}$. Now $eW_{1}^{2}$ and $e^{-1}W_5^{1}$ are two $2$-dimensional vector spaces that both lie in $W_3^{3}$, a $3$-dimensional vector space. Hence, their intersection is nontrivial. This condition is carried over to the image of $F_3$ in $X_{(5,3,3,5,3,1,1,1,-1)}$ as $\dim (e^{-1}W_3^{1}\cap W_{1}^{2})\geqslant 1$. We do not have such a restriction for the elements of $X_{(5,3,3,5,3,1,1,1,-1)}$. Similarly, when we consider $F_5: X_{(7, 5, 3, 3, 7, 5, 5, 3, 1, 1, 3, 1, -1)}\rightarrow X_{(5,3,3,5,3,1,1,3,1,-1)}$, the conditions that $W_{3}^2\supset v_3^{4}$ is carried to the image, so $F_5$ is not surjective. 

    As we have seen, neither $F_5$ nor $F_3$ makes $\alpha$ reducible. The highlight of this example is that we introduce a reduction map $s: X_\alpha \rightarrow X_{\alpha'}$, we call it the \textit{swapping map}. Then we show that $X_{\alpha'}$ is reducible with respect to $F_5$, so $X_\alpha$ is reducible. 

    In particular, we swap the positions of the first weight $5$ and the first weight $3$ (both from the left), to obtain $\alpha'= (7, 3, 5, 3, 7, 5, 5, 3, 1, 1, 3, 1, -1)$. Elements of $X_{\alpha'}$ satisfy the condition that $W_3^{1}$ is annihilated by $e$, so $W_3^{1}= \bC v_3^{4}$. Consequently, we have that $$X_{\alpha'}\cong X_{(7,\Box, 5,3, 7, 5, 5, 3, 1, 1, 3, 1, -1)} \cong  X_{(7, 5,\Box , 3, 7, 5, 5, 3, 1, 1, 3, 1, -1)}$$ (see the notation for the boxes in \Cref{decompositions}). And for $X_\alpha$, the map that forgets $W_3^{1}$ gives the projection from a $\bP^1$-bundle: $X_\alpha \rightarrow X_{(7, 5,\Box , 3, 7, 5, 5, 3, 1, 1, 3, 1, -1)}$. Thus, we have $X_\alpha\cong (\bP^1, X_{\alpha'})$. Next, we have $X_{\alpha'}\cong X_{(7, 5, 3, 7, 5, 5, 3, 1, 1, 3, 1, -1)}=: X_{\alpha''}$ by distinguished reduction (see \Cref{distinguished reduction}). Then the map $F_5: X_{\alpha''}\rightarrow X_{\alpha''_5}= X_{(5, 3, 5, 3, 1, 1, 3, 1, -1)}$ is the projection from a projective $\bP^1$-bundle. In summary, after a chain of reduction maps, we have $X_\alpha \cong (\bP^1, \bP^1, X_{(5, 3, 5, 3, 1, 1, 3, 1, -1)})$. 
\end{exa}
    In general, consider a weight $2m-1>0$ such that $\dim V_{2m-1}= \dim V_{2m+1}+1$, and a tuple $\alpha$ such that in $\alpha$ the first two weights $2m-1$ follow after the first weight $2m+1$. Let $\alpha$ be the tuple obtained from $\alpha$ by swapping the position of the first $2m+1$ and the first $2m-1$.
    \begin{defin} \label{swap map}
        We have a \textit{swapping map}, $s_{2m-1, 2m+1}: X_\alpha \rightarrow X_{\alpha'}$, replacing the subspace $W_{2m-1}^1$ of $X_\alpha$ to be ker $(e)\cap V_{2m-1}$. Then it realizes $X_\alpha$ as $(\bP^1, X_{\alpha'})$ by \Cref{forget only one}.
    \end{defin}
    
The next example describes $F_i$ as the composition of two blow-ups. We will see that the description of the centers of the blow-ups is slightly more complicated than what we had in the case where $\lambda$ has $3$ parts.
\begin{exa} \label{double blow up}
    (This example provides details for case V-3 in the proof of \Cref{subcollections}).
    
    We consider $\lambda= (6,6,6,4)$ and $\alpha= (5,3,5,5,3,1,3,1,3,1,1)$. In the notation of \Cref{decompositions}, we consider the following decomposition of $F_3$    $$X_{(5,3,5,5,3,1,3,1,3,1,1)}\xrightarrow{F_3^{3,2,1}}  X_{(5,3,5,5,3,1,\Box,1,3,1,1)} \xrightarrow{F_3^{2,1}}  X_{(5,3,5,5,\Box,1,\Box,1,3,1,1)} \xrightarrow{F_3^{1}}  X_{(5,\Box,5,5,\Box,1,\Box,1,3,1,1)}= X_{\alpha_3}$$
    First, the map $F_3^{3,2,1}$ forgets the space $W_3^3$. To recover this subspace $W_3^3$, the constraint is $W_3^3\supset W_3^{2}\cup eW_1^{2}$. As we see that both $W_3^{2}$ and $eW_1^{2}$ contain $eW_1^{1}$, we can rewrite the constraint for $W_3^3$ as $W_3^3/ eW_1^{1}\supset W_3^{2}/ eW_1^{1} \cup eW_1^{2}/ eW_1^{1}$. This condition fits in the context of \Cref{main blow up} where $V$ is $V_3/eW_1^{1}$. We deduce that $F_3^{3,2,1}$ is a blow-up map. The center $Z$ of this blow-up is the subvariety of $X_{(5,3,5,5,3,1,\Box,1,3,1,1)}$ that satisfies $W_3^{2}= eW_1^{2}$. Hence, the variety $Z$ is isomorphic to $X_{(5,3,5,5,3,1,1,\Box,3,1,1)}$. In addition, we have an isomorphism $X_{(5,3,5,5,3,1,1,\Box,3,1,1)} \cong X_{(5,3,5,5,3,1,1,3,1,3,1)}$, since the space $W_3^{3}$ in the latter is determined by $eW_1^{3}$. In summary, we have shown that $X_\alpha \cong \Bl_{X_{(5,3,5,5,3,1,1,3,1,3,1)}} X_{(5,3,5,5,3,1,\Box,1,3,1,1)}$ and that the blow-up morphism is $F_3^{3,2,1}$. 

    Next, we consider the map $X_{(5,3,5,5,3,1,\Box,1,3,1,1)} \xrightarrow{F_3^{2,1}}  X_{(5,3,5,5,\Box,1,\Box,1,3,1,1)}$. This map $F_3^{2,1}$ forgets $W_3^{2}$, and the constraint to recover $W_3^2$ is $W_3^2\supset W_3^{1}\cup eW_1^{1}$. Similarly to the previous paragraph, an application of \Cref{main blow up} and \Cref{transversal} gives us that $F_3^{2,1}$ is a blow-up map. The center of this blow-up is the subvariety of $X_{(5,3,5,5,\Box,1,\Box,1,3,1,1)}$ that satisfies $W_3^{1}= eW_1^{1}$. This subvariety is isomorphic to $X_{(5,3,1,5,5,\Box,\Box,1,3,1,1)}$, which is isomorphic to $X_{(3,1,3,3,1,1,1)}$ since we can now disregard the spaces $W_3^{j}$ and then replace the weights $5$ by the weights $3$. For the last map in the decomposition of $F_3$, it is straightforward that $X_{(5,3,5,5,\Box,1,\Box,1,3,1,1)} \xrightarrow{F_3^{1}}  X_{(5,\Box,5,5,\Box,1,\Box,1,3,1,1)}$ is the projection from a $\bP^1$-bundle. Therefore, the tuple $\alpha$ is reducible.
\end{exa}

Next, we introduce some notation which describes properties of reduction maps in a more convenient way.
\begin{notation}\label{reduction diagrams} 
    We write $X\xrightarrow{\bP^k} Y$ for the projection from a $\bP^k$-bundle. We write $X\xrightarrow {\Bl (Z)} Y$ for the blow-up map along $Z\subset Y$. 
\end{notation} 

\begin{exa} \label{Example reduction diagrams}
We give some examples of describing reduction maps using diagrams (some of the examples come from the next subsections). 
\begin{enumerate}
    \item Consider $\alpha= (5,3,3,1,1,5,3,1,5,3,1)$. Note that $W_3^{2}$ and $W_3^{3}$ are determined by $eW_1^{2}$ and $eW_1^{2}$. Hence $F_3$ is the map forgetting $W_3^{1}$, and it is the projection from a $\bP^1$-bundle. Using the notation in \Cref{decompositions} and the notation we just introduced, we have the reduction diagram that describes $X_\alpha$ as follows.
    $$X_{(5,3,3,1,1,5,3,1,5,3,1)}\cong X_{(5,3,3,1,1,5,\Box,1,5,3,1)}\cong X_{(5,3,\Box,1,1,5,\Box,1,5,3,1)}\xrightarrow{\bP^1}X_{(5,\Box,\Box,1,1,5,\Box,1,5,3,1)}= X_{\alpha_3}$$
Here, how we have decomposed $F_3$ can be read from the diagram. At each step, we forget a subspace $W_3^{j}$, and replace the $j$-th weight $3$ by a box.
    \item We have    $$X_{(5,3,3,5,5,3,1,1,3,1,1)}\xrightarrow{\bP^1}X_{(3,5,3,5,5,3,1,1,3,1,1)}\cong X_{(5,3,5,5,3,1,1,3,1,1)}.$$
    The first morphism $s_{3,5}$ is explained in \Cref{swap map}. The second morphism, which is an isomorphism, comes from a distinguished reduction map (see \Cref{distinguished reduction}). 
    \item In \Cref{reduction six six four}, case I.3b, we have 
    $$X_{(5,3,5,3,1,1,3,1)}\xrightarrow{\Bl(X_{(3,1,1,3,1)})} X_{(5,\Box,5,3,1,1,3,1)}\cong X_{(5,\Box,5,\Box,1,1,3,1)}= X_{(3,3,1,1,1)}.$$
    \item In \Cref{one negative}, we have    $$X_{(3,3,1,3,1,-1,-3,1)}\xrightarrow{\Bl(X_{(1,1,-1,1)})}X_{(3,3,1,3,\Box,-1,-3,1)}\xrightarrow{\bP^1} X_{(3,3,\Box,3,\Box,-1,-3,1)}\cong X_{(1,1,1,-1)}.$$
    \item When the description is more complicated like in \Cref{swap} we can specify the weight tuples, and the reduction steps with their properties as follows.

    $$X_{(7, 5, 3, 3, 7, 5, 5, 3, 1, 1, 3, 1, -1)}\xrightarrow[s_{3,5}]{\bP^1} X_{(7, 3, 5, 3, 7, 5, 5, 3, 1, 1, 3, 1, -1)} \xrightarrow[d]{\cong} X_{(7, 5, 3, 7, 5, 5, 3, 1, 1, 3, 1, -1)} $$
    $$\xrightarrow[F_5]{\bP^1} X_{(5, 3, 5, 3, 1, 1, 3, 1, -1)}$$
    Here we write $d$ for the distinguished reduction map.
    \item When the description of the center of the blow-up is not immediate like in \Cref{double blow up}. We include more details in the diagram as follows.
    $$X_{(5,3,5,5,3,1,3,1,3,1,1)}\xrightarrow{\Bl(X_{(5,3,5,5,3,1,1,\Box,3,1,1)}\cong X_{(5,3,5,5,3,1,1,3,1,3,1)})}  X_{(5,3,5,5,3,1,\Box,1,3,1,1)} $$
    $$\xrightarrow{\Bl(X_{(5,3,1,5,5,\Box,\Box,1,3,1,1)}\cong X_{(3,1,3,3,1,1,1)})}  X_{(5,3,5,5,\Box,1,\Box,1,3,1,1)} \xrightarrow{\bP^1}  X_{(5,\Box,5,5,\Box,1,\Box,1,3,1,1)}= X_{\alpha_3}$$    
\end{enumerate}
\end{exa}
\subsection{Geometry of $\cB_{(2x_1,2x_2,2x_3)}^{gr}$}
We have described $\cB_{(2x_1,2x_1,2x_1)}^{gr}$ in \Cref{3 row equal}. By distinguished reductions, the two cases left to consider are $(2x_1=2x_2> 2x_3)$ and $(2x_1>2x_2= 2x_3)$. If $w_{2M-1}= w_{2M-3}$ (equivalently $x_1- x_3\geqslant 2$), then $\alpha$ is reducible by \Cref{equal part}. Thus, our consideration reduces to the case $x_1- x_3= 1$. In other words, we have $\lambda= (2m+2, 2m, 2m)$ or $\lambda= (2m+2, 2m+2, 2m)$ for some $m\geqslant 1$. 
\subsubsection{Base cases}
This first subsection describes the components of $\fix$ in the smallest cases where $m=1$ (base cases).
\begin{exa} \label{ four two two}
    Consider $\lambda = (4,2,2)$. Thanks to \Cref{no negative}, we are left with the cases where $\alpha$ contains the weight $-1$. All standard tuples $\alpha$ with weight $-1$ are $(3,1,1,-1)$, $(3,1,-1,1)$, $(1,-1,3,1)$, $(1,3,1,-1)$. First, we recall the basis diagram of $V_\lambda$ below. 
    $$\ytableausetup{centertableaux, boxsize= 3em}
    \begin{ytableau}
    \none &v_{1}^{3} &v_{-1}^{3}&\none\\
    \none &v_{1}^{2} &v_{-1}^{2} &\none\\
    v_{3}^{1} &v_{1}^{1} &v_{-1}^{1} &v_{-3}^{1} \\
    \end{ytableau}\\    
    $$

We now describe the varieties $X_\alpha= (\fix)_\alpha$ using the notation in \Cref{alternative}. First, the space $W_{3}^1$ is $\bC v_{3}^{1}$. Next, the space $W_{1}^{2}$ is determined by the orthogonal complement of $W_{-1}^{1}$. Hence, our focus is the two spaces $W_{1}^{1}$ and $W_{-1}^{1}$. Recall that we use sequences of vectors $(u_1, u_2, u_3, u_4)$ as expressions for flags in $\fix$.
    \begin{itemize}
        \item We have $X_{(3,1,-1,1)}\cong X_{(1,-1,1)}\cong \OG(1,3)$ and  $X_{(3,1,1,-1)}\cong X_{(1,1,-1)} \cong (\bP_1,\OG(1,3))$. The isomorphisms that we use in both cases are $F_3$. And the varieties $X_{(1,-1,1)}$ and $X_{(1,1,-1)}$ are described in \Cref{equal partition}. 
        \item For $\alpha= (1,-1,3,1)$, we require $e^{-1}\bC u_1= \bC u_2$ and $u_1$ to be annihilated by $e$. Therefore, $\bC u_2$ is an isotropic line in $\Span(v_{-1}^{2}, v_{-1}^{3})$. We get $X_{(1,-1,3,1)}\cong \OG(1,2)$, which consists of two points.
        \item $\alpha= (1,3,1,-1)$. Note that $u_1$ is annihilated by $e$, so it belongs to the intersection $\Span(v_{1}^{2}, v_{1}^3)\cap (eu_4)^\perp$. This intersection always has dimension $1$ when $u_4\notin \bC v_{-1}^{1}$. This is the case since $u_4$ is isotropic in $V_{-1}$. Here, we use the dot product with respect to the basis $\{v_{-1}^{1}, v_{-1}^{2},v_{-1}^{3}\}$. As a consequence, we have $X_{(1,3,1,-1)}\cong \OG(1,3)$.
    \end{itemize}
    The case $\alpha= (1,-1,3,1)$ gives an explanation for the appearance of 2-point orbits that have the stabilizer $\{1,z_1,z_{23}, z_{123}\}$ in $Y_e$. This type of orbit did not appear in the case $\lambda= (2x_1,2x_1,2x_1)$. 
\end{exa}
We now proceed to the case $\lambda= (4,4,2)$. The technical details are similar, so we will be brief in our explanation. 
\begin{exa} \label{four four two}
    Consider $\lambda = (4,4,2)$. Thanks to \Cref{no negative}, we only need to describe the variety $X_\alpha= (\fix)_\alpha$ when $\alpha$ contains the weight $-1$ or $-3$. We use the expression by eigenvectors $(u_1,u_2,...,u_5)$ of $T_e$ for a flag in $\fix$. The basis diagram of $V_\lambda$ is as follows.
    $$\ytableausetup{centertableaux, boxsize= 3em}
    \begin{ytableau}
    \none &v_{1}^{3} &v_{-1}^{3}&\none\\
    v_{3}^{2} &v_{1}^{2} &v_{-1}^{2} &v_{-3}^{2}\\
    v_{3}^{1} &v_{1}^{1} &v_{-1}^{1} &v_{-3}^{1} \\
    \end{ytableau}\\    
    $$    
We have two main cases. 
    \begin{itemize}
        \item Case 1: The tuple $\alpha$ contains the weight $-3$. The cases are $(1,3,1,-1,-3)$, $(3,1,1,-1,-3)$, $(3,1,-1,-3,1)$. In these cases, we have $W_{-1}^{1}= eW_{-3}^{1}$, and $W_{-3}^{1}\in \OG(1,V_{-3})= \OG(1,2) $ because $\dim V_{-3}= 2$. Once $W_{-3}^{1}$ is determined, we have $W_{3}^{1}=e^3 W_{-3}^{1}$. Therefore, it is left to determine $W_1^{1}$. 
        
        Working out the details, we get $X_{(1,3,1,-1,-3)}\cong X_{(3,1,-1,-3,1)}\cong \OG(1,2)$ and $X_{(3,1,1,-1,-3)}\cong (\bP^1, \OG(1,2))$.
        \item Case 2: The tuple $\alpha$ does not contain the weight $-3$. We write $(p_1,p_2,...,p_5)$ for $\alpha$.
        \begin{itemize}
            \item Case 2.1: If $p_1=1$, we have $U_1= \bC v_{1}^{3}$. Then $X_\alpha$ is isomorphic to $X_{\alpha'}$ where $\alpha'$ is $(p_2,p_3,p_4,p_5)$.
            \item Case 2.2: If $p_1= p_2= 3$, we have $X_\alpha \cong (\bP^1, X_{(p_3,p_4,p_5)})$. The isomorphism is given by the morphism $F_3$. The varieties $X_{(p_3,p_4,p_5)}\subset \cB_{(2,2,2)}^{gr}$ have been discussed in \Cref{equal partition}. We get that $X_\alpha$ is a tower of projective bundles over a point.
            \item Case 2.3: If $\alpha= (3,1,-1,3,1)$, then $\bC u_1$ and $\bC u_2$ are determined by $e^2 \bC u_3$ and $e\bC u_3$ (note that $e^2 \bC u_3\neq 0$ because $u_3$ is isotropic in $V_{-1}$). Therefore, $X_{\alpha}\cong \OG(1,3)$.
        \end{itemize}         
    \end{itemize}

The last case $\alpha= (3,1,3,1,-1)$ was discussed in \Cref{first blowup}. There, we have shown that $X_\alpha$ is isomorphic to a blow-up, $\Bl_{X_{(1,3,1,-1,3)}} X_{(1,1,-1)}$. We have $X_{(1,1,-1)}\cong (\bP^1, \OG(1,3))$, a projective $\bP^1$-bundle over $\OG(1,3)$. The variety $X_{(1,3,1,-1,3)}$ is two points (how we embed it in $X_{(1,1,-1)}$ is explained in \Cref{first blowup}). Therefore, $X_\alpha$ is the blow-up along two points of $(\bP^1, \OG(1,3))$.

\end{exa}

\subsubsection{Reductions to base cases}

The main result of this section is \Cref{geometric 3 rows}. A key ingredient of the proof is the following proposition.
\begin{pro} \label{3 row main}
    Consider the cases $\lambda= (2m+2,2m+2,2m)$ or $\lambda= (2m+2,2m,2m)$. Let $\alpha$ be a standard tuple of weights. Then $\alpha$ is reducible.
\end{pro}
\begin{proof}
    The proposition is proved when $m= 1$ in \Cref{four four two} and \Cref{ four two two}. Consider $m> 1$. Here, we present the proof for the case $\lambda= (2m+2, 2m+2, 2m)$. The case $\lambda= (2m+2,2m,2m)$ is proved with the same techniques and is somewhat simpler. 
    
    When $\lambda= (2m+2, 2m+2, 2m)$, we have $n= 3m+ 1$. Consider a standard $n$-tuple of weights $\alpha= (p_1,...,p_n)$. Thanks to \Cref{no negative}, we focus on the cases where $\alpha$ has negative weight(s). We use the notation $U^\bullet$ for the flags in $X_\alpha$. Recall that we write $d_i$ for the multiplicities of the weights $i$ in $\alpha$. In other words, $d_i= \dim U^\bullet\cap V_i$.

    \textbf{Step 1.} We first show that $\alpha$ is reducible if it contains the weight $-3$. By \Cref{isotropic}, the spaces $U^\bullet \cap V_{-1}$ (resp. $U^\bullet \cap V_{-3}$) are isotropic in $V_{-1}$ (resp. $V_{-3}$). As the dimensions of both two weight spaces $V_{-1}, V_{-3}$ are $3$, we get $d_{-1}= d_{-3}= 1$. Then $d_1= d_3= 3-1= 2$. We now consider the reduction map $F_1: X_\alpha \rightarrow X_{\alpha_1}$ (see \Cref{ Reduction maps F} for a description of $\alpha_1$). 
    
    First, $W_{-1}^{1}$ is determined by $eW_{-3}^{1}$, then $W_{1}^{2}$ is determined as the orthogonal complement of $eW_{-1}^{1}$ in $V_{1}$. Hence, the fiber of $F_1$ depends on how we determine the space $W_{1}^{1}\subset W_{1}^{2}$. By a case-by-case consideration of relative positions of the weights $3,1,-1,-3$ in $\alpha$, we see that $F_1$ is either an isomorphism or the projection from a $\bP^1$-bundle. The details are similar to the case where $\lambda$ has equal parts (Section 8.3).

    \textbf{Step 2.} We now consider the case where $\alpha$ has exactly one negative weight, $-1$ with multiplicity $1$. Since we consider $m>1$, we have that $d_{1-2m}$ and $d_{-1-2m}$ are both $0$. Hence, we have $d_{2m+1}= 2$ and $d_{2m-1}= 3$. Because $U^1$ is annihilated by $e$, we have $p_1= 2m+1$ or $2m- 1$. 
    
    If $p_1= 2m-1$, we have $U^1= \bC v_{2m-1}^{3}$. By distinguished reduction, the map sending $U^\bullet$ to $U^\bullet/U_1$ is an isomorphism between $X_\alpha$ and $X_{(p_2,...,p_n)}$. Hence, $\alpha$ is reducible if $p_1= 2m-1$. 
    
    If $p_1= 2m+1$, then $p_2= 2m+1$ or $2m-1$ for $\alpha$ is standard. Assume $p_1=p_2= 2m+1$. Similarly to Case 2.2 in \Cref{four four two}, we have $X_\alpha \cong (\bP^1, X_{(p_3,...,p_n)})$. Therefore, it is left to prove that $\alpha$ is reducible when $p_1= 2m+1$ and $p_2= 2m-1$. 
    
    We show that $\alpha$ is reducible by considering the map $F_{2m-1}$. In the context of \Cref{Cartesian}, we let $A$ be $\{2m-1\}$ and $B$ be $\{2m-3, 2m+1\}$ and $C$ be $\{1,3,..., 2m-5\}$ (it is possible that $C=\emptyset $). We use the second expression of the Cartesian square in \Cref{Cartesian}. In this expression, we have several different nilpotent elements $e_D$ for subsets $D\subset \{1,3,..., 2M-1\}$. The process of obtaining the associated partition $\lambda_D$ is discussed at the beginning of \Cref{subsect 8.2.2}. Working out the details, we have $\lambda_{A\cup B}= (6,6,4)$ and $\lambda_B= (4,4,2)$. Plugging in the second diagram in \Cref{Cartesian}, we have a Cartesian square:

    $$\begin{tikzcd} 
    (\fix)_\alpha \arrow{r}{i_{A\cup B,C}}\arrow{d}{p_{B\cup C = F_{2m-1}}} \& (\mathcal{B}^{\text{gr}}_{(6,6,4)})_{\alpha_{A\cup B}}\times (\mathcal{B}^{\text{gr}}_{e_C})_{\alpha_C} \arrow{d}{p^{A\cup B}_{B}\times \Id= F_3\times \Id}\\
     (\mathcal{B}^{\text{gr}}_{e_{B\cup C}})_{\alpha_{B\cup C}} \arrow{r}{i_{B, C}} \& (\mathcal{B}^{\text{gr}}_{(4,4,2)})_{\alpha_{B}}\times (\mathcal{B}^{\text{gr}}_{e_{C}})_{\alpha_{C}} \\
     \end{tikzcd} $$
    In this diagram, the map $p_B^{A\cup B}$ on the right is just $F_3$. Our first aim is to show when $\lambda= (6,6,4)$, the map $F_3$ is reducible. Then we explain how the reducibility of $F_{2m-1}$ follows.

    \textbf{Step 3}. Thanks to Step 1 and Step 2, we have restricted ourselves to the case of a standard tuple $\alpha$ that has exactly one negative weight $-1$, and $\alpha$ has the form $(2m+1, 2m-1,...)$. As a consequence, the induced partition $\alpha_{A\cup B}$ is a tuple of $8$ weights that have the form $(5,3...)$ with at most one negative weight (which is $-1$). To finish the proof of the proposition, we verify the following two statements.
    \begin{claim} \label{claim six six four}
    For $\lambda= (6,6,4)$, any standard tuple $\alpha$ having the form $(5,3,...)$ with at most one negative weight is reducible with respect to $F_3$.
    \end{claim}
    \begin{claim}  \label{claim reducible}
    Consider the tuples $\alpha$ in \Cref{claim six six four}. In the above diagram, if the induced tuple $\alpha_{A\cup B}$ is reducible with respect to $F_3$, the tuple $\alpha$ is reducible with respect to $F_{2m-1}$.
    \end{claim}
    The proof for \Cref{claim six six four} requires a case-by-case consideration; we give the details in \Cref{reduction six six four} below. For \Cref{claim reducible}, since the square is Cartesian, if $F_3$ is the projection from a tower of projective bundles, then $F_{2m-1}$ is. Hence, it is left to show that $F_{2m-1}$ is a blow-up along a smooth subvariety when $F_3$ is. This is an application of \Cref{transversal}, the verification of transversality is given after \Cref{reduction six six four}.   
\end{proof}

\begin{exa}{(Proof of \Cref{claim six six four})} \label{reduction six six four}

    Consider the case $\lambda= (6,6,4)$. We want to study the properties of the map $F_3$ for certain tuples $\alpha$ to finish the proof of \Cref{3 row main}. As analyzed in Step 3 of the proof of \Cref{3 row main}, we consider $\alpha$ of the form $(5,3,...)$ with at most one negative weight.
    
    To describe $F_3$, we need to further decompose the map as in \Cref{decompositions}. Here, we explain how the decomposition in \Cref{decompositions} works for $F_3$. The map $F_3$ forgets the spaces $W_{3}^{1}$ and $W_{3}^{2}$ (because $W_3^{3}$ is just $V_3$). Therefore, we have two ways to write $F_3$ as compositions: $F_3^{2}\circ F_3^{1,2}$ or $F_3^{1}\circ F_3^{2,1}$. The rule is that $F_3^{i,j}$ first forget the space $W_3^{i}$ and then $F_3^{j}$ forget the space $W_3^{j}$. For the sources and the targets of the maps in the decomposition, see \Cref{decompositions} for more details. We will show that $\alpha$ is reducible with respect to the map $F_3$ by case-by-case consideration. Recall that we only need to consider the cases $\alpha$ has $p_1= 5, p_2=3$.
    
    \textbf{Case I}. The tuple $\alpha$ consists only of nonnegative weights. Note that we have shown in \Cref{no negative} that $\alpha$ is reducible with respect to $F_1$, here we specifically focus on $F_3$.
    
    We consider the cases based on the relative position of the weights $3$ and $1$ in $\alpha$. All possible relative positions are $(3,3,3,1,1,1)$, $(3,3,1,3,1,1)$, $(3,3,1,1,3,1)$, $(3,1,3,3,1,1)$, $(3,1,3,1,3,1)$. Recall that we have restricted ourselves to considering the case $\alpha$ being standard (see \Cref{standard tuple} and \Cref{standard}).
    \begin{enumerate}
        \item If the relative position is $(3,1,3,1,3,1)$, then $F_3$ is an isomorphism. If the relative position is $(3,3,3,1,1,1)$, the decomposition $X_{\alpha}\xrightarrow{F_3^{2,1}}X_{\alpha_3}^{1}\xrightarrow{F_3^{1}} X_{\alpha_3}$ gives us the projection from a projective bundle in each step. The argument is similar to the proof of \Cref{no negative} (note that the ranks of the projective bundles depend on the relative positions of the weights $5$ and $3$). Thus, in these two cases, $\alpha$ is reducible by $F_3$.
        \item Suppose that the relative position is $(3,1,3,3,1,1)$. The standard tuples $\alpha$ we consider are $(5,3,1,5,3,3,1,1)$ and $(5,3,1,3,5,3,1,1)$.
        \begin{enumerate}
            \item Consider $\alpha = (5,3,1,5,3,3,1,1)$. Since $W_3^{1}$ is determined by $eW_{1}^{1}$, the map $F_3$ is the map that forgets the space $W_3^{2}$. We get $X_{(5,3,1,5,3,3,1,1)} \cong (\bP^1, X_{\alpha_3}= X_{(3,1,3,1,1)})$ and $\alpha$ is reducible with respect to $F_3$.
            \item Consider $\alpha = (5,3,1,3,5,3,1,1)$. Since $W_3^{1}$ is determined by $eW_{1}^{1}$ and $W_3^{2}$ is determined by $e^{-1} W_{5}^{1}$, the map $F_3$ is an isomorphism. Therefore, $\alpha$ is irreducible with respect to $F_3$.
        \end{enumerate}
        \item Suppose that the relative position is $(3,3,1,1,3,1)$. The standard tuples $\alpha$ we consider are $(5,3,3,1,1,5,3,1)$ and $(5,3,5,3,1,1,3,1)$.
        \begin{enumerate}
            \item Consider $\alpha= (5,3,3,1,1,5,3,1)$. Since $W_{3}^2$ is determined by $e^{-1}W_5^{1}$ (or $eW_{1}^2$), the map $F_3$ is the map that forgets $W_{3}^1$. We find that $X_{(5,3,3,1,1,5,3,1)} \cong (\bP^1, X_{\alpha_3}= X_{(3,1,1,3,1)})$, and $\alpha$ is reducible with respect to $F_3$.
            \item Consider $\alpha = (5,3,5,3,1,1,3,1)$. This case was considered in \Cref{blow up six six four}, the map $F_3$ gives us a realization $X_\alpha\cong \Bl_{X_{(3,1,1,3,1)}} X_{(3,3,1,1,1)}$. Here, the variety $X_{(3,1,1,3,1)}$ is embedded into $X_{(3,3,1,1,1)}$ as follows. In the alternative description in \Cref{alternative} of $X_{(3,3,1,1,1)}$, we add the constraint that $eW_{1}^{2}= W_{3}^1$.
            \end{enumerate}
        \item If the relative position is $(3,3,1,3,1,1)$. The standard tuples $\alpha$ we consider are $(5,3,5,3,1,3,1,1)$ and $(5,3,3,1,5,3,1,1)$.
        \begin{enumerate}
            \item Consider $\alpha= (5,3,3,1,5,3,1,1)$. Since $W_{3}^2$ is determined by $e^{-1}W_5^{1}$, we get $X_{(5,3,3,1,5,3,1,1)}\cong (\bP^1, X_{(3,1,3,1,1)})$.             
            \item Consider $\alpha= (5,3,5,3,1,3,1,1)$. We decompose $F_3$ as $X_{\alpha}\xrightarrow{F_3^{1,2}}X_{\alpha_3}^{2}\xrightarrow{F_3^{2}} X_{\alpha_3}$. Here, the variety $X_{\alpha_3}^{2}$ is obtained by recovering the space $W_{3}^{2}$. And the condition is $W_{3}^{2}\supset eW_1^{1}$, so $X_{\alpha_3}^{2}\cong (\bP^1, X_{\alpha_3})$. Next, for given $X_{\alpha_3}^{2}$, to recover $W_{3}^{1}$, the condition is $W_{3}^{1}\subset W_{3}^{2}\cap e^{-1} W_{5}^{1}$. This fits the context of \Cref{intersection} for $w=3$ and $j=1$. Arguing as in \Cref{blow up six six four}, we have $X_\alpha \cong \Bl_{X_{(5,3,3,1,5,3,1,1)}} X_{\alpha_3}^{2}$. So $F_3$ is a composition of the projection from a projective bundle and a blow-up map. 
        \end{enumerate}
    \end{enumerate}
    
    \textbf{Case II}. The tuple $\alpha$ contains the weight $-1$ with multiplicity $1$.\\
    This case is similar to Case I for the following reasons. When we describe the map $F_3$, the task is to recover the spaces $W_{3}^1$ and $W_{3}^2$. Recall that in the description in \Cref{alternative}, these spaces are determined by their relations with $W_{1}^{1}, W_{1}^{2}$ and $W_{5}^{1}$. Hence, the property of the map $F_3: X_\alpha \rightarrow X_{\alpha_3}$ does not change when we add the space $W_{-1}^{1}$ and remove the space $W_{1}^{3}$. For example, the properties of $F_3$ for the relative position $(3,3,1,3,1,-1)$ or $(3,3,1,-1,3,1)$ are the same for the case where the relative position is $(3,3,1,3,1,1)$.
\end{exa}

With the descriptions of the exceptional sets from \Cref{reduction six six four} (Cases 3.b and 4.b), we are ready to prove \Cref{claim reducible}.
\begin{proof}[Proof of \Cref{claim reducible}]
    We explain that for $\alpha_{A\cup B}= (5,3,5,3,1,1,3,1)$ (Case 3.b), the diagram in the proof of \Cref{3 row main} satisfies the transversality condition in \Cref{transversal}. The explanation for $\alpha_{A\cup B}= (5,3,5,3,1,3,1,1)$ (Case 4.b) is similar. First, recall that we have a diagram $$\begin{tikzcd} 
    (\fix)_\alpha \arrow{r}{i_{A\cup B,C}}\arrow{d}{p_{B\cup C = F_{2m-1}}} \& (\mathcal{B}^{\text{gr}}_{(6,6,4)})_{\alpha_{A\cup B}}\times (\mathcal{B}^{\text{gr}}_{e_C})_{\alpha_C} \arrow{d}{p^{A\cup B}_{B}\times \Id= F_3\times \Id}\\
     (\mathcal{B}^{\text{gr}}_{e_{B\cup C}})_{\alpha_{B\cup C}} \arrow{r}{i_{B, C}} \& (\mathcal{B}^{\text{gr}}_{(4,4,2)})_{\alpha_{B}}\times (\mathcal{B}^{\text{gr}}_{e_{C}})_{\alpha_{C}} \\
     \end{tikzcd} $$

    In the context of \Cref{transversal}, we consider $X= X_{\alpha_B}\times X_{\alpha_C}$ and $Z= X_{B\cup C}$. Here $\alpha_B$ is $(3,3,1,1,1)$, and the variety $Y$ is isomorphic to $X_{(3,1,1,3,1)}\times X_{\alpha_C}$. It is defined in $X$ by adding the condition $W_3^{1}= eW_{1}^{2}$. Similarly, we deduce that the variety $Y\cap Z$ is defined in $Z$ by adding the condition $W_{2m-1}^{1}= eW_{2m-3}^{2}$, the shift of weights here comes from how we define $i_{B,C}$ and the induced tuples (see \Cref{induced tuple}). In terms of \Cref{alternative}, we see that $Y\cap Z$ is isomorphic to a variety $X_{\alpha_{B\cup C}'}$ where $\alpha_{B\cup C}'$ can be obtained from $\alpha_{B\cup C}$ by changing the subtuple $(2m-1, 2m-1, 2m-3, 2m-3, 2m-3)$ of $\alpha_{B\cup C}$ to $(2m-1, 2m-3, 2m-3, 2m-1, 2m-1)$ and keeping other weights intact. Since $X_{\alpha_{B\cup C}'}$ is smooth, the intersection $Y\cap Z$ is smooth. Moreover $Y$ is defined by adding relations between $\{W_{2m-1}^{j}\}$ and $\{W_{2m-3}^{l}\}$, while $Z$ are defined by adding relations between $\{W_{2m+1}^{j}\}$ and $\{W_{2m-1}^{l}\}$ (see \Cref{alternative}). Therefore, $Y$ and $Z$ intersect transversally.  
\end{proof}

As \Cref{claim six six four} and \Cref{claim reducible} are verified, \Cref{3 row main} is proved. In the process, we have shown that we can choose the reduction maps so that the centers of the blow-ups have the forms $X_{\beta}$ for some tuples $\beta$.

Next, we define a collection $\cC_3$ of $A_e$-varieties as follows. We consider a vector space $V$ of dimension $3$ with a basis $(v_1, v_2, v_3)$. Equip $V$ with the dot product. Consider $A_e$ as the group of diagonal matrices with entries $1$ or $-1$, then $A_e$ is realized as a subgroup of $O(V)$. We have natural actions of $A_e$ on the varieties $\OG(1, \Span(v_2, v_3))$, $\OG(1, \Span(v_1, v_2))$ and $\OG(1,V)$. These are the base varieties for the collection $\cC_3$.
\begin{defin} \label{collection C_3}
    The collection $\cC_3$ is the minimal collection of $A_e$-varieties containing $4$ varieties $\pt$, $\OG(1,V)$, $\OG(1, \Span(v_2, v_3))$, $\OG(1, \Span(v_1, v_2))$ and stable under the following operations.
\begin{itemize}    
    \item If we have $X\in \cC_3$, then $\bP(\cV)\in \cC_3$ for any $A_e$-equivariant vector bundle $\cV$ over $X$.
    \item For an $A_e$-equivariant embedding $Z\hookrightarrow X$ with $X, Z\in \cC_3$, we have the blow-up $\Bl_Z X$ belonging to $\cC_3$.
\end{itemize}
\end{defin}
With the geometric analysis of $(\fix)_\alpha$ in this subsection, we have proved the following theorem.
\begin{thm} \label{geometric 3 rows}
For $\lambda= (2x_1, 2x_2, 2x_3)$, the variety $(\fix)_\alpha$ belongs to $\cC_3$. In other words, we can obtain $X_\alpha$ from the following process after a finite number of steps. We start with $X_0 \subset \{\pt, \OG(1,\Span(v_1, v_2)),$ $ \OG(1,\Span(v_2, v_3)), \OG(1,3)\}$, then at each step we either build a projective bundle or blow-up along a subvariety isomorphic to some $X_{\beta}$. Moreover, each of the steps can be done $A_e$- equivariantly.     
\end{thm}

From \Cref{projective bundle} and \Cref{finite model blow up}, the varieties in the collection $\cC_3$ admit categorical finite models. Thus, \Cref{geometric 3 rows} gives the verification of our main theorem (\Cref{main}) for the case $\lambda$ has $3$ rows. Recall that in \Cref{unique} we have explained that if finite models $Y_\alpha$ exist, they are unique. Therefore, the finite model $Y_e$ of $\spr$ (and of $\fix$) is the union of the $Y_\alpha$'s where $\alpha$ runs through the set of weight tuples with respect to $\lambda$. We now proceed to the discussion of finite models $Y_\alpha$ and $Y_e$. The result agrees with the numerical results that we obtained in \Cref{subsect 4.2} as follows.

The finite model $Y_{(2x_1,2x_1,2x_1)}$ was discussed after \Cref{4 row equal}. We now consider the case where $\lambda$ has the form $(2x_1, 2x_1, 2x_3)$. Recall that a reduction map $F$ from $(\fix)_\alpha$ has the target $(\cB_{e'})_{\alpha'}$ where $e'$ is smaller than or equal to $e$ (see \Cref{reduction maps}). Specifically, the reduction maps that we use in the proof of \Cref{3 row main} have the property that they preserve the type of partitions as follows. When $\lambda= (2x_1, 2x_1, 2x_3)$, the element $e'$ has the associated partition $\lambda'= (2x_1', 2x_1', 2x_3')$ for some $x_1'\leqslant x_1$ and $x_3'\leqslant x_3$. Moreover, if $F$ involves a blow-up along some $(\cB_{e''})_{\beta}$ then $e''$ has associated partition $\lambda''= (2x_1'', 2x_1'', 2x_3'')$ for some $x_1''\leqslant x_1$ and $x_3''\leqslant x_3$. We then deduce that a type of centrally extended orbit appears in $Y_{(2x_1,2x_1,2x_3)}$ if and only if it appears in $Y_{(4,4,2)}$. In \Cref{four four two}, we have seen that the finite model $Y_{(4,4,2)}$ has three types of $A_e$-orbits (or $Q_e$-orbits). Hence $Y_{(2x_1,2x_1,2x_3)}$ has the same three types of orbit. They are the trivial orbit, the 2-point orbit fixed by $z_{12}$ and the special orbit with the unique Schur multiplier of $A_e/\{z_{123},1\}$. 

Similarly, we obtain the following table that records the types of centrally extended orbits in $Y_{\lambda}$ for different types of $\lambda$ with $3$ parts. All notation is from \Cref{subsect 4.2}. Here $a_0$ is the multiplicity of the trivial orbit, $a_1$ is the multiplicity of the 2-point orbit fixed by $z_1$, $a_{12}$ is the multiplicity of the 2-point orbit fixed by $z_{12}$, and $s$ is the multiplicity of the special orbit. 

\begin{center}
\begin{tabular}{ |c|c|c|c|c|c| } 
 \hline
 $\lambda$ & non-zero multiplicities\\
 \hline
 $(2x_1, 2x_1, 2x_1)$ & $a_0, s$\\
 \hline
 $(2x_1, 2x_1, 2x_3), x_1> x_3$ & $a_{12}, a_0, s$\\
 \hline
 $(2x_1, 2x_3, 2x_3), x_1> x_3$ & $a_{1}, a_0, s$\\
 \hline
 $(2x_1, 2x_2, 2x_3), x_1> x_2> x_3$ & $a_{1}, a_{12}, a_0, s$\\
 \hline 
 \end{tabular}
\end{center}
The result in this table agrees with the numerical caculations before \Cref{not same fix point}. For example, when $\lambda= (2k,2j,2i)$, we have the formula $a_{1}= \frac{{i+j+k\choose i}EC(2j,2k)- {i+j+k\choose j}EC(2i,2k)}{2}$ which gives us $a_1= 0$ when $i=j$. This corresponds to $\lambda= (2x_1,2x_1,2x_1)$ or $\lambda= (2x_1,2x_1,2x_3)$ in the table.
\begin{rem} \label{slogan 3 rows}
    As in \Cref{geometric 3 rows}, the varieties $(\fix)_\alpha$ belong to $\cC_3$. The collection $\cC_3$ is built up from four varieties $\pt$, $\OG(1,\Span(v_1,v_2))$, $\OG(1,\Span(v_2,v_3))$ and $\OG(1,3)$. And for a generic partition $\lambda= (2x_1, 2x_2, 2x_3)$ with $x_1> x_2> x_3$, the finite model $Y_\lambda$ has $4$ types of orbit $\bO_0$, $\bO_1$, $\bO_{12}$, $\bO_{0}^{+}$. Here, we explain the relations between these two observations in the following table. 

\begin{center}
\begin{tabular}{ |c|c|c|c|c|c| } 
 \hline
Generators & $\pt$ & $\OG(1,\Span(v_1,v_2)) \cong 2\pt$& $\OG(1,\Span(v_2,v_3)) \cong 2\pt$ & $\OG(1,3) \cong \bP^1$\\
 \hline
Finite models & $\bO_0$ & $\bO_1$ & $\bO_2$& $\bO_0^{+}\cup$ $\bO_0$\\
\hline
 \end{tabular}
\end{center}   
\end{rem}
\subsection{Geometry of $\cB_{(2x_1,2x_2,2x_3,2x_4)}^{gr}$} 
In Section 8.6.1, we first state the main result of this section, \Cref{geometric main 4 rows}. We then discuss the interpretation of this theorem in terms of finite models of $\fix$. The rest of the section (and Appendix A) is devoted to proving \Cref{geometric main 4 rows}. We first reduce our consideration to $X_\alpha$ with $\alpha$ having at most $3$ negative weights in Section 8.6.2. We explain the reduction steps in Section 8.6.3, and the base cases are studied in Appendix A.
\subsubsection{Main results}
Consider $\lambda$ having the form $(2x_1, 2x_2, 2x_3, 2x_4)$. Recall from \Cref{4 row Uniqueness} that the possibilities for orbits in $Y_\lambda$ are $\bO_{0}$, $\bO_{12}$,  $\bO_{23}$,  $\bO_{1,2}$, $\bO_{1,23}$, $\bO_{3,12}$, $\bO_{12,23}$, $\bO_{0}^{1}$, $\bO_{0}^{12}$, $\bO_{0}^{123}$, $\bO_{2,3}^+$, $\bO_{12,23}^+$. For details, see \Cref{Notation for types of $A_{e}^{'}$ orbits}. An application of the result in this subsection is that we do not have orbits $\bO_{2,3}^{+}$ or $\bO_{0}^{12}$ in $Y_e$. Then it follows that we can solve for the multiplicites of the other ten orbits in $Y_\lambda$ (see \Cref{4 row Uniqueness} and \Cref{can solve equations}). In other words, the numerical invariants that we obtain in Section 3 are enough to determine the structure of $Y_\lambda$. These results follow from the main theorem of this subsection, a $4$-row version of \Cref{geometric 3 rows}. 

To state the theorem, we first define the collection $\cC_4$ in a similar way to how we define the collection $\cC_3$ (see \Cref{collection C_3}). Consider a vector space $V$ of dimension $4$ with a basis $(v_1, v_2, v_3, v_4)$. Equip $V$ with the dot product. We realize $A_e$ as the group of diagonal matrices with diagonal entries $1$ or $-1$ in $O(V)$.

\begin{defin} \label{collection C_4}    
    First, we consider the following varieties: $\pt$, $\OG(1, \Span(v_1, v_2))$, $\OG(1, \Span(v_2, v_3))$, $\OG(1, \Span(v_3, v_4))$, $\OG(1, \Span(v_2, v_3, v_4))$, $\OG(1, \Span(v_1, v_2, v_3))$, $\OG(1, V)$, $\OG(1, \Span(v_1,v_2))\times \OG(1, \Span(v_3,v_4))$, $\OG(1, \Span(v_1,v_4))\times \OG(1, \Span(v_2,v_3))$ and $\OG(2,V)$. The action of $A_e$ on $V$ induces natural actions of $A_e$ on these ten varieties. The collection $\cC_4$ is the minimal collection of $A_e$-varieties containing these ten varieties and stable under the following operations.
\begin{itemize}
    \item If we have $X\in \cC_4$, then $\bP(\cV)\in \cC_4$ for any $A_e$-equivariant vector bundle $\cV$ over $X$.
    \item For an $A_e$-equivariant embedding $Z\hookrightarrow X$ with $X, Z\in \cC_4$, we have the blow-up variety $\Bl_Z X$ belonging to $\cC_4$.
\end{itemize}
\end{defin}
Then we have our main theorem. 
\begin{thm} \label{geometric main 4 rows}
    For $\lambda= (2x_1, 2x_2, 2x_3, 2x_4)$, the varieties $(\fix)_\alpha$ belong to $\cC_4$. 
\end{thm}
The next remark is similar to \Cref{slogan 3 rows}. We give the readers a brief idea of how the ten orbits appear from the categorical finite models of the ten generators of $\cC_4$.  

\begin{rem} \label{types of orbits}
   The ten varieties that generate $\cC_4$ are quadrics and orthogonal grassmannians. Their finite models with respect to $O(V')$-actions (for corresponding subspaces $V'\subset V$) are explained in \Cref{quadrics}. These spaces $V'$  are stable under $A_e$-actions, so we have maps $A_e\rightarrow O(V')$. Pulling back the Schur multipliers from $O(V')$ to $A_e$, a simple deduction gives us the following table of $A_e$-finite models.
    \begin{center}
\begin{tabular}{ |c|c|c|c|c|c| } 
 \hline
Varieties & Finite models\\
\hline
$\pt$ & $\bO_0$\\
\hline
$\OG(1, \Span(v_1, v_2))\cong 2\pt$ & $\bO_{12,3}$\\
\hline
$\OG(1, \Span(v_2, v_3))\cong 2\pt$ & $\bO_{1,23}$\\
\hline
$\OG(1, \Span(v_3, v_4))\cong 2\pt$ & $\bO_{1,2}$\\
\hline
$\OG(1, \Span(v_2, v_3, v_4))\cong \bP^1$ & $\bO_0^{1}\cup \bO_0$\\
\hline
$\OG(1, \Span(v_1, v_2, v_3))\cong \bP^1$ & $\bO_0^{123}\cup \bO_0$\\
\hline
$\OG(1,V)\cong \bP^1\times \bP^1$ & $\bO_{12,23}^{+}\cup \bO_0\cup \bO_0$\\
\hline
$\OG(1, \Span(v_1,v_2))\times \OG(1, \Span(v_3,v_4))\cong 4 \pt$ & $\bO_{12}$\\
\hline
$\OG(1, \Span(v_2,v_3))\times \OG(1, \Span(v_1,v_4)) \cong 4 \pt$ & $\bO_{23}$\\
\hline
$\OG(2,V)\cong \bP^1\cup \bP^1$ & $\bO_{12,23}\cup \bO_{12,23}^{+}$\\
\hline

 \end{tabular}
\end{center} 
Hence, thanks to \Cref{projective bundle} and \Cref{finite model blow up}, we can construct a categorical finite model of any $X\in \cC_4$ from the ten types of orbits in the table. From \Cref{unique}, the finite models of $(\fix)_\alpha$ for the $4$-row case exist uniquely. Therefore, we see that the types of orbit that can appear in $Y_\lambda$ are precisely those ten. The table below records the types of centrally extended orbits in $Y_{\lambda}$ for different $\lambda$.  
\begin{center}
\begin{tabular}{ |c|c|c|c|c|c| } 
 \hline
 $\lambda$ & The orbits in $Y_\lambda$\\
 \hline
 $(2x_1, 2x_1, 2x_1,2x_1)$ & $\pt$, $\bO_{12,23}^+$, $\bO_{12,23}$ \\
 \hline
 $(2x_1, 2x_1, 2x_1,2x_4), x_1> x_4$ & $\pt$, $\bO_{12,23}^+$, $\bO_{12,23}$, $\bO_{0}^{123}$   \\
 \hline
 $(2x_1, 2x_1, 2x_3, 2x_3), x_1> x_3$ & $\pt$, $\bO_{12,23}^+$, $\bO_{12,23}$, $\bO_{12,3}$, $\bO_{1,2}$, $\bO_{12}$\\
 \hline
 $(2x_1, 2x_2, 2x_2, 2x_2), x_1> x_2$ & $\pt$, $\bO_{12,23}^+$, $\bO_{12,23}$, $\bO_{0}^{1}$\\
 \hline
 $(2x_1, 2x_1, 2x_3,2x_4), x_1> x_3> x_4$ & $\pt$, $\bO_{12,23}^+$, $\bO_{12,23}$, $\bO_{12,3}$, $\bO_{1,2}$, $\bO_{12}$, $\bO_{0}^{123}$\\
 \hline 
 $(2x_1, 2x_2, 2x_2,2x_4), x_1> x_2> x_4$ & $\pt$, $\bO_{12,23}^+$, $\bO_{12,23}$, $\bO_{1,23}$, $\bO_{0}^{1}$, $\bO_{0}^{123}$, $\bO_{23}$\\
 \hline 
 $(2x_1, 2x_2, 2x_3,2x_3), x_1> x_3> x_4$ & $\pt$, $\bO_{12,23}^+$, $\bO_{12,23}$, $\bO_{12,3}$, $\bO_{1,2}$, $\bO_{12}$, $\bO_{0}^{1}$\\
 \hline 
 $(2x_1, 2x_2, 2x_3,2x_4), x_1> x_2> x_3> x_4$ & $\pt$, $\bO_{12,23}^+$, $\bO_{12,23}$, $\bO_{1,23}$, $\bO_{0}^{1}$, $\bO_{0}^{123}$, $\bO_{2,3}$, $\bO_{23}$, $\bO_{1,2}$, $\bO_{12}$\\
 \hline 
 \end{tabular}
\end{center}
\end{rem}
We will see how this table comes from the process of describing the varieties $X_\alpha:= (\fix)_\alpha$. The geometry of $X_\alpha$ is studied in a similar way as in the case $\lambda$ having $3$ parts. We show that almost all tuples of weights $\alpha$ are reducible, so our consideration is reduced to some base cases. Although the technical details are mostly similar to Section 8.5, the case-by-case consideration requires more laboring work. 
\subsubsection{Reductions to the cases $\alpha$ having few negative weights}
First, by distinguished reductions (\Cref{distinguished reduction}), we can assume that $e$ is not distinguished. And by \Cref{no negative}, our focus is restricted to the tuples $\alpha$ with some negative weights. Recall that we write $d_i$ for the multiplicity of weight $i$ in the tuple $\alpha$ (so $d_i= \dim V_{-i}\cap U^\bullet$). And we have $d_i+d_{-i}= \dim V_i$, $d_{-i}\leqslant d_i$ for $i> 0$ (see \Cref{alternative}). In particular, since $\dim V_1= 4$, we have $d_{-1}\leqslant 2$. The next lemma is an analogue of Step 1 in the proof of \Cref{3 row main}.

\begin{lem} \label{one negative}
    Consider the case where $\lambda$ has $4$ parts and consider a standard tuple of weights $\alpha$. In the following cases, the tuple $\alpha$ is reducible.
    \begin{enumerate}
        \item The tuple $\alpha$ has both weights $-1$ and $-3$ with multiplicities one, then it is reducible with respect to $F_{1}$.
        \item The tuple $\alpha$ has both weights $-1$ and $-3$ with multiplicities two, then it is reducible with respect to $F_{1}$.
        \item The tuple $\alpha$ has weight $-1$ with multiplicity two, and weight $-3$ and $-5$ with multiplicity one, then it is reducible with respect to $F_{3}$.
    \end{enumerate}
    As a consequence, any tuple $\alpha$ having weight $-5$ or having weight $-3$ of multiplicity two is reducible.
\end{lem}
\begin{proof}
    We give the proof of the first part of the lemma. We claim that it is enough to prove the lemma in the case where the basis diagram of $\lambda$ has $4$ columns. This is done similarly to Step 2 of the proof of \Cref{3 row main}. We choose $A= \{1\}$, $B= \{3\}$ and $C= \{5,...,2m-1\}$ in the context of \Cref{Cartesian} to obtain a Cartesian square. We show that $X_{\alpha_{A\cup B}}$ is reducible with respect to $F_1$, then it follows that $X_\alpha$ is reducible with respect to $F_{1}$. 

    Next, we consider the case where the basis diagram of $\lambda$ has $4$ columns. In other words, there are four weight spaces $V_3, V_1, V_{-1}$ and $V_{-3}$ in $V_\lambda$. Since $\lambda$ has $4$ rows, we have $\dim V_1= \dim V_{-1}= 4$. We further assume that $\dim V_{3}= 4$, the other cases are proved similar (and simpler). As we are proving Part 1, we now consider the case $d_{-1}= d_{-3}= 1$. Since $\alpha$ is standard and we are considering the case $\alpha$ has weights $3$ and $1$ both of multiplicities one, $\alpha$ contains a sequence of $3$ consecutive weights $(1,-1,-3)$. Next, we consider the relative positions of the weights $3$ and $1$, each with multiplicity three.
    \begin{itemize}
        \item We have that $F_1$ is the projection from a tower of projective bundles for the following relative positions of the weights $3$ and $1$: $(3,3,3,1,1,1)$, $(3,3,1,1,3,1)$. The verification is similar to case I of \Cref{reduction six six four}.        
        \item For the relative position $(3,3,1,3,1,1,3)$, we have three possibilities of $\alpha$. The geometric descriptions of the maps $F_1$ in these cases are similar to the cases $\lambda$ having $3$ rows.
        \begin{itemize}
            \item $\alpha= (3,3,1,3,1,1,-1,-3)$. We have $F_1$ is the projection from a tower of projective bundles by decomposing it as $F_1^{1} \circ F_1^{2,1}$. 
            \item $\alpha= (3,3,1,3,1,-1,-3,1)$. In this case, we consider the decomposition $$F_1: X_{(3,3,1,3,1,-1,-3,1)}\xrightarrow{F_1^{2,1}} (\bP^1, X_{(1,1,1,-1)}) \xrightarrow{F_1^{1}} X_{(1,1,1,-1)}.$$ The map $F_1^{2,1}$ is a blow-up along a subvariety of $(\bP^1, X_{(1,1,1,-1)})$ that is isomorphic to $X_{(1,1,-1,1)}$.
            \item $\alpha= (3,3,1,-1,-3,3,1,1)$. We have $F_1= F_1^{2}$, which is the projection from a $\bP^1$-bundle. 
        \end{itemize}  
        \item If the relative position is $(3,1,3,3,1,1)$, the situation is similar to the previous case. Since $W_3^{1}= eW_1^{1}$, the map $F_3$ is the map that forgets $W_3^{2}$. If $\alpha= (3,1,-1,-3,3,3,1,1)$ or $(3,1,3,3,1,1-1,-3)$, then $F_3$ is the projection from a $\bP^1$-bundle. If $\alpha= (3,1,3,3,1,-1,-3,1)$, we have $X_\alpha\cong \Bl_{X_{(1,-1,1,1)}} X_{(1,1,1,-1)}$.
    \end{itemize}
    Thus, Part 1 is verified, Part 2 and Part 3 are proved similarly as follows. For Part 2, we consider the relative positions of weights $\pm 1, \pm 3$ each with multiplicity two. The map $F_1$ is the projection from a tower of projective bundles in all cases. For Part 3, we consider the relative positions of weights $5,3$ (with multiplicity $3$) and weight $1$ (with multiplicity $2$). The map $F_3$ is the projection from a tower of projective bundles, except for the case of $(5,5,3,5,3,1,3,1)$. In this case, $F_3$ is a blow-up. The details of the argument for Part $2$ (resp. Part 3) is similar to the case $\lambda$ having $2$ rows (resp. $3$ rows) that we have elaborated on in previous subsections.
\end{proof}

\subsubsection{Proofs of \Cref{geometric main 4 rows}}
We are now ready to describe $X_\alpha$ in the case where $\lambda$ has $4$ rows. We start by defining six collections of $A_e$-varieties $\cC_4^{(6,6,6,4),0}$, $\cC_4^{(6,6,6,4),1}$ $\cC_4^{(6,6,4,4),0}$, $\cC_4^{(6,6,4,4),1}$, $\cC_4^{(6,4,4,4),0}$, and $\cC_4^{(6,4,4,4),1}$. Similarly to how we define $\cC_4$ in \Cref{collection C_4}, these collections have the properties of being closed under taking projective bundles or blowing up along smooth varieties. And now we list the varieties that generate these collections.
\begin{itemize}
    \item For $i=0, 1$, the collection $\cC_4^{(6,6,6,4),i}$ contains the connected components of $\fix$ for $e$ has the associated partition $\lambda\in $ $\{$ $(6+ 2i,6+ 2i,6+ 2i,4+ 2i)$, $(8+ 2i,6+ 2i,6+ 2i,4+ 2i),(8+ 2i,8+ 2i,6+ 2i,4+ 2i)\}$. 
    \item For $i=0, 1$, the collection $\cC_4^{(6,6,4,4),i}$ contains the connected components of $\fix$ for $e$ has the associated partition $\lambda\in $ $\{$ $(6+ 2i,6+ 2i,4+ 2i,4+ 2i)$, $(8+ 2i,6+ 2i,4+ 2i,4+ 2i)\}$.
    \item For $i=0, 1$, the collection $\cC_4^{(6,4,4,4),i}$ contains the connected components of $\fix$ for $e$ has the associated partition $\lambda\in $ $\{$ $(6+ 2i,4+ 2i,4+ 2i,4+ 2i)\}$.
\end{itemize}

\begin{rem}
    The readers may notice the absence of some small partitions $\lambda$ in the definitions of these collections. For example, we require $\cC_4^{(6,6,6,4),0}$ to contain connected components of $\fix$ for $\lambda= (6,6,6,4)$, but not $(4,4,4,2)$. The reason is that for a tuple of weights $\alpha$ of $\lambda= (2x_1,2x_2,2x_3,2x_4)$, we have $X_\alpha\cong X_{\alpha'}$ for certain $\alpha'$ that comes from $\lambda'= (2x_1+2,2x_2+2,2x_3+2,2x_4+2)$. 
    
    In particular, assuming $x_4= m$, we choose $\alpha'$ to be the tuple such that $\alpha'_{2m+1}= \alpha$, and the relative position of the weights $2m+1$ and $2m-1$ in $\alpha'$ is $(2m+1, 2m-1, 2m+1, 2m-1, 2m+1, 2m-1, 2m+1, 2m-1)$ (if there are weights $1-2m$ in $\alpha$, we add weights $-1-2m$ in $\alpha'$ accordingly). Then we have $F_{2m+1}: X_{\alpha'}\rightarrow X_\alpha$ is an isomorphism.
\end{rem}

Our strategy to prove \Cref{geometric main 4 rows} is to first prove that $X_\alpha$ belongs to some of the collection that we have defined. Then in Appendix A, we show that the above six collections are subcollections of $\cC_4$. Now we consider $\lambda$ of the form $(2x_1, 2x_2, 2x_3, 2x_4)$ with $x_4= m$ and $x_1>x_4$ (the case $\lambda$ having equal parts is addressed in \Cref{4 row equal}). Thanks to \Cref{one negative}, we only need to consider three cases for the negative weights in $\alpha$: $(-1)$, $(-1,-1)$ or $(-1,-1,-3)$, so $d_{-3}= 0$ or $1$. For such a tuple of weights $\alpha$, the following theorem specifies which collection of the above six collections $X_\alpha$ belongs. 

In the context of \Cref{Cartesian}, we choose $A= \{2m-1\}$ and $B= \{2m+1, 2m-3\}$, then $\lambda_{A\cup B}$ has one of the forms $(6,6,6,4)$, $(6,6,4,4)$ or $(6,4,4,4)$.
\begin{thm} \label{subcollections}
    Consider $\lambda$ of the form $(2x_1, 2x_2, 2x_3, 2m)$ with $x_1> m$, and a tuple of weights $\alpha$ as above. Assuming $m\geqslant d_{-3}+2$, we have $X_\alpha \in \cC_4^{\lambda_{A\cup B}, d_{-3}}$.
\end{thm}
\begin{proof}
We first prove the theorem in the case $\lambda_{A\cup B} =(6,6,6,4)$, (equivalently $x_3-x_4\geqslant 1$). Later, we explain at the end of the proof how the other two cases of $\lambda_{A\cup B}$ can be dealt with similarly. 

We will consider all the relative positions of the weights $2m+1$, $2m-1$ and $2m-3$ in $\alpha$, and we mainly use $F_{2m-1}$ to prove that $\alpha$ is reducible. For some of the cases where $F_{2m-1}$ is not surjective, we consider $F_{2m+1}$, $F_{2m-3}$, or the swapping map (as in \Cref{use another reduction map} and \Cref{swap}). Note that we have $d_{2m-1}=4$ as a consequence of our assumptions on $\alpha$ and $m$.

Similarly to the case where $\lambda$ has $3$ rows, we can study the properties of $F_{2m-1}$ by reducing to $\lambda= (6,6,6,4)$ and studying $F_3$. As $d_{2m-1}= 4$, the tuples $\alpha_{A\cup B}$ (see \Cref{induced tuple} for notation) we consider have the property $d_3= 4$. In the following, we first consider the case where the multiplicity of weight $1$ in $\alpha_{A\cup B}$ is $4$ (correspondingly, $d_{2m-3}=4$ in $\alpha$). The case $d_{2m-3}= 2,3$ is explained at the end of the proof.

Similarly to \Cref{reduction six six four}, we consider possibilities for the relative positions of the weights $5$, $3$ and $1$ in $\alpha_{A\cup B}$. To systematically cover the cases, we first give the list of beginnings (relative positions of $5$ and $3$) and ends (relative positions of $3$ and $1$).

\begin{center}
\begin{tabular}{c|c}
 \textbf{List of beginnings} &  \textbf{List of ends}\\
 
\textbf{I.}     5,3,3,5,3,5,3 &    \textbf{1.} 3,3,1,1,3,1,3,1\\
\textbf{II.}     5,5,3,3,3,5,3 &  \textbf{2.}  3,3,1,1,3,3,1,1\\ 
\textbf{III.}    5,3,3,5,5,3,3 &  \textbf{3.}  3,3,1,3,1,3,1,1\\ 
\textbf{IV.}     5,3,5,3,3,5,3 &  \textbf{4.}  3,3,1,3,1,1,3,1\\
\textbf{V.}      5,3,5,5,3,3,3 &  \textbf{5.}  3,3,1,3,3,1,1,1\\
\textbf{VI.}     5,5,3,5,3,3,3 &  \textbf{6.}  3,3,3,1,3,1,1,1\\
 \textbf{VII.}   5,5,3,3,5,3,3 &  \textbf{7.}  3,3,3,1,1,3,1,1\\
\textbf{VIII.}   5,3,5,3,5,3,3 &  \textbf{8.}  3,3,3,1,1,1,3,1\\   
 & \textbf{9.} 3,1,3,3,3,1,1,1\\
 & \textbf{10.}3,1,3,3,1,3,1,1\\
 & \textbf{11.} 3,1,3,3,1,1,3,1\\
 & \textbf{12.} 3,1,3,1,3,3,1,1\\
\end{tabular}
         
\end{center}
As we now consider $\lambda= (6,6,6,4)$, we write $\alpha$ instead of $\alpha_{A\cup B}$ for the sake of notation simplicity. For example, Case I-12 corresponds to $\alpha= (5,3,1,3,1,5,3,5,3,1,1)$. We omit the trivial end ($3,1,3,1,$ $3,1,3,1$) because $F_3$ is clearly an isomorphism in this case. And we have omitted the cases where the beginnings have the form $(3,...)$ since distinguished reduction maps (\Cref{distinguished reduction}) are applicable in these cases. We now consider other cases of $\alpha$.

\textbf{Case I}, the relative position of $5$ and $3$ is $(5,3,3,5,3,5,3)$. In this case, the subspaces $W_3^{2}$ and $W_3^{3}$ are uniquely determined by $e^{-1}W_5^{1}$ and $e^{-1}W_5^{2}$. So the map $F_3$ is the map that forgets $W_3^{1}$. If the relative position of $3$ and $1$ has the forms $(3,1,...)$, then $W_{3}^{1}= eW_{1}^1$; $F_3$ is an isomorphism. If the relative position of $3$ and $1$ has the forms $(3,3,...)$, then $F_3$ is the projection from a projective bundle by \Cref{forget only one}.

\textbf{Case II}, the relative position of $5$ and $3$ is $(5,5,3,3,3,5,3)$, so $W_3^{3}$ is determined by $e^{-1}W_5^{2}$. We have two subcases.
\begin{itemize}
    \item If the relative position of the weights $3$ and $1$ takes the form $(3,1,3,...)$ or $(3,3,1,1,...)$, then $W_3^{1}$  (resp. $W_3^{2}$) is determined. In these cases, the map $F_3$ is the map that forgets $W_3^{2}$ (resp. $W_3^{1}$), which is the projection from a projective bundle by \Cref{forget only one}.
    \item In the other cases, the relative position of the weights $3$ and $1$ takes the form $(3,3,1,3,...)$. We decompose $F_3$ as $F_3^{1}\circ F_3^{2,1}$ first by forgetting $W_3^{1}$ and then forgetting $W_3^{2}$. Both $F_3^{2,1}$ and $F_3^{1}$ are projections from projective $\bP^1$-bundles.
\end{itemize}    

\textbf{Case IV} (we left Case III as the last case), the relative position of $5$ and $3$ is $(5,3,5,3,3,5,3)$. In this case, $W_3^{3}$ is determined by $e^{-1}W_5^{2}$. 

If the ends have the form $(3,1,...)$, then $W_3^{1}$ is determined. We get $F_3$ is the map that forgets $W_3^{2}$. Then $\alpha$ is reducible by \Cref{forget only one}. If the ends have the form $(3,3,1,1...)$, then $W_3^{2}$ is determined, and $F_3$ is the map that forgets $W_3^{1}$. The condition to recover this space is $W_3^{1}\subset W_3^{2}\cap e^{-1}W_{5}^1$. This fits the context of \Cref{intersection}, so $F_3$ is a blow-up map as follows.
$$X_{(5,3,5,3,1,1,3,5,3,1,1)}\xrightarrow{\Bl(X_{(5,\Box,3,1,1,5,3,5,3,1,1)}\cong X_{(3,5,3,1,1,5,3,5,3,1,1)}) } X_{(5,\Box,5,3,1,1,3,5,3,1,1)}$$
$$X_{(5,3,5,3,1,1,3,1,5,3,1)} \xrightarrow{\Bl(X_{(5,\Box,3,1,1,5,3,1,5,3,1)}\cong X_{(3,5,3,1,1,5,3,1,5,3,1)}) } X_{(5,\Box,5,3,1,1,3,1,5,3,1)}$$

We are left with Cases 3, 4, 5, 6, 7 and 8. For Cases III-6 (7 or 8), all the restrictions for $W_3^{1}$ and $W_3^{2}$ are $W_3^{1}\subset e^{-1}W_5^{1}$ and $W_3^{1}\subset W_3^{2}\subset W_3^{3}$. We then have $F_3$ is the projection from a tower of projective bundles by first forgetting $W_3^{1}$ and then $W_3^{2}$. 

The last three cases are considered below. The results are similar, $F_3= F_3^{1,2,3}\circ F_3^{2,3}$ is a composition of a blow-up map and a projective $\bP^1$-bundle projection. In summary, the condition to recover $W_3^{2}$ is $eW_{1}^{1}\subset W_3^{2}\subset W_3^{3}$, so $F_3^{2,3}$ is the projection from a $\bP^1$-bundle. Given $W_3^{2}$, the condition to recover $W_{3}^{1}$ is $W_{3}^{1}\subset W_{3}^{2}\cap e^{-1}W_{5}^1$. Both $W_{3}^{2}$ and $e^{-1}W_{5}^1$ are subspaces of $W_3^{3}$, a three dimensional vector space. This fits the context of \Cref{intersection}, so $F_3^{1,2,3}$ is a blow-up map; details are given below.
\begin{itemize}
    \item IV-3, $\alpha= (5,3,5,3,1,3,1,5,3,1,1)$. The corresponding reduction diagram is
    $$X_{(5,3,5,3,1,3,1,5,3,1,1)}\xrightarrow{\Bl(X_{(3,1,3,1,3,1,1)})} X_{(5,\Box,5,3,1,3,1,5,3,1,1)}\xrightarrow{\bP^1} X_{(5,\Box,5,\Box,1,3,1,5,3,1,1)}= X_{\alpha_3}.$$
    \item IV-4, $\alpha= (5,3,5,3,1,3,1,1,5,3,1)$. The corresponding reduction diagram is
    $$X_{(5,3,5,3,1,3,1,1,5,3,1)}\xrightarrow{\Bl(X_{(3,1,3,1,1,3,1)})} X_{(5,\Box,5,3,1,3,1,1,5,3,1)}\xrightarrow{\bP^1} X_{(5,\Box,5,\Box,1,3,1,1,5,3,1)}= X_{\alpha_3}.$$
    \item IV-5, $\alpha= (5,3,5,3,1,3,5,3,1,1,1)$. The corresponding reduction diagram is
    $$X_{(5,3,5,3,1,3,5,3,1,1,1)}\xrightarrow{\Bl(X_{(3,1,3,3,1,1,1)})} X_{(5,\Box,5,3,1,3,3,5,3,1,1,1)}\xrightarrow{\bP^1} X_{(5,\Box,5,\Box,1,3,5,3,1,1,1)}= X_{\alpha_3}.$$
\end{itemize}

\textbf{Case V}, the relative position of $5$ and $3$ is $(5,3,5,5,3,3,3)$. 
\begin{itemize}
    \item If the end takes the form $(3,3,1,1,...)$ (in particular V-1 and V-2), the map $F_3$ is not surjective. We shift our focus from $F_{2m-1}$ (which corresponds to $F_3$) to $F_{2m-3}$ (which corresponds to $F_1$). Then the two ends $1$ and $2$ become the beginnings $I$ and $III$. In these cases, $\alpha$ is reducible if $m> 2$ with respect to $F_{2m-3}$.    
    \item V-3, we have the reduction diagrams as follows (see details in \Cref{double blow up}).  $$X_{(5,3,5,5,3,1,3,1,3,1,1)}\xrightarrow{\Bl(X_{(5,3,5,5,3,1,1,\Box,3,1,1)}\cong X_{(5,3,5,5,3,1,1,3,1,3,1)})}  X_{(5,3,5,5,3,1,\Box,1,3,1,1)} $$
    $$\xrightarrow{\Bl(X_{(5,3,1,5,5,\Box,\Box,1,3,1,1)}\cong X_{(3,1,3,3,1,1,1)})}  X_{(5,3,5,5,\Box,1,\Box,1,3,1,1)} \xrightarrow{\bP^1}  X_{(5,\Box,5,5,\Box,1,\Box,1,3,1,1)}= X_{\alpha_3}$$
    \item V-4, we have the reduction diagrams as follows.
    $$ X_{(5,3,5,5,3,1,3,1,1,3,1)} \cong X_{(5,3,5,5,3,1,\Box,1,1,3,1)} \xrightarrow{\Bl(X_{(5,3,1,5,5,\Box,\Box,1,1,3,1)}\cong X_{(3,1,3,3,1,1,1)})} $$ 
    $$X_{(5,3,5,5,\Box,1,\Box,1,1,3,1)} \xrightarrow{\bP^1} X_{(5,\Box,5,5,\Box,1,\Box,1,1,3,1)}= X_\alpha  $$
    \item V-5, $\alpha= (5,3,5,5,3,1,3,3,1,1,1)$. We have $X_\alpha \xrightarrow{\bP^1} X_{(5,3,5,5,3,1,\Box,3,1,1,1)}\cong$ $ X_{(5,3,5,5,3,1,\Box,1,1,3,1)}\cong$ $ X_{(5,3,5,5,3,1,3,1,1,3,1)}$. The last variety is $X_\alpha$ of case V-4.
    \item V-6,7,8. Case V-6, $\alpha= (5,3,5,5,3,3,1,3,1,1,1)$, the reduction diagram is
    $$X_{(5,3,5,5,3,3,1,3,1,1,1)}\xrightarrow{\bP^1} X_{(5,3,5,5,\Box,3,1,3,1,1,1)} \xrightarrow{\Bl(X_{(3,5,3,5,5,3,1,3,1,1,1)}\cong X_{(\Box,5,\Box,5,5,3,1,3,1,1,1)})} $$    
    $$X_{(5,\Box,5,5,\Box,3,1,3,1,1,1)}\xrightarrow{\bP^2} X_{(5,\Box,5,5,\Box,\Box,1,3,1,1,1)}= X_{\alpha_3}.$$
    For V-7 (resp. V-8), we have a similar diagram with the last map being the projection from a $\bP^1$-bundle (resp. an isomorphism).
    \item If the end has the form $(3,1,3,3,...)$,in particular cases V-9,10,11, we have $W_3^{1}= eW_1^{1}$. Therefore, the map $F_3$ is the map that forgets $W_3^{2}$ and $W_3^{3}$. By first forgetting $W_3^{3}$ and then forgetting $W_3^{2}$ (using $F_3^{2,1} \circ F_3^{3,2,1}= F_3$ in the notation of \Cref{ Reduction maps F}), we get $F_3$ is projection from a tower of projective bundles.
    \item If the end is $(3,1,3,1,3,3,1,1)$, case V-12, the map $F_3$ is the map that forgets $W_3^{3}$, and $X_\alpha \xrightarrow{\bP^1} X_{\alpha_3}$.
    
\end{itemize}
For the next three cases (Cases VI, VII, and VIII), the general strategy is to decompose $F_3$ as $F_3^{1,2,3}\circ F_3^{2,3}\circ F_3^{3}$. The map $F_3^{3}$ is an isormorphism or the projection from a projective bundle. The other two maps may involve blow-ups.

\textbf{Case VI}. The relative position of $5$ and $3$ is $(5,5,3,5,3,3,3)$.
\begin{itemize}
    \item Consider the ends that have the form $(3,1,...)$, in particular VI-9,10,11,12. Similarly to V-9,10,11, we have that $F_3$ is the projection from a tower of projective bundles.
    \item VI-8, $\alpha= (5,5,3,5,3,3,1,1,1,3,1)$. The corresponding reduction diagram is
    $$X_\alpha \xrightarrow{\Bl(X_{(5,5,\Box,3,5,3,1,1,1,3,1)}\cong \Bl_{X_{(3,3,1,1,1,3,1)}} X_{(3,3,3,1,1,1,1)})} X_{(5,5,\Box,5,3,3,1,1,1,3,1)}$$
    $$\xrightarrow{\bP^2}X_{(5,5,\Box,5,\Box,3,1,1,1,3,1)} \cong X_{(5,5,\Box,5,\Box,\Box,1,1,1,3,1)}= X_{\alpha_3}.$$
    For VI-7 (resp. VI-6), the reduction diagrams are similar; we replace the last isomorphism ($F_{3}^{3}$) with the projection from a $\bP^1$-bundle (resp. $\bP^2$-bundle).
    \item VI-4, $\alpha= (5,5,3,5,3,1,3,1,1,3,1)$. Then the corresponding reduction diagram is
    $$X_\alpha \xrightarrow{\Bl(X_{(5,5,\Box,3,1,5,3,1,1,3,1)}\cong \Bl_{X_{(3,3,1,1,1,3,1)}} X_{(3,3,1,3,1,1,1)})} X_{(5,5,\Box,5,3,1,3,1,1,3,1)}$$
    $$ \xrightarrow{\bP^1} X_{(5,5,\Box,5,\Box,1,3,1,1,3,1)} \cong X_{(5,5,\Box,5,\Box,1,\Box,1,1,3,1)}= X_{\alpha_3}.$$
    For VI-3 (resp. VI-5), the reduction diagrams are similar; we replace the last isomorphism with the projection from a $\bP^1$-bundle (resp. $\bP^2$-bundle).
    \item VI-1, $\alpha= (5,5,3,5,3,1,1,3,1,3,1)$. Then the corresponding reduction diagram is
    $$X_\alpha \cong X_{(5,5,3,5,\Box,1,1,\Box,1,3,1)}\xrightarrow{\Bl(X_{(3,3,1,1,3,1,1)})} X_{(5,5,\Box,5,\Box,1,1,\Box,1,3,1)}= X_{\alpha_3}.$$
    For VI-2, we have $X_\alpha= X_{(5,5,3,5,3,1,1,3,3,1,1)}\xrightarrow{\bP^1} X_{(5,5,3,5,3,1,1,\Box,3,1,1)}\cong X_{(5,5,3,5,3,1,1,3,1,3,1)}$, the last variety is $X_\alpha$ in case VI-1.
\end{itemize}

\textbf{Case VII}. The relative position of $5$ and $3$ is $(5,5,3,3,5,3,3)$.
\begin{itemize}
    \item VII-1,2, the space $W_3^{2}$ is determined and the map that forgets $W_3^{1}$ is the projection from a $\bP^1$-bundle. Also, the map that forgets $W_3^3$ (with $W_3^{2}$ given) is an isomorphism in case VII-1 and is the projection from a $\bP^1$-bundle in case VII-2. So we have $F_3$ is the projection from a tower of projective bundles.
    \item VII-4, $\alpha= (5,5,3,3,1,5,3,1,1,3,1)$. The corresponding reduction diagram is $$ X_\alpha \xrightarrow{\bP^1} X_{(5,5,\Box,3,1,5,3,1,1,3,1)}
    \xrightarrow{\Bl(X_{(3,3,1,1,1,3,1)})} X_{(5,5,\Box,\Box,1,5,3,1,1,3,1)}\cong X_{(5,5,\Box,\Box,1,5,\Box,1,1,3,1)}= X_{\alpha_3}.$$
    For VI-3 (resp. VI-5), the reduction diagrams are similar; we replace the last isomorphism with the projection from a $\bP^1$-bundle (resp. $\bP^2$-bundle).
    \item VII-8, $\alpha= (5,5,3,3,5,3,1,1,1,3,1)$. The corresponding reduction diagram is
    $$X_\alpha \xrightarrow{\bP^1} X_{(5,5,\Box,3,5,3,1,1,1,3,1)}
    \xrightarrow{\Bl(X_{(3,3,1,1,1,3,1)})} X_{(5,5,\Box,\Box,5,3,1,1,1,3,1)}\cong X_{(5,5,\Box,\Box,5,\Box,1,1,1,3,1)}= X_{\alpha_3}.$$
    For VII-7 (resp. VII-6), the reduction diagrams are similar; we replace the last isomorphism with the projection from a $\bP^1$-bundle (resp. $\bP^2$-bundle). 
    \item VII-11, $\alpha= (5,5,3,1,3,5,3,1,1,3,1)$. The corresponding reduction diagram is 
    $$X_\alpha \xrightarrow{\Bl(X_{(3,3,1,1,1,3,1)})} X_{(5,5,3,1,\Box,5,3,1,1,3,1)} \cong X_{(5,5,\Box,1,\Box,5,\Box,1,1,3,1)}= X_{\alpha_3}.$$
    For VII-10 (resp. VII-11), the reduction diagrams are similar; we replace the last isomorphism with the projection from a $\bP^1$-bundle (resp. $\bP^2$-bundle). 
    \item For VII-12, the spaces $W_3^{1}$ and $W_3^{2}$ are determined, the map $F_3$ is the map that forgets $W_3^{3}$. It is the projection from a $\bP^1$-bundle.
\end{itemize}

\textbf{Case VIII}. The relative position of $5$ and $3$ is $(5,3,5,3,5,3,3)$. 
\begin{itemize}
    \item  VIII-1, $\alpha= (5,3,5,3,1,1,5,3,1,3,1)$. Then the corresponding reduction diagram is
    $$X_{\alpha} \xrightarrow{\Bl(X_{(3,1,1,3,3,1,1)})} X_{(5,\Box,5,3,1,1,5,3,1,3,1)} \cong X_{(5,\Box,5,\Box,1,1,5,\Box,1,3,1)}= X_{\alpha_3}.$$
    For VIII-2, the diagram is similar; we replace the last isomorphism with the projection from a $\bP^1$-bundle.    
    \item  VIII-4, then $\alpha= (5,3,5,3,1,5,3,1,1,3,1)$. The corresponding reduction diagram is
    $$X_\alpha \xrightarrow{\Bl(X_{(5,\Box,3,5,1,5,3,1,1,3,1)})\cong X_{(3,5,3,1,5,5,3,1,1,3,1)})} X_{(5,\Box,5,3,1,5,3,1,1,3,1)}$$
    $$\xrightarrow{\Bl(X_{(3,3,1,1,1,3,1)})} X_{(5,\Box,5,\Box,1,5,3,1,1,3,1)} \cong X_{(5,\Box,5,\Box,1,5,\Box,1,1,3,1)}= X_{\alpha_3}.$$
    For VIII-3 (resp. VIII-5), the reduction diagrams are similar; we replace the last isomorphism with the projection from a $\bP^1$-bundle (resp. $\bP^2$-bundle).
    \item  VIII-8, then $\alpha= (5,3,5,3,5,3,1,1,1,3,1)$. The corresponding reduction diagram is
    $$X_\alpha \xrightarrow{\Bl(X_{(5,\Box,3,5,5,3,1,1,1,3,1)}\cong X_{(5,3,5,5,3,1,1,1,3,1)})} X_{(5,\Box,5,3,5,3,1,1,1,3,1)} $$
    $$\xrightarrow{\Bl(X_{(3,3,1,1,1,3,1)})} X_{(5,\Box,5,\Box,5,3,1,1,1,3,1)}\cong X_{(5,\Box,5,\Box,5,\Box,1,1,1,3,1)}=X_{\alpha_3}.$$
    For VIII-7 (resp. VIII-6), the reduction diagrams are similar; we replace the last isomorphism with the projection from a $\bP^1$-bundle (resp. $\bP^2$-bundle). 
    \item  VIII-11, then $\alpha= (5,3,1,5,3,5,3,1,1,3,1)$. The corresponding reduction diagram is 
    $$X_{\alpha} \xrightarrow{\Bl(X_{(3,1,3,1,1,3,1)})} X_{(5,3,1,5,\Box,5,3,1,1,3,1)} \cong X_{(5,\Box,1,5,\Box,5,\Box,1,1,3,1)}=X_{\alpha_3}.$$
    For VIII-10 (resp. VIII-11), the reduction diagrams are similar; we replace the last isomorphism with the projection from a $\bP^1$-bundle (resp. $\bP^2$-bundle).
    \item For VIII-12, we have $W_3^{1}$ and $W_3^{2}$ are determined, then $F_3$ is the map that forgets $W_3^{3}$. Thus, it is the projection from a $\bP^1$-bundle.
\end{itemize}    
\textbf{Case III}, the relative position of $5$ and $3$ is $(5,3,3,5,5,3,3)$. In this case, $W_3^{2}$ is determined by $e^{-1}W_5^{1}$. We proceed by considering the following subcases.
\begin{itemize}
    \item If there is no weight $1$ between the last two weights $3$, $F_3$ is the projection from a tower of projective bundles.
    \item If there is one weight $1$ between the last two weights $3$, we decompose $F_3$ as $F_3^{3,1,2}\circ F_3^{1,2}\circ F_3^{2}$. Since $W_3^{2}= e^{-1}W_5^{1}$, the map $F_3^{2}$ is an isomorphism. Next, if there is no weight $1$ between the first two weights $3$, $F_3^{1,2}$ is the projection from a $\bP^1$-bundle. If there is a weight $1$ between the first two weights $3$, $F_3^{1,2}$ is an isomorphism. It is left to describe $F_3^{3,1,2}$. In particular, if the end is $(3,3,1,1,3,1,3,1)$, then $F_3^{3,1,2}$ is an isomorphism. For other cases, the condition to recover $W_3^{3}$ fits the context of \Cref{main blow up}, so it is a blow-up map as follows.
    $$X_{(5,3,3,5,5,3,1,3,1,1,1)}\xrightarrow{\Bl( X_{(5,3,3,1,5,5,\Box,3,1,1,1)}\cong X_{(5,3,3,1,5,3,5,3,1,1,1)})} X_{(5,3,3,5,5,\Box,1,3,1,1,1)}$$

    $$X_{(5,3,1,3,5,5,3,1,3,1,1)}\xrightarrow{\Bl( X_{(5,3,1,3,1,5,5,\Box,3,1,1)}\cong X_{(5,3,1,3,1,5,3,5,3,1,1)})} X_{(5,3,1,3,5,5,\Box,1,3,1,1,1)}$$

    $$X_{(5,3,3,1,5,5,3,1,3,1,1)}\xrightarrow{\Bl( X_{(5,3,3,1,1,5,5,\Box,3,1,1)}\cong X_{(5,3,3,1,1,5,3,5,3,1,1)})} X_{(5,3,3,1,5,5,\Box,1,3,1,1,1)}$$

    \item If there are two or three weights $1$ between the last two weights $3$, the map $F_3$ is not surjective. If $m>2$, we consider $F_{2m-3}$, the beginning (the relative position of the weights $2m-1$ and $2m-3$) now becomes one of the following cases: II, IV, and VII. We have shown that $\alpha$ is reducible in these cases. Now, we consider $m=2$.
    \begin{itemize}
        \item If there is no weight $1$ between the first two weights $3$, let $\alpha'$ be the tuple obtained from $\alpha$ by swapping the first weight $5$ and the first weight $3$. We then have $X_\alpha \cong (\bP^1, X_{\alpha'})$ (see \Cref{swap map}).
        \item If there is a weight $1$ between the first two weights $3$, the tuple $\alpha$ has the form $(5,3,1,3,5,...)$. Changing the first four weights to obtain $\alpha'$ of the form $(3,5,3,1,5,..)$, we have $X_\alpha\cong X_{\alpha'}$. 
    \end{itemize}
     In both cases, let $\alpha''$ be the tuple obtained by removing the first weight $3$ in $\alpha'$. By a distinguished reduction map, $X_\alpha' \cong X_{\alpha''}$. For $\alpha''$, the situation now is $\dim V_3=\dim V_5= 3$. Therefore, $X_{\alpha''}$ is reducible with respect to $F_5$ (the details are discussed in the case where $\lambda$ has $3$ parts). 

\end{itemize}
So far, we have shown that for $\lambda = (2x_1, 2x_2, 2x_3, 2x_4)$ and $x_3> x_4= m$ (equivalently $\lambda_{A\cup B}= (6,6,6,4)$), and for a tuple of weights $\alpha$ with the multiplicity of weight $2m-3$ being four, the tuple $\alpha$ is reducible with respect to $F_{2m-1}$ or $F_{2m-3}$. We further claim that this statement holds for the case $d_{2m-3}= 3$ or $d_{2m-3}= 2$ using similar arguments. The change here is that in the list of ends we have three or two weights $1$ (instead of four). Note that in our arguments, only $W_1^{j}$ and the dimension of $eW_1^{j}$ play a role. So, from an end with $d_{2m-3}$ weights $1$, we can add $4-d_{2m-3}$ weights $1$ to its end. And then apply the argument of the case $d_{2m-3}= 4$. For example, the end $(3,3,1,1,3,3)$ gives the same constraint as the end $(3,3,1,1,3,3,1,1)$ does in terms of studying $F_3$.
 
In summary, for $\lambda = (2x_1, 2x_2, 2x_3, 2x_4)$, $x_3> x_4= m$, and for $\alpha$ that do not have weight $-3$, the tuple $\alpha$ is reducible if $m\geqslant 2$. Next, if $x_3-x_4> 1$ (resp. $x_2-x_3>1$ or $x_1-x_2>1$), then $\alpha$ is reducible with respect to $F_{2x_3-1}$ (resp. $F_{2x_2-1}$ or $F_{2x_1-1}$); the descriptions of these reduction maps were given in the cases where $\lambda$ has $3$ (resp. $2$ or $1$) parts. Moreover, in our arguments, all centers of blow-ups have the forms $(\cB^{gr}_{\lambda'})_{\alpha'}$ with $\lambda'_{A\cup B}= (6,6,6,4)$. Therefore, by induction, we get $X_\alpha\in \cC_4^{(6,6,6,4),0}$ for $\alpha$ with $d_{-3}= 0$. Similar arguments apply for the case where the weight $-3$ has multiplicity one, and the theorem is proved for the case $\lambda_{A\cup B}$= $(6,6,6,4)$.

Consider the other two cases of the theorem, $\lambda_{A\cup B}= (6,6,4,4)$ and $\lambda_{A\cup B}= (6,4,4,4)$. The change when we consider $F_{2m-1}$ is that $\dim V_{2m+1}$ is now $2$ or $1$ (instead of $3$). For the case-by-case consideration, we have the same list of ends and fewer numbers of beginnings. For each beginning in these two cases, we can turn them into a beginning in the case $\dim V_{2m+1}= 3$ by adding $3- \dim V_{2m+1}$ weights $5$ to the beginning of these beginnings. For example, from $(5,3,5,3,3,3)$, we have $(5,5,3,5,3,3,3)$. Our arguments depend only on the dimensions of $e^{-1}W_5^{j}$, and this process of adding weights $5$ does not change these dimensions. Therefore, similar arguments apply to describe the properties of $F_3$.
\end{proof}

Recall that our main theorem of this subsection, \Cref{geometric main 4 rows}, states that the varieties $(\fix)_\alpha$ belong to the collection $\cC_4$ when the associated partition of $e$ has four rows. Therefore, with \Cref{subcollections} verified, to prove \Cref{geometric main 4 rows}, it is left to show that the collections $\cC_4^{\lambda_{A\cup B},i}$ are subcollections of $\cC_4$. From the way we define these collections, the task is to show that the generators of the collections $\cC_4^{\lambda_{A\cup B},i}$ belong to $\cC_4$. In other words, we want to show that the varieties $(\fix)_\alpha$ belong to $\cC_4$ for certain small $e$. 

The techniques to prove that these $X_\alpha$ can be built by taking projective bundles and blow-ups from $\OG(j,d), 2\leqslant d\leqslant 4, j\leqslant \frac{d}{2}$ are similar to the techniques we used in the proof of \Cref{subcollections}. We list the geometric descriptions of these generators in Appendix A. 

Recall that we have two cases for the multiplicity of weight $-1$ in $\alpha$, either $1$ or $2$. In general, the geometry of $X_\alpha$ is somewhat simpler when the weight $-1$ in $\alpha$ has multiplicity $2$. For example, there are fewer blow-ups involved in the process of constructing $X_\alpha$. As the details are saved for the Appendix, the next remark only gives a heuristic explanation for this phenomenon in terms of finite models.
\begin{rem} \label{2 is easier than 1}
    We consider the case where $\lambda$ has $4$ parts and the weight $-1$ in $\alpha$ has multiplicity $2$. We have a morphism $\pi_{-1}: X_\alpha\rightarrow \OG(2,4)$ which sends a flag $U^\bullet$ to the space $W_{-1}^2= U^\bullet \cap V_{-1}$. The variety $\OG(2,4)$ has two connected components, each isomorphic to $\bP^1$. The group $A_e$ acts transitively on this set of components as follows, $z_i$ permutes the two components and $z_{ij}$ fixes each component. The morphism $\pi_{-1}$ is $A_e$-equivariant, so $X_\alpha$ must have at least two connected components if it is not empty. As a consequence, the finite model of $X_\alpha$ can only contain orbits that have stabilizers lying in the subgroup $A_e^{+}\subset A_e$ generated by $\{z_{ij}\}_{1\leqslant i\leqslant j\leqslant 4}$. 

    From \Cref{finite model blow up}, if our process of building $X_\alpha$ involves a blow-up along some variety $X_{\alpha'}$, the categorical finite model $Y_{\alpha'}$ is a subset of the categorical finite model $Y_\alpha$. Therefore, provided that the points of $Y_\alpha$ cannot be fixed by elements $z_i\notin A_e^{+}$, we have the same restriction on $Y_{\alpha'}$. This condition limits the possibilities of the centers of the blow-ups (if there is one) when we describe $X_\alpha$.    
\end{rem}

A corollary of this remark is that we can prove \Cref{can solve equations} without studying the geometry of $X_\alpha$ when the weight $-1$ has multiplicity $2$ in $\alpha$. Indeed, the key ingredient of the proof of \Cref{can solve equations} is that the two orbits $\bO_{0}^{1,2}$ and $\bO_{2,3}^+$ do not appear in $Y_e$. And both of these orbits are fixed by $z_2,z_3\in A_e$, so they cannot appear in the finite models of $X_\alpha$ when $-1$ has multiplicity $2$ in $\alpha$. Therefore, for \Cref{can solve equations}, it is enough to study in detail the geometry of $X_\alpha$ when $-1$ has multiplicity $1$ in $\alpha$.

\section{Predictions and conjectures for general nilpotent elements $e$}
Throughout Section 2 to Section 8, we have assumed that the Lie algebra $\fg$ is $\fsp(2n)$ and the associated partition $\lambda$ consists only of even parts. The first subsection explains how we can derive the finite model $Y_e$ for a general $e\in \fsp_{2n}$ from the finite model $Y_{e_0}$ for an even nilpotent element $e_0\in \fsp_{2n'}$, $n'\leqslant n$. The second subsection treats the cases $\fg= \fso_{2n+1}$ and $\fg= \fso_{2n}$. The last subsection consists of predictions and conjectures related to finite models $Y_e$.

\subsection{Finite models of general nilpotent elements $e\in \fsp_{2n}$} \label{subsect 9.1}
We consider $e\in \fsp_{2n}$, and we no longer assume the associated partition $\lambda$ of $e$ only has even parts. We write $\lambda= \lambda_{even}+ \lambda_{odd}$ where $\lambda_{even}$ (resp. $\lambda_{odd}$) consists of even (resp. odd) parts of $\lambda$. Similarly to \Cref{A_e action}, we have a direct sum decomposition $V_\lambda= V_{\lambda_{even}}\oplus V_{\lambda_{odd}}$. We view $\lambda_{even}$ and $\lambda_{odd}$ as the associated partitions of two nilpotent elements $e_{0}\in \fsp(V_{\lambda_{even}})$ and $e_{1}\in \fsp(V_{\lambda_{odd}})$. In the notation of \Cref{subsect 2.1}, we have three reductive groups $Q_e$, $Q_{e_0}$ and $Q_{e_1}$ as the reductive parts of the centralizers of $e, e_0$ and $e_1$ in the corresponding symplectic groups. We have $Q_e\cong Q_{e_0}\times Q_{e_1}$.

Consider a reductive group $H_1$ and a $H_1$-centrally extended set $Y$. Assume that we have a group $H_2$ such that the pullback morphism of Schur multipliers $M(H_1)\rightarrow M(H_1\times H_2)$ is an isomorphism. Then we have a natural $H_1\times H_2$-centrally extended structure on $Y$ as follows. At the level of sets, $H_2$ acts trivially on the points of $Y$; and when we take the Schur multipliers into account, the central extensions are read from the pullback isomorphism $M(H_1)\rightarrow M(H_1\times H_2)$. We call this $H_1\times H_2$-centrally extended structure the lift of the original $H_1$-centrally extended structure on $Y$. On the other hand, assume that we have two groups $H_1$, $H_2$ satisfying the same condition, and a $H_1\times H_2$-centrally extended set $Y$ such that $H_2$ acts trivially on $Y$. Then this $H_1\times H_2$-centrally extended structure can be obtained by lifting the natural $H_1$-centrally extended structure on $Y$.

Let $\ell_1$ be the number of distinct parts in $\lo$. According to \cite[Chapter 5]{collingwood1993nilpotent}, the group $Q_{e_1}$ takes the form $\prod_{j=1}^{\ell_1} Sp_{2i_j}$ for some $i_j$. Applying the second part of \Cref{easy schur} to the factors of $Q_{e_1}$ one by one, we obtain $M(Q_{e_0}\times Q_{e_1})= M(Q_{e_0})$. In addition, $Q_{e_1}$ is connected, so $Q_{e_1}$ acts trivially on $Y_e$. Therefore, the $Q_e$-centrally extended structure of $Y_e$ is fully recovered from the $Q_{e_0}$-centrally extended structure of $Y_e$ by the lifting process explained in the above paragraph.

From \cite[Chapter 5]{collingwood1993nilpotent}, we have that in $\lambda_{odd}$, each part has multiplicity even. Assuming that $\lo$ has $2k_1$ parts, we write $\lo$ as $(b_1, b_1, b_2, b_2,...,b_{k_1}, b_{k_1})$. Next, we explain the relation between $Y_e$ and $Y_{e_0}$ as two $Q_{e_0}$-centrally extended sets.

The decomposition $V_\lambda= V_{\lambda_{even}}\oplus V_{\lo}$ can be viewed as a mod-2 weight decomposition of $T_e$. Similarly to \Cref{subsect 3.2}, we have a description of $\fix$ as the disjoint union of $n \choose |\lambda_{odd}|$ copies of $\cB_{\lambda_{even}}^{gr} \times \cB_{\lambda_{odd}}^{gr}$. Since the group $Q_{e_1}$ is a product of symplectic groups, the $Q_{e_1}$-centrally extended structure of $Y_{e_1}$ is trivial. Therefore, $Y_{e_1}$ consists of $\chi(\cB_{\lambda_{odd}})$ trivial points. The Euler characteristic of $\cB_{\lambda_{odd}}$ can be calculated as in \Cref{subsect 3.2}. If $\lambda_{odd}= (b_1,b_1,....,b_{k_1}, b_{k_1})$, then $\chi(\cB_{\lambda_{odd}})=$ $2^{b_1+...+ b_{k_1}}\times$ ${|\lambda_{odd}| \choose b_1,...,b_{k_1}}$. Therefore, we have the following proposition.

\begin{pro} \label{separate odd even}
    The $Q_{e_0}$-centrally extended set $Y_e$ consists of $2^{b_1+...+ b_{k_1}}\times$ $ n\choose {b_1, b_2, ..., b_{k_1}}$ disjoint copies of $Y_{e_0}$.
\end{pro}

\subsection{The cases $\fso_{2n}$ and $\fso_{2n+1}$} \label{subsect 9.2}
In this subsection, we consider a nilpotent element $e$ of $\fso_{2n}$ or $\fso_{2n+1}$. Let $\lambda$ be the associated partition of $e$. We return to the case where $\lambda$ has up to $4$ parts. Similarly to the case $e\in \fsp_{2n}$, we consider a vector space $V_\lambda$ equipped with a symmetric bilinear form. We have the groups $A_e\subset Q_e\subset O(V_\lambda)$ as in \Cref{subsect 2.1}. To be consistent with the notation when $\fg= \fsp_{2n}$, we consider the Springer fiber $\cB_e$ as the variety of maximal isotropic flags in $V_\lambda$ stabilized by $e$ (this notation may be different from the standard notation in the sense that our Springer fibers may have two connected components permuted by the group of components of $O(V)$). Recall that we have defined a set of numerical invariants $\{F_a^{A'}(X)\}$ for any $A_e$-variety $X$ in \Cref{numeric F}. We have the following theorem.
\begin{thm} \label{can solve equation 2}
    In the above setting, the $A_e$ (or $Q_e$)-centrally extended structure of the finite set $Y_e$ is determined by the numerical invariants $\{F_a^{A'}(\spr)\}$.
\end{thm}
\begin{proof}
    When $\lambda$ has one, two or three parts, the proof of the theorem is essentially similar to the case of $\fsp_{2n}$ (see \Cref{numeric 2 row} and \Cref{subsect 4.2}). For the case $\lambda$ has $4$ parts, the proof is much simpler than in the case $\fsp_{2n}$, details are given below.

    Similarly to the case $\fsp_{2n}$ discussed in Section 9.1, if the partition $\lambda $ has even parts, the centrally extended finite set $Y_e$ can be obtained by taking the disjoint union of $Y_{e'}$ for smaller $e'$. Therefore, we consider $\lambda\in \fso_{2n}$ to have $4$ odd parts. 
    
    As in \Cref{4 row Uniqueness}, we study the characters of $A_e\cong (\cyclic{2})^{\oplus 4}$ that appear in the total cohomologies $H^*(\spr^s)$, $s\in A_e$. The group $A_e$ is generated by four elements $z_1,z_2,z_3,z_4$ (see \Cref{subsect 2.1}), and we consider $A_e' =A_e/ \{z_{1234},1\}$. Following \cite[Section 13.3]{carter1993finite}, we find that the characters of $A_{e}^{'}$ that appear in $K_0(\spr)$ are $(1,1,1)$, $(-1,-1,1)$, $(1,-1,-1)$, $(-1,-1,-1)$, $(1,1,-1)$, $(-1,1,1)$. This is precisely the list of characters we get in type C (\Cref{4 row Uniqueness}). Furthermore, we observe that $\spr^{s}= \emptyset$ for $s\in \{z_1, z_2, z_3, z_{123}\}$. Therefore, $\dim V_{(s, \rho)}= 0$ for $s\in \{z_1, z_2, z_3, z_{123}\}$ and $\rho$ such that $\rho(z_i)= 1$. As a consequence, the system of equations in this case can be obtained by plugging $\dim V_{(z_i, \rho)}= 0$ into the two systems (\ref{system 1}) and (\ref{system 2}) in \Cref{4 row Uniqueness}. In other words, we have
    
    \begin{equation} \label{system 3}
  \left\{
    \begin{aligned}
&a_{1,2}+ a_{0}^{12}= \dim V_{(z_1,(1,1,-1))}=0\\
&a_0+ a_{1,2}= \dim V_{(z_2,(1,1,1))}=0\\
&a_0+ a_{3,12}= \dim V_{(z_3,(1,1,1))}=0\\
&a_0+ a_{12,23}= \dim V_{(z_{13},(1,1,1))}\\
&a_{12}+a_{3,12}=\dim V_{(z_{12},(1,1,1))} 
    \end{aligned}
  \right.
\end{equation}
and 
\begin{equation} \label{system 4}
  \left\{
    \begin{aligned}
&a_{1,23}+ a_{0}^{123}= \dim V_{(z_1,(1,-1,-1))}= 0\\
&a_0^{123}+ a_{2,3}^{+}= \dim V_{(z_2,(-1,1,-1))}= 0\\
&a_0^{1}+a_{2,3}^+= \dim V_{(z_3,(1,-1,1))}= 0\\
&a_{23}+ a_{1,23}= \dim V_{(z_{23},(1,1,1))}- \dim V_{(z_{13},(1,1,1))}\\
&a_{12,23}^{+} + a_{0}^{123}= \dim V_{(z_{12},(-1,-1,1))}+ \dim V_{(z_{13},(1,1,1))}- \dim V_{(z_{12},(1,1,1))}
    \end{aligned}
  \right.
\end{equation}
Therefore, we get $a_{12,23}= \dim V_{(z_{13},(1,1,1))}$, $a_{12}=\dim V_{(z_{12},(1,1,1))}$, $a_{23}= \dim V_{(z_{23},(1,1,1))}- \dim V_{(z_{13},(1,1,1))}$ and $a_{12,23}^{+}= \dim V_{(z_{12},(-1,-1,1))}+ \dim V_{(z_{13},(1,1,1))}$ $- \dim V_{(z_{12},(1,1,1))}$. All other multiplicities are $0$, we have up to four types of $A_e$-centrally extended orbits in $Y_e$.
\end{proof}
\begin{rem}
    It is noted that the proof of this theorem can be obtained without mentioning the numerical invariants $S_i^{A'}$. As shown above, the main reason is that we have $\dim V_{(s, \rho)}= 0$ for more pairs $(s,\rho)$. The heuristic reason for this numerical fact comes from the geometry of $\fix$ as follows. Recall that the variety $\fix$ is defined by the fixed point loci of $T_e$ in $\spr$, here $T_e$ is the torus coming from an $\fsl_2$-triple (see \Cref{subsect 6.1}). As $T_e$ acts on $V_\lambda$, we have the corresponding weight decomposition of $V_\lambda$. Since all parts of $\lambda$ are odd, the dimension of the weight space $V_0$ is $4$. Furthermore, an isotropic flag $U^\bullet$ in $\fix$ must satisfy $\dim W_0= \dim (U^n\cap V_0)= 2$. Therefore, from \Cref{2 is easier than 1}, we see that any variety $(\fix)_\alpha$ has at least two connected components. So, the stabilizers of the orbits in $Y_e$ are subgroups of the subgroups of $A_e$ generated by the elements $z_{ij}$, $1\leqslant i<  j\leqslant 4$. As a consequence, we have less nonzero multiplicities.
\end{rem}

Here, we state a conjecture that generalizes \Cref{can solve equations} and \Cref{can solve equation 2}. Let $e$ be a nilpotent element in $\fsp_{2n}$ (or $\fso_{2n}$, $\fso_{2n+1}$). Let $c$ be the corresponding two-sided cell. At the end of \Cref{subsect 5.1}, we have a realization $J_c\cong K_0(Sh^{Q_e}(Y_e\times Y_e))$ and the left cell modules are of the forms $K_0(Sh^{Q_e}(Y_e\times \bO))$ where $\bO$ is a $Q_e$-centrally extended orbit that appears in $Y_e$. Hence, for each type of $Q_e$-centrally extended orbit $\bO$ in $Y_e$, we have a corresponding cell module $\cC_\bO$. 

We consider the case where $e$ is distinguished, so $A_e= Q_e$. 
\begin{conj} \label{combi conjecture} 
In the Grothendieck group of the category of $J_c$-modules, the classes of the cell modules $K_0(Sh^{Q_e}(Y_e\times \bO))$ for non-isomorphic $\bO$'s are linearly independent.
\end{conj}

When the partition $\lambda$ has less than or equal to four parts, we have listed the types of orbits that appear in $Y_e$. It can be checked that \Cref{combi conjecture} is true. Note that even for the case where $\lambda$ has $4$ rows, the validity of the conjecture is highly non-trivial. It requires geometric understanding of $\fix$ in Section 8 to exclude certain types of orbits (see the discussion after \Cref{2 is easier than 1}). 

In the following, we discuss some consequences that will follow if \Cref{combi conjecture} is true. Write $m(\bO)$ for the multiplicity of the $A_e$-centrally extended orbit $\bO$ in $Y_e$. In the Grothendieck group of $J_c$-modules, we have $[J_c]= [Sh^{Q_e}(Y_e\times Y_e)]$ $=\sum m(\bO)[\cC_\bO]$. Because $J_c$ is semisimple over $\bQ$, we have $[J_c]= $ $\sum (\dim V_{(s,\rho)})[V_{(s,\rho)}]$. Therefore, we have $\sum m(\bO)[\cC_\bO]$ $=\sum (\dim V_{(s,\rho)})[V_{(s,\rho)}]$. Then the linear independency of the left cell modules in \Cref{combi conjecture} implies that we can solve for $m(\bO)$ in terms of $\dim V_{(s,\rho)}$. Next, we have seen that the numbers $\{\dim V_{(s,\rho)}\} $ can be computed from $\{F_a^{A'}(\spr)\}$ as in the proof of \Cref{dimensions from F_i}. Therefore, another consequence is that the multiplicities of the orbits in $Y_e$ can be fully recovered from the set of numerical invariants $\{F_a^{A'}(\spr)\}$. 

\begin{rem}
    Our motivation for Statement 1 of the conjecture comes from the result of \cite[Lemma 6.16]{Losev_2014} which studies cell modules in the case of finite Weyl groups. 
\end{rem}

\subsection{Predictions and conjectures}

This subsection consists of conjectures related to the centrally extended finite set $Y_e$. In Section 8, we have seen that the orbits that appear in $Y_e$ for general $e$ are the orbits that appear in $Y_{e'}$ for some small $e'$ (base cases). We first generalize this phenomenon to a conjecture.

Recall that we write $\lambda$ for the partition of $e$ and $\lambda^\intercal$ for the dual partition. As discussed in \Cref{subsect 9.1}, we can reduce our consideration to the case where $\lambda$ consists only of even (resp. odd) parts for $e\in \fg$ of type C (resp. B,D). If $\fg$ is of type C, all parts of $\lambda^\intercal$ have even multiplicities. If $\fg$ is of type $B$ or $D$, the largest part of $\lambda^\intercal$ has odd multiplicity and the other parts have even multiplicities.

We then form a partition $\lambda'$ as follows. From $\lambda^\intercal$, if a part has multiplicity $2l+1$ or $2l+2$ for some $l\geqslant 0$, we remove $2l$ of these parts. After this process, we obtain a partition $\lambda_\text{red}$, let $\lambda'= \lambda_\text{red}^\intercal$. Then we have a corresponding nilpotent element $e'\in \fg'$ with partition $\lambda'$ where $\fg'$ is of the same type as $\fg$. We give some examples below.
\begin{exa}  \leavevmode
    \begin{itemize}
        \item If $\lambda= ((2a)^{2k})$, then $\lambda'= (2^{2k})$.
        \item If $\lambda= (2a=2a> 2b)$, then $\lambda'= (4,4,2)$.
        \item If $\lambda= (2a> 2b >2c> 2d)$, then $\lambda'= (8,6,4,2)$.
        \item If $\lambda= (2a+1=2a+1> 2b+1=2b+1> 2c+1)$, then $\lambda'= (5,5,3,3,1)$.
    \end{itemize}
\end{exa}

It is straightforward to check that $Q_e\cong Q_{e'}$ and $A_e\cong A_{e'}$. We can then state our conjecture.

\begin{conj}
    An $Q_e$-centrally extended orbit appears in $Y_e$ if and only if it appears in $Y_{e'}$.
\end{conj}
In this paper, we have proved the conjecture in the cases where $\lambda$ has $2,3$ or $4$ parts. For $\lambda$ having equal parts, the conjecture follows from the discussion in Section 8.3. Using similar techniques, one can verify that the conjecture is true for the cases where $\lambda$ has two distinct parts.

In the rest of this subsection, under the assumption that \Cref{union} is true, we describe a list of $Q_e$-centrally extended orbits that appear in $Y_e$. This list is expected to be the complete list of $Q_e$ orbits in $Y_e$ when $\lambda$ has $5$ parts. However, when $\lambda$ has $6$ parts or more, there is evidence that the list is not exhaustive.

We first consider $e\in \fsp_{2n}$. As discussed in \Cref{subsect 9.1}, we can reduce our consideration to the case where $\lambda$ consists only of even parts. 

Now, we recall the setting at the beginning of \Cref{subsect 2.1}. We consider $\lambda= (2x_1,...,2x_k)$, a partition of $2n$. We have groups $A_e\subset Q_e\subset Sp(V_\lambda)$, here $Q_e$ takes the form $\prod_{i=1}^{\ell} O_{{i_m}- {i_{m-1}}}$ where $\ell$ is the number of distinct parts in $\lambda$. The group $A_e$ is isomorphic to $(\cyclic{2})^{\oplus k}$ with the generators $z_1,z_2,..., z_k$. Consider a vector space $V$ of dimension $k$, and a basis $u_1,...,u_k$. Equipping $V$ with the standard dot product, we then have an embedding $A_e\hookrightarrow O(V)$ sending $z_i$ to the diagonal matrix with entries $a_{jj}= (-1)^{\delta_{ij}}$. 

For $1\leqslant i< j\leqslant k$, let $V_{[i,j]}$ be the vector subspace of $V$ spanned by $v_i, v_{i+1},..., v_{j}$ and write $\Span(z_i,...,z_j)$ for the subgroup of $A_e$ generated by $z_i,...,z_j$. We then have an embedding $\Span(z_i,...,z_j)$ $\hookrightarrow O(V_{[i,j]})\cong O_{j-i+1}$. We also have the projection $A_e\rightarrow \Span(z_i,...,z_j)$ that sends $v_l$ to $1$ if $l\notin [i,j]$ and sends $v_l$ to $v_l$ if $l\in [i,j]$. Hence, we obtain a composition $A_e\rightarrow \Span(z_i,...,z_j) \hookrightarrow O_{j-i+1}$. We now give notation for the $A_e$-centrally extended orbits induced by these embeddings.

\begin{notation} \label{basic orbits}
\begin{itemize} We label some of the orbits as follows.
    \item Let $\text{pt}$ denote a point with trivial $A_e$-action (no Schur multiplier).
    \item Consider $j>i$ and $j-i$ even, let $[i,j]$ denote the point with the Schur multiplier given by pulling back the unique nontrivial Schur multiplier of $O_{j-i+1}$ along the map $A_e\rightarrow $ Span$(z_i,...,z_j)$ $\hookrightarrow O_{j-i +1}$.
    \item Consider $j>i$ and $j-i$ odd; let $\pm [i,j]$ denote the two points permuted by $A_e$ without Schur multiplier, where the stabilizer of each point is the preimage of $\SO_{j-i+1}$ under the map $A_e\rightarrow $ Span$(z_i,...,z_j)$ $\hookrightarrow O_{j-i +1}$. If $j-i\geqslant 3$, let $[i,j]^{\pm}$ denote two points with the same stabilizer, while the centrally extended structure is given by pulling back the unique Schur multiplier of $\SO_{j-i+1}$.    
\end{itemize}   
\end{notation}
\begin{rem}
Our notation can be interpreted in the context of \Cref{Schurz2} as follows. Consider a subgroup $A'$ of $A_e$. The group $H^2(A', \mathbb{C}^\times)$ is identified with the space of 2-forms on $A'$. In terms of 2-forms, the Schur multiplier $\omega_{[i,j]}$ of the orbit $[i,j]$ satisfies $\omega_{[i,j]}(z_h, z_l)= 1$ if $i\leqslant h\neq l\leqslant j$ and $0$ in other cases. For the orbit $[i,j]^\pm$, the corresponding 2-form is obtained by pulling $\omega_{[i,j]}$ back to the subgroup of $A'$ generated by $z_h$ with $h<i$ or $h>j$ and $z_{m}z_l$ with $i\leqslant m,l\leqslant j$.
\end{rem}

Note that the parity of $j-i$ is already specified when we write $[i,j]$, $\pm[i,j]$, or $[i,j]^\pm$. And it is natural to regard $[i,i]$ as $\text{pt}$. Next, we define the notion of \textit{good multiplication}. Consider two types of $A_e$-centrally extended orbits, $\mathcal{O}_1$ and $\mathcal{O}_2$. Associated to these orbits, we have two pairs $(A_i, \psi_i)$ where $A_i\subset A_e$, $\psi_i \in H^2(A_i, \bC^\times)$ for $i=1, 2$. Write $A_{1,2}$ for $A_1\cap A_2$. Let $\psi_{1,2}$ be the 2-form on $A_{1,2}$ that is given by $\psi_1|_{A_{1,2}}+ \psi_2|_{A_{1,2}}$. 
\begin{defin} \label{orbit multiplication}
    We define the orbit $\cO_1\times \cO_2$ to be the centrally extended orbit stabilized by $A_{1,2}$ with the Schur multiplier $\psi_{1,2}$. Consider $\cO_1$ and $\cO_2$ that are given by intervals as in \Cref{basic orbits}. The multiplication is called \textit{good multiplication} if the two intervals are disjoint or if one interval contains the other. 
\end{defin}

This multiplication is clearly commutative. Let $S_e$ be the set of centrally extended orbits generated by good multiplication from $[i,j]$, $\pm[i,j]$, and $[i,j]^\pm$ for $1\leqslant i\leqslant j\leqslant k$. When it is clear, we will drop the "$\times$" sign. For example, we write $[i,j][j+1,h]$ for $[i,j]\times [j+1,h]$.
\begin{pro} \label{relation}
We have the following relations between the generators.
\begin{itemize}
    \item $[i,i]= \text{pt}$ for all $1\leqslant i\leqslant k$. And $[i,i]\times \mathcal{O}= \mathcal{O}$.
    \item $[i,j]\times [i,j]= \text{pt}$, $\pm[i,j]\times \pm [i,j]= \pm [i,j]$ \text{ and } $[i,j]^\pm\times [i,j]^\pm= \pm [i,j]$.
    \item $[i,j+h]\times [i,j]\times [j+1, j+h]^\pm= \pm [j+1, j+h]$ and $[i,j+h]^\pm\times [i,j]^\pm\times [j+1, j+h]^\pm= \pm[i,j] \pm [j+1, j+h]$ for any $h\geq 2$ even. 
    \item $[i,j]^\pm\times[i, j-1]= \pm[i,j]$ and $[i,j]^\pm\times[i+1, j]= \pm[i,j]$.
\end{itemize}
\end{pro}
\begin{proof}
    The verification of these relations is straightforward by working with the corresponding forms.
\end{proof}
\begin{rem} \label{last relation}
The last relation can be rewritten as $[i,j+h]\times [j+1, j+h]^\pm= [i,j]\pm [j+1, j+h]$ or $[i,j+h]\times \pm[j+1, j+h]= [i,j] [j+1, j+h]^\pm$. This means that we can always substitute an inclusion of an even interval in an odd interval with the same beginning or end by two disjoint intervals.
\end{rem}
When $e$ is distinguished, let $S_e$ be the set of orbits that are generated by good multiplications from $[i,j]$, $\pm[i,j]$, and $[i,j]^\pm$ for $1\leqslant i,j\leqslant k$. When $e$ is not distinguished, we define $S_e$ as follows. We write the associated partition $\lambda$ in the form $[{2x'_{1}}^{j_1},...,{2x'_{l}}^{j_l}]$ with $x'_{1}>...> x'_{l}$ and $j_1+...+ j_l= k$. We restrict the set of generators to $[i,j]$, $\pm[i,j]$ and $[i,j]^\pm$ so that there exists $p<q$ satisfying $i= j_1+...+j_p+1$ and $j= j_1+...+j_q$.
\begin{conj} \label{5-row conjecture}
Assume that the partition of $e\in \fsp_{2n}$ has five parts. The centrally extended orbits that appear in $Y_e$ are precisely the orbits in the set $S_e$. 
\end{conj}
\begin{rem}
    When $\lambda$ has five parts, the set $S_e$ has $25$ orbits. The numerical constraints in Section 5 give us a list $S'_{e}$ of $36$ possible orbits. Let $S'\subset K_0(J_c)$ be the span of the left cell modules that correspond to the orbits in $S'_{e}$. It can be checked that the left cell modules that correspond to the orbits in $S_e$ form a basis of $S'$. In other words, \Cref{5-row conjecture} is equivalent to \Cref{combi conjecture} in this case.  
\end{rem}
\begin{rem}
    When $\lambda$ has six parts or more, it can be shown that the set $S_e$ does not exhaust the list of centrally extended orbits that appear in $Y_e$. In particular, we can find certain irreducible modules $V_{(s,\rho)}$ of $K_0(\Sh^{A_e}(Y_e\times Y_e)$ such that there is no orbit in $S_e$ associated with the pair $(s,\rho)$. On the other hand, from \Cref{equality of dimensions}, there must be some orbit in $Y_e$ that is associated with $(s,\rho)$.
\end{rem}

We give examples for the cases where $\lambda$ has $2, 3$ or $4$ rows, demonstrating how the set $S_e$ gives us the list of orbits that appear in $Y_e$.
\begin{exa} We list the cases below.
\begin{enumerate}
    \item When $\lambda= (2x_1, 2x_2)$. There are two types of orbit: $\pm[1,2]$ and $\text{pt}$.
    \item Consider $\lambda= (2x_1, 2x_2, 2x_3)$.
    \begin{enumerate}
        \item When $x_1= x_2= x_3$, we have two types of orbit $\text{pt}$ and $[1,3]$.
        \item When $x_1< x_2= x_3$, we have $\text{pt}$, $[1,3]$, and $\pm [1,2]$. Similarly, when $x_1= x_2< x_3$, we have $\text{pt}$, $[1,3]$, and $[2,3]$.
        \item For the case where $\lambda$ has distinct parts, we have $\text{pt}$, $\pm[1,2]$, $\pm[2,3]$ and $[1,3]$.
    \end{enumerate}
    \item Consider $\lambda= (2x_1, 2x_2, 2x_3,2x_4)$.
    \begin{enumerate}
        \item When $x_1=x_2=x_3< x_4$, we have $\text{pt}$, $[1,4]^\pm$, $\pm[1,4]$, $[2,4]$ and $[1,4]^\pm[2,4]$ and $\pm[1,4][2,4]$. The case $x_1<x_2=x_3= x_4$ is similar.
        \item When $x_1 =x_2< x_3= x_4$, we have $\pm[1,2]$, $\pm[3,4]$, $\pm[1,4]$, $[1,4]^\pm$ and $\text{pt}$.
        \item When $x_1< x_2= x_3< x_4$, we have $\pm[2,3]$, $\pm[1,4]\pm[2,3]$, $\pm[1,4]$, $[1,4]^\pm$, $[1,3]$, and $[2,4]$.
        \item When $x_1<x_2<x_3= x_4$, we have $\pm[1,4]$, $[1,4]^\pm$, $\pm[3,4]$, $\pm[1,2]$, $[2,4]$, $\pm[1,2]\pm[3,4]$, and $\text{pt}$. Similarly for $x_1=x_2< x_3< x_4$.
        \item For the distinguished case, we take the union of orbit types from all previous cases. There are $10$ of them: $\pm[1,4]$, $[1,4]^\pm$, $\pm[3,4]$, $\pm[1,2]$, $[2,4]$, $\pm[1,2]\pm[3,4]$, $\pm[2,3]$, $[1,3]$, $\pm[2,3]\pm[1,4]$ and $\text{pt}$.
    \end{enumerate}
\end{enumerate}
\end{exa}

In the rest of this subsection, we explain the geometric origin of the orbits in $S_e$.
\begin{lem} \label{evidence}
    For each element $s\in S_e$, we have a tuple of weights $\alpha$ such that the variety $(\fix)_\alpha$ has a finite model containing an orbit corresponding to $s$.
\end{lem}
To prove this lemma, we first introduce the reduced forms of an element $s\in S_e$.
\begin{defin}
First, since the multiplication is commutative, any element $s\in S_e$ can be represented by a union of $[i,j]$, $\pm[i,j]$, and $[i,j]^\pm$ such that for arbitrary two intervals, they are disjoint or one contains the other. Let $w$ be such an expression of $s$. We count the number of pairs $([i_1, j_1], [i_2, j_2])$ in $w$ such that $[i_1, j_1]\subseteq [i_2, j_2]$. Let this number be $p(w)$. We define the reduced forms of $s$ as the expressions with the smallest $p(w)$. From \Cref{relation}, a reduced form $R(s)$ of $s$ has the following properties:
\begin{itemize}
\item Each $[i,j]$, $\pm[i,j]$ or $[i,j]^\pm$ appears at most once in $R(s)$. There is no $[i,i]$ in $R(s)$ unless $s= \text{pt}$.
\item $\pm[i,j]$ and $[i,j]^\pm$ do not appear at the same time in $R(s)$.
\item Any chain of inclusion $[i_1, j_1]\subset... \subset [i_l, j_l]$ in $R(s)$ has length less than or equal to $\frac{1}{2}$ $(j_l-i_l+1)$.
\end{itemize}

The first two properties easily follow from the first two properties in \Cref{relation}. The restriction on the lengths of the chains comes from the last property in \Cref{relation}. Indeed, any chain of inclusion $[i_1, j_1]\subset... \subset [i_l, j_l]$ with length larger than $\frac{1}{2}$ $(j_l-i_l+1)$ must contain two intervals with lengths' difference $1$ or $0$. Thus, we can perform the substitution mentioned in \Cref{last relation} or use the last property of \Cref{relation} to reduce the quantity $p$.
\end{defin}

Before going into the details of the proof of \Cref{evidence}, we briefly explain the ideas to the readers. First, an orbit $\bO$ given by a chain of inclusions of intervals can be found in the finite model of a tower of quadric bundles $\cQ_\bO$. For example, the generators $[i,j]$, $[i,j]^{\pm}$ (or $\pm[i,j]$) come from the variety $\OG(1,V_{[i,j]})$ (or $\OG(\frac{j-i+1}{2},V_{[i,j]})$). The product of two intervals comes from a quardric bundle over a quadric, and so on. Then we choose $\alpha$ so that $X_\alpha$ is a tower of projective bundles over $\cQ_\bO$.  
Next, when there are $l>1$ disjoint chains of intervals in $s$, we view $s$ as a product $\bO_1\times ...\times \bO_l$. We construct $\alpha_1, \alpha_2,...,\alpha_l$ so that the orbit $\bO_i$ belongs to the finite model of $X_{\alpha_i}$. The tuple $\alpha$ is then assembled from $\alpha_1, \alpha_2,...,\alpha_l$ in a certain order so that $X_\alpha\cong \prod_{i=1}^{l} X_{\alpha_i}$.
\begin{proof}
Consider an element $s$ in $S_e$ and write $s$ in a reduced form $R(s)$. We will construct the connected components with respect to $s$ by induction on the maximal length $l$ of the inclusion chains in $R(s)$. Note that this induction process does not depend on $e$. Consider an element $e$ with the associated partition $(2x_1> 2x_2> ...>2x_k)$. We write $[2w_1-1,...,2w_l-1]^h$ for the chain of weights obtained by adjoining $h$ chains of weights $(2w_1-1,...,2w_l-1)$. If there is no interval in $R(s)$, the corresponding orbit is a point. In this case, consider $\alpha= ([2x_1-1,...,1], [2x_2,...,1],..., [2x_k- 1,...,1])$, then the corresponding variety $X_\alpha$ is a single point.

Next, consider $s\in S_e$ which is represented by an interval $[i,j]$ (or $\pm [i,j], [i,j]^\pm$). In terms of the basis diagram (see \Cref{basis diagram}), our strategy is to first exhaust the boxes from the $i$-th row to the $j$-th row and then fill in the boxes from other rows. For $[i,j]$ and $[i,j]^\pm$, consider $\alpha'= ([2x_i-1,...,2x_{i+1}+ 1], [2x_{i+1}-1,...,2x_{i+2}+1]^2,..., [2x_{j-1}-1,..., 2x_j+ 1]^{j-i},[2x_{j}-1,...3,1]^{j-i},[-1,-3,...,1-2x_{j}])$. For $\pm[i,j]$, we change the last two chains to $[2x_{j}-1,...3,1]^{(j-i+1)/2},[-1,-3,...,1-2x_{j}]^{(j-i+1)/2}$. Then the variety $X_{\alpha'}$ is a tower of projective bundles over $\OG(1,V_{[i,j]})$ (or $\OG(\frac{j-i+1}{2}, V_{[i,j]})$, so we have a finite model of $X_{\alpha'}$ that contains the orbits $[i,j]$, $[i,j]^\pm$ (or $\pm [i,j])$. Next, we add more weights to $\alpha'$ to get $\alpha$ so that $X_\alpha\cong X_{\alpha'}$. This is done by adding the chain $([2x_{j+1}-1,...,1],...,[2x_{k}-1,...,1])$ to the beginning of $\alpha'$ and adding the chain $[2x_1-1,..., 1],...,[2x_{i-1}-1,...,1]$ to the end. The following is an example of this construction.
\begin{exa}
    Consider $\lambda= (8,6,4,2)$, we have $k=4$. We have the corresponding basis diagram as follows.
   
    $$\ytableausetup{centertableaux, boxsize= 3em}
\begin{ytableau}
\none & \none & \none &v_{1}^{4} &v_{-1}^{4} &\none   &\none &\none \\
\none & \none &v_{3}^{3} &v_{1}^{3} &v_{-1}^{3} &v_{-3}^{3}   & \none &\none\\
\none & v_{5}^{2} &v_{3}^{2} &v_{1}^{2} &v_{-1}^{2} &v_{-3}^{2}   &v_{-5}^{2} &\none\\
 v_{7}^{1} & v_{5}^{1} &v_{3}^{1} &v_{1}^{1} &v_{-1}^{1} &v_{-3}^{1}   &v_{-5}^{1} &v_{-7}^{1}\\
\end{ytableau}\\
$$
    \begin{enumerate}
        \item The orbit $s= [1,4]^\pm$ appears in the finite model $Y_\alpha$ of $(\fix)_\alpha$ for $\alpha= (7,5,5,3,3,3,1,1,1,-1)$. 
        \item The orbit $s= \pm [1,4]$ appears in the finite model $Y_\alpha$ of $(\fix)_\alpha$ for $\alpha= (7,5,5,3,3,3,1,1,-1,-1)$. 
        \item For the orbit $s= [2,4]$, we first disregard the first row of the diagram, and let $\alpha'= (5,3,3,1,1,-1)$. Adding the chain $(7,5,3,1)$ to the end of $\alpha'$, we have $\alpha= (5,3,3,1,1,-1,7,5,3,1)$. With this arrangement of weights, the spaces $W_1^{3}$ (resp. $W_3^{3}$, $W_5^{2}$) are uniquely determined by adding the vectors $v_1^{1}$ (resp. $v_3^{1}$, $v_5^{1}$) to the spaces $W_1^{2}$ (resp. $W_3^{2}$, $W_5^{1}$). Hence $X_\alpha\cong X_{\alpha'}$, and it is straightforward that $X_{\alpha'}\cong (\bP^1, \bP^1, \OG(1, \Span(\{v_1^{j}\}_{2\leqslant j\leqslant 4})))$.
        \item Similarly to the case $s= [2,4]$, for $s= [1,3]$, our construction adds the weight $1$ to the beginning, giving $\alpha= (1,7,5,5,3,3,3,1,1,-1)$. And we have $X_\alpha \cong X_{\alpha'}\cong (\bP^1, \bP^1, \bP^2, \bP^1, \OG(1,  \Span(\{v_1^{j}\}_{1\leqslant j\leqslant 3})))$.
        \item For the orbit $s= [2,3]$, we have $\alpha'= (5,3,3,1,-1)$. As described in \Cref{subsect 9.2}, we have $X_{\alpha'}\cong (\bP^1, \OG(1, \Span(v_1^{2}, v_1^{3})))$.
        Adding the weights of the first row and the last row, we have $\alpha= (1,5,3,3,1,-1,7,5,3,1)$. By distinguished reduction, we first have $X_\alpha\cong X_{\alpha''}$ where $\alpha''$ is obtained by removing the first weight $1$ in $\alpha$. An argument similar to the case $s=[2,4]$ gives us $X_{\alpha''}\cong X_{\alpha'}$.      
    \end{enumerate}

Our geometric description of $(\fix)_\alpha$ can be realized in the setting of \Cref{basic orbits} as follows. We identify the weight space $V_1$ with $V$. The subspaces $V_{[i,j]}$ are then realized as the span of the basis vectors $\{v_1^{h}\}_{i\leqslant h\leqslant j}$.
\end{exa}
We say that an interval in a reduced expression of $s$ is maximal if this interval is not contained in another interval. When we mention the maximal interval, we will omit the notation $\pm$ because we only focus on the length of the interval. We show that if $s$ has a reduced form that contains two or more disjoint maximal intervals, then we can assemble $\alpha$ from smaller tuples. 

For simplicity of notation, we explain the case $s= s_1 \times s_2$ where $s_1$ and $s_2$ have unique maximal intervals $[i_1, j_1]$ and $[i_2, j_2]$, the other cases are similar. Assume $j_1< i_2$, recall that we have the vector space $V_\lambda$ corresponds to $e$ and a basis $(v_i^{j})$ at the beginning of Section 8. We then have a direct sum decomposition $V_\lambda= V_{\lambda_1}\oplus V_{\lambda_2}$ where $V_{\lambda_1}= \Span(x_i^{j})_{j\leqslant j_1}$ and $V_{\lambda_2}= \Span(x_i^{j})_{j> j_1}$. In terms of basis diagrams, $V_{\lambda_1}$ is the span of the vectors in the first $j_1$ rows and $V_{\lambda_2}$ is the span of the vectors in the last $k-j_1$ rows. With these two subdiagrams, we obtain two corresponding nilpotent elements $e_l\in \fsp(V_{\lambda_l})$ for $l= 1, 2$. As $s_1$ and $s_2$ both have lengths smaller than $s$, by the induction hypothesis, we have two tuples of weights $\alpha_1$ and $\alpha_2$ such that $(\cB_{e_1})_{\alpha_1}$ and $(\cB_{e_1})_{\alpha_2}$ have finite models that contain the oribts $s_1$ and $s_2$. We take $\alpha= (\alpha_2, \alpha_1)$, then $(\fix)_\alpha\cong (\cB_{e_1})_{\alpha_1} \times (\cB_{e_2})_{\alpha_2}$. As a consequence, $(\fix)_\alpha$ admits a finite model that contains the orbit $s= s_1\times s_2$.

The above argument reduces our consideration to the case $s$ has a unique maximal interval $[i,j]$. Furthermore, if $i>1$ or $j< k$ we can think of $[i,j]$ as $[i,j][1,1]...[i-1,i-1][j+1, j+1]...[k,k]$, and the argument from the previous paragraph applies. Hence, it suffices to consider that $s$ has a unique maximal interval $[1,k]$. Taking away $[1,k]$, we consider the rest as an expression of $s'\in S_{e'}$ in which $e'$ has the basis diagram obtained by removing two middle columns from the basis diagram of $e$ (corresponding to the weights $1$ and $-1$). The induction hypothesis gives us a tuple of weights $\alpha'$ corresponding to the pair $(e', s')$. From $\alpha'$, we substitute all weights $2i+1$ by $2i+3$ if $i\geqslant 0$ and by $2i-1$ if $i< 0$. We then add a pair of weights $[1,-1]$ to the left of each weight $-3$, and add the chain of weights $[1^{w_1- 1- 2d_{-3}}, -1]$ to the end of $\alpha'$. Here $w_1= \dim V_1$ and $d_{-3}$ is the multiplicity of weight $-3$ in $\alpha$ (which is the multiplicity of weight $-1$ in $\alpha'$). This construction of $\alpha$ together with the map $X_\alpha \xrightarrow{F_1} X_{\alpha'}$ realizes $X_\alpha$ as a tower of projective bundles over a quadric bundle over $X_{\alpha'}$. Let $Y_{\alpha'}$ be a finite model of $X_{\alpha'}$ that contains the orbit $s'$. Then we see from \Cref{quadric fibration} that $X_\alpha$ has a finite model that contains $Y_{\alpha'}\times [i,j]$ (or $[i,j]^\pm$, $\pm[i,j]$) as a subset. The model $Y_{\alpha'}$ contains the orbit $s'$, so $X_\alpha$ has a finite model that contains the orbit $s$.
\end{proof}

For $e\in \fso_{2n}, \fso_{2n+1}$, we consider the case where $\lambda$ consists only of odd parts. If $\lambda$ has an odd number of parts (so $e\in \fso_{2n+1}$), the set $S_e$ is defined similarly. When $\lambda$ has an even number of parts (so $e\in \fso_{2n}$), we change the set $S_e$ to the set $\{s\in S_e, R(s) \text{ contains the interval } \pm[1,k] \text{ or } [1,k]^\pm\}$ (here $k$ is the number of parts in $\lambda$). 

Lastly, we state a conjecture about the geometry of $\fix$ itself. Let Con$(\fix)$ denote the set of connected components of $\fix$. This is an $A_e$-set. Forgetting the centrally extended structures of the orbits in $S_e$, we obtain a set $R_e$ of ordinary $A_e$-orbits. \Cref{evidence} has shown that all the $A_e$-orbits in $R_e$ appear in Con$(\fix)$. 
\begin{conj} \label{stab conj}
    The set $R_e$ gives the complete list (up to isomorphism) of $A_e$-orbits that appear in Con$(\fix)$.
\end{conj}
This conjecture is motivated by \cite[Section 3]{lusztig2024constructible}. In particular, \Cref{stab conj} for $e\in \fsp_{2n}$ with the associated partition $((2d)^2, (2d-2)^2,...,(2)^2)$ will imply the conjecture in \cite[Section 3]{lusztig2024constructible}.

\appendix
\section{Appendix} \label{appendix}
In A.1, we show a detailed analysis of three cases of $X_\alpha$ as representatives. These cases include almost all typical characteristics of varieties $X_\alpha$ for small $e$. For the base cases, the techniques are similar (and simpler).The reduction diagrams of the base cases are given in A.2.
\subsection{Some examples of $X_\alpha$ for $\lambda$ having $4$ parts}

Recall that in the definition of $\cC_4$ (see \Cref{collection C_4}), we use an abstract vector space $V$ with a basis $\{v_1,v_2,v_3,v_4\}$. In our examples, $V$ is realized as the weight space $V_{-1}$ with the basis $v_{-1}^1, v_{-1}^2,v_{-1}^3,v_{-1}^4$.
\begin{exa} \label{base case 1}
    Consider $\lambda= (8,6,6,4)$ and $\alpha= (5,3,7,5,5,3,1,1,-1,3,1,3)$. The basis diagram of $V_\lambda$ is given below.
    $$\ytableausetup{centertableaux, boxsize= 3em}
\begin{ytableau}
\none & \none&v_{3}^{4} &v_{1}^{4} &v_{-1}^{4} &v_{-3}^{4}   &\none &\none \\
\none & v_{5}^{3} &v_{3}^{3} &v_{1}^{3} &v_{-1}^{3} &v_{-3}^{3}   &v_{-5}^{3} &\none\\
\none & v_{5}^{2} &v_{3}^{2} &v_{1}^{2} &v_{-1}^{2} &v_{-3}^{2}   &v_{-5}^{2} &\none\\
 v_{7}^{1} & v_{5}^{1} &v_{3}^{1} &v_{1}^{1} &v_{-1}^{1} &v_{-3}^{1}   &v_{-5}^{1} &v_{-7}^{1}\\
\end{ytableau}\\
$$

    First, we have $X_\alpha \cong (\bP^1, X_{(5,3,7,\Box,5,3,1,1,-1,3,1,3)})$ since the constraint for $W_5^2$ is $W_5^{1}\subset W_5^{2} \subset W_5^{3}$. Now we consider the map $X_{(5,3,7,\Box,5,3,1,1,-1,3,1,3)} \xrightarrow{F_5^{1}} X_{(\Box,3,7,\Box,5,3,1,1,-1,3,1,3)}\cong X_{\alpha_5}$. This map forgets the space $W_5^{1}$, and the constraint to recover $W_5^{1}$ is that $e$ annihilates $W_5^{1}$ and $e^{-1}W_5^{1}\supset W_3^{1}$. In terms of the basis vectors of $V_\lambda$, the first condition is $e^{-1}W_5^{1}\subset \Span(v_3^{2}, v_3^{3}, v_3^{4})$. Now we are in the situation where $e^{-1}W_5^{1}\supset W_3^{1} \cup \bC v_3^{4}$, this fits the setting of \Cref{main blow up} for $V= \Span(v_3^{2}, v_3^{3}, v_3^{4})$. Similar to the explanation of Case 3.b after \Cref{intersection}, we have $F_{5}^{1}$ is a blow up map. The center of this blow-up lives in $X_{(\Box,3,7,\Box,5,3,1,1,-1,3,1,3)}$ by the determining condition $W_3^{1}= \bC v_3^{4}$, so this subvariety is isomorphic to $X_{(3,\Box,\Box,7,5,3,1,1,-1,3,1,3)}$. Next, we have the following chain of isomorphisms by distinguished reductions and $F_5$.
    $$X_{(3,\Box,\Box,7,5,3,1,1,-1,3,1,3)}\xrightarrow{d} X_{(\Box,\Box,7,5,3,1,1,-1,3,1,3)}\xrightarrow{F_5} X_{(5,3,1,1,-1,3,1,3)}\xrightarrow{d} X_{(3,1,1,-1,3,1,3)}$$ 
    Hence, the first reduction diagram is 
    $$X_\alpha \xrightarrow{\bP^1} X_{(5,3,7,\Box,5,3,1,1,-1,3,1,3)}\xrightarrow{\Bl(X_{(3,1,1,-1,3,1,3)})} X_{\alpha_5}.$$
    Next, we consider $X_{\alpha_5}$, here $\alpha_5= (3,5,3,1,1,-1,3,1,3)$. Consider the map $F_3: X_{\alpha_5}\rightarrow X_{(\alpha_5)_3}$, here $(\alpha_5)_3= (3,1,1,-1,1)$ and $X_{(\alpha_5)_3}= X_{(3,1,1,-1,1)}$. Using an argument similar to the previous paragraph, we see that $F_3$ is a blow-up along a subvariety of $X_{(3,1,1,-1,1)}$. This subvariety is determined by the condition $eW_{1}^{2}= 0$. So, we have $X_{\alpha_5}\cong \Bl_{X_{(1,1,-1,3,1)}}X_{(3,1,1,-1,1)}$.
    
    So far we have obtained the expression $X_\alpha\cong (\bP^1,\Bl_{X_{\alpha'}}(\Bl_{X_{\alpha''}} X_{(3,1,1,-1,1)})$ where $\alpha''= (1,1,-1,3,1)$ and $\alpha'= (3,1,1,-1,3,1,3)$. We show that the three varieties involved in this expression are in $\cC_4$, then it follows that $X_\alpha\in \cC_4$. For $X_{(3,1,1,-1,1)}$ and $X_{(1,1,-1,3,1)}$, their reduction diagrams are as follows.
    $$X_{(3,1,1,-1,1)}\cong X_{(1,1,-1,1)}\xrightarrow{\bP^1} X_{(\Box,1,-1,1)}\xrightarrow{\bP^1} X_{(\Box,\Box,-1,1)}\cong \OG(1,4)$$
    $$X_{(1,1,-1,3,1)}\xrightarrow{\bP^1} X_{(\Box,1,-1,3,1)}\xrightarrow{\bP^1} X_{(\Box,\Box,-1,3,1)}\cong \OG(1,\Span(v_{-1}^{2}, v_{-1}^{3}, v_{-1}^{4}))= OG(1,3)$$
    Here the variety $\OG(1,3)$ parametrizes the isotropic lines of $V_{-1}$ which are annihilated by $e^2$. The variety $\OG(1,\Span(v_{-1}^{2}, v_{-1}^{3}, v_{-1}^{4}))$ is one of the generators of $\cC_4$. Therefore, the two above diagrams prove that $X_{(3,1,1,-1,1)}$ and $X_{(1,1,-1,3,1)}$ are in $\cC_4$.

    Next, consider the variety $X_{\alpha'}= X_{(3,1,1,-1,3,1,3)}$. The condition $e^{-1}W_3^{1}= W_1^{2}$ is equivalent to that $W_{1}^{2}$ contains the vector $v_1^{4}$. And we need $W_{-1}^1$ to be the orthogonal complement of $e^{-1}W_1^{3}$, so $W_{-1}^{1}\subset (\bC v_{-1}^{4})^\perp$. In other words, $W_{-1}$ is an isotropic line in $\Span(v_{-1}^{1},v_{-1}^{2},v_{-1}^{3})$. We deduce that $X_{(3,1,1,-1,3,1,3)}\cong (\bP^1, \bP^1, \OG(1,\Span(v_{-1}^{2}, v_{-1}^{3}, v_{-1}^{4})))$. This variety belongs to $\cC_4$ since $\OG(1,\Span(v_{-1}^{2}, v_{-1}^{3}, v_{-1}^{4}))$ is one of the generators of $\cC_4$. This finishes the proof of the statement $X_\alpha\in \cC_4$ for $\alpha= (5,3,7,5,5,3,1,1,-1,3,1,3)$.
\end{exa}
The next two examples can be considered as a slightly modified version of the previous example. They demonstrate the phenomenon that the reduction maps for $X_\alpha$ when $d_{-1}= 2$ are similar to the case $d_{-1}= 1$ (and somewhat simpler).
\begin{exa} \label{base case 2}
    Consider $\lambda= (8,6,6,4)$, and $\alpha= (5,3,7,5,5,3,1,1,-1,-1,3,3)$. The basis diagram for this partition $\lambda$ is given in the previous example. The main feature of this example is that $W_1^{2}$ is identified with $eW_{-1}^2$, so it is an isotropic subspace of $V_1$. This constraint has the effect that some of the birational reduction maps will be an isomorphism instead of a blow-up map. The result is that $X_\alpha$ is a tower of projective bundles over $\OG(2,4)$. Details are presented below.
    
    First, we notice that the relative position of the weights $7,5,3$ is the same as in \Cref{base case 1}. Hence, by a similar argument, we have a description of $F_5$ as follows.
    $$X_\alpha \xrightarrow{\bP^1} X_{(5,3,7,\Box,5,3,1,1,-1,-1,3,3)}\xrightarrow{\Bl(X_{(3,1,1,-1,-1,3,3)})} X_{\alpha_5}$$
    It turns out that the center of the above blow-up, $X_{(3,1,1,-1,-1,3,3)}$, is empty for the following reason. The condition $W_1^{2}= e^{-1}W_3^{1}$ implies that $W_1^{2}$ contains $v_1^{4}$, so $W_{-1}^{2}= e^{-1} W_1^{2}$ contains $v_{-1}^{4}$, a non-isotropic vector. But the vectors of $W_{-1}^2$ are isotropic, so we have a contradiction, i.e. $X_{(3,1,1,-1,-1,3,3)}= \emptyset$. As a consequence, we have $X_\alpha= (\bP^1, X_{\alpha_5})$.
    
    Next, we have $\alpha_5= (3,5,3,1,1,-1,-1,3,3)$. In \Cref{base case 1}, for $\alpha_5= (3,5,3,1,1,3,1,3)$, both spaces $W_3^2$ and $W_3^3$ are determined, and $F_3$ was a blow-up. The difference in this example is that, when we consider $F_3$, the choice of $W_3^3$ yields the projection from a $\bP^1$-bundle. We now explain the details. Arguing as in \Cref{base case 1}, we then have $X_{\alpha_5}\cong (\bP^1, \Bl_{X_{(1,1,-1,-1,3)}}(X_{(3,1,1,-1,-1)})$. The center of this blow-up, $X_{(1,1,-1,-1,3)}$, is empty since we cannot have a two-dimensional isotropic subspace of $\Span(v_{-1}^{2},v_{-1}^{3},v_{-1}^{4})$. So we have $X_{\alpha_5}\cong (\bP^1, X_{(3,1,1,-1,-1)})\cong (\bP^1, \bP^1, \bP^1, \OG(2,4))$, the latter belongs to $\cC_4$.
    
\end{exa}
\begin{exa} \label{base case 3}
    We consider $\lambda= (10,8,8,6)$ and $\alpha= (7,5,9,7,7,5,3,3,1,-1,-3,5,3,1,-1,5)$. This tuple of weights has the property $\alpha_{1}= (5,3,7,5,5,3,1,1,-1,3,1,3)$, the latter was studied in \Cref{base case 1}. In other words, the relative position of the weights $9,7,5,3$ in $\alpha$ is the same as the relative position of $7,5,3,1$ in $\alpha_1$. We have described $F_5$ and $F_3$ for $\alpha_1$. By the same argument, we have the following reduction diagram
    $$X_\alpha\xrightarrow{\bP^1} X_{(7,5,9,\Box,7,5,3,3,1,-1,-3,5,3,1,-1,5)} \xrightarrow{\Bl(X_{5,3,3,1,-1,-3,5,3,1,-1,5})} X_{\alpha_7},$$
    and $X_{\alpha_7}= \Bl_{X_{(3,3,1,-1,-3,5,3,1,-1)}}X_{(5,3,3,1,-1,-3,3,1,-1)}$. It is straightforward that $X_{(5,3,3,1,-1,-3,3,1,-1)}$ $\cong X_{(3,3,1,-1,-3,3,1,-1)}$ $\cong (\bP^1,\bP^2,\bP^1, \bP^1,\OG(2,4))$. This variety is a tower of projective bundles over $\OG(2,4)$, so it belongs to $\cC_4$. We now describe the centers of the blow-ups.
    \begin{enumerate}
        \item For $X_{(3,3,1,-1,-3,5,3,1,-1)}$, we have that $W_{-3}^1$ is determined by $e^{-1}W_{-3}^1$ and $W_{1}^3= (eW_{-1})^\perp$. Therefore, $F_3$ is the map that forgets $W_{3}^{1}$ and $W_3^{2}$. We have 
    $$X_{(3,3,1,-1,-3,5,3,1,-1)}\xrightarrow{\bP^1} X_{(\Box,3,1,-1,\Box,5,\Box,1,-1)}\xrightarrow{F_3^{2}} X_{(1,-1,3,1,-1)}.$$
    We claim that the map $F_3^{2}$ is an isomorphism. One option is to realize the map as a blow-up along an empty variety (similar to \Cref{base case 2}). Here, we give a direct argument. The constraints to recover the subspace $W_3^{2}$ are that $W_3^{2}\subset W_3^{3}$ and that $W_3^{2}$ is annihilated by $e$. In terms of basis vectors, we have $W_3^{2}\subset W_3^{3}\cap \Span(v_3^{2},v_3^{3},v_3^{4})$. Next, we cannot have $W_3^{3}= \Span(v_3^{2},v_3^{3},v_3^{4})$ since $W_3^{3}$ is the orthogonal complement of an isotropic vector. Therefore, there is a unique way to recover $W_3^2$, and the map $F_3^{2}$ is an isomorphism. Consequently, we have $X_{(3,3,1,-1,-3,5,3,1,-1)}\cong (\bP^1, X_{(1,-1,3,1,-1)})$. The variety $X_{(1,-1,3,1,-1)}$ is isomorphic to $\OG(2,4)$, since after choosing $W_{-1}^2$, other spaces are determined uniquely. In particular, $W_{-1}^{1}= W_{-1}^2\cap \Span(v_{-1}^{2},v_{-1}^{3},v_{-1}^{4})$, and $W_{1}^{1}= eW_{-1}^{1}$, $W_{1}^{2}= eW_{-1}^{2}$.

        \item For $X_{(5,3,3,1,-1,-3,5,3,1,-1,5)}$, we have the condition $W_{3}^2= e^{-1}W_5^{1}$, so $v_3^{4}\in W_{3}^2$. A consequence is $v_3^{4}\in W_{3}^3$, which implies $W_{-3}^{1}\subset \Span(v_{-3}^{1},v_{-3}^{2},v_{-3}^{3})$. Then $W_{-1}^{1}= eW_{-3}^{1}\subset \Span(v_{-1}^{1},v_{-1}^{2},v_{-1}^{3})$, so $W_{-1}^{1}$ is determined by $W_{-1}^{2}\cap \Span(v_{-1}^{1},v_{-1}^{2},v_{-1}^{3})$. The choice of $W_3^{2}\subset W_3^{3}$ such that it contains $v_3^{4}$ gives us a $\bP^1$-bundle over $\OG(2,4)$. The choice of $W_3^{1}\subset W_3^{2}$ gives us another $\bP^1$-bundle. Next, note that $W_5^{1}$ and $W_5^{2}$ are determined by $eW_3^{2}$ and $eW_3^{3}$. In summary, we have a chain of isomorphisms
    $$X_{(5,3,3,1,-1,-3,5,3,1,-1,5)}\cong X_{(\Box,3,3,1,-1,-3,\Box,3,1,-1,5)}\cong (\bP^1, \bP^1, \OG(2,4)).$$
    \end{enumerate}
        
    Because the three varieties involved in the description of $X_\alpha$ belong to $\cC_4$, we get $X_\alpha\in \cC_4$.        
\end{exa}
\subsection{Reduction diagrams of base varieties}
Let $B= \{(6,6,6,4), (8,6,6,4), (8,8,6,4), (6,6,4,4), (8,6,4,4), (6,4,4,4)\}$. In this section, we give reduction diagrams for $X_\alpha$ in the following cases. 
\begin{enumerate}
    \item $\lambda\in B-2$, and $\alpha$ has only one weight $-1$.
    \item $\lambda\in B$, case V-1 and V-2 (see the proof of \Cref{geometric main 4 rows}).
    \item $\lambda\in B-2$, $\alpha$ contains weights $-1$ of multiplicities $2$.
\end{enumerate}
Our strategy is similar to the proof of \Cref{geometric main 4 rows}. For each $\lambda$, we only consider the tuples $\alpha$ that we cannot apply distinguished reduction maps (\Cref{distinguished reduction}) or swapping maps (\Cref{swap map}). The order in which we list these tuples is based on the relative positions of the weights $3$ and $1$.  
\subsubsection{Cases $\lambda\in B-2$ and $d_{-1}= 1$}
\textbf{1}. Case $\lambda= (4,2,2,2)$. 
\begin{enumerate}
    \item $\alpha= (1,3,1,1,-1)$ or $\alpha= (1,3,1,-1,1)$.
    $$X_{(1,3,1,1,-1)}\xrightarrow{\bP^1} X_{(1,3,\Box,1,-1)}\xrightarrow{\bP^1} X_{(\Box,3,\Box,1,-1)}=\OG(1,4)$$ $$X_{(1,3,1,-1,1)}\xrightarrow{\Bl(OG(1,3))} X_{(1,3,\Box,-1,1)}=(\bP^1,\OG(1,4))$$
    \item $\alpha= (1,-1,3,1,1)$.
    $$X_\alpha \xrightarrow{\bP^1} \OG(1,3)$$
    \item $\alpha= (1,1,3,1,-1)$.
    $$X_\alpha \xrightarrow{\bP^1} X_{(\Box,1,3,1,-1)}\cong \OG(1,4)$$
    \item $\alpha= (1,1,-1,3,1)$ or $\alpha= (1,-1,1,3,1)$.
    $$X_{(1,1,-1,3,1)}\xrightarrow{\bP^1} X_{(1,-1,1,3,1)}\cong \OG(1,3)$$
\end{enumerate}
\textbf{2}. Case $\lambda= (4,4,2,2)$. 
\begin{enumerate}
    \item $\alpha= (3,1,3,1,1,-1), (3,1,3,1,-1,1), (3,1,-1,3,1,1)$.
    $$X_{(3,1,-1,3,1,1)}\xrightarrow{\Bl((\bP^1, \OG(1,2))} X_{(\Box,1,-1,3,1,1)}\cong (\bP^1, \OG(1,4))$$
    $$X_{(3,1,3,1,-1,1)}\xrightarrow{\Bl(X_{(1,\Box,3,1,-1,1)})} X_{(\Box,1,3,1,-1,1)}\cong (\bP^1,\bP^1, \OG(1,4))$$
    $$X_{(3,1,3,1,1,-1)}\xrightarrow{\Bl(X_{(1,\Box,3,1,1,-1)})} X_{(\Box,1,3,1,1,-1)}\cong (\bP^2,\bP^1, \OG(1,4))$$
    \item $\alpha= (1,3,3,1,1,-1), (1,3,3,1,-1,1)$.
    $$X_{(1,3,3,1,1,-1)}\xrightarrow{\bP^1} X_{(1,\Box,3,1,1,-1)}\cong \Bl_{(\bP^1, \OG(1,2))} (\bP^1, \OG(1,4))$$
    $$X_{(1,3,3,1,-1,1)}\xrightarrow{\bP^1} X_{(1,\Box,3,1,-1,1)}\cong \Bl_{(\OG(1,2)\cup \OG(1,2))} \OG(1,4)$$
    \item $\alpha= (1,-1,3,3,1,1)$.
    $$X_\alpha \xrightarrow{\bP^1} X_{(\Box,-1,\Box,3,1,1)} \xrightarrow{\bP^1} X_{(\Box,-1,\Box,3,\Box,1)}\cong \OG(1,2)$$
    \item $\alpha= (1,3,1,1,-1,3)$, $(1,3,1,-1,1,3)$. Then $W_{1}^{3}$ must contain $v_{1}^{3}$ and $v_{1}^{4}$, so $W_{-1}^1\subset $ Span$(v_{-1}^{1}, v_{-1}^{2})$. Therefore, we have the followings.
    $$X_{(1,3,1,-1,1,3)}\cong (\bP^1, \OG(1,2))$$  
    $$X_{(1,3,1,1,-1,3)}\cong (\bP^1, \bP^1, \OG(1,2))$$
    \item $\alpha= (1,3,1,-1,3,1)$, $(1,3,1,3,1,-1)$. 
    $$X_{(1,3,1,-1,3,1)}\cong \Bl_{\OG(1,2)}\OG(1,4)$$
    $$X_{(1,3,1,3,1,-1)}\xrightarrow{\Bl(\bP^1\cup \bP^1)} X_{(\Box,3,1,3,1,-1)}\xrightarrow{\Bl(\bP^1\cup \bP^1)} X_{(\Box,3,\Box,3,1,-1)}\cong \bP^1\times \OG(1,4)$$
    \item $\alpha= (1,-1,3,1,3,1)$, $X_\alpha\cong X_{(\Box,-1,3,\Box,3,\Box)}\cong \OG(1,2)\times \bP^1$.
\end{enumerate}
\textbf{3}. Case $\lambda= (4,4,4,2)$. First, if $\alpha$ starts with $1$, then $X_\alpha$ is reducible with respect to distinguished reduction map (see \Cref{distinguished reduction}). Next, if $\alpha$ has two consecutive weights $1$ after the first weight $3$, we can use the swapping map in \Cref{swap map} to reduce to case $\lambda= (4,4,4)$. In the following, we consider other cases.
\begin{enumerate}
    \item $\alpha= (3,1,-1,1,3,1,3)$. Then $W_{1}^{3}= e^{-1}W_{3}^{2}$, containing $v_{1}^{4}$. Therefore, $W_{-1}^{1}$ $\subset$ Span$(v_{-1}^{1},v_{-1}^{2},v_{-1}^{3})$. And we have
    $$X_\alpha \cong X_{(\Box,1,-1,1,\Box,1,3)}\cong (\bP^1,\OG(1,3)).$$
    \item $\alpha= (3,3,1,1,1,-1,3)$, $(3,3,1,1,-1,1,3)$, $(3,3,1,-1,1,1,3)$. Similarly to the previous case, we have $W_{-1}^{1}$ $\subset$ Span$(v_{-1}^{1},v_{-1}^{2},v_{-1}^{3})$. 
    $$X_{(3,3,1,1,1,-1,3)} \xrightarrow{(\bP^2, \bP^1)} X_{(3,3,\Box,\Box,1,-1,3)}\xrightarrow{\bP^1} X_{(\Box,3,\Box,\Box,1,-1,3)}\cong \OG(1,3)$$
    $$X_{(3,3,1,1,-1,1,3)} \xrightarrow{(\bP^1, \bP^1)} X_{(3,3,\Box,\Box,-1,1,3)}\xrightarrow{\bP^1} X_{(\Box,3,\Box,\Box,-1,1,3)}\cong \OG(1,3)$$
    $$X_{(3,3,1,-1,1,1,3)} \xrightarrow{\bP^1} X_{(3,3,\Box,-1,\Box,1,3)}\xrightarrow{\bP^1} X_{(\Box,3,\Box,-1,\Box,1,3)}\cong \OG(1,3)$$
    \item $\alpha= (3,1,-1,1,3,3,1)$. Similarly to the previous case, we have $W_{-1}^{1}$ $\subset$ Span$(v_{-1}^{1},v_{-1}^{2},v_{-1}^{3})$. So we have
    $$X_\alpha\xrightarrow{(\bP^1,\bP^1)}  X_{(\Box,\Box,-1,\Box,\Box,3,1)}\cong \OG(1,3).$$
    \item $\alpha= (3,1,3,1,1,-1,3)$, $(3,1,3,1,-1,1,3)$, $(3,1,-1,3,1,1,3)$. Similarly to the previous case, we have $W_{-1}^{1}$ $\subset$ Span$(v_{-1}^{1},v_{-1}^{2},v_{-1}^{3})$. 
    $$X_{(3,1,3,1,1,-1,3)}\xrightarrow{\Bl(X_{(\Box,3,1,1,-1,3)})} X_{(\Box,1,3,1,1,-1,3)}\xrightarrow{(\bP^2, \bP^1) } X_{(\Box,\Box,\Box,\Box,1,-1,3)}\cong (\OG(1,3))$$
    $$X_{(3,1,3,1,-1,1,3)} \xrightarrow{\Bl(X_{(\Box,3,1,-1,1,3)})} X_{(\Box,1,3,1,-1,1,3)} \xrightarrow{(\bP^1, \bP^1)} \OG(1,3)$$
    $$X_{(3,1,-1,3,1,1,3)} \cong X_{(\Box,\Box,-1,3,1,1,3)}\xrightarrow{\bP^1} \OG(1,3)$$
    \item $\alpha= (3,1,3,3,1,1,-1)$, $(3,1,3,3,1,-1,1)$, $(3,1,-1,3,3,1,1)$. 
    $$X_{(3,1,3,3,1,1,-1)}\xrightarrow{\bP^1}X_{(3,1,\Box,3,1,1,-1)}\xrightarrow{\Bl(X_{(\Box,\Box,3,1,1,-1)})} X_{(\Box,1,\Box,3,1,1,-1)}\xrightarrow{(\bP^2, \bP^1)}\OG(1,4)$$
    $$X_{(3,1,3,3,1,-1,1)}\xrightarrow{\bP^1} X_{(3,1,\Box,3,1,-1,1)}\xrightarrow{\Bl(X_{(\Box,\Box,3,1,-1,1)})} X_{(\Box,1,\Box,3,1,-1,1)} \xrightarrow{(\bP^1, \bP^1)}\OG(1,4)$$
    $$X_{(3,1,-1,3,3,1,1)} \xrightarrow{(\bP^1, \bP^1)} X_{(3,1,-1,\Box,3,\Box,1)} \cong \OG(1,4)$$    
    \item $\alpha= (3,3,1,3,1,1,-1)$, $(3,3,1,3,1,-1,1)$, $(3,3,1,-1,3,1,1)$.

    $$X_{(3,3,1,3,1,1,-1)}\xrightarrow{\Bl(X_{(3,3,\Box,1,3,1,-1)})} X_{(3,3,\Box,3,1,1,-1)}\xrightarrow{(\bP^2, \bP^1, \bP^2)} \OG(1,4)$$
    
    $$X_{(3,3,1,3,1,-1,1)}\xrightarrow{\Bl(X_{(3,3,\Box,1,3,-1,1)})} X_{(3,3,\Box,3,1,-1,1)}\xrightarrow{(\bP^1, \bP^1, \bP^2)} \OG(1,4)$$
    
    $$X_{(3,3,1,-1,3,1,1)}\xrightarrow{(\bP^1,\bP^1)} X_{(\Box,\Box,1,-1,3,1,1)}\xrightarrow{\bP^1} \OG(1,4)$$    
    \item $\alpha= (3,3,1,1,3,1,-1)$, $(3,3,1,1,-1,3,1)$, $(3,3,1,-1,3,1,1)$. 

    $$X_{(3,3,1,1,3,1,-1)}\xrightarrow{\bP^1} X_{(3,3,\Box,1,3,1,-1)}\xrightarrow{\Bl(X_{(3,3,\Box,\Box,1,3,-1)}\cong (\bP^1,\OG(1,3))} X_{(3,3,\Box,\Box,3,1,-1)} \xrightarrow{(\bP^1, \bP^2)}\OG(1,4) $$

    $$X_{(3,3,1,1,-1,3,1)}\xrightarrow{\bP^1} X_{(\Box,3,1,1,-1,3,1)} \xrightarrow{\Bl((\bP^1, \OG(1,3)))}  X_{(\Box,\Box,1,1,-1,3,1)}\cong (\bP^1, \bP^1, \OG(1,4))$$

    $$X_{(3,3,1,-1,3,1,1)}\xrightarrow{(\bP^1, \bP^1)} X_{(\Box,\Box,1,-1,3,1,1)}\xrightarrow{} (\bP^1, \OG(1,4))$$
    
    \item $\alpha=$ $(3,1,3,1,3,1,-1)$, $(3,1,3,1,-1,3,1)$, $(3,1,-1,3,1,3,1)$. 

    $$X_{(3,1,3,1,3,1,-1)}\xrightarrow{\Bl(X_{(1, \Box,3,1,3,1,-1)})} X_{(\Box,1,3,1,3,1,-1)} \xrightarrow{\Bl((\bP^1, \bP^1,\OG(1,3))} X_{(\Box,1,\Box,1,3,1,-1)}\cong (\bP^2, \bP^1, \OG(1,4))$$

    $$X_{(3,1,3,1,-1,3,1)} \xrightarrow{\Bl(X_{(1,\Box,3,1,-1,3,1)})} X_{(\Box,1,3,1,-1,3,1)} \xrightarrow{\Bl(\bP^1, \OG(1,3))} X_{(\Box,1,\Box,1,-1,3,1)} \cong (\bP^1, \bP^1, \OG(1,4))$$

    $$X_{(3,1,-1,3,1,3,1)} \xrightarrow{\Bl(X_{(3,1,-1,3,\Box,1,3)})} X_{(3,1,-1,3,\Box,3,1)} \cong (\bP^1, \OG(1,4))$$
\end{enumerate}
\textbf{4}. Case $\lambda= (6,4,2,2)$. We only need to consider the relative position $(3,5,3)$ of weights $5$ and $3$. If there is no weight $1$ in front of the first weight $3$, $\alpha$ is reducible by distinguished reduction. If there are two weights $1$ in front of the first weight $3$, then $W_1^{2}=$ Span$(v_1^{3}, v_1^{4})$. Hence $W_{-1}^{1}\in \OG(1,$Span$(v_{-1}^{1}, v_{-1}^{2})$), and the descriptions of $X_\alpha$ easily follow. Below, we consider the cases where there is precisely one weight $1$ before the first weight $3$.
\begin{enumerate}
    \item $\alpha=$ $(1,3,5,3,1,1,-1)$, $(1,3,5,3,1,-1,1)$, $(1,-1,3,5,3,1,1)$. We have 
    $X_{(1,3,5,3,1,1,-1)} \cong$ $ X_{(1,\Box,3,1,1,-1)}$ where the latter variety has been described in the case $\lambda= (4,4,2,2)$. The other two cases are similar.
    \item $\alpha=$ $(1,3,1,5,3,1,-1)$. Then $X_{(1,3,1,5,3,1,-1)}\cong \Bl_{\OG(1,2)}(\OG(1,4))$. 
    \item $\alpha=$  $(1,3,1,-1,5,3,1)$. Then $X_{(1,3,1,-1,5,3,1)}\cong X_{(1,3,1,-1)}$.
    \item $\alpha=$ $(1,-1,3,1,5,3,1)$, $(1,-1,3,1,5,3,1)$. We have $X_{(1,-1,3,1,5,3,1)}\cong \OG(1,2)\times X_{(3,1,5,3,1)}$ and $X_{(1,-1,3,1,5,3,1)}\cong \OG(1,2)\times X_{(3,1,5,3,1)}$. The variety $X_{(3,1,5,3,1)}$ is a tower of projective bundle by \Cref{no negative}.
\end{enumerate}
\textbf{5}. Case $\lambda= (6,4,4,2)$.
If the first weight of $\alpha$ is $5$, then $X_\alpha$ is reducible by distinguished reduction. If there are two weights $3$ in front of the first weight $3$, then $W_3^{2}=$Span$(v_{3}^{2}, v_3^{3})$, the descriptions of $X_\alpha$ are straightforward. In the following, we consider the relative position $(3,5,3,3)$ of the weights $5$ and $3$. And if there are two consecutive weights $1$ after the first weight $3$, $X_\alpha$ is reducible by the swapping map in \Cref{swap map}. Other cases are given below.
\begin{enumerate}
    
    \item $\alpha=$ $(3,5,3,1,1,1,-1,3)$, $(3,5,3,1,1,-1,1,3)$, $(3,5,3,1,-1,1,1,3)$. In these cases, $W_1^{3}= e^{-1}W_3^{2}$, the latter contains the vector $v_1^{4}$. Hence $W_{-1}^1$ is a subspace of Span$(v_{-1}^{1},v_{-1}^{2},v_{-1}^{3})$. Once $W_{-1}^1$ is determined, we have $W_{1}^3$= $(eW_{-1}^1)^\perp$, $W_3^{2}= eW_{1}^3$, and $W_{3}^1= W_3^{2}\cap $ Span$(v_3^{2}, v_3^{3})$. Therefore, these varieties $X_\alpha$ are towers of projective bundles over $\OG(1,3)$.
    \item $\alpha= (3,1,5,3,1,1,-1,3)$, $(3,1,5,3,1,-1,1,3)$, $(3,1,-1,5,3,1,1,3)$. Similarly to the previous case, $W_{-1}^1\subset $ Span$(v_{-1}^{1},v_{-1}^{2},v_{-1}^{3})$.
    $$X_{(3,1,5,3,1,1,-1,3)}\xrightarrow{\Bl(X_{(3,1,1,-1)})} X_{(\Box,1,5,\Box,1,1,-1,3)}\cong (\bP^1, \bP^1, \OG(1,3))$$
    $$X_{(3,1,5,3,1,-1,1,3)}\xrightarrow{\Bl(X_{(3,1,-1,1)})} X_{(\Box,1,5,\Box,1,-1,1,3)}\cong X_{(1,3,1,-1,1)}$$

    $$X_{(3,1,-1,5,3,1,1,3)}\xrightarrow{\bP^1} X_{(3,1,-1,5,\Box, \Box,1,3)}\cong \OG(1,3)$$
    \item $\alpha= (3,1,5,3,3,1,1,-1)$, $(3,1,5,3,3,1,-1,1)$, $(3,1,-1,5,3,3,1,1)$. 
    $$X_{(3,1,5,3,3,1,1,-1)}\xrightarrow{\bP^1} X_{(3,1,5,\Box,3,1,1,-1)}\xrightarrow{\Bl(X_{(3,1,1,-1)})} X_{(\Box,1,5,\Box,3,1,1,-1)} \cong X_{(1,3,1,1,-1)}$$
    
    $$X_{(3,1,-1,5,3,3,1,1)} \xrightarrow{(\bP^1, \bP^1)} X_{(3,1,-1,5,\Box,3,\Box,1)} \cong \OG(1,3)$$
    \item $\alpha= (3,5,3,1,3,1,1,-1)$, $(3,5,3,1,3,1,-1,1)$, $(3,5,3,1,-1,3,1,1)$.
    $$X_{(3,5,3,1,3,1,1,-1)}\xrightarrow{\Bl(X_{(3,1,5,\Box,3,1,1,-1)})} X_{(3,5,\Box,1,3,1,1,-1)}\xrightarrow{\bP^1}(\bP^2, \bP^1, \OG(1,4))$$

    $$X_{(3,5,3,1,3,1,-1,1)}\xrightarrow{\Bl(X_{(3,1,5,\Box,3,1,-1,1)})} X_{(3,5,\Box,1,3,1,-1,1)} \xrightarrow{\bP^1}(\bP^1, \bP^1, \OG(1,4))$$

    $$X_{(3,5,3,1,-1,3,1,1)}\xrightarrow{\Bl(X_{(3,1,-1,5,\Box,3,1,1)})} X_{(3,5,\Box,1,-1,3,1,1)}\xrightarrow{\bP^1}(\bP^1, \OG(1,4))$$
    
    \item $\alpha= (3,5,3,1,1,3,1,-1)$, $(3,5,3,1,1,-1,3,1)$, $(3,5,3,1,-1,3,1,1)$.
    
    $$X_{(3,5,3,1,1,3,1,-1)}\xrightarrow{\bP^1} X_{(3,5,3,\Box,1,3,1,-1)}\xrightarrow{\Bl(X_{(3,5,3,\Box,\Box,1,-1,3)})} X_{(3,5,3,\Box,\Box,3,1,-1)}\cong (\bP^1, \bP^1, \OG(1,4))$$

    $$X_{(3,5,3,1,1,-1,3,1)}\xrightarrow{\Bl(X_{(\Box,3,5,1,1,-1,3,1)})} X_{(\Box,5,3,1,1,-1,3,1)}\xrightarrow{\Bl((\bP^1, \OG(1,3))} X_{(\Box,5,\Box,1,1,-1,3,1)}\cong X_{(3,1,1,-1,1)} $$

    $$X_{(3,5,3,1,-1,3,1,1)}\xrightarrow{\Bl(X_{(3,1,-1,5,\Box,3,1,1)})} X_{(3,5,\Box,1,-1,3,1,1)}\xrightarrow{\bP^1}  X_{(\Box,5,\Box,1,-1,3,1,1)}\cong X_{(3,1,-1,1,1)}$$
    
    \item $\alpha= (3,1,5,3,1,3,1,-1)$, $(3,1,5,3,1,-1,3,1)$, $(3,1,-1,5,3,1,3,1)$. Similarly to the previous case, these $X_\alpha$ are reducible with respect to $F_3$. 
\end{enumerate}

\textbf{6}. $\lambda= (6,6,4,2)$. If there are two weights $3$ between the two weights $5$, then we can use a swapping map. Hence, we only need to consider the relative position of weight $5$ and weight $3$ is $(5,3,5,3,3)$. In this case, we claim that $X_\alpha$ is reducible using $F_5$. Note that $F_5$ is the map that forgets $W_5^{1}$, and the condition to recover this space is $e^{-1}W_5^{1}\supset W_3^{1}$. In other words, $e^{-1}W_5^{1}$ is a 2-dimensional space containing $\bC v_{3}^{3}$ and $W_3^{1}$, which fits the context of \Cref{main blow up}.

\subsubsection{$\lambda\in B$, case V-1 and V-2}

First, we consider the case where there is no weight $7$. We claim that in these cases, the varieties $X_\alpha$ are reducible with respect to $F_5$. Recall that we are considering the cases with the beginning V. Therefore, the relative positions of the weights $5$ and $3$ are $(5,3,5,5,3,3,3)$. Similarly to the proof of \Cref{3 row main}, we only need to prove that $X_{(3,1,3,3,1,1,1)}$ is reducible with respect to $F_3$. Indeed, we have
$$X_{(3,1,3,3,1,1,1)}\xrightarrow{\bP^1} X_{(3,1,\Box,3,1,1,1)} \xrightarrow{\Bl(X_{(1,1,1)})} X_{(\Box,1,\Box,3,1,1,1)}\cong X_{(1,1,1,1)}.$$
Next, we consider the cases where there is weight $7$ in $V_\lambda$. Similarly to the previous case, to prove that $X_\alpha$ is reducible with respect to $F_5$, we prove that the following varieties are reducible with respect to $F_3$: $X_{(5,3,1,5,3,3,1,1,1)}$ and $X_{(5,3,1,3,5,3,1,1,1)}$. 
$$X_{(5,3,1,5,3,3,1,1,1)}\xrightarrow{\bP^1} X_{(5,3,1,5,\Box,3,1,1,1)}\xrightarrow{\Bl(X_{(3,3,1,1,1)})} X_{(5,\Box,1,5,\Box,3,1,1,1)}\cong X_{(3,1,3,1,1,1)}$$

$$X_{(5,3,1,3,5,3,1,1,1)}\xrightarrow{\Bl(X_{(3,3,1,1,1)})} X_{(5,\Box,1,3,5,3,1,1,1)}\cong X_{(3,1,3,1,1,1)}$$

\subsubsection{Cases $\alpha$ contains two weights $-1$}
There are two relative position of weights $1$ and $-1$: $(1,1,-1,-1)$ and $(1,-1,1,-1)$. In the following, we consider the relative position $(1,-1,1,-1)$, the other case can be dealt with similarly.\\
\textbf{1}. $\lambda= (4,2,2,2)$. We have $X_{(1,-1,3,1,-1)}\cong  \OG(2,4)$.\\
\textbf{2}. $\lambda= (4,4,2,2)$.
\begin{enumerate}
    \item $\alpha= (3,1,-1,3,1,-1)$, $X_\alpha\cong $ $\Bl_{\OG(1,2)\times \OG(1,2)} (\bP^1,\OG(2,4))$.
    \item $\alpha= (3,1,-1,1,-1,3)$, $X_\alpha \cong \OG(1,2)\times \OG(1,2)$.
\end{enumerate}
\textbf{3}. $\lambda= (6,4,2,2)$, $X_{(1,-1,3,5,3,1,-1)}\cong \OG(1,2)\times \OG(1,2)$.

Next, in the following cases, the dimension of $V_3$ is $3$. As we need $W_{1}^{2}$ to be an isotropic subspace of $V_1$, the map $e:W_{1}^{2}\rightarrow V_3$ is injective. This feature makes the description of $X_\alpha$ much simpler than in the cases where the weight $-1$ has multiplicity one in $\alpha$.\\
\textbf{4}. $\lambda= (4,4,4,2)$. 
\begin{enumerate}
    \item $\alpha= (3,1,-1,3,1,-1,3)$, $X_\alpha\cong X_{(1,-1,1,-1)}\cong (\bP^1,\OG(2,4))$.
    
    \item $\alpha= (3,3,1,-1,1,-1,3)$, $X_\alpha \cong (\bP^1, X_{(1,-1,1,-1)})\cong (\bP^1,\bP^1,\OG(2,4))$.
    \item $\alpha= (3,3,1,-1,3,1,-1)$, $X_\alpha\cong (\bP^1, \bP^1, X_{(1,-1,1,-1)})$.
\end{enumerate}
\textbf{5}. $\lambda= (6,4,4,2)$. We only need to consider the case where the relative position of $5$ and $3$ is $(3,5,3,3)$.
\begin{enumerate}
    \item $\alpha= (3,1,-1,5,3,1,-1,3)$, $X_\alpha \cong X_{\alpha_3}$.
    \item $\alpha= (3,5,3,1,-1,1,-1,3)$.
    $$X_\alpha \cong X_{(\Box,5,3,1,-1,1,-1,3)} \cong X_{(3,1,-1,1,-1)}$$
    \item $\alpha= (3,5,3,1,-1,3,1,-1)$.
    $$X_\alpha \xrightarrow{\Bl(X_{(\Box,3,5,1,-1,3,1,-1)})} X_{(\Box,5,3,1,-1,3,1,-1)}\xrightarrow{\bP^1} X_{\alpha_3}$$
\end{enumerate}
\textbf{6}. $\lambda= (6,6,4,2)$. 
\begin{enumerate}
    \item $\alpha= (5,3,1,-1,5,3,1,-1,3)$, $(5,3,1,-1,3,1,-1,5,3)$. In these cases, $F_3$ is an isomorphism, $X_\alpha\cong X_{\alpha_3}$.
    \item $\alpha= (5,3,5,3,1,-1,1,-1,3)$, $(5,3,3,1,-1,1,-1,5,3)$. 
    $$X_{(5,3,5,3,1,-1,1,-1,3)} \xrightarrow{\Bl( X_{(3,5,3,5,1,-1,1,-1,3)})} X_{(5,\Box,5,3,1,-1,1,-1,3)}= X_{\alpha_3}$$
    $$ X_{(5,3,3,1,-1,1,-1,5,3)} \xrightarrow{\bP^1} X_{(3,5,3,1,-1,1,-1,5,3)} $$
    \item $\alpha= (5,3,5,3,1,-1,3,1,-1)$, $(5,3,3,1,-1,5,3,1,-1)$. 
    $$X_{(5,3,5,3,1,-1,3,1,-1)}\xrightarrow{\Bl(X_{(3,5,3,5,1,-1,3,1,-1)})} X_{(5,\Box,5,3,1,-1,3,1,-1)}\xrightarrow{\bP^1} X_{(5,\Box,5,\Box,1,-1,3,1,-1)}= X_{\alpha_3}$$
    $$X_{(5,3,3,1,-1,5,3,1,-1)}\xrightarrow{\bP^1} X_{(5,3,5,3,1,-1,3,1,-1)}$$
\end{enumerate}

\printbibliography
\end{document}